%% file: sps_for_orbifolds_arbitrary_weights.tex
\documentclass[11pt]{amsart}
\usepackage{amscd,amsfonts,amssymb,amsmath,amsthm,latexsym}

\usepackage{graphicx}
\usepackage{graphics}
\usepackage{epstopdf}
\usepackage{mathrsfs}

\usepackage{geometry}
\usepackage{mathtools}
\usepackage{enumerate}
\usepackage{float}
\usepackage{color}
\usepackage{url}
\usepackage{array}
\usepackage{multirow}
\usepackage{bigstrut}
\usepackage{makecell}

\usepackage{tikz}
\usetikzlibrary{calc}
\usetikzlibrary{decorations.markings}

\usepackage[all]{xy}
\xyoption{curve}
\xyoption{import}
\xyoption{arc}
\xyoption{ps}

\input{00_commands.tex}

\title[SPs arising from surfaces with orbifold points, part II]{Species with potential arising from surfaces with orbifold points of order 2, Part II: arbitrary weights}


\author{Jan Geuenich}
\address{Mathematisches Institut, Universit\"at Bonn, Germany}
\email{jagek@math.uni-bonn.de}

\author{Daniel Labardini-Fragoso}
\address{Instituto de Matem\'aticas, Universidad Nacional Aut\'onoma de M\'exico}
\email{labardini@matem.unam.mx}

\date{\today}

\begin{document}

  \input{00_abstract.tex}

  \maketitle
  \tableofcontents

  \input{00_intro.tex}
  \input{00_acknowledgments.tex}
  \input{01_surfaces_and_triangs.tex}

  \input{02_weighted_quiver.tex}
  \input{03_chain_complexes.tex}
  \input{04_colored_triangs_and_flips.tex}
  \input{05_species.tex}
  \input{06_potentials.tex}

  \input{07_flips_vs_mutation.tex}
  \input{08_relation_to_surface_homology.tex}
  \input{09_flip_graph.tex}
  \input{10_cohomology_and_jacobian_algs.tex}

  \input{11_uniqueness_of_sps.tex}

  \input{14_problems.tex}
  \input{12_examples.tex}

  \input{13_technical.tex}

\input{00_references.tex}
\end{document}

%% file: 00_commands.tex
\numberwithin{equation}{section}

\theoremstyle{plain}
    \newtheorem{thm}{Theorem}[section]
    \newtheorem{lemma}[thm]{Lemma}
    \newtheorem{coro}[thm]{Corollary}
    \newtheorem{prop}[thm]{Proposition}
    \newtheorem{problem}[thm]{Problem}
    \newtheorem{conjecture}[thm]{Conjecture}
    \newtheorem{question}[thm]{Question}
\theoremstyle{definition}
    \newtheorem{defi}[thm]{Definition}
    \newtheorem{ex}[thm]{Example}
    \newtheorem{remark}[thm]{Remark}
\theoremstyle{remark}
    \newtheorem{case}{Case}

\newcommand{\suchthat}{\ | \ }

\newcommand{\N}{\mathbb{N}}
\newcommand{\Z}{\mathbb{Z}}

\newcommand{\F}{\mathbb{F}}
\newcommand{\R}{\mathbb{R}}
\newcommand{\C}{\mathbb{C}}

\newcommand{\B}{\mathcal{B}}

\newcommand{\Hom}{\operatorname{Hom}}

\newcommand{\Gal}{\operatorname{Gal}}

\newcommand{\marked}{\mathbb{M}}
\newcommand{\orb}{\mathbb{O}}

\newcommand{\surf}{(\Sigma,\marked,\orb)}
\newcommand{\overlineSigma}{\widetilde{\Sigma}}
\newcommand{\SSigma}{{\mathbf{\Sigma}}}
\newcommand{\SSigmaw}{{\mathbf{\Sigma}_\omega}}
\newcommand{\genus}{\mathbf{g}}

\newcommand{\dtuple}{\mathbf{d}}

\newcommand{\dtauwi}[1][i]{{d(\tauw)_{#1}}}

\newcommand{\tauw}[1][\tau]{{#1, \omega}}

\newcommand{\Qtauomega}{Q(\tauw)}

\newcommand{\dtauomega}[1]{d(\tauw)_{#1}}

\newcommand{\myhat}{\widehat}

\newcommand{\Xhat}{\myhat{X}}
\newcommand{\Yhat}{\myhat{Y}}
\newcommand{\Sigmahat}{\myhat{\Sigma}}

\newcommand{\xihat}[1][\xi]{\myhat{#1}}

\newcommand{\Ctauw}[1][\tau]{{C_\bullet(\tauw[#1])}}
\newcommand{\Ctauwhat}[1][\tau]{{\myhat{C}_\bullet(\tauw[#1])}}

\newcommand{\CtauwFtwo}[1][\tau]{{C_\bullet(\tauw[#1])}}
\newcommand{\CtauwhatFtwo}[1][\tau]{{\myhat{C}_\bullet(\tauw[#1])}}

\newcommand{\CCtauwFtwo}[1][\tau]{{C^\bullet(\tauw[#1])}}
\newcommand{\CCtauwhatFtwo}[1][\tau]{{\myhat{C}^\bullet(\tauw[#1])}}

\newcommand{\Htauw}[1][\tau]{{H_\bullet(C_\bullet(\tauw[#1]))}}
\newcommand{\Htauwhat}[1][\tau]{{H_\bullet(\myhat{C}_\bullet(\tauw[#1]))}}

\newcommand{\CHtauwFtwo}[1][\tau]{{H^\bullet(C^\bullet(\tauw[#1]))}}
\newcommand{\CHtauwhatFtwo}[1][\tau]{{H^\bullet(\myhat{C}^\bullet(\tauw[#1]))}}

\newcommand{\CZonetauw}[1][\tau]{{Z^1(\tauw[#1])}}

\newcommand{\CZonetauwFtwo}[1][\tau]{{Z^1(\tauw[#1])}}
\newcommand{\CZonetauwhatFtwo}[1][\tau]{{\myhat{Z}^1(\tauw[#1])}}

\newcommand{\ConetauwFtwo}[1][\tau]{{C_1(\tauw[#1])}}
\newcommand{\ConetauwhatFtwo}[1][\tau]{{\myhat{C}_1(\tauw[#1])}}

\newcommand{\CConetauwFtwo}[1][\tau]{{C^1(\tauw[#1])}}

\newcommand{\CHonetauwFtwo}[1][\tau]{{H^1(C^\bullet(\tauw[#1]))}}
\newcommand{\CHonetauwhatFtwo}[1][\tau]{{H^1(\myhat{C}^\bullet(\tauw[#1]))}}

\newcommand{\HtauwFtwo}[1][\tau]{{H_\bullet(C_\bullet(\tauw[#1]))}}
\newcommand{\HtauwhatFtwo}[1][\tau]{{H_\bullet(\myhat{C}_\bullet(\tauw[#1]))}}

\newcommand{\HonetauwFtwo}[1][\tau]{{H_1(C_\bullet(\tauw[#1]))}}
\newcommand{\HonetauwhatFtwo}[1][\tau]{{H_1(\myhat{C}_\bullet(\tauw[#1]))}}

\newcommand{\deltatau}[1][\tau]{\delta^{#1}}
\newcommand{\epstau}[1][\tau]{\varepsilon^{#1}}

\newcommand{\Ypoint}{y}
\newcommand{\Ycurve}{c}
\newcommand{\Yface}{f}

\newcommand{\flip}{f}

\newcommand{\RA}[1]{R\langle\hspace{-0.05cm}\langle #1\rangle\hspace{-0.05cm}\rangle}

\newcommand{\myid}{1\hspace{-0.15cm}1}

\newcommand{\maxid}{\mathfrak{m}}
\newcommand{\Jacalg}[1]{\mathcal{P}(#1)}

\newcommand{\tauc}{(\tau,\xi)}
\newcommand{\sigmad}{(\sigma,\zeta)}
\newcommand{\Atauc}{A\tauc}
\newcommand{\Stauc}{S\tauc}
\newcommand{\AStauc}{(\Atauc,\Stauc)}
\newcommand{\Asigmad}{A\sigmad}
\newcommand{\Ssigmad}{S\sigmad}
\newcommand{\ASsigmad}{(\Asigmad,\Ssigmad)}

\newcommand{\AAA}{\mathbb{A}}
\newcommand{\BB}{\mathbb{B}}
\newcommand{\CC}{\mathbb{C}}
\newcommand{\tildeB}{\widetilde{\mathbb{B}}}
\newcommand{\tildeC}{\widetilde{\mathbb{C}}}
\newcommand{\tildeBC}{\widetilde{\mathbb{BC}}}

\newcommand{\diag}{\operatorname{diag}}
\newcommand{\lcm}{\operatorname{lcm}}

\newdir{((}{{}*!/-2\jot/\dir^{(}}
\newdir{(((}{{}*!/-2.8\jot/\dir^{(}}

%% file: 00_abstract.tex
\begin{abstract}
Let $\SSigma=\surf$ be either an unpunctured surface with marked points and order-2 orbifold points, or a once-punctured closed surface with order-2 orbifold points.
For each pair $(\tau,\omega)$ consisting of a triangulation $\tau$ of $\SSigma$ and a function $\omega:\orb\rightarrow\{1,4\}$, we define a chain complex $\Ctauw$ with coefficients in $\F_2=\mathbb{Z}/2\mathbb{Z}$. Given $\SSigma$ and $\omega$, we define a \emph{colored triangulation} of $\SSigmaw=(\Sigma,\marked,\orb,\omega)$ to be a pair $(\tau,\xi)$ consisting of a triangulation of $\SSigma$ and a 1-cocycle in the cochain complex which is dual to $\Ctauw$; the combinatorial notion of \emph{colored flip} of colored triangulations is then defined as a refinement of the notion of flip of triangulations. Our main construction associates to each colored triangulation a species and a potential, and our main result shows that colored triangulations related by a flip have species with potentials (SPs) related by the corresponding SP-mutation as defined in \cite{Geuenich-Labardini-1}.

We define the \emph{flip graph} of colored triangulations of $\SSigmaw$ as the graph whose vertices are the pairs $(\tau,x)$ consisting of a triangulation $\tau$ and a cohomology class $x\in \CHonetauwFtwo$, with an edge connecting two such pairs $(\tau,x)$ and $(\sigma,z)$ if and only if there exist 1-cocycles $\xi\in x$ and $\zeta\in z$ such that $(\tau,\xi)$ and $(\sigma,\zeta)$ are colored triangulations related by a colored flip; then we prove that this flip graph is always disconnected provided the underlying surface $\Sigma$ is not contractible.

In the absence of punctures, we show that the Jacobian algebras of the SPs constructed are finite-dimensional and that whenever two colored triangulations have the same underlying triangulation, the Jacobian algebras of their associated SPs are isomorphic if and only if the underlying 1-cocycles have the same cohomology class; we also give a full classification of the non-degenerate SPs one can associate to any given pair $(\tau,\omega)$ over cyclic Galois extensions with primitive $4^{\operatorname{th}}$ roots of unity.

The species constructed here are species realizations of the $2^{|\orb|}$ skew-symmetrizable matrices that Felikson-Shapiro-Tumarkin associated in \cite{FeShTu-orbifolds} to any given triangulation of $\SSigma$. In the prequel \cite{Geuenich-Labardini-1} to this paper we constructed a species realization of only one of these matrices, but therein we allowed the presence of arbitrarily many punctures.
\end{abstract}

%% file: 00_intro.tex
\section{Introduction}

In \cite{FeShTu-orbifolds}, Felikson-Shapiro-Tumarkin showed how $2^{|\orb|}$ different cluster algebras can be associated to a surface with marked points and order-2 orbifold points $\SSigma=\surf$. They did so by assigning $2^{|\orb|}$ skew-symmetrizable matrices to each (tagged) triangulation of $\SSigma$, and by showing that if two (tagged) triangulations are related by a flip, then the two assigned $2^{|\orb|}$-tuples of matrices are related by the corresponding matrix mutation of Fomin-Zelevinsky. They then used the alluded $2^{|\orb|}$ cluster algebras to study the ``\emph{lambda length} coordinate ring'' of the corresponding \emph{decorated Teichm\"uller space} in a way similar to the one established by Fomin-Shapiro-Thurston \cite{FST} and Fomin-Thurston in \cite{FT} in the case of surfaces with marked points and without orbifold points.

In \cite{Geuenich-Labardini-1} we associated a species and a potential to each triangulation of a possibly punctured surface with marked points and order-2 orbifold points, and proved that triangulations related by a flip have species with potential related by the SP-mutation also defined in \cite{Geuenich-Labardini-1}. The species constructed therein for any given triangulation $\tau$ is a species realization of only one of the $2^{|\orb|}$ skew-symmetrizable matrices assigned by Felikson-Shapiro-Tumarkin to $\tau$. In this paper we consider surfaces that are either unpunctured with order-2 orbifold points, or once-punctured closed with order-2 orbifold points, and for every ideal triangulation of such surfaces we realize the $2^{|\orb|}$ matrices of Felikson-Shapiro-Tumarkin via SPs. Our main result, Theorem \ref{thm:flip<->SP-mutation}, asserts that the SP-mutations of the SPs we construct are compatible with the flips of triangulations. More precisely, we introduce the notions of \emph{colored triangulation} and \emph{colored flip}, which are refinements of the notions of ideal triangulation and flip, associate an SP to each colored triangulation, and show that whenever two colored  triangulations are related  by a colored flip, then the associated SPs are related by the corresponding SP-mutation defined in \cite{Geuenich-Labardini-1}.

Let us say some words about the context within which this paper has been written; this context dates back at least ten years, to the works of Derksen-Weyman-Zelevinsky \cite{DWZ1} and Fomin-Shapiro-Thurston \cite{FST}, which themselves go back to the turn-of-the-century discovery and invention by Sergey Fomin and Andrei Zelevinsky of nowadays pervasive cluster algebras \cite{FZ1}.

The mutation theory of quivers with potential developed by Derksen-Weyman-Zelevinsky in \cite{DWZ1} has turned out to be very useful both inside and outside cluster-algebra theory. In cluster algebras, it provided representation-theoretic means for solutions, in the case of skew-symmetric cluster algebras, of several conjectures stated by Fomin-Zelevinsky in \cite{FZ4} (see, for instance, \cite{DWZ2} and \cite{CKLP}). Outside cluster algebras, one of the reasons of its usefulness is that, thanks to a powerful
result of Keller-Yang \cite{Keller-Yang}, it can be thought of as a down-to-earth counterpart of tilting in certain triangulated categories.

Motivated by \cite{DWZ1} and by Fomin-Shapiro-Thurston's assignment of a \emph{signed-adjacency quiver} to every tagged triangulation of a surface with marked points and without orbifold points (cf. \cite{FST}) and by the compatibility between flips and quiver mutations this assignment possesses, the second author of the present paper noticed in \cite{Labardini1} and \cite{Labardini4} that every tagged triangulation of a surface with marked points and without orbifold points comes with a ``natural'' potential on its signed-adjacency quiver, and proved that if two tagged triangulations are related by a flip, then the associated quivers with potential are related by the corresponding QP-mutation of Derksen-Weyman-Zelevinsky. A combination of this result with the result of Keller-Yang referred to above allows to associate to any surface with marked points and without orbifold points a 3-Calabi-Yau triangulated category, the combinatorics of whose ``canonical'' hearts and tilts is parametrized and governed by the tagged triangulations of the surface and the flips of triangulations. Similarly, combining the results of \cite{Labardini1,Labardini4} with the constructions and results given by Amiot in \cite{Amiot-gldim2,Amiot-survey} allows to associate to the surface a 2-Calabi-Yau Hom-finite triangulated category whose ``canonical'' cluster-tilting objects are parametrized by tagged triangulations, and with the IY-mutation\footnote{\emph{IY-mutation} is a mutation operation defined by Iyama-Yoshino on cluster-tilting subcategories inside certain triangulated categories, cf. \cite{Iyama-Yoshino}.} of cluster-tilting objects modeled by the combinatorial operation of flip of arcs in triangulations.
The present paper has been written in the belief and hope that similar associations of triangulated categories to surfaces with marked points and order-2 orbifold points can be made and that, this way, the SPs constructed here can turn out to be as useful as the QPs associated to triangulations of surfaces without orbifold points have turned out to be.


Now, we describe the contents of the paper in more detail. In Section \ref{sec:surfaces-and-triangs} we recall the framework of surfaces with marked points and order-2 orbifold points, to which Felikson-Shapiro-Tumarkin have associated several cluster algebras in \cite{FeShTu-orbifolds}. The framework we recall is less general than Felikson-Shapiro-Tumarkin's framework since we will work only with surfaces that either are unpunctured (with an arbitrary number of orbifold points) or once-punctured with empty boundary (again with arbitrarily many orbifold points).

Letting $\SSigma=\surf$ be any surface as in the previous paragraph\footnote{With the further exception of 8 surfaces of genus zero, which we shall explicitly list.}, in Section \ref{sec:weighted-quiver} we describe the $2^{|\orb|}$ weighted quivers (that correspond under \cite[Lemma 2.3]{LZ} to the $2^{|\orb|}$ skew-symmetrizable matrices) associated by Felikson-Shapiro-Tumarkin to any given triangulation $\tau$ of $\surf$. More precisely, we describe a rule that for each pair $(\tau,\omega)$ consisting of a triangulation $\tau$ of $\surf$ and a function $\omega:\orb\rightarrow\{1,4\}$ allows us to obtain a weighted quiver $(Q(\tau,\omega),\dtuple(\tau,\omega))$, where $Q(\tau,\omega)$ is defined in terms of the signed adjacencies between the arcs in $\tau$ (and also takes the function $\omega$ into account), and the tuple $\dtuple(\tau,\omega)$ is defined so as to attach the integer $2$ to every arc not containing any orbifold point, and the integer $\omega(q)$ to each arc containing an orbifold point $q\in\orb$. For a fixed $\tau$, the functions $\omega:\orb\rightarrow\{1,4\}$, parametrize all the weighted quivers that correspond under \cite[Lemma 2.3]{LZ} to the $2^{|\orb|}$ skew-symmetrizable matrices attached to $\tau$ by Felikson-Shapiro-Tumarkin\footnote{They parametrize their matrices with a slightly different parameter though, namely the functions $w:\orb\rightarrow\{\frac{1}{2},2\}$, see \cite[Remark 4.7-(4)]{Geuenich-Labardini-1}.}. In Section \ref{sec:weighted-quiver} we also define two auxiliary quivers $Q'(\tau,\omega)$ and $Q''(\tau,\omega)$, which we use in Section \ref{sec:chain-complexes} to define two chain complexes $C_\bullet(\tau,\omega)$ and $\widehat{C}_\bullet(\tau,\omega)$ with coefficients in the field $\F_2=\mathbb{Z}/2\mathbb{Z}$ for each fixed pair $(\tau,\omega)$.

Let us stress at this point that, starting in Section \ref{sec:weighted-quiver} and to the very end of the paper, we will work not only with a fixed surface $\SSigma=\surf$, but also with a fixed function $\omega:\orb\rightarrow\{1,4\}$; only the triangulations will be allowed to vary from Section \ref{sec:weighted-quiver} to the end of the paper.

Just as a quiver alone does not suffice to define a path algebra, for the further specification of a ground field is necessary, having a weighted quiver is not enough to define a path algebra, for the further specifications of an appropriate ground field extension and of a \emph{modulating function} are needed (see \cite[Subsection 2.4 and Section 3.1 up to Definition 3.5]{Geuenich-Labardini-1}). So, besides the weighted quiver $(Q(\tau,\omega),\dtuple(\tau,\omega))$, we need an appropriate field extension and a modulating function in order to be able to associate a species and a (complete) path algebra to $(Q(\tau,\omega),\dtuple(\tau,\omega))$. In Section \ref{sec:chain-complexes} we define a chain complex $C_\bullet(\tau,\omega)$ in terms of the quiver $Q'(\tau,\omega)$ and of the triangles of the triangulation $\tau$; each choice of an element of the first cocycle group of the cochain complex which is $\F_2$-dual to $C_\bullet(\tau,\omega)$ will allow us to read off a modulating function for $(Q(\tau,\omega),\dtuple(\tau,\omega))$ whenever we are given a degree-$d$ cyclic Galois field extension $E/F$, with $F$ having primitive $d^{\operatorname{th}}$ roots of unity, where $d=\lcm(\dtuple(\tau,\omega))\in\{2,4\}$. This means, in particular, that we will attach not only one species to a given pair $(\tau,\omega)$ even when the field extension $E/F$ is fixed, but as many as 1-cocycles the cochain complex $C^\bullet(\tau,\omega):=\Hom_{\F_2}(C_\bullet(\tau,\omega),\F_2)$ has. Let $ Z^1(\tau,\omega)$ be the $\F_2$-vector subspace of $C^1(\tau,\omega):=\Hom_{\F_2}(C_1(\tau,\omega),\F_2)$ whose elements are such 1-cocycles. In order to be able to define the combinatorial operation of flip of a pair $(\tau,\xi)$ consisting of a triangulation $\tau$ and a 1-cocycle $\xi\in Z^1(\tau,\omega)$, in Section \ref{sec:chain-complexes} we  use the quiver $Q''(\tau,\omega)$ to define an auxiliary chain complex $\widehat{C}_\bullet(\tau,\omega)$ whose $\F_2$-dual cochain complex will help us to define the desired notion of flip of $(\tau,\xi)$ with respect to an arc of $\tau$.


In Section \ref{sec:colored-triangulations-and-flips} we introduce the notions of \emph{colored triangulations} and their flips, which we distinguish from the flips of ordinary triangulations by calling them \emph{colored flips}. Given a fixed $\SSigmaw=(\Sigma,\marked,\orb,\omega)$, a colored triangulation of $\SSigmaw$ is defined to be any pair $(\tau,\xi)$ consisting of a triangulation $\tau$ of $\SSigma$ and a 1-cocycle $\xi\in\CZonetauwFtwo\subseteq C^1(\tau,\omega) $. The colored flip of a colored triangulation $(\tau,\xi)$ with respect to an arc $k\in\tau$ is defined in Section \ref{sec:colored-triangulations-and-flips} as well and produces another colored triangulation $(\sigma,\zeta)=\flip_k(\tau,\xi)$, with $\sigma$ defined to be the triangulation obtained by applying the ordinary flip of $k$ to $\tau$, and $\zeta$ defined to be a certain 1-cocycle inside the cochain complex $C^1(\sigma,\omega)$, see Definition \ref{def:colored-flips-of-colored-triangulations}. The assignment $\xi\mapsto \zeta$ does not always constitute a group homomorphism $\CZonetauwFtwo\rightarrow Z^1(\sigma,\omega)$, and is defined after passing to the larger auxiliary cochain complexes $\widehat{C}^\bullet(\tau,\omega)=\Hom_{\F_2}(\widehat{C}_\bullet(\tau,\omega),\F_2)$ and $\widehat{C}^\bullet(\sigma,\omega)=\Hom_{\F_2}(\widehat{C}_\bullet(\sigma,\omega),\F_2)$.

Section \ref{sec:sp-of-a-colored-triangulation} is devoted to associating to each colored triangulation $(\tau,\xi)$ an SP over any degree-$d$ cyclic Galois field extension $E/F$ with the property that $F$ has a primitive $d^{\operatorname{th}}$ root of unity\footnote{One can always find a suitable extension $E/F$ of finite or $p$-adic fields satisfying the desired properties. For example, if $p$ is a positive prime number congruent to $1$ modulo $4$, and $F$ is either $\mathbb{Q}_p$ or finite with $\operatorname{char}(F)=p$, then $F$ definitely has a primitive $d^{\operatorname{th}}$ root of $1$ and a degree-$d$ cyclic Galois extension $E/F$. Furthermore, if $d$ happens to be $2$, one can take $E/F$ to be $\mathbb{C}/\mathbb{R}$.
}, where $d=\lcm(\dtuple(\tau,\omega))\in\{2,4\}$. Given such an extension $E/F$ and a colored triangulation $(\tau,\xi)$ of our fixed $\SSigmaw$, in Subsection \ref{subsec:species-of-a-triangulation} we attach to each $k\in\tau$ the unique subfield $F_k$ of $E$ such that $[F_k:F]=d(\tau,\omega)_k$, and show how the 1-cocycle $\xi\in Z^1(\tau,\omega)\subseteq C^1(\tau,\omega)$ naturally gives rise to a modulating function $g(\tau,\xi):Q(\tau,\omega)_1\rightarrow\bigcup_{j,k\in\tau}\Gal(F_j\cap F_k/F)$ and hence to a species $A(\tau,\xi)$. In Subsection \ref{sec:potentials-for-colored-triangulations} we locate some `obvious' cycles on $A(\tau,\xi)$ for the different types of triangles that $\tau$ can have, and define a potential $S(\tau,\xi)\in\RA{A(\tau,\xi)}$ as the sum of such `obvious' cycles. The definition of $S(\tau,\xi)$ follows the same basic idea of \cite{Labardini1}, \cite{Labardini4} and \cite{Geuenich-Labardini-1} in that it is the sum of cyclic paths on $A(\tau,\xi)$ that are `as obvious as possible'.

In Section \ref{sec:flip=>SP-mutation} we arrive at Theorem \ref{thm:flip<->SP-mutation}, the main result of this paper, stated for surfaces that are either unpunctured with order-2 orbifold points, or once-punctured closed with order-2 orbifold points: for such a surface $\SSigma=\surf$ and any fixed function $\omega:\orb\rightarrow\{1,4\}$, if $(\tau,\xi)$ and $(\sigma,\zeta)$ are colored triangulations of $\SSigmaw=(\Sigma,\marked,\orb,\omega)$ that are related by the colored flip of an arc $k\in\tau$, then the associated SPs $(A(\tau,\xi),S(\tau,\xi))$ and $(A(\sigma,\zeta),S(\sigma,\zeta))$ are related by the $k^{\operatorname{th}}$ SP-mutation defined in \cite{Geuenich-Labardini-1}. More precisely, the SPs $\mu_k(A(\tau,\xi),S(\tau,\xi))$ and $(A(\sigma,\zeta),S(\sigma,\zeta))$ are right-equivalent, where $\mu_k(A(\tau,\xi),S(\tau,\xi))$ is obtained from $(A(\tau,\xi),S(\tau,\xi))$ by applying \cite[Definitions 3.19 and 3.22]{Geuenich-Labardini-1}. Consisting of a careful and detailed case-by-case analysis of the possible configurations that $\tau$ and $\omega$ can present around the arc $k$, the proof of Theorem \ref{thm:flip<->SP-mutation} is rather lengthy and hence deferred to Section \ref{sec:proof-of-main-theorem}.

In Section \ref{sec:relation-to-surface-homology} we show that the homology groups of the chain complex $C_\bullet(\tau,\omega)$  (resp. $\widehat{C}_\bullet(\tau,\omega)$) and the cohomology groups of its dual chain complex $C^\bullet(\tau,\omega)$ (resp. $\widehat{C}^\bullet(\tau,\omega)$) are respectively isomorphic to the singular homology and cohomology groups with coefficients in $\F_2$ of a surface $\overlineSigma$ (resp. $\widehat{\Sigma}$) which is closely related to $\Sigma$. These isomorphisms will turn out to be canonical in the sense that for any two triangulations $\tau$ and $\sigma$ of $\SSigma$ that are related by the flip of an arc $k\in\tau$ we will have several commutative diagrams, the most important one being the diagram
$$
\xymatrix{ & H^1(\widehat{\Sigma};\F_2) \ar[dl] \ar[dr]  &  \\
H^1(\widehat{C}^\bullet(\tau,\omega)) \ar@<.75ex>[rr]  & & H^1(\widehat{C}^\bullet(\sigma,\omega)) \ar@<.75ex>[ll]}
$$
(see Proposition \ref{prop:theta-and-phi-compose-to-theta}). The desired canonical isomorphisms are obtained by giving concrete ``geometric realizations'' of the complexes $C^\bullet(\tau,\omega)$ (resp. $\widehat{C}^\bullet(\tau,\omega)$); more precisely, by realizing these complexes as cellular complexes with coefficients in $\F_2$ of some very concrete topological subspaces $Y(\tau,\omega)\subseteq \widehat{Y}(\tau,\omega)$ of $\Sigma$.

In Section \ref{sec:flip-graph} we define two different flip graphs for a given $\SSigmaw$. The \emph{cocycle flip graph} is the graph that has the colored triangulations of $\SSigmaw$ as its vertices, with an edge joining two colored triangulations precisely if they are related by a colored flip. The \emph{flip graph} is the simple graph obtained from the cocycle flip graph by identifying two vertices $(\tau,\xi)$ and $(\tau',\xi')$ precisely when $\tau=\tau'$ and $[\xi]=[\xi']$ as elements of $H^1(C^\bullet(\tau,\omega))$. With the aid of the isomorphisms from Section \ref{sec:relation-to-surface-homology}, we show that both of these graphs are disconnected if $\Sigma$ happens to be not contractible as a topological space.

Section \ref{sec:cohomology-and-Jacobian-algebras} is devoted to 
showing that if we are given colored triangulations $(\tau,\xi)$ and $(\tau,\xi')$ with the same underlying triangulation, then the corresponding Jacobian algebras $\mathcal{P}(A(\tau,\xi),S(\tau,\xi))$ and $\mathcal{P}(A(\tau,\xi'),S(\tau,\xi'))$ are isomorphic as rings through an $F$-algebra isomorphism that fixes each of the basic idempotents if and only if the 1-cocycles $\xi$ and $\xi'$ have the same cohomology class in $H^1(C^\bullet(\tau,\omega))$. Consequently, for a fixed $(\tau,\omega)$ the (isomorphism classes of the) Jacobian algebras we construct for colored triangulations of $(\tau,\omega)$ are bijectively parametrized by the first cohomology group $H^1(C^\bullet(\tau,\omega))$, see Theorem \ref{thm:comologous<=>isomorphic-Jacobian-algs}. 

In Section \ref{sec:classification-of-nondeg-SPs} we classify all possible realizations that the skew-symmetrizable matrices $B(\tau,\omega)$ can have via non-degenerate SPs up to right-equivalence. More precisely, we show that if $\SSigma$ is an unpunctured surface different from a once-marked torus without orbifold points, $(\tau,\omega)$ is any pair consisting of a triangulation $\tau$ of $\SSigma$ and a function $\omega:\orb\rightarrow\{1,4\}$, $A$ is a species realization of $B(\tau,\omega)$ over $E/F$ (where as before $E/F$ is a degree-$d$ cyclic Galois field extension with $d=\lcm(\dtuple(\tau,\omega))$) that admits a non-degenerate potential $S$ (in the sense of Derksen-Weyman-Zelevinsky), then there exists a 1-cocycle $\xi\in Z^1(\tau,\omega)\subseteq C^1(\tau,\omega)$ such that the SP $(A,S)$ is right-equivalent to the SP $(A(\tau,\xi),S(\tau,\xi))$.

Taking into account what we have said in the third, fourth and fifth paragraphs of this introduction, in Section \ref{sec:problems} we state some problems and questions that, in our opinion, arise naturally from the constructions and results of this paper. For instance, we believe that if $\SSigma$ is an unpunctured disc with at most two orbifold points, then for any $\omega:\orb\rightarrow\{1,4\}$ and any colored triangulation $(\tau,\xi)$ of $\SSigmaw$ the resulting Jacobian algebra $\mathcal{P}(A(\tau,\xi),S(\tau,\xi))$ is a \emph{cluster-tilted algebra} in the sense  \cite{BMRRT,BMR} of Buan-Marsh-Reineke-Reiten-Todorov and Buan-Marsh-Reiten. We also wonder whether the constructions and results of Amiot \cite{Amiot-gldim2,Amiot-survey}, Keller-Yang \cite{Keller-Yang} and Plamondon \cite{Plamondon-characters,Plamondon} can be applied as are or adapted in order to associate 2- and 3-Calabi-Yau triangulated categories to $\SSigmaw$ via the SPs we have constructed in this paper. We ask whether our constructions of SPs can be generalized so as to encompass all arbitrarily punctured surfaces with order-2 orbifold points.

Section \ref{sec:examples} is devoted to discussing how the matrices that are mutation-equivalent to a matrix of Dynkin type $\AAA_n$, $\CC_n$ or $\BB_n$,  or of affine type $\tildeC_n$, $\tildeB_n$, or $\tildeBC_n$, can be realized by SPs associated to colored triangulations of polygons with at most two orbifold points. We also point out that, over the complex numbers, our constructions associate two non-Morita-equivalent path algebras to the Kronecker quiver, one of them being the well known path algebra on which $\mathbb{C}$ acts centrally.  Finally, we briefly outline the relation that the constructions and results of this paper keep with those given previously by Assem-Br\"ustle-Charbonneau-Plamondon \cite{ABCP}, by those who this write \cite{Geuenich-Labardini-1} and by the second author of the present paper \cite{Labardini1,Labardini4}.

Section \ref{sec:proof-of-main-theorem} closes the paper with a case-by-case proof of Theorem \ref{thm:flip<->SP-mutation}, which, as we have said above, is the main result of this work.

%% file: 00_acknowledgments.tex
\section*{Acknowledgments}

We thank Marcelo Aguilar, Christof Geiss, Max Neumann-Coto and Jan Schr\"{o}er for helpful discussions. The financial support provided by the first author's \emph{BIGS} scholarship and by the second author's grants \emph{CONACyT-238754} and \emph{PAPIIT-IA102215} is gratefully acknowledged, for it allowed a five month visit of the first author to \emph{IMUNAM}, the second author's institution, a visit during which we came up with the main constructions and results of this paper.

%% file: 01_surfaces_and_triangs.tex
\section{Triangulated surfaces with weighted orbifold points}

\label{sec:surfaces-and-triangs}

Recall from~%
 \cite[Definition~2.1]{Geuenich-Labardini-1}
that a \emph{surface with marked points and orbifold points of order $2$} is a triple~$\SSigma = (\Sigma, \marked, \orb)$ where
\begin{itemize}
 \item $\Sigma$ is an oriented connected compact real surface with boundary~$\partial \Sigma$,
 \item $\marked \subseteq \Sigma$ is a non-empty finite set meeting each connected component of~$\partial \Sigma$ at least once,
 \item $\orb \subseteq \Sigma \setminus (\partial \Sigma \cup \marked)$ is a finite set (possibly empty).
\end{itemize}
The elements in $\marked$ are called \emph{marked points} and those in $\orb$ \emph{orbifold points of order $2$}.
Marked points belonging to the interior~$\Sigma \setminus \partial \Sigma$ are known as \emph{punctures}. We shall refer to $\SSigma$ simply as a \emph{surface}

In this article we will only work with the following two types of surfaces:
\begin{enumerate}\item \emph{unpunctured surfaces}, i.e. those $\SSigma$ that satisfy that the boundary $\partial\Sigma$ of $\Sigma$ is not empty, and~$\marked \subseteq \partial \Sigma$, the finite set $\orb$ being arbitrary;
\item \emph{once-punctured closed surfaces}, i.e. those $\SSigma$ that satisfy that the boundary $\partial\Sigma$ of $\Sigma$ is empty, and~$|\marked|=1$, the finite set $\orb$ being arbitrary.
\end{enumerate}
As usual, there will be a few surfaces that we shall completely exclude from our considerations. These are those the following 8 surfaces:
\begin{itemize}
\item once-punctured closed spheres with $|\orb|<4$;
\item the unpunctured disc with $|\marked| =1$ and $|\orb|=1$;
\item the unpunctured discs with $|\marked|\in\{1,2,3\}$ and $|\orb|=0$.
\end{itemize}

\begin{remark}\label{rem:why-terminology-order-2}
 For an explanation of the terminology ``orbifold points of order~$2$'' we kindly refer the reader to the introduction and to Remark~2.2 of~\cite{Geuenich-Labardini-1}; see also~\cite{FeShTu-orbifolds}.
 For the sake of brevity, we will often omit the attribute ``of order $2$'' throughout the rest of the article. We will call the elements in~$\orb$ simply \emph{orbifold points}.
\end{remark}

Fomin, Shapiro, and Thurston introduced in~\cite{FST} the notions of ideal and tagged triangulations for surfaces with marked points and orbifold points~$\SSigma = (\Sigma, \marked, \orb)$ in the case~$\orb = \varnothing$.
Their definitions were generalized by Felikson, Shapiro, and Tumarkin \cite{FeShTu-orbifolds} to the case where $\orb$ may be non-empty. Detailed definitions of ideal and tagged triangulations can also be found in~\cite[Definition~2.3]{Geuenich-Labardini-1}.
Since such definitions are rather lengthy we will not reproduce them here.
Observe, however, that if $\SSigma$ has no punctures, then every tagged arc is necessarily tagged plain, fact that implies that the notions of ideal and tagged triangulation coincide. On the other hand, if $\SSigma$ is closed with exactly one puncture, then every flip of any ideal triangulation produces an ideal triangulation. These are the reasons for the following definition, given the setting we shall work with.

\begin{defi}\label{def:triangulation-of-SSigma}
 A \emph{triangulation}~$\tau$ of~$\SSigma$ is an ideal triangulation~$\tau$ of $\SSigma$ as defined in \cite[Definition~2.3]{Geuenich-Labardini-1}  (see \cite[Section 4]{FeShTu-orbifolds} as well).
 An arc of a triangulation~$\tau$ is called \emph{pending} if it contains an orbifold point, otherwise it is \emph{non-pending}.
\end{defi}

One of the reasons that make triangulations combinatorially so interesting, is the following theorem by Felikson, Shapiro, and Tumarkin~\cite{FeShTu-orbifolds}, which builds on results from~\cite{FST}.

\begin{thm}\label{thm:ordinary-flip-graph-is-connected}
 Let~$\SSigma$ be a surface with marked points and orbifold points.
 If $\SSigma$ is either unpunctured or once-punctured closed, then:
 \begin{enumerate}
  \itemsep 4pt \parskip 0pt \parsep 0pt

  \item
  For every triangulation~$\tau$ of~$\SSigma$ and every arc~$i \in \tau$, there is a unique arc~$j$ on $\SSigma$, different from $i$, such that~$\flip_i(\tau) := (\tau \setminus \{i\}) \cup \{j\}$ is again a triangulation of~$\SSigma$.
  One says that~$\flip_i(\tau)$ is obtained from~$\tau$ by the \emph{flip} of $i$.
  \item
  Every two triangulations of~$\SSigma$ can be obtained from each other by a finite sequence of flips.
 \end{enumerate}
\end{thm}

\begin{ex} In Figure \ref{Fig:some_flips} we can see four triangulations of a hexagon with one orbifold point.
        \begin{figure}[!ht]
                \centering
                \includegraphics[scale=.175]{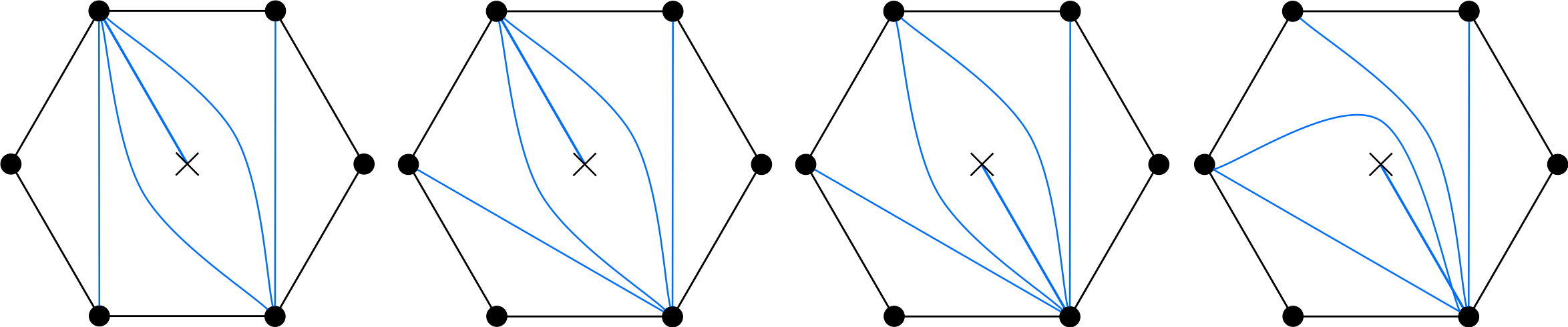}
                \caption{}
                \label{Fig:some_flips}
        \end{figure}
Every two consecutive triangulations in the figure are clearly related by a flip.
\end{ex}

It will be quite essential for our arguments that every triangulation can be glued from a finite number of ``puzzle pieces''.
For a more precise statement of what this means see \cite{FeShTu-orbifolds} or \cite{Geuenich-Labardini-1}.
Each of the puzzle pieces needed in the gluing occurs in the list depicted in Figure \ref{Fig:unpunct_puzzle_pieces}.
        \begin{figure}[!ht]
                \centering
                \includegraphics[scale=.175]{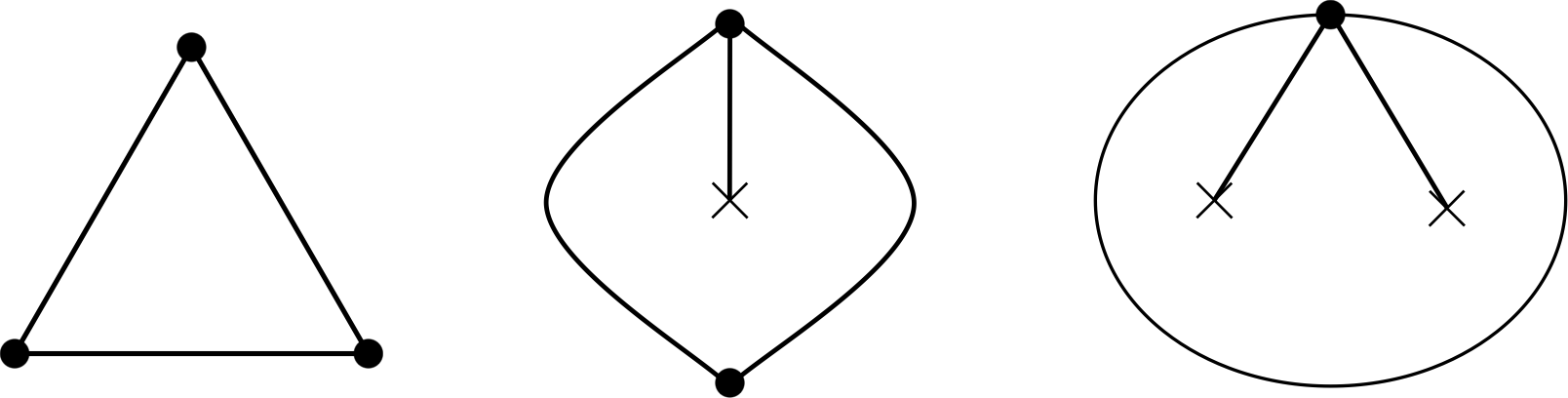}
                \caption{}
                \label{Fig:unpunct_puzzle_pieces}
        \end{figure}

In particular, the possibilities for how a triangle in a triangulation of one of the surfaces in our setting (unpunctured, or once-punctured closed) can look like are limited. More precisely, we can distinguish the following three types of triangles:

\begin{enumerate}
 \itemsep 4pt \parskip 0pt \parsep 0pt

 \item
 \emph{Ordinary triangles}, i.e.\ triangles containing no orbifold points.

 \item
 \emph{Once orbifolded triangles}, i.e.\ triangles containing exactly one orbifold point.

 \item
 \emph{Twice orbifolded triangles}, i.e.\ triangles containing containing exactly two orbifold points.
\end{enumerate}

As already mentioned in the Introduction, our input information will consist not only of a surface $\SSigma$, but of an assignment of a weight to each orbifold point.

\begin{defi}\label{def:surface-with-weighted-orb-points}
 A \emph{surface with marked points and weighted orbifold points}~$\SSigmaw$, is a surface $\SSigma=\surf$ together with a function $\omega : \orb \to \{1, 4\}$.
\end{defi}

Given a function $\omega : \orb \to \{1, 4\}$ and a triangulation~$\tau$ of $\SSigma$, we denote by $\tau^{\omega=1}$ the subset of arcs in $\tau$ consisting of all pending arcs whose orbifold point $q$ has weight~\smash{$\omega(q) = 1$}.

\begin{remark}\begin{enumerate}\item The idea of taking a function $\omega:\orb\rightarrow\{1,4\}$ as part of the input information comes from~\cite{FeShTu-orbifolds}; letting $\omega$ vary allows to associate $2^{|\orb|}$ skew-symmetrizable matrices to any given $\tau$.
\item The number $\omega(q)\in\{1,4\}$ is never the \emph{order} of $q$ as an orbifold point (as already mentioned, the order is always~$2$ in our setting).
\item The reader may wonder where it is exactly that the order of the orbifold points being 2 plays a role. The role is subtly played in Felikson-Shapiro-Tumarkin's definition of the notion of triangulation of $\SSigma$ (Definition \ref{def:triangulation-of-SSigma}), as this notion is subtly tailored so that if a Fuchsian group $\Gamma\subseteq\operatorname{Iso}(\mathbb{H}^2)$ is given with the properties that all its non-trivial finite subgroups have order 2 and some fundamental domain $D\subseteq\mathbb{H}^2$ of $\Gamma$ has finitely many sides and finite hyperbolic area, then any triangulation of $D$ by hyperbolic geodesics is mapped to a combinatorial triangulation of the corresponding surface with orbifold points under the projection $\mathbb{H}^2\rightarrow\mathbb{H}^2/\Gamma$, and the flip of a hyperbolic geodesic passing through a fixed point of an elliptic M\"obius transformation (necessarily of order 2) is mapped to the combinatorial flip of a pending arc. See \cite[the discussion that follows Definition 2.8]{Geuenich-Labardini-1}
 \end{enumerate}
\end{remark}

For the rest of the article,~$\SSigma_\omega=(\Sigma,\marked,\orb,\omega)$ will be part of our \emph{a priori} given input, and we will put the triangulations of $\SSigma$ to vary only after~$\SSigma_\omega$ is fixed.
We will denote by $\genus$ the genus and by $b$ the number of boundary components of $\Sigma$.
Moreover, we set $m = |\marked|$, $o = |\orb|$, and $u = |\{q \in \orb \suchthat \omega(q) = 1\}|$.
Observe that by our definitions $m \geq b > 0$ and $o \geq u \geq 0$.

%% file: 02_weighted_quiver.tex
\section{The weighted quivers of a triangulation}

\label{sec:weighted-quiver}

Let $\SSigma=\surf$ be a surface with orbifold points. In this section we will associate a weighted quiver $(Q(\tauw), \dtuple(\tauw))$ to each pair $(\tauw)$ consisting of a triangulation $\tau$ and a function $\omega:\orb\rightarrow\{1,4\}$. We will also define two more quivers~$Q'(\tauw)$ and $Q''(\tauw)$, closely related to~$Q(\tauw)$, that will turn out to be useful for our constructions of chain complexes in Section~\ref{sec:chain-complexes}. Our starting point to define $(Q(\tauw), \dtuple(\tauw))$, $Q'(\tauw)$ and $Q''(\tauw)$, will be to define a quiver $\overline{Q}(\tau)$ that does not depend on the function $\omega$.

Let us recall what a weighted quiver is.

\begin{defi}\cite{LZ}
 A \emph{weighted quiver} is a pair~$(Q, \dtuple)$ consisting of a quiver~$Q$
 and a tuple $\dtuple = (d_i)_{i \in Q_0}$ of positive integers.
 The integer~$d_i$ is called the \emph{weight} of the vertex~$i \in Q_0$.
\end{defi}

\begin{defi}
Let $\tau$ be a triangulation of $\SSigma$. We construct a
 quiver $\overline{Q}(\tau)$ as follows:
 \begin{enumerate}
  \itemsep 3pt \parskip 0pt \parsep 0pt

  \item
  We take for the vertex set of $\overline{Q}(\tau)$ the set of arcs of $\tau$, that is,  $\overline{Q}_0(\tau) = \tau$.

  \item
  The arrows of $\overline{Q}(\tau)$ are induced by the triangles of $\tau$.
  Namely, for each triangle $\triangle$ of $\tau$ and every pair $i, j \in \tau$ of arcs in $\triangle$ such that $j$ succeeds $i$ in $\triangle$ with respect to the orientation of $\Sigma$, we draw a single arrow
  from $i$ to $j$.
  Explicitly, for the three possible kinds of triangles depicted in Figure~\ref{Fig:unpunct_puzzle_pieces} we draw arrows according to the rule depicted in Figure \ref{Fig:rule-for-arrows-of-overlineQ}, with the understanding that no arrow incident to a boundary segment is drawn.
        \begin{figure}[!ht]
                \centering
                \includegraphics[scale=.1]{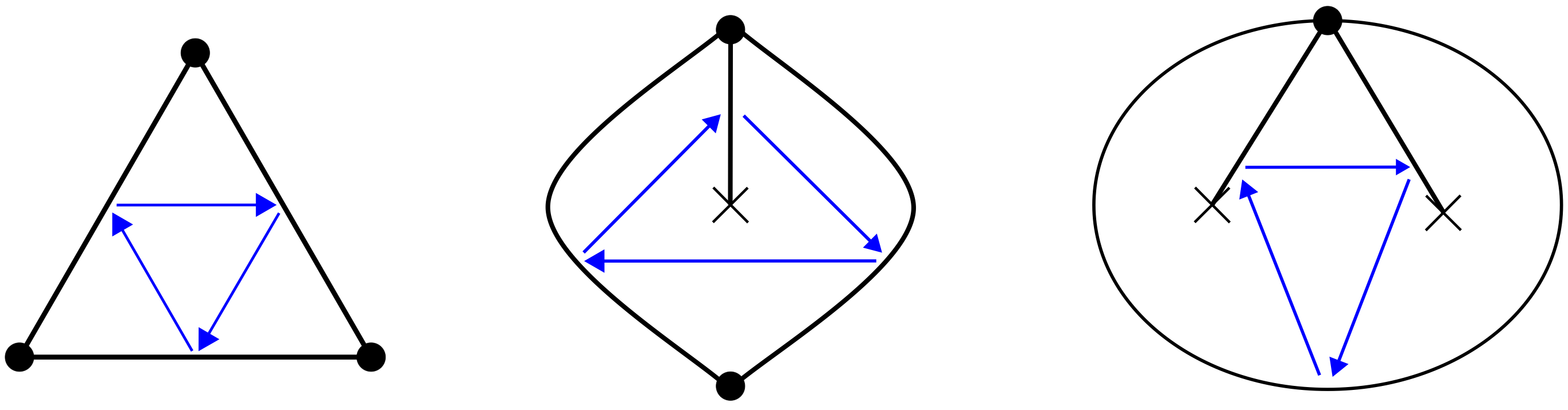}
                \caption{}
                \label{Fig:rule-for-arrows-of-overlineQ}
        \end{figure}
 \end{enumerate}
\end{defi}

\begin{defi}\label{def:Q(tau,omega)}
Let $\SSigma=\surf$ be a surface with orbifold points, $\omega:\orb\rightarrow\{1,4\}$ a function, and $\tau$ a triangulation of $\SSigma$.
  For each arc $i \in Q_0(\tau)$ we define an integer $\dtauwi$, the \emph{weight of $i$ with respect to $\omega$}, by the rule
   \[
    \dtauwi
    \:=\:
    \begin{cases}
     2         & \text{if $i$ is a non-pending arc,} \\
     \omega(q) & \text{if $i$ is a pending arc with $q \in i \cap \orb$.}
    \end{cases}
   \]
We set $\dtuple(\tauw)=(\dtauwi)_{i\in\tau}$, and define the \emph{weighted quiver of $\tau$ with respect to $\omega$} to be the weighted quiver $(Q(\tauw), \dtuple(\tauw))$ on the vertex set $Q_0(\tauw)=\tau$, where $Q(\tauw)$ is the quiver obtained from $\overline{Q}(\tau)$ by adding an extra arrow $j\rightarrow i$ for each pair of pending arcs $i$ and $j$ that satisfy $\dtauwi = \dtauwi[j]$ and for which $\overline{Q}(\tau)$ has an arrow from $j$ to $i$.
\end{defi}

\begin{thm}\label{thm:flip->mut-of-weighted-quivers} Let $\SSigma$, $\omega$ and $\tau$ be as in Definition \ref{def:Q(tau,omega)}. For any arc $k\in\tau$ we have $\mu_k(Q(\tauw),\dtuple(\tauw))=(Q(\flip_k(\tau),\omega),\dtuple(\flip_k(\tau),\omega))$, where $\flip_k(\tau)$ is the triangulation of $\SSigma$ obtained from $\tau$ by flipping the arc $k$. Here, $\mu_k$ is the $k^{\operatorname{th}}$ \emph{mutation of weighted quivers}\footnote{Which is nothing but the weighted-quiver version of Fomin-Zelevinsky's matrix mutation.} of \cite[Definition 2.5]{LZ}.
\end{thm}

\begin{remark}\label{rem:all-possible-matrices}\begin{enumerate}\item If $\omega:\orb\rightarrow\{1,4\}$ is the constant function taking the value $4$, then the quiver $Q(\tauw)$ is the same quiver as the one defined in~\cite{Geuenich-Labardini-1} and the weight tuple $\dtuple(\tauw)$ is obtained from the tuple defined in~\cite{Geuenich-Labardini-1} by multiplying every entry of the latter by $2$.
\item According to \cite[Lemma 2.3]{LZ}, each 2-acyclic weighted quiver $(Q,\dtuple)$ determines, and is determined by, a unique pair $(B,D)$ consisting of an integral skew-symmetrizable matrix $B$ and an integral diagonal matrix $D$ with positive diagonal entries such that $DB$ is skew-symmetric. On the other hand, in \cite[Subsection 4.3]{FeShTu-orbifolds}, Felikson-Shapiro-Tumarkin associated a skew-symmetrizable matrix $B(\tau,w)$ to each pair $(\tau,w)$ consisting of a triangulation of $\SSigma$ and a function $w:\orb\rightarrow\{\frac{1}{2},2\}$. From such a $w$ one can obtain a function $\omega:\orb\rightarrow\{1,4\}$ by setting $\omega(q)=\frac{2}{w(q)}$ for $q\in\orb$ (see \cite[Remark 4.7(4)]{Geuenich-Labardini-1}). The weighted quiver $(Q(\tauw), \dtuple(\tauw))$ turns out to correspond to the pair $(B(\tau,w),\diag(\dtuple(\tauw)))$ under \cite[Lemma 2.3]{LZ}. So, Definition \ref{def:Q(tau,omega)} is really due to Felikson-Shapiro-Tumarkin \cite{FeShTu-orbifolds}, and so is Theorem \ref{thm:flip->mut-of-weighted-quivers}. Actually, Theorem \ref{thm:flip->mut-of-weighted-quivers} is one of the main motivations for the present work.
\item  In \cite{FZ2}, Fomin-Zelevinsky associate a so-called \emph{diagram} to each pair $(B,D)$ as above. The quiver $\overline{Q}(\tau)$ turns out to be the diagram of $(B(\tau,w),\diag(\dtuple(\tauw)))$.
\end{enumerate}
\end{remark}

\begin{defi}
Let $\SSigma$, $\omega$ and $\tau$ be as in Definition \ref{def:Q(tau,omega)}.
Define quivers $Q'(\tauw)$ and $Q''(\tauw)$ as follows:
 \begin{enumerate}
  \itemsep 4pt \parskip 0pt \parsep 0pt

  \item
  $Q'(\tauw)$ is the full subquiver of $\overline{Q}(\tau)$ spanned by the vertices~$\tau \setminus \tau^{\omega=1}$.
  We denote by $\kappa$ the quiver morphism~$Q'(\tauw) \to Q(\tauw)$ obtained as the composition of the canonical inclusions $Q'(\tauw) \hookrightarrow \overline{Q}(\tau) \hookrightarrow Q(\tauw)$.

  \item
  $Q''(\tauw)$ is the quiver obtained from $Q'(\tauw)$ by adding, for each arc~$k \in \tau^{\omega=1}$, an arrow~$\epstau_k$ with head and tail given by the following description:
  \\ [-1em]
  \begin{itemize}
   \item
     The triangle~$\triangle$ containing $k$ either contains exactly one element~$i$ of the set $\tau \setminus \tau^{\omega=1} = Q''_0(\tauw)$ (in which case $\triangle$ either contains exactly two orbifold points of weight~$1$, or is a non-interior orbifolded triangle containing exactly one pending arc, being $k$ this one; see~\ref{Fig:unpunct_puzzle_pieces}), or $\triangle$ contains exactly two elements~$i, j$ from~$\tau \setminus \tau^{\omega=1} = Q''_0(\tauw)$ and induces an arrow~$\alpha$ of $Q(\tauw)$ that goes from one of these arcs to the other, say from $i = t(\alpha)$ to $j = h(\alpha)$.
     In the former case, we set~$h(\epstau_k) = t(\epstau_k) = i$, and, in the latter case, $h(\epstau_k) = i$ and $t(\epstau_k) = j$.
  \end{itemize}

 \end{enumerate}
\end{defi}

\begin{remark}
 Note that the quivers $Q(\tauw)$, $Q'(\tauw)$, and $Q''(\tauw)$ are connected.
\end{remark}

\begin{ex}\label{ex:all-possible-associated-quivers} In Figure \ref{Fig:pentagon_two_orb_points} we can see two triangulations $\tau$ and $\sigma$ of the pentagon with two orbifold points.
        \begin{figure}[!ht]
                \centering
                \includegraphics[scale=.2]{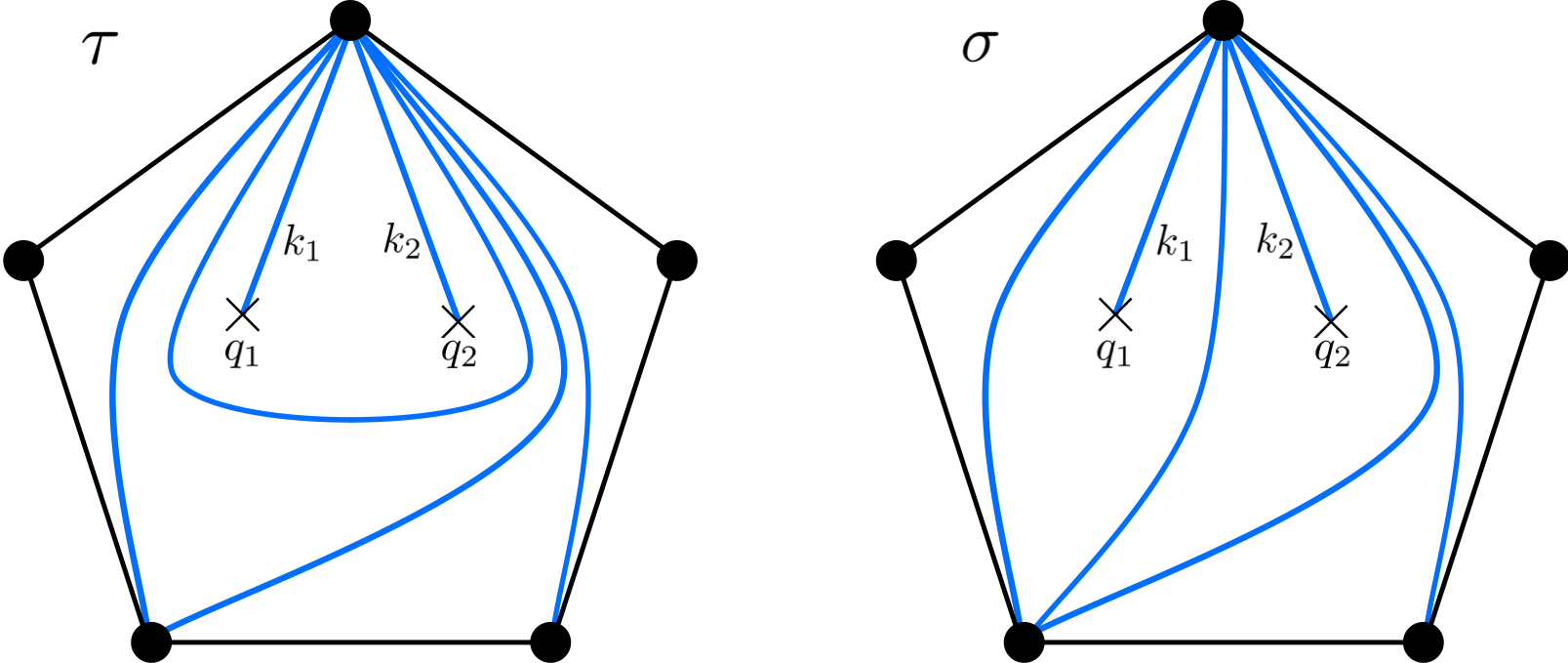}
                \caption{}
                \label{Fig:pentagon_two_orb_points}
        \end{figure}\\
The quivers $\overline{Q}(\tau)$ and $\overline{Q}(\sigma)$ are:
$$
\xymatrix{
 & k_1 \ar[rr] & & k_2 \ar[dl] & &     &  & k_1 \ar[dr] & & k_2 \ar[dr] & & \\
\overline{Q}(\tau): & & \bullet \ar[ul] \ar[dr] &  & &    \overline{Q}(\sigma):& \bullet \ar[ur] & & \bullet \ar[ll] \ar[ur] & & \bullet \ar[ll] \ar[r] &\bullet\\
& \bullet \ar[ur] & & \bullet \ar[ll] \ar[r] & \bullet   &     & & & & &
}
$$
The quivers $Q(\tauw)$, $Q'(\tauw)$, $Q''(\tauw)$, $Q(\sigma,\omega)$, $Q'(\sigma,\omega)$ and $Q''(\sigma,\omega)$ can be seen in Figures \ref{Fig:example_1_Q_Qp_Qpp} and \ref{Fig:example_2_Q_Qp_Qpp} for all possible functions $\omega:\orb\rightarrow\{1,4\}$.
        \begin{figure}[!ht]
                \centering
                \includegraphics[scale=.165]{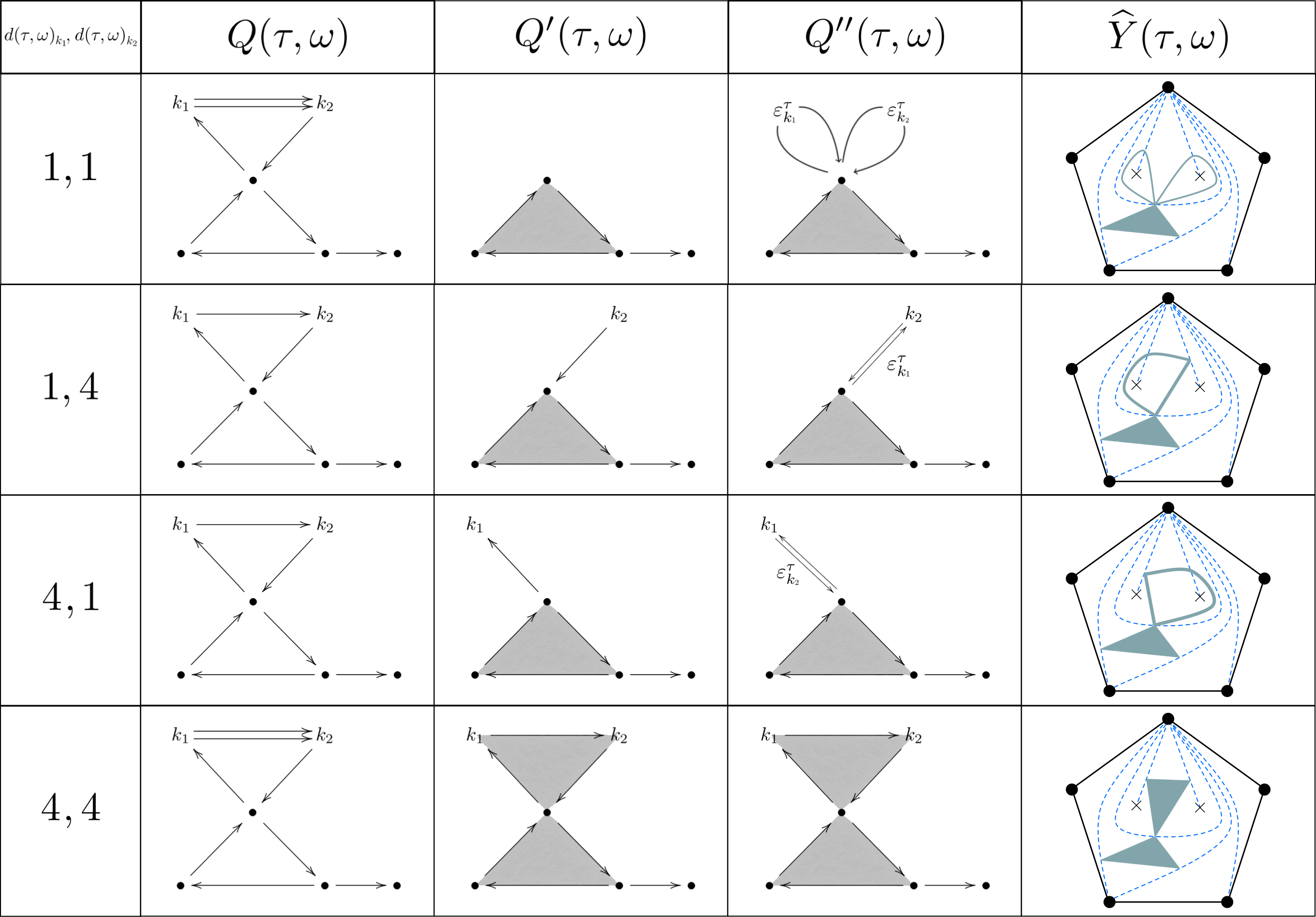}
                \caption{}
                \label{Fig:example_1_Q_Qp_Qpp}
        \end{figure}
        \begin{figure}[!ht]
                \centering
                \includegraphics[scale=.165]{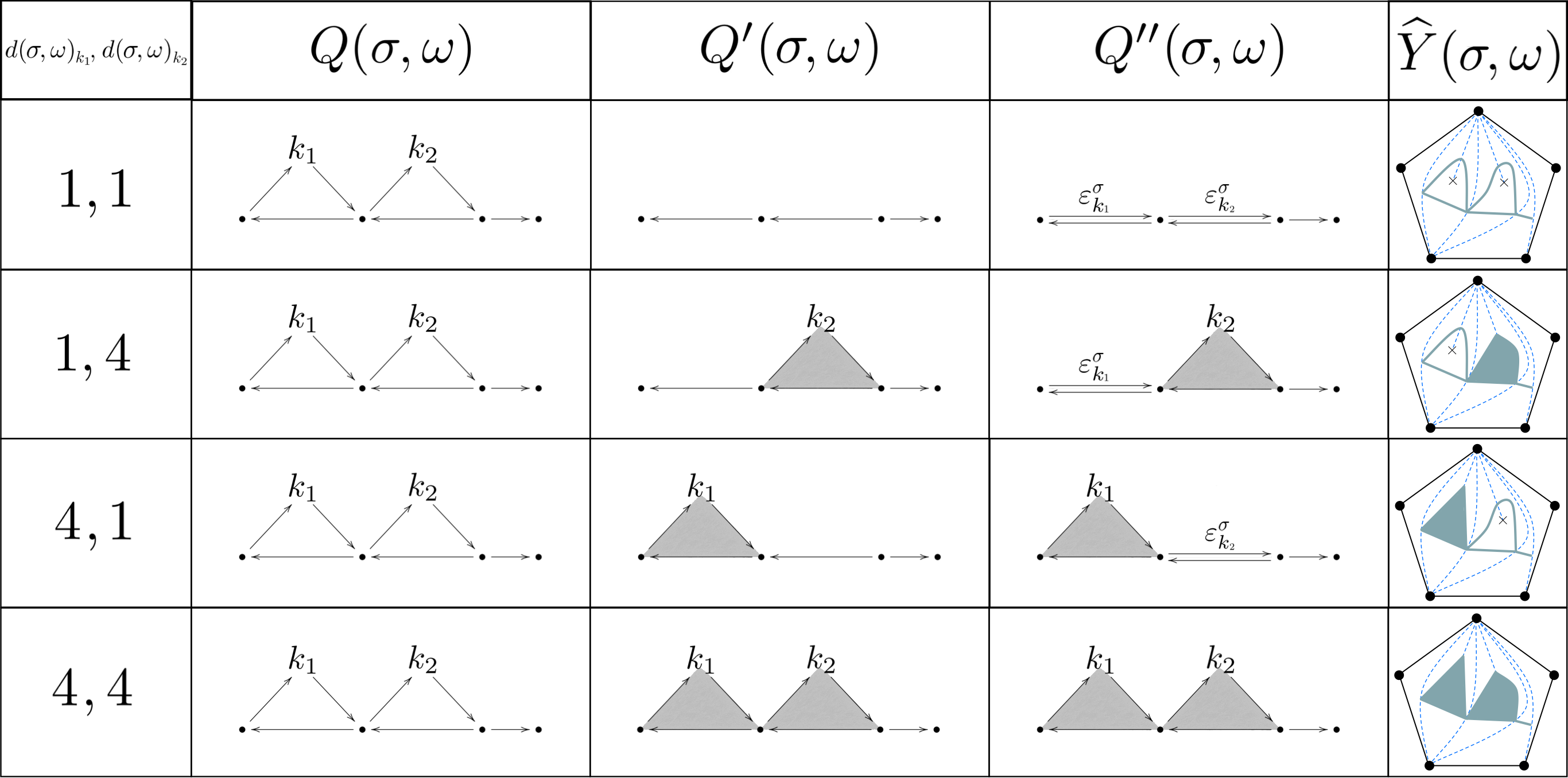}
                \caption{}
                \label{Fig:example_2_Q_Qp_Qpp}
        \end{figure}
Note that no triangle of $\sigma$ contains more than one orbifold point, hence $Q(\sigma,\omega)=\overline{Q}(\sigma)$ for every function $\omega:\orb\rightarrow\{1,4\}$.
\end{ex}

%% file: 03_chain_complexes.tex
\section{Chain complexes associated with a triangulation}

\label{sec:chain-complexes}

We begin this section with an example intended to motivate the constructions that are to come.

\begin{ex}\label{ex:pentagon-2orbs-cocycle-condition-needed} Let $\SSigma=\surf$ be the unpunctured pentagon with two orbifold points. Consider the triangulation $\sigma$ of $\SSigma$ depicted in Figure \ref{Fig:pentagon_two_orb_points} and the function $\omega:\orb\rightarrow\{1,4\}$ given by $\omega(q_1)=4$ and $\omega(q_2)=1$. Let $F$ be a field containing a primitive $4^{\operatorname{th}}$ root of unity, $E/F$ a degree-$4$ cyclic Galois extension, and $L$ the unique subfield of $E$ that has degree $2$ over $F$. In order to define a species realization of the weigted quiver $(Q(\sigma,\omega),\dtuple(\sigma,\omega))$ using this data, we still need a modulating function $g:Q_1(\sigma,\omega)\rightarrow \bigcup_{i,j\in\sigma}\Gal(F_{i,j}/F)=\Gal(E/F)\cup\Gal(L/F)\cup\Gal(F/F)$ (see \cite[Section 3]{Geuenich-Labardini-1} for definitions and notation). An easy count shows that there are $2^5$ such functions. However, \cite[Example 3.12]{Geuenich-Labardini-1} shows that for some of these $2^5$ modulating functions the corresponding species fails to admit a non-degenerate potential.

Let us be more precise about the last statement. As pointed out in Example \ref{ex:all-possible-associated-quivers}, the quiver of $(\sigma,\omega)$ is
$$
\xymatrix{
   Q(\sigma,\omega):  &  & k_1 \ar[dr]^{\beta} & & k_2 \ar[dr]^{\nu} & & \\
    & \bullet \ar[ur]^{\gamma} & & \bullet \ar[ll]^{\alpha} \ar[ur]^{\rho} & & \bullet . \ar[ll]^{\eta} \ar[r]_{\phi} &\bullet
}
$$
Let $g:Q_1(\sigma,\omega)\rightarrow \bigcup_{i,j\in\sigma}\Gal(F_{i,j}/F)$ be any modulating function for $Q(\sigma,\omega)$, and let $A_g$ be the corresponding species (cf. \cite[Definitions 3.2 and 3.3]{Geuenich-Labardini-1}, where $g$ does not appear as subindex). Then $g_\alpha,g_\beta,g_\gamma\in\Gal(L/F)\cong\mathbb{Z}/2\mathbb{Z}$. Applying \cite[Example 3.12]{Geuenich-Labardini-1} to the species $\widetilde{\mu}_{k_1}(A_g)$ shows that if $g_\alpha g_\beta g_\gamma\neq\myid_{L}$, then for any potential $S\in\RA{A_g}$ the SP $(A_g,S)$ is degenerate.

So, if we want a species realization of $(Q(\sigma,\omega),\dtuple(\sigma,\omega))$ that has the chance to admit a non-degenerate potential, we must restrict our attention to those modulating functions $g$ which satisfy $g_\alpha g_\beta g_\gamma = \myid_{L}\in\Gal(L/F)$. A moment of thought tells us that this is a cocycle condition inside some cochain complex with coefficientes in $\Gal(L/F)\cong\mathbb{Z}/2\mathbb{Z}$.

One is thus tempted to work with (the cochain complex which is $\F_2$-dual to) the chain complex with coefficients in $\F_2:=\mathbb{Z}/2\mathbb{Z}$ defined by taking $Q_0(\sigma,\omega)$ and $Q_1(\sigma,\omega)$ as $\F_2$-bases of the $0^{\operatorname{th}}$ and $1^{\operatorname{st}}$ chain groups, and the set consisting of the two ``obvious'' 3-cycles on $Q(\sigma,\omega)$ as a $\F_2$-basis of the $2^{\operatorname{nd}}$ chain group, with the differentials defined in an obvious way (so that the image of each of the two ``obvious'' 3-cycles under the $2^{\operatorname{nd}}$ differential is the sum of the three arrows that appear in it). This is, however, not quite the chain complex whose dual cochain complex we will need to consider. Indeed, since $\omega(q_2)=1$ in the current example, every modulating function $g:Q_1(\sigma,\omega)\rightarrow \bigcup_{i,j\in\sigma}\Gal(F_{i,j}/F)$ satisfies $g_\eta|_F g_\nu g_\rho=\myid_F\in\Gal(F/F)$, so there is no need to impose any no cocycle condition on the 3-cycle passing through $k_2$.
\end{ex}

Let us go back to general considerations; fix again a triangulation~$\tau$ of~$\SSigma$ and a function $\omega:\orb\rightarrow\{1,4\}$.
We define a family of sets $X_\bullet(\tauw) = (X_n(\tauw))_{n \in \N}$ by setting $X_n(\tauw) = \varnothing$ for $n \not\in \{0, 1, 2\}$ and
\begin{equation}\label{eq:def-sets-X}
 \begin{array}{lcl}
  X_0(\tauw)
  &=&
  Q'_0(\tauw)
  ,
  \\ [0.5em]
  X_1(\tauw)
  &=&
  Q'_1(\tauw)
  ,
  \\ [0.5em]
  X_2(\tauw)
  &=&
  \{ \triangle \suchthat \text{$\triangle$ is an interior triangle of $\tau$ all of whose arcs belong to $X_0(\tauw)$} \}
  .
 \end{array}
\end{equation}

We define $\Xhat_\bullet(\tauw) = (\Xhat_n(\tauw))_{n \in \N}$ as the family of sets given by
\begin{equation}\label{eq:def-Xhat}
 \Xhat_n(\tauw) =
 \begin{cases}
  Q''_1(\tauw)
  =
  X_1(\tauw) \cup \{ {{\epstau_k}}  \suchthat \text{$k \in \tau^{\omega=1}$} \}
  &
  \text{if $n = 1$,}
  \\
  X_n(\tauw)
  &
  \text{if $n \neq 1$.}
 \end{cases}
\end{equation}
Similarly to \cite[Subsection~2.2]{Amiot-Grimeland}, we can define a chain complex $\Ctauw$ associated with $X_\bullet(\tauw)$ that embeds canonically into a chain complex $\Ctauwhat$ associated with $\Xhat_\bullet(\tauw)$ as follows:
\begin{equation}\label{eq:def-chain-complexes}
 \xymatrix{
  \Ctauw :
  \:\:\:
  \cdots \ar[r] & 0
  \ar[r]^{\partial_3\phantom{/10pt/}}
  &
  \F_2 X_2(\tauw)
  \ar[r]^{\partial_2}
  \ar@{((->}[d]
  &
  \F_2 X_1(\tauw)
  \ar[r]^{\partial_1}
  \ar@{((->}[d]
  &
  \F_2 X_0(\tauw)
  \ar[r]  \ar[r]^/9pt/{\partial_0}
  \ar@{((->}[d]
  &
  0
  \\
  \Ctauwhat :
  \:\:\:
  \cdots \ar[r] & 0
  \ar[r]^{\partial_3\phantom{/10pt/}}
  &
  \F_2 \Xhat_2(\tauw)
  \ar[r]^{\partial_2}
  &
  \F_2 \Xhat_1(\tauw)
  \ar[r]^{\partial_1}
  &
  \F_2 \Xhat_0(\tauw)
  \ar[r]  \ar[r]^/9pt/{\partial_0}
  &
  0,
 }
\end{equation}
where $\F_2X$ stands for the vector space with basis $X$ over $\F_2=\Z/2\Z$.
The non-zero differentials are given on basis elements as follows:
\[
 \arraycolsep 2pt
 \begin{array}{lcll}
  \partial_2(\triangle)
  &=&
  {{\alpha}} + {{\beta}} + {{\gamma}}
  &
  \hspace{0.5em}
  \text{if \smash{$\triangle \in \Xhat_2(\tauw)$} induces
  $\alpha, \beta, \gamma \in Q'_1(\tauw)$,}
  \\ [0.5em]
  \partial_1({{\alpha}})
  &=&
  h(\alpha) - t(\alpha)
  &
  \hspace{0.5em}
  \text{for \smash{${{\alpha}} \in \Xhat_1(\tauw)$}.}
 \end{array}
\]

\begin{remark} We will use the 1-cocycles of the cochain complex which is $\F_2$-dual to $\Ctauw$ to choose modulating functions on $(Q(\tau,\omega),\dtuple(\tau,\omega))$. It is to define the flips on these 1-cocycles that we will use the auxiliary chain complex $\Ctauwhat$, for the flip rule on 1-cocycles will not always be the obvious $\F_2$-linear function one may come up with.
\end{remark}

\begin{ex} Let $\SSigma$ be the unpunctured pentagon with two orbifold points, and let $\tau$ and $\sigma$ be the triangulations of $\SSigma$ depicted in Figure \ref{Fig:pentagon_two_orb_points}. In the third column of Figure \ref{Fig:example_1_Q_Qp_Qpp}  (resp. Figure \ref{Fig:example_2_Q_Qp_Qpp}) we can visualize the sets $X_0(\tauw)$, $X_1(\tauw)$ and $X_2(\tauw)$ (resp. $X_0(\sigma,\omega)$, $X_1(\sigma,\omega)$ and $X_2(\sigma,\omega)$) for all possible functions $\omega:\orb\rightarrow\{1,4\}$, while in the fourth column we can visualize the sets $\Xhat_0(\tauw)$, $\Xhat_1(\tauw)$ and $\Xhat_2(\tauw)$ (resp. $\Xhat_0(\sigma,\omega)$, $\Xhat_1(\sigma,\omega)$ and $\Xhat_2(\sigma,\omega)$). The shaded triangular regions correspond to the elements of the respective set $\Xhat_2(\tauw)$ (or $\Xhat_2(\sigma,\omega)$).
Note that for those functions $\omega$ that take the value $1$ at least once, the dimension over $\F_2$ of the first homology group of the corresponding chain complex $\Ctauwhat$ (resp. $\widehat{C}_\bullet(\sigma,\omega)$) is strictly greater than the dimension over $\F_2$ of the first homology group of the chain complex $\Ctauw$ (resp. $C_\bullet(\sigma,\omega)$).
\end{ex}

%% file: 04_colored_triangs_and_flips.tex
\section{Colored triangulations and their flips}

\label{sec:colored-triangulations-and-flips}

Let $\SSigma=\surf$ be a surface with (arbitrarily many) orbifold points which is either unpunctured with non-empty boundary, or once-punctured closed. The main combinatorial input for our construction of species and potentials will be colored triangulations.

\begin{defi}\label{defi:colored-triangulations}
 A \emph{colored triangulation} of $\SSigmaw$ is a pair~$(\tau, \xi)$ consisting of a triangulation $\tau$ of $\SSigma$ and a $1$-cocycle~$\xi$ of the cochain complex $\CCtauwFtwo = \Hom_{\F_2}(\CtauwFtwo, \F_2)$.
\end{defi}

For each triangulation $\tau$ of $\SSigma$ and each function $\omega:\orb\rightarrow\{1,4\}$ we will denote by $\CZonetauwFtwo$ the set of $1$-cocycles of $\CCtauwFtwo$. Thus, $\CZonetauwFtwo$ is an $\F_2$-vector subspace of $C^1(\tau,\omega)$.

\begin{remark} By its very definition,  $C_1(\tau,\omega)$ is the $\F_2$-vector space with basis $Q'_1(\tauw)$. Let $\{{{\alpha}}^\vee\suchthat\alpha\in Q'_1(\tauw)\}$ be the $\F_2$-vector space basis of $C^1(\tau,\omega)=\Hom_{\F_2}(C_1(\tau,\omega), \F_2)$ which is dual to $Q'_1(\tauw)$.
Then, choosing a cocycle $\xi = \sum_{{{\alpha}}} \xi({{\alpha}}) {{\alpha}}^\vee \in \CZonetauwFtwo$ amounts to fixing, for each arrow \smash{${{\alpha}} \in Q'_1(\tauw)$}, an element $\xi({{\alpha}}) \in \{0,1\} = \F_2$ in such a way that whenever ${{\alpha}}, {{\beta}}, {{\gamma}}$ are arrows of $Q'(\tauw)$ induced by an interior triangle $\triangle$
 one has
 \[
  \xi({{\alpha}}) + \xi({{\beta}}) + \xi({{\gamma}}) = 0 \:\in\: \F_2 \,.
 \]
 As we will see in Section~\ref{sec:classification-of-nondeg-SPs}, without this condition the corresponding species would fail to admit a non-degenerate potential, a fact that was hinted already in Example \ref{ex:pentagon-2orbs-cocycle-condition-needed}.
\end{remark}

In this section we want to explain how to obtain one colored triangulation from another by \emph{flipping} an arc.
For this purpose let us fix two triangulations~$\tau$ and $\sigma$ of $\SSigma$ that are related by the flip of an arc $k \in \sigma$ in the sense of \cite{FeShTu-orbifolds}.
Our goal is to define a ``natural'' bijection $\CZonetauwFtwo \to \CZonetauwFtwo[\sigma]$.

Consider the morphism \smash{$\rho_\tau : \CtauwhatFtwo \to \CtauwFtwo$} defined as $(\rho_\tau)_n = \mathrm{id}_{\widehat{C}_n(\tauw)}$ for $n \neq 1$, and by the rule
\[
 (\rho_\tau)_1({{\alpha}}) \:=\:
 \begin{cases}
  \phantom{-}{{\alpha}} & \text{if ${{\alpha}} \in X_1(\tauw)$,}       \\
  -{{\beta}}            & \text{if ${{\alpha}} = {{\epstau_k}}$ and there exists ${{\beta}} \in Q_1'(\tauw)$ induced by a triangle containing~$k$,} \\
  \phantom{-\,\,} 0   & \text{otherwise.,}
 \end{cases}
\]
for ${{\alpha}} \in \Xhat_1(\tauw)$.
It is easy to check that~$\rho_\tau$ is indeed a morphism of chain complexes.
Let us denote by~$\rho^\tau$ the dual morphism \smash{$\CCtauwFtwo \to \CCtauwhatFtwo$} induced by $\rho_\tau$, where $\CCtauwhatFtwo = \Hom_{\F_2}(\CtauwhatFtwo, \F_2)$.

\begin{remark}
 Clearly, $\rho_\tau$ is surjective and $\rho^\tau$ injective.
 Moreover, in cohomology $\rho^\tau$ is a section of the morphism \smash{$\CHtauwhatFtwo \twoheadrightarrow \CHtauwFtwo$} induced by the canonical inclusion~\smash{$X_\bullet(\tauw) \hookrightarrow \Xhat_\bullet(\tauw)$}.
\end{remark}

Recall from \eqref{eq:def-sets-X}, \eqref{eq:def-Xhat} and \eqref{eq:def-chain-complexes} that
$\Xhat_1(\tauw)=Q'_1(\tauw)\cup \{ {{\epstau_k}}  \suchthat k \in \tau^{\omega=1}\}$ is an $\F_2$-basis of $ \widehat{C}_1(\tau,\omega)$; let $\{\alpha^\vee\suchthat\alpha\in\Xhat_1(\tauw)\}$ be the corresponding dual $\F_2$-basis of $\widehat{C}^1(\tau,\omega)=\Hom_{\F_2}(\widehat{C}_1(\tau,\omega),\F_2)$. Denote by \smash{$\CZonetauwhatFtwo$} the set of $1$-cocycles of~\smash{$\CCtauwhatFtwo$}.
We associate with each cocycle~$\xi \in \CZonetauwFtwo$ the cocycle
\[
 \xihat
 \:=\:
 \rho^\tau(\xi) + \varepsilon^\tau
 \:\in\:
 \CZonetauwhatFtwo
 \,,
\]
where $\varepsilon^\tau$ is the cocycle \smash{$\sum_{k \in \tau^{\omega=1}} ({{\epstau_k}})^\vee$}.

\begin{ex} Consider once more the triangulation $\tau$ depicted in Figure \ref{Fig:pentagon_two_orb_points}. In Figure \ref{Fig:big_cocycle} we illustrate the function $\xi\mapsto\widehat{\xi}$ for all possible functions $\omega:\orb\rightarrow\{1,4\}$.
        \begin{figure}[!ht]
                \centering
                \includegraphics[scale=.125]{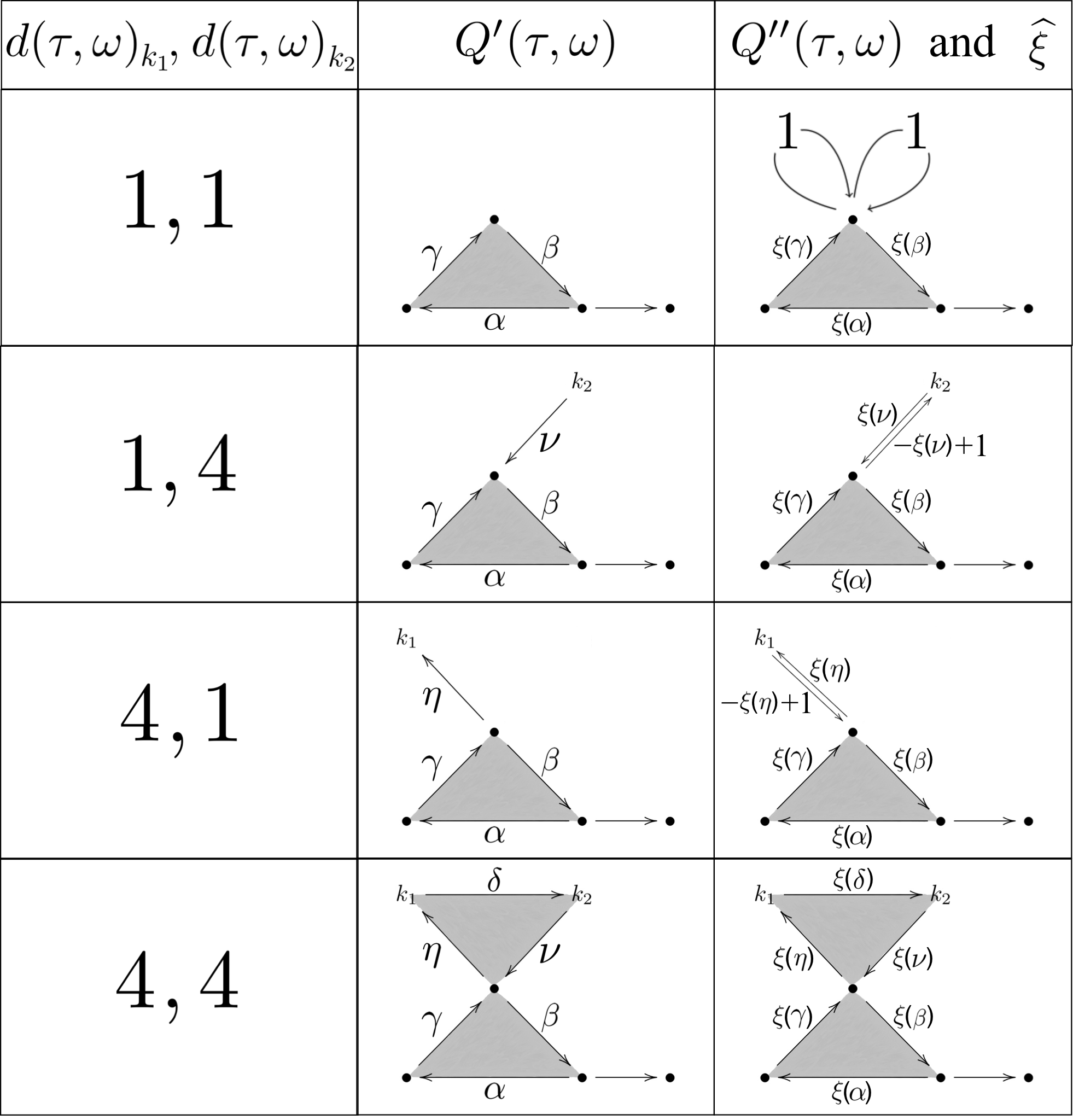}
                \caption{}
                \label{Fig:big_cocycle}
        \end{figure}
More precisely, in the left-most column we have written the possible values of $\omega$, the middle column depicts the corresponding quiver $Q'(\tau,\omega)$, with its arrows labeled with their names, and the right-most column depicts the quiver $Q''(\tau,\omega)$, with its arrows labeled not with their names, but with the values that $\widehat{\xi}$ takes at them for any given $\xi\in Z^1(\tau,\omega)$.
\end{ex}

The following easy observation will become important in Section~\ref{sec:flip-graph}.

\begin{lemma}
 \label{lem:hat-preserves-being-cohomologous}
 Let $(\tau, \xi)$ and $(\tau, \xi')$ be two colored triangulations of $\SSigmaw$.
 Then $[\xi] = [\xi']$ in $\CHonetauwFtwo$ if and only if \smash{$\big[\xihat\,\big] = \big[\xihat'\big]$} in \smash{$\CHonetauwhatFtwo$}.
\end{lemma}

\begin{proof}
 Because of the injectivity of $\rho^\tau$ in cohomology, $\big[\xihat - \myhat{\xi'}\big] = \rho^\tau([\xi - \xi']) = 0$ if and only if $[\xi - \xi'] = 0$.
\end{proof}

To compare the complexes $\CtauwhatFtwo[\sigma]$ and $\CtauwhatFtwo$, we define a morphism $\varphi_{\tau,\sigma} : \CtauwhatFtwo[\sigma] \to \CtauwhatFtwo$ as follows.
For notational simplicity we will make the identification $k = k'$ for the unique arc $k' \in \tau \setminus \sigma$, so that we have $C_0(\tauw[\sigma]) = C_0(\tauw)$.
Thus, we can set $(\varphi_{\tau,\sigma})_0 = \mathrm{id}_{C_0(\tauw[\sigma])}$.

Next, define $(\varphi_{\tau,\sigma})_2:\CtauwhatFtwo[\sigma] \to \CtauwhatFtwo$ following the rule
$$
(\varphi_{\tau,\sigma})_2(\triangle)=
\begin{cases}
\triangle & \text{if $\triangle$ does not contain $k$,}\\
0 & \text{if $\triangle$ contains $k$,}
\end{cases}
$$
for $\triangle\in\Xhat_2(\tauw[\sigma])$.

Finally, we shall define $\varphi_{\tau,\sigma}({{\alpha}})=(\varphi_{\tau,\sigma})_1({{\alpha}})$ for every ${{\alpha}} \in \Xhat_1(\tauw[\sigma])$.

For $\rho \in \{\sigma, \tau\}$ let $Q^k(\tauw[\rho])$ be the subquiver of $Q''(\tauw[\rho])$ spanned by all arrows that are induced by the triangles containing $k$,
including arrows of the form~$\epstau_{i}$ for $i$ contained in any triangle containing $k$. Notice that $Q^k(\tauw[\rho])$ may fail to be a full subquiver of $Q''(\tauw[\rho])$.

Let us call a triangle of $\rho \in \{\sigma, \tau\}$ \emph{exceptional} if 
it is an interior triangle and contains exactly one arc $\ell$ with $d(\rho,\omega)_\ell=1$.
Note that any exceptional triangle of $\rho$ is necessarily orbifolded (although it may contain two orbifold points), and induces a unique arrow ${{\deltatau[\rho]_\ell}} \in Q_1(\tauw[\rho])$ not incident to $\ell$.

For every ${{\alpha}} \in Q''_1(\tauw[\sigma]) \setminus Q^k_1(\tauw[\sigma]) = Q''_1(\tauw) \setminus Q^k_1(\tauw)$ we set $\varphi_{\tau,\sigma}({{\alpha}}) = {{\alpha}}$.

Let ${{\alpha}} \in Q^k_1(\tauw[\sigma])$.
If~$k$ is a pending arc of weight $1$ in an exceptional triangle of $\sigma$, set
\[
 \varphi_{\tau,\sigma}({{\alpha}})
 \:=\:
 \begin{cases}
  {{\epstau_{k}}}
  & \text{if ${{\alpha}} = {{\deltatau[\sigma]_k}}$,}                                                                              \\
  {{\deltatau_k}}
  & \text{if ${{\alpha}} = {{\epstau[\sigma]_k}}$.}
 \end{cases}
\]
If~$k$ is a pending arc of weight $4$ in an exceptional triangle of $\sigma$ whose other pending arc is $\ell \in \tau^{\omega=1}$, set
\[
 \varphi_{\tau,\sigma}({{\alpha}})
 \:=\:
 \begin{cases}
  -{{\deltatau_\ell}}
  & \text{if ${{\alpha}} = {{\deltatau[\sigma]_\ell}}$,}                                                                              \\
  -{{\epstau_{\ell}}}
  & \text{if ${{\alpha}} = {{\epstau[\sigma]_\ell}}$,}
 \end{cases}
\]
and, finally, if $k$ is not a pending arc in an exceptional triangle of $\sigma$, we set
\begin{equation}\label{eq:phi(alpha)-k-not-pending-in-exceptional-triangle}
 \varphi_{\tau,\sigma}({{\alpha}})
 \:=\:
 \begin{cases}
  {{\mu^\tau_\ell}} + {{\epstau_\ell}} + {{\nu^\tau_\ell}}
  & \text{if ${{\alpha}} = {{\epstau[\sigma]_\ell}}$,}
  \\
  \:\:\:\:\:\:\: -{{\alpha^*}}
  & \text{if ${{\alpha}} \in X_1(\tauw[\sigma])$ with $h({{\alpha}}) = k$ or $t({{\alpha}}) = k$,}                                                                              \\
  \:\:\:\: {{\beta^*}} + {{\gamma^*}}
  & \text{if ${{\alpha}} \in X_1(\tauw[\sigma])$ with $h({{\alpha}}) \neq k$ and $t({{\alpha}}) \neq k$, where ${{\beta}}$, ${{\gamma}}$ with $h(\gamma) = k$}
  \\
  & \raisebox{2pt}{\text{and $t(\beta) = k$ are induced by the same $\triangle \in X_2(\tauw[\sigma])$ as ${{\alpha}}$.}}
  \vspace{-2pt}
 \end{cases}
\end{equation}
Here, if ${{\alpha}} = {{\epstau[\sigma]_\ell}}$, then the elements ${{\mu^\tau_\ell}}$ and ${{\nu^\tau_\ell}}$ of $\ConetauwhatFtwo$ are given as follows:
\begin{equation}\label{eq:mu-and-nu}
 \arraycolsep 2pt
 \begin{array}{lll}
 {{\mu^\tau_\ell}}
 &=&
 \begin{cases}
  {{\mu}}
  & \text{if $t(\epstau_\ell) \neq t(\epstau[\sigma]_\ell)$ and ${{\mu}} \in Q_1(\tauw)$ with $t({{\mu}}) = t(\epstau[\sigma]_\ell)$ and $h({{\mu}}) = t(\epstau_\ell)$,}
  \\
  0
  & \text{otherwise,}
 \end{cases}
 \\ [1.7em]
 {{\nu^\tau_\ell}}
 &=&
 \begin{cases}
  {{\nu}}
  & \text{if $h(\epstau_\ell) \neq h(\epstau[\sigma]_\ell)$ and ${{\nu}} \in Q_1(\tauw)$ with $t({{\nu}}) = h(\epstau_\ell)$ and $h({{\nu}}) = h(\epstau[\sigma]_\ell)$,}
  \\
  0
  & \text{otherwise.}
 \end{cases}
 \end{array}
\end{equation}
The following lemma will be essential later on.

\begin{lemma}
 \label{lem:phi-sigma-tau-induces-isos}
 Let $\sigma$ be a triangulation of $\SSigma$ related to $\tau$ by the flip of an arc~$k$. The $\F_2$-linear maps we have just defined constitute a homomorphism of chain complexes $\varphi_{\tau,\sigma}:\widehat{C}_\bullet(\sigma,\omega)\rightarrow\widehat{C}_\bullet(\tau,\omega)$. This chain complex homomorphism is a homotopy equivalence, hence
the $\F_2$-linear maps $\varphi_{\tau,\sigma}^*=(\varphi_{\tau,\sigma})_n^*:\Hom_{\F_2}(\widehat{C}_n(\tau,\omega),\F_2)\rightarrow\Hom_{\F_2}(\widehat{C}_n(\sigma,\omega),\F_2)$ given by $\varphi_{\tau,\sigma}^*(f)=f\circ\varphi_{\tau,\sigma}$ 
induce an isomorphism in cohomology
 \[
  H^\bullet(\widehat{C}^\bullet(\tau,\omega))
  \xrightarrow{\hspace{15pt}}
  H^\bullet(\widehat{C}^\bullet(\sigma,\omega)),
 \]
 whose inverse is induced by $\varphi_{\sigma,\tau}^*$.
 Actually, $\varphi_{\tau,\sigma}$ and $\varphi_{\sigma,\tau}$ already induce a pair of inverse isomorphisms
 \[
  \xymatrix{
   \CZonetauwhatFtwo
   \ar@<2pt>[rr]^{\varphi_{\tau,\sigma}^*}
   &&
   \CZonetauwhatFtwo[\sigma]
   \ar@<2pt>[ll]^{\varphi^*_{\sigma,\tau}}
  }
  .
 \]
\end{lemma}


Before turning to the proof of Lemma~\ref{lem:phi-sigma-tau-induces-isos}, let us shortly illustrate and motivate the definition of $\varphi_{\tau,\sigma}$.
For simplicity, we assume for this illustration that the triangles of $\sigma$ with side~$k$ are interior.

As will be pointed out in the proof of Theorem~\ref{thm:flip<->SP-mutation}, there are in principle (up to interchanging $\sigma$ and $\tau$) $24$ possibilities for $(Q^k(\tauw[\sigma]),Q^k(\tauw))$ corresponding to the configurations
 $1$, $2$, $3$, $4$, $5$, $6$,
 $11$, $13$, $14$,
 $18$, $19$, $20$, $21$, $24$, $25 \:\widehat{=}\: 26$, $27$,
 $32$, $34$, $36$, $37$, $38$, $39$, $44$, $45$
 depicted in Figure~\ref{Fig:all_possibilities_for_k_and_weights}.
These configurations can be grouped together as done in Figure~\ref{fig:cases-phi-tau-sigma}, leaving us with $9$ cases.
The action of $\varphi_{\tau,\sigma}$ on $Q^k(\tauw[\sigma])$ and of $\varphi_{\sigma,\tau}$ on $Q^k(\tauw)$ has been unraveled in the last two columns of the table.

\begin{remark}
 To get a conceptual idea why $\varphi_{\tau,\sigma}$ ``should'' be defined as we just did it and why Lemma~\ref{lem:phi-sigma-tau-induces-isos} is true, you might want to take a look at Proposition~\ref{prop:theta-and-phi-compose-to-theta} and the discussion in Section~\ref{sec:relation-to-surface-homology} preceding it.
\end{remark}

\begin{figure}
 \centering

 \caption{}
 \label{fig:cases-phi-tau-sigma}

  ~ \\

 \newcolumntype{^}{@{\vline width 1pt \hskip\tabcolsep}}
 \newcolumntype{"}{@{\hskip\tabcolsep \vline width 1pt}}
 \newcolumntype{M}{ >{\centering\arraybackslash} m{4cm} }
 \newcolumntype{N}{ >{\centering\arraybackslash} m{2.9cm} }
 \scalebox{0.88}{
 \begin{tabular}{^m{0.65cm}|m{1.53cm}|M|M|N|N"}
  \Xhline{2\arrayrulewidth}

  Case
  &
  Configs.
  &
  $Q^k(\tauw[\sigma])$
  &
  $Q^k(\tauw)$
  &
  $\varphi_{\tau,\sigma}$
  &
  $\varphi_{\sigma,\tau}$
  \\
  \Xhline{2\arrayrulewidth}

  I
  &
  1, 4
  &
  \vspace{3pt}
  \scalebox{0.63}{
  \begin{tikzpicture}[scale=0.8]

    \begin{scope}[every node/.style={circle, inner sep=2pt}]
    \node (i) at (-3, 0) {$i$};
    \node (j) at ( 0, 0) {$j$};
    \end{scope}

    \begin{scope}[
                  every node/.style={circle, fill=white},
                  every edge/.style={draw, thick}
                ]
    \path [->, bend left] (i) edge node {${{\deltatau[\sigma]}}$} (j);
    \path [->, bend left] (j) edge node {$\epstau[\sigma]$}   (i);
    \end{scope}

  \end{tikzpicture}
  }
  &
  \vspace{3pt}
  \scalebox{0.63}{
  \begin{tikzpicture}[scale=0.8]

    \begin{scope}[every node/.style={circle, inner sep=2pt}]
    \node (i) at (-3, 0) {$i$};
    \node (j) at ( 0, 0) {$j$};
    \end{scope}

    \begin{scope}[
                  every node/.style={circle, fill=white},
                  every edge/.style={draw, thick}
                ]
    \path [->, bend left] (i) edge node {$\epstau$}   (j);
    \path [->, bend left] (j) edge node {${{\deltatau}}$} (i);
    \end{scope}

  \end{tikzpicture}
  }
  &
  \scalebox{0.8}{
  $
  \arraycolsep 2pt
  \begin{array}{lcl}
   {{\deltatau[\sigma]}}
   &\mapsto&
   \epstau
   \\
   \epstau[\sigma]
   &\mapsto&
   {{\deltatau}}
  \end{array}
  $
  }
  &
  \scalebox{0.8}{
  $
  \arraycolsep 2pt
  \begin{array}{lcl}
   {{\deltatau}}
   &\mapsto&
   \epstau[\sigma]
   \\
   \epstau
   &\mapsto&
   {{\deltatau[\sigma]}}
  \end{array}
  $
  }
  \\
  \hline

  II
  &
  3
  &
  \vspace{-11pt}
  \scalebox{0.63}{
  \begin{tikzpicture}[scale=0.8]

    \begin{scope}[every node/.style={circle, inner sep=2pt}]
    \node (i) at ( 0, 0) {$i$};
    \end{scope}

    \begin{scope}[
                  every node/.style={circle, fill=white},
                  every edge/.style={draw, thick}
                ]
    \path [->, out=218, in=142, looseness=20] (i) edge node {$\epstau[\sigma]_0$} (i);
    \path [->, out= 38, in=-38, looseness=20] (i) edge node {$\epstau[\sigma]_1$} (i);
    \end{scope}

  \end{tikzpicture}
  }
  \vspace{-13pt}
  &
  \vspace{-11pt}
  \scalebox{0.63}{
  \begin{tikzpicture}[scale=0.8]

    \begin{scope}[every node/.style={circle, inner sep=2pt}]
    \node (i) at ( 0, 0) {$i$};
    \end{scope}

    \begin{scope}[
                  every node/.style={circle, fill=white},
                  every edge/.style={draw, thick}
                ]
    \path [->, out=218, in=142, looseness=20] (i) edge node {$\epstau_0$} (i);
    \path [->, out= 38, in=-38, looseness=20] (i) edge node {$\epstau_1$} (i);
    \end{scope}

  \end{tikzpicture}
  }
  \vspace{-13pt}
  &
  \scalebox{0.8}{
  $
  \arraycolsep 2pt
  \begin{array}{lcl}
   \epstau[\sigma]_i
   &\mapsto&
   \epstau_i
  \end{array}
  $
  }
  &
  \scalebox{0.8}{
  $
  \arraycolsep 2pt
  \begin{array}{lcl}
   \epstau_i
   &\mapsto&
   \epstau[\sigma]_i
  \end{array}
  $
  }
  \\
  \hline

  III
  &
  5
  &
  \vspace{3pt}
  \scalebox{0.63}{
  \begin{tikzpicture}[scale=0.8]

    \begin{scope}[every node/.style={circle, inner sep=2pt}]
    \node (i) at (-3, 0) {$i$};

    \node[draw, thick, outer sep=3pt, fill=black!10!white] (k) at ( 0, 0) {$k$};
    \end{scope}

    \begin{scope}[
                  every node/.style={circle, fill=white},
                  every edge/.style={draw, thick}
                ]
    \path [->, bend left] (i) edge node {${{\deltatau[\sigma]}}$} (k);
    \path [->, bend left] (k) edge node {$\epstau[\sigma]$}   (i);
    \end{scope}

  \end{tikzpicture}
  }
  &
  \vspace{3pt}
  \scalebox{0.63}{
  \begin{tikzpicture}[scale=0.8]

    \begin{scope}[every node/.style={circle, inner sep=2pt}]
    \node (i) at (-3, 0) {$i$};

    \node[draw, thick, outer sep=3pt, fill=black!10!white] (k) at ( 0, 0) {$k$};
    \end{scope}

    \begin{scope}[
                  every node/.style={circle, fill=white},
                  every edge/.style={draw, thick}
                ]
    \path [->, bend left] (i) edge node {$\epstau$}   (k);
    \path [->, bend left] (k) edge node {${{\deltatau}}$} (i);
    \end{scope}

  \end{tikzpicture}
  }
  &
  \scalebox{0.8}{
  $
  \arraycolsep 2pt
  \begin{array}{lcl}
   {{\deltatau[\sigma]}}
   &\mapsto&
   -{{\deltatau}}
   \\
   \epstau[\sigma]
   &\mapsto&
   -\epstau
  \end{array}
  $
  }
  &
  \scalebox{0.8}{
  $
  \arraycolsep 2pt
  \begin{array}{lcl}
   {{\deltatau}}
   &\mapsto&
   -{{\deltatau[\sigma]}}
   \\
   \epstau
   &\mapsto&
   -\epstau[\sigma]
  \end{array}
  $
  }
  \\
  \hline

  IV
  &
  2, 6
  &
  \scalebox{0.63}{
  \begin{tikzpicture}[scale=0.6]

    \begin{scope}[every node/.style={circle, inner sep=2pt}]
    \node (i) at (-3, 3) {$i$};
    \node (j) at ( 3, 3) {$j$};

    \node[draw, thick, outer sep=3pt, fill=black!10!white] (k) at ( 0, 0) {$k$};
    \end{scope}

    \begin{scope}[
                  every node/.style={circle, fill=white},
                  every edge/.style={draw, thick}
                ]
    \path [->] (i) edge node {${{\alpha_0}}$} (j);
    \path [->] (j) edge node {${{\gamma_0}}$} (k);
    \path [->] (k) edge node {${{\beta_0}}$}  (i);
    \end{scope}

  \end{tikzpicture}
  }
  &
  \scalebox{0.63}{
  \begin{tikzpicture}[scale=0.6]

    \begin{scope}[every node/.style={circle, inner sep=2pt}]
    \node (i) at (-3, 3) {$i$};
    \node (j) at ( 3, 3) {$j$};

    \node[draw, thick, outer sep=3pt, fill=black!10!white] (k) at ( 0, 0) {$k$};
    \end{scope}

    \begin{scope}[
                  every node/.style={circle, fill=white},
                  every edge/.style={draw, thick}
                ]
    \path [->] (j) edge node {${{\delta_0}}$}   (i);
    \path [->] (k) edge node {${{\gamma_0^*}}$} (j);
    \path [->] (i) edge node {${{\beta_0^*}}$}  (k);
    \end{scope}

  \end{tikzpicture}
  }
  &
  &
  \bigstrut
  \\
  \cline{1-4}

  V
  &
  11, 13, 19, 27, 39
  &
  \vspace{1pt}
  \scalebox{0.63}{
  \begin{tikzpicture}[scale=0.6]

    \begin{scope}[every node/.style={circle, inner sep=2pt}]
    \node (l) at (-3,-3) {$l$};
    \node (i) at (-3, 3) {$i$};
    \node (m) at ( 3,-3) {$m$};
    \node (j) at ( 3, 3) {$j$};

    \node[draw, thick, outer sep=3pt, fill=black!10!white] (k) at ( 0, 0) {$k$};
    \end{scope}

    \begin{scope}[
                  every node/.style={circle, fill=white},
                  every edge/.style={draw, thick}
                ]
    \path [->] (i) edge node {${{\alpha_0}}$} (j);
    \path [->] (j) edge node {${{\gamma_0}}$} (k);
    \path [->] (k) edge node {${{\beta_0}}$}  (i);
    \path [->] (m) edge node {${{\alpha_1}}$} (l);
    \path [->] (l) edge node {${{\gamma_1}}$} (k);
    \path [->] (k) edge node {${{\beta_1}}$}  (m);
    \end{scope}

  \end{tikzpicture}
  }
  \vspace{-3pt}
  &
  \scalebox{0.63}{
  \begin{tikzpicture}[scale=0.6]

    \begin{scope}[every node/.style={circle, inner sep=2pt}]
    \node (l) at (-3,-3) {$l$};
    \node (i) at (-3, 3) {$i$};
    \node (m) at ( 3,-3) {$m$};
    \node (j) at ( 3, 3) {$j$};

    \node[draw, thick, outer sep=3pt, fill=black!10!white] (k) at ( 0, 0) {$k$};
    \end{scope}

    \begin{scope}[
                  every node/.style={circle, fill=white},
                  every edge/.style={draw, thick}
                ]
    \path [->] (l) edge node {${{\delta_0}}$}   (i);
    \path [->] (i) edge node {${{\beta_0^*}}$}  (k);
    \path [->] (k) edge node {${{\gamma_1^*}}$} (l);
    \path [->] (j) edge node {${{\delta_1}}$}   (m);
    \path [->] (m) edge node {${{\beta_1^*}}$}  (k);
    \path [->] (k) edge node {${{\gamma_0^*}}$} (j);
    \end{scope}

  \end{tikzpicture}
  }
  \vspace{-3pt}
  &
  &
  \bigstrut
  \\
  \cline{1-4}

  VI
  &
  14, 18, 21, 25, 37, 38, 45
  &
  \vspace{3pt}
  \scalebox{0.63}{
  \begin{tikzpicture}[scale=0.6]

    \begin{scope}[every node/.style={circle, inner sep=2pt}]
    \node (i) at (-4, 0) {$i$};
    \node (j) at ( 3,-3) {$j$};
    \node (l) at ( 3, 3) {$l$};

    \node[draw, thick, outer sep=3pt, fill=black!10!white] (k) at ( 0, 0) {$k$};
    \end{scope}

    \begin{scope}[
                  every node/.style={circle, fill=white},
                  every edge/.style={draw, thick}
                ]
    \path [->, bend left] (i) edge node {$\epstau[\sigma]_0$} (k);
    \path [->, bend left] (k) edge node {${{\beta_0}}$}         (i);

    \path [->] (l) edge node {${{\alpha_1}}$} (j);
    \path [->] (j) edge node {${{\gamma_1}}$} (k);
    \path [->] (k) edge node {${{\beta_1}}$}  (l);
    \end{scope}

  \end{tikzpicture}
  }
  &
  \vspace{3pt}
  \scalebox{0.63}{
  \begin{tikzpicture}[scale=0.6]

    \begin{scope}[every node/.style={circle, inner sep=2pt}]
    \node (i) at (-3, 3) {$i$};
    \node (j) at (-3,-3) {$j$};
    \node (l) at ( 4, 0) {$l$};

    \node[draw, thick, outer sep=3pt, fill=black!10!white] (k) at ( 0, 0) {$k$};
    \end{scope}

    \begin{scope}[
                  every node/.style={circle, fill=white},
                  every edge/.style={draw, thick}
                ]
    \path [->] (j) edge node {${{\delta_0}}$}   (i);
    \path [->] (i) edge node {${{\beta_0^*}}$}  (k);
    \path [->] (k) edge node {${{\gamma_1^*}}$} (j);

    \path [->, bend left] (l) edge node {${{\beta_1^*}}$} (k);
    \path [->, bend left] (k) edge node {$\epstau_0$}   (l);
    \end{scope}

  \end{tikzpicture}
  }
  &
  &
  \bigstrut
  \\
  \cline{1-4}

  VII
  &
  24, 36
  &
  \vspace{3pt}
  \scalebox{0.63}{
  \begin{tikzpicture}[scale=0.8]

    \begin{scope}[every node/.style={circle, inner sep=2pt}]
    \node (i) at (-3, 0) {$i$};
    \node (j) at ( 3, 0) {$j$};

    \node[draw, thick, outer sep=3pt, fill=black!10!white] (k) at ( 0, 0) {$k$};
    \end{scope}

    \begin{scope}[
                  every node/.style={circle, fill=white},
                  every edge/.style={draw, thick}
                ]
    \path [->, bend left] (i) edge node {$\epstau[\sigma]_0$}  (k);
    \path [->, bend left] (k) edge node {${{\beta_0}}$} (i);

    \path [->, bend left] (j) edge node {$\epstau[\sigma]_1$}  (k);
    \path [->, bend left] (k) edge node {${{\beta_1}}$} (j);
    \end{scope}

  \end{tikzpicture}
  }
  &
  \vspace{3pt}
  \scalebox{0.63}{
  \begin{tikzpicture}[scale=0.8]

    \begin{scope}[every node/.style={circle, inner sep=2pt}]
    \node (i) at (-3, 0) {$i$};
    \node (j) at ( 3, 0) {$j$};

    \node[draw, thick, outer sep=3pt, fill=black!10!white] (k) at ( 0, 0) {$k$};
    \end{scope}

    \begin{scope}[
                  every node/.style={circle, fill=white},
                  every edge/.style={draw, thick}
                ]
    \path [->, bend left] (i) edge node {${{\beta_0^*}}$}          (k);
    \path [->, bend left] (k) edge node {$\epstau_1$} (i);

    \path [->, bend left] (j) edge node {${{\beta_1^*}}$}          (k);
    \path [->, bend left] (k) edge node {$\epstau_0$} (j);
    \end{scope}

  \end{tikzpicture}
  }
  &
  \multirow{-4}[8]*[8.0em]{
  \scalebox{0.8}{
  $
  \arraycolsep 2pt
  \begin{array}{lcl}
   \epstau[\sigma]_i
   &\mapsto&
   {{\beta_i^*}} + \epstau_i + {{\beta_{1-i}^*}}
   \\ [0.5em]
   {{\alpha_i}}
   &\mapsto&
   {{\beta_i^*}} + {{\gamma_i^*}}
   \\ [0.5em]
   {{\beta_i}}
   &\mapsto&
   -{{\beta_i^*}}
   \\
   {{\gamma_i}}
   &\mapsto&
   -{{\gamma_i^*}}
  \end{array}
  $
  }
  }
  &
  \multirow{-4}[8]*[8.0em]{
  \scalebox{0.8}{
  $
  \arraycolsep 2pt
  \begin{array}{lcl}
   \epstau_i
   &\mapsto&
    {{\beta_i}} + \epstau[\sigma]_i + {{\beta_{1-i}}}
   \\ [0.5em]
   {{\delta_i}}
   &\mapsto&
   {{\gamma_{1-i}}} + {{\beta_i}}
   \\ [0.5em]
   {{\gamma_i^*}}
   &\mapsto&
   -{{\gamma_i}}
   \\
   {{\beta_i^*}}
   &\mapsto&
   -{{\beta_i}}
  \end{array}
  $
  }
  }
  \\
  \hline

  VIII
  &
  20, 34, 44
  &
  \scalebox{0.63}{
  \begin{tikzpicture}[scale=0.8]

    \begin{scope}[every node/.style={circle, inner sep=2pt}]
    \node (i) at (-3, 0) {$i$};
    \node (j) at ( 3, 0) {$j$};

    \node[draw, thick, outer sep=3pt, fill=black!10!white] (k) at ( 0, 0) {$k$};
    \end{scope}

    \begin{scope}[
                  every node/.style={circle, fill=white},
                  every edge/.style={draw, thick}
                ]
    \path [->, bend left] (i) edge node {$\epstau[\sigma]_0$} (k);
    \path [->, bend left] (k) edge node {${{\beta_0}}$}           (i);

    \path [->, bend left] (j) edge node {${{\gamma_1}}$}          (k);
    \path [->, bend left] (k) edge node {$\epstau[\sigma]_1$} (j);
    \end{scope}

  \end{tikzpicture}
  }
  &
  \hspace{-2em}
  \parbox{4cm}{
  \vspace{-2.2em}
  \scalebox{0.63}{
  \begin{tikzpicture}[scale=0.8]

    \begin{scope}[every node/.style={circle, inner sep=2pt}]
    \node (i) at (-3, 0) {$i$};
    \node (j) at ( 3, 0) {$j$};

    \node[draw, thick, outer sep=3pt, fill=black!10!white] (k) at ( 0, 0) {$k$};
    \end{scope}

    \begin{scope}[
                  every node/.style={circle, fill=white},
                  every edge/.style={draw, thick}
                ]
    \path [->] (i) edge node {${{\beta_0^*}}$}  (k);
    \path [->] (k) edge node {${{\gamma_1^*}}$} (j);

    \path [->, out=-135, in=-45] (j) edge node {${{\delta_0}}$} (i);

    \path [->, out=165, in= 98, looseness=20] (k) edge node {$\epstau_0$} (k);
    \path [->, out= 82, in= 15, looseness=20] (k) edge node {$\epstau_1$} (k);
    \end{scope}

  \end{tikzpicture}
  }
  }
  &
  \scalebox{0.8}{
  $
  \arraycolsep 2pt
  \begin{array}{lclll}
   \epstau[\sigma]_0
   &\mapsto&
   {{\beta_0^*}} \: + &\epstau_0
   \\
   \epstau[\sigma]_1
   &\mapsto&
   &\epstau_1 + {{\gamma_1^*}}
   \\ [0.5em]
   {{\beta_i}}
   &\mapsto&
   -{{\beta_i^*}}
   \\
   {{\gamma_i}}
   &\mapsto&
   -{{\gamma_i^*}}
  \end{array}
  $
  }
  &
  \scalebox{0.8}{
  $
  \arraycolsep 2pt
  \begin{array}{lclll}
   \epstau_0
   &\mapsto&
   {{\beta_0}} \: + &\epstau[\sigma]_0
   \\
   \epstau_1
   &\mapsto&
   &\epstau[\sigma]_1 + {{\gamma_1}}
   \\ [0.5em]
   {{\delta_i}}
   &\mapsto&
   {{\gamma_{1-i}}} + {{\beta_i}}
   \\ [0.5em]
   {{\gamma_i^*}}
   &\mapsto&
   -{{\gamma_i}}
   \\
   {{\beta_i^*}}
   &\mapsto&
   -{{\beta_i}}
  \end{array}
  $
  }
  \\
  \hline

  IX
  &
  32
  &
  \parbox{4cm}{
  \vspace{-2.2em}
  \scalebox{0.63}{
  \begin{tikzpicture}[scale=0.8]

    \begin{scope}[every node/.style={circle, inner sep=2pt}]
    \node (i) at (-3, 0) {$i$};

    \node[draw, thick, outer sep=3pt, fill=black!10!white] (k) at ( 0, 0) {$k$};
    \end{scope}

    \begin{scope}[
                  every node/.style={circle, fill=white},
                  every edge/.style={draw, thick}
                ]
    \path [->, bend left] (i) edge node {$\epstau[\sigma]_0$}  (k);
    \path [->, bend left] (k) edge node {${{\beta_0}}$} (i);

    \path [->, out= 75, in=  8, looseness=20] (k) edge node {$\epstau[\sigma]_1$} (k);
    \path [->, out= -8, in=-75, looseness=20] (k) edge node {$\epstau[\sigma]_2$} (k);
    \end{scope}

  \end{tikzpicture}
  }
  \vspace{-2.2em}
  }
  &
  \parbox{4cm}{
  \vspace{-2.2em}
  \scalebox{0.63}{
  \begin{tikzpicture}[scale=0.8]

    \begin{scope}[every node/.style={circle, inner sep=2pt}]
    \node (i) at (-3, 0) {$i$};

    \node[draw, thick, outer sep=3pt, fill=black!10!white] (k) at ( 0, 0) {$k$};
    \end{scope}

    \begin{scope}[
                  every node/.style={circle, fill=white},
                  every edge/.style={draw, thick}
                ]
    \path [->, bend left] (i) edge node {${{\beta_0^*}}$}          (k);
    \path [->, bend left] (k) edge node {$\epstau_1$} (i);

    \path [->, out= 75, in=  8, looseness=20] (k) edge node {$\epstau_0$} (k);
    \path [->, out= -8, in=-75, looseness=20] (k) edge node {$\epstau_2$} (k);
    \end{scope}

  \end{tikzpicture}
  }
  \vspace{-2.2em}
  }
  &
  \scalebox{0.8}{
  $
  \arraycolsep 2pt
  \begin{array}{lclll}
   \epstau[\sigma]_0
   &\mapsto&
   {{\beta_0^*}} \: + &\epstau_0
   \\
   \epstau[\sigma]_1
   &\mapsto&
   &\epstau_1 + {{\beta_0^*}}
   \\
   \epstau[\sigma]_2
   &\mapsto&
   &\epstau_2
   \\ [0.5em]
   {{\beta_i}}
   &\mapsto&
   -{{\beta_i^*}}
  \end{array}
  $
  }
  &
  \scalebox{0.8}{
  $
  \arraycolsep 2pt
  \begin{array}{lclll}
   \epstau_0
   &\mapsto&
   {{\beta_0}} \: + &\epstau[\sigma]_0
   \\
   \epstau_1
   &\mapsto&
   &\epstau[\sigma]_1 + {{\beta_0}}
   \\
   \epstau_2
   &\mapsto&
   &\epstau[\sigma]_2
   \\ [0.5em]
   {{\beta_i^*}}
   &\mapsto&
   -{{\beta_i}}
  \end{array}
  $
  }
  \\
  \Xhline{2\arrayrulewidth}
 \end{tabular}
 }
\end{figure}

\begin{proof}[Proof of Lemma~\ref{lem:phi-sigma-tau-induces-isos}]
That $\varphi_{\tau,\sigma}=((\varphi_{\tau,\sigma})_n)_{n\in\mathbb{Z}}$ is a homomorphism of chain complexes and actually a homotopy equivalence we leave in the hands of the reader. From this we immediately deduce that $\varphi_{\tau,\sigma}$ induces isomorphisms in cohomology; it remains to show that the inverses are induced by $\varphi_{\sigma,\tau}$ and that $\varphi_{\tau,\sigma}$ is already an isomorphism at the level of 1-cocycles.

 Observe that $k$ is a pending arc of weight $d$ in an exceptional triangle of $\sigma$ if and only if $k$ is a pending arc of weight $d$ in an exceptional triangle of $\tau$.
 Because of this, it is clear from the definition that $\varphi_{\tau,\sigma} : \ConetauwhatFtwo[\sigma] \to \ConetauwhatFtwo$ and $\varphi_{\sigma,\tau} : \ConetauwhatFtwo \to \ConetauwhatFtwo[\sigma]$ are inverse to each other if $k$ is such an arc.
 In particular, the lemma holds trivially true in this case.

 From now on, we assume $k$ is not a pending arc in an exceptional triangle.
 For ${{\alpha}} \in X_1(\tauw[\sigma]) \cap Q^k_1(\tauw[\sigma])$ with $h(\alpha) = k$ or $t(\alpha) = k$ one clearly has $\varphi_{\sigma,\tau} \circ \varphi_{\tau,\sigma}({{\alpha}}) = {{\alpha}}$ since $(\alpha^*)^* = \alpha$.
 The same is true for all arrows of the form $\alpha = \epstau[\sigma]_\ell \in Q^k_1(\tauw[\sigma])$ using the observations that $\lambda^\tau_\ell = 0 \Leftrightarrow \lambda^\sigma_\ell = 0$ and that $\lambda^\tau_\ell \neq 0 \Rightarrow \lambda^\tau_\ell = (\lambda^\sigma_\ell)^*$ for $\lambda \in \{\mu, \nu\}$, where $\mu$ and $\nu$ are given by \eqref{eq:mu-and-nu}.
 Finally, if ${{\alpha}} \in X_1(\tauw[\sigma]) \cap Q^k_1(\tauw[\sigma])$ with $h(\alpha) \neq k$ and $t(\alpha) \neq k$, let ${{\beta}}, {{\gamma}} \in X_1(\tauw[\sigma]) \cap Q^k_1(\tauw[\sigma])$ with $h(\gamma) = k$ and $t(\beta) = k$ be the arrows induced by the same triangle $\triangle \in X_2(\tauw[\sigma])$ that induces $\alpha$.
 Then $\varphi_{\sigma,\tau} \circ \varphi_{\tau,\sigma}({{\alpha}}) = \varphi_{\sigma,\tau}(\beta^* + \gamma^*) = -(\beta + \gamma)$.

 It remains to verify the last part of the lemma.
 For this, we observe that the morphism $\varphi_{\tau,\sigma}^* : \CConetauwFtwo \to \CConetauwFtwo[\sigma]$ induced by $\varphi_{\tau,\sigma}$ acts as
 $\varphi_{\tau,\sigma}^*\big(({{\epstau_\ell}})^\vee\big) = ({{\epstau[\sigma]_\ell}})^\vee$ and
 $\varphi_{\tau,\sigma}^*\big({{\alpha}}^\vee\big) = 0$ for ${{\alpha}} \in Q^k_1(\tauw) \cap X_1(\tauw)$ with $h({{\alpha}}) \neq k$ and $t({{\alpha}}) \neq k$, whereas for ${{\alpha}} \in Q^k_1(\tauw) \cap X_1(\tauw)$ with $h({{\alpha}}) = k$ or $t({{\alpha}}) = k$ the action is given by
 \[
  \varphi_{\tau,\sigma}^*\big({{\alpha}}^\vee\big)
  \:=\:
  \hspace{20pt}
  \sum_{\mathclap{\ell : \alpha \in \{\mu^\tau_\ell, \nu^\tau_\ell\}}}
  ({{\epstau[\sigma]_\ell}})^\vee
  \hspace{5pt}
  +
  \hspace{15pt}
  \sum_{\mathclap{\substack{\triangle \in X_2(\tauw[\sigma]) \\ \text{inducing $\alpha^*$}}}} {{\delta}}_\triangle^\vee
  \hspace{5pt}
  -
  \hspace{3pt}
  ({{\alpha^*}})^\vee,
 \]
 where $\delta_\triangle$ denotes the unique arrow $\delta \in Q^k(\tauw[\sigma])$ with $k \not\in \{h(\delta), t(\delta)\}$ induced by the triangle $\triangle \in X_2(\tauw[\sigma])$ containing $k$.
The action of $\varphi_{\sigma,\tau}^*$ is explicitly determined by the same rules.

 Now we are ready to check $\varphi_{\sigma,\tau}^* \circ \varphi_{\tau,\sigma}^*(\xi) = \xi$ for every \smash{$\xi \in \CZonetauwhatFtwo$}.
 To do this, we may assume, without loss of generality, $\xi(\alpha) = 0$ for every 
 $\alpha = \epstau_\ell \in Q^k(\tauw)$.
 Then $\xi$ can be written as
 \[
  \xi
  \hspace{4pt}
  =
  \hspace{25pt}
  \sum_{\mathclap{\alpha:k \in \{h(\alpha),t(\alpha)\}}} \xi(\alpha) \, \alpha^\vee
  \hspace{5pt}
  +
  \hspace{15pt}
  \sum_{\mathclap{\triangle \in X_2(\tauw)}} \xi(\delta_\triangle) \, \delta_\triangle^\vee
  \hspace{5pt}
  .
 \]
 A straightforward computation yields
 \[
  \begin{array}{lcl}
    \varphi_{\sigma,\tau}^* \circ \varphi_{\tau,\sigma}^*(\xi)
    &=&
    \hspace{20pt}
    \displaystyle
    \sum_{\mathclap{\alpha:k \in \{h(\alpha),t(\alpha)\}}} \hspace{10pt} \xi(\alpha) \,
    \left(
      \hspace{15pt}
      \sum_{\mathclap{\ell : \alpha \in \{\mu^\tau_\ell, \nu^\tau_\ell\}}}
      ({{\epstau_\ell}})^\vee
      \hspace{5pt}
      -
      \left(
        \hspace{18pt}
        \sum_{\mathclap{\ell : \alpha^* \in \{\mu^\sigma_\ell, \nu^\sigma_\ell\}}}
        ({{\epstau_\ell}})^\vee
        \hspace{5pt}
        +
        \hspace{15pt}
        \sum_{\mathclap{\substack{\triangle \in X_2(\tauw) \\ \text{inducing $\alpha$}}}} {{\delta}}_\triangle^\vee
        \hspace{3pt}
        -
        \hspace{2pt}
        {{\alpha}}^\vee
      \right)
    \right)
    \\ [2.5em]
    &=&
    \hspace{20pt}
    \displaystyle
    \sum_{\mathclap{\alpha:k \in \{h(\alpha),t(\alpha)\}}} \hspace{10pt} \xi(\alpha) \, {{\alpha}}^\vee
    \hspace{5pt}
    -
    \hspace{15pt}
    \sum_{\mathclap{\substack{\triangle \in X_2(\tauw) \\ \text{$\alpha \neq \delta_\triangle$ induced by $\triangle$}}}} \xi(\alpha) \, {{\delta}}_\triangle^\vee
    \hspace{10pt}
    ,
  \end{array}
 \]
 where the last equality follows from the fact that $\alpha \in \{\mu^\tau_\ell, \nu^\tau_\ell\} \Leftrightarrow \alpha^* \in \{\mu^\sigma_\ell, \nu^\sigma_\ell\}$.
 This proves $\varphi_{\sigma,\tau}^* \circ \varphi_{\tau,\sigma}^*(\xi) = \xi$, since the cocycle condition for $\xi$ means that for every $\triangle \in X_2(\tauw)$
 \[
  \xi(\delta_\triangle)
  \hspace{5pt}
  +
  \hspace{20pt}
  \sum_{\mathclap{\substack{\text{$\alpha$ induced by $\triangle$} \\ \alpha \neq \delta_\triangle}}} \xi(\alpha)
  \hspace{10pt}
  =
  \hspace{20pt}
  \sum_{\mathclap{\text{$\alpha$ induced by $\triangle$}}} \xi(\alpha)
  \hspace{10pt}
  =
  \hspace{5pt}
  0
  \hspace{5pt}
  .
 \]
 Analogously, one can show $\varphi_{\tau,\sigma}^* \circ \varphi_{\sigma,\tau}^* = \mathrm{id}_{\CZonetauwhatFtwo[\sigma]}$ and the proof of the lemma is complete.
\end{proof}

We have all the prerequisites to define when two colored triangulations are related by a flip.

\begin{defi}\label{def:colored-flips-of-colored-triangulations}
 We say that two colored triangulations $(\tau, \xi)$ and $(\sigma, \zeta)$ of~$\SSigmaw$ are \emph{related by the colored flip} of an arc~$k$ if the triangulations $\tau$ and $\sigma$ of $\Sigma$ are related by the flip of $k$ in the sense of \cite{FeShTu-orbifolds} and $\myhat{\zeta} = \varphi^*_{\tau,\sigma}(\xihat)$.
\end{defi}

\begin{ex} Each of Figures \ref{Fig:colored_flip_non_pending_1} and \ref{Fig:colored_flip_non_pending_2} depicts a table. In the top row of each table we can see the result of gluing some pairs of puzzle pieces, with the corresponding portion of the quiver $Q'(\tau,\omega)$ drawn for all the possible values of $\omega$ at the orbifold points contained in the two puzzle pieces respectively involved.
 \begin{figure}[!ht]
                \centering
                \includegraphics[scale=.15]{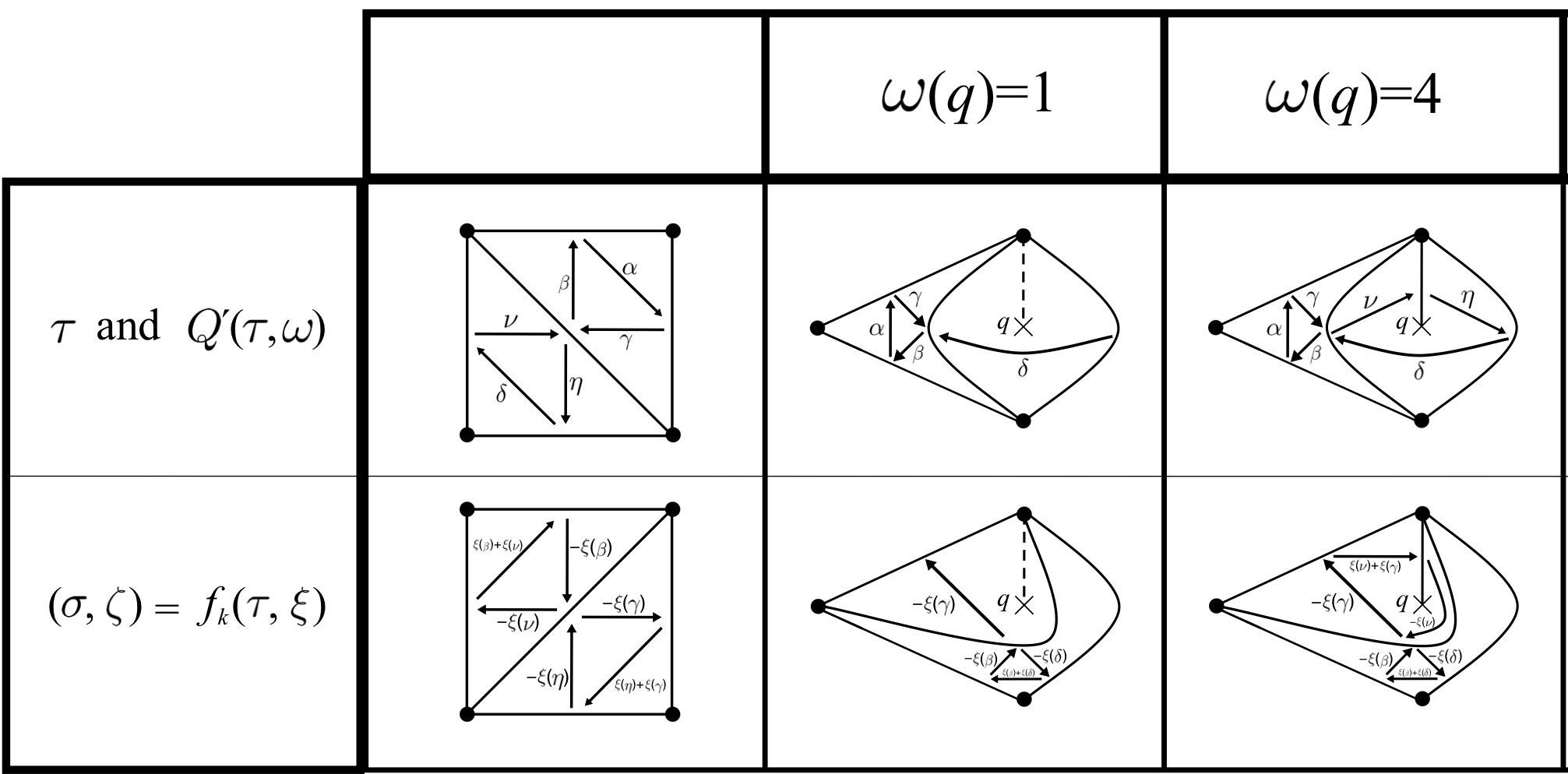}
                \caption{}
                \label{Fig:colored_flip_non_pending_1}
        \end{figure}
In the bottom row of each table we see the effect that the colored flip of $k$ has on $(\tau,\xi)$ for any 1-cocycle $\xi\in Z^1(\tau,\omega)$ (in the top row, the arrows are labeled by their names, while in the bottom row the arrows are labeled by the corresponding values of $\zeta$, where $\zeta\in Z^1(\sigma,\omega)$ is the underlying 1-cocycle of the colored triangulation $(\sigma,\zeta):=\flip_k(\tau,\xi)$; the values of $\zeta$ are written in terms of the values of $\xi$).
\begin{figure}[!ht]
                \centering
                \includegraphics[scale=.25]{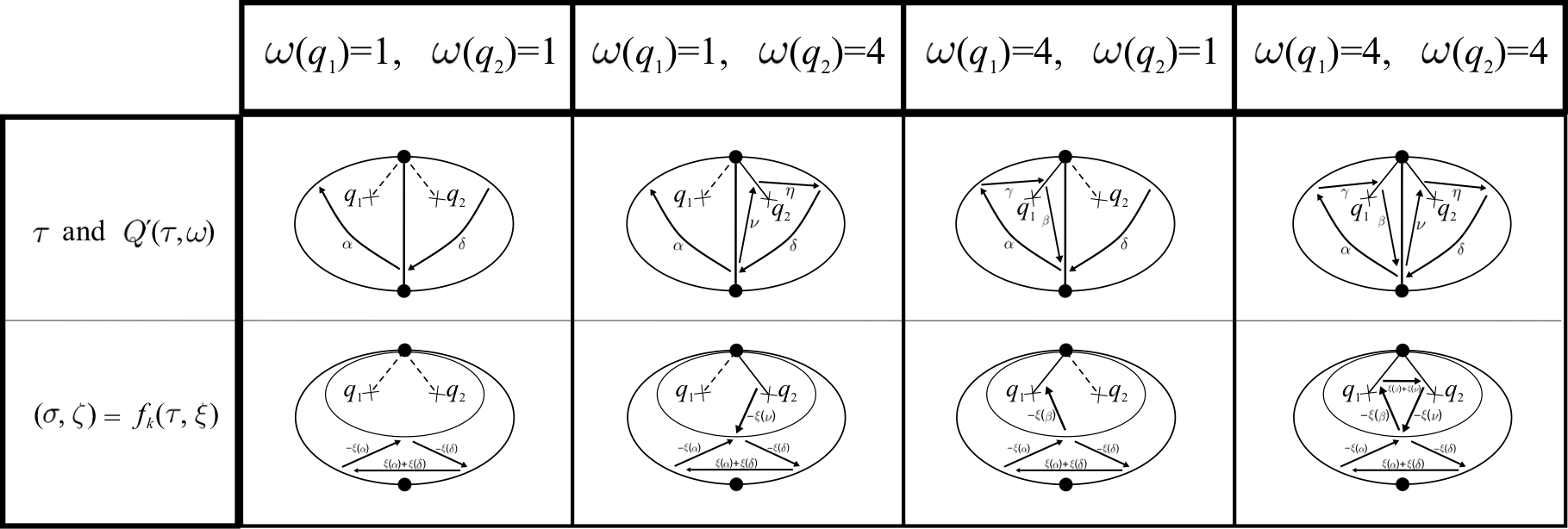}
                \caption{}
                \label{Fig:colored_flip_non_pending_2}
        \end{figure}

In this example we can glimpse a general fact which is actually easy to prove, namely, that if $k$ is non-pending, then the function $Z^1(\tau,\omega)\rightarrow Z^1(\sigma,\omega)$, $\xi\mapsto\zeta$, is linear.
\end{ex}

\begin{ex} Figure \ref{Fig:colored_flip_pending} depicts a table. In the top row of the table we can see a puzzle piece, with the corresponding portion of the quiver $Q'(\tau,\omega)$ drawn on it for all the possible values of $\omega$ at the orbifold points contained in the puzzle piece respectively involved.
\begin{figure}[!ht]
                \centering
                \includegraphics[scale=.3]{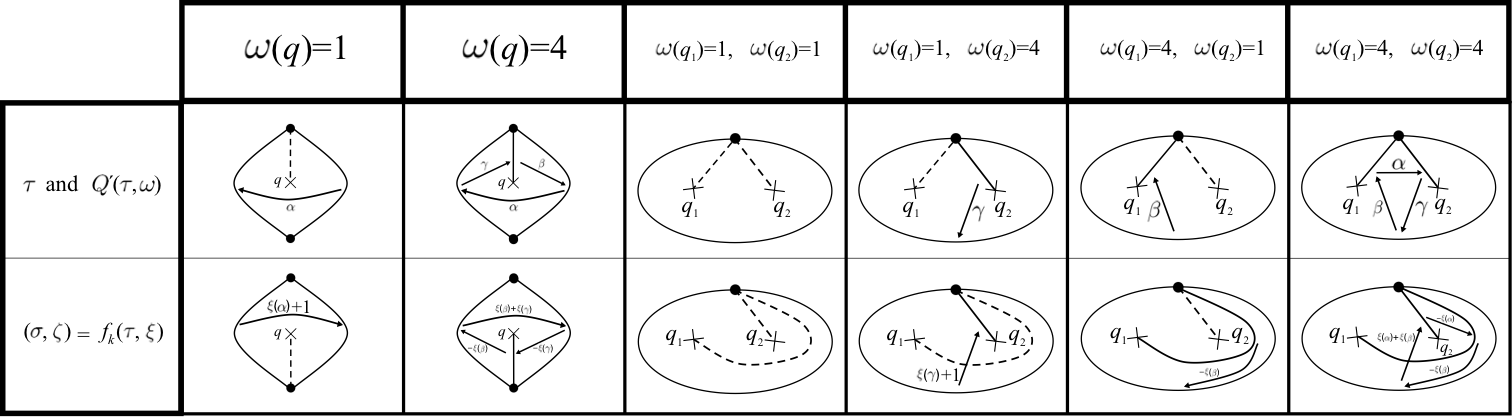}
                \caption{}
                \label{Fig:colored_flip_pending}
        \end{figure}
In the bottom row of the table we can again see the effect that flipping $k$ has on $(\tau,\xi)$ for any 1-cocycle $\xi\in Z^1(\tau,\omega)$ (as before, the arrows in the top row are labeled with their names, while in the bottom row the arrows are labeled by the corresponding values of $\zeta$, where $\zeta\in Z^1(\sigma,\omega)$ is the underlying 1-cocycle of the colored triangulation $(\sigma,\zeta):=\flip_k(\tau,\xi)$; the values of $\zeta$ are again written in terms of the values of $\xi$).

In this example we can glimpse a general fact which is in fact easy to prove, namely, that if $k$ is pending and $d(\tau,\omega)_k=4$, then the function $Z^1(\tau,\omega)\rightarrow Z^1(\sigma,\omega)$, $\xi\mapsto\zeta$, is linear. Furthermore, if $k$ is pending and $d(\tau,\omega)_k=1$, then the function $Z^1(\tau,\omega)\rightarrow Z^1(\sigma,\omega)$, $\xi\mapsto\zeta$, may fail to be linear.
\end{ex}

%% file: 05_species.tex

\section{The species with potential of a colored triangulation}
\label{sec:sp-of-a-colored-triangulation}

In this section we associate a species with potential to each colored triangulation of a surface with weighted orbifold points. The reader is kindly asked to recall from Definition \ref{defi:colored-triangulations} that a colored triangulation is a pair $(\tau,\xi)$ consisting of a triangulation $\tau$ and a choice of a 1-cocycle $\xi\in \CZonetauwFtwo\subseteq\CConetauwFtwo = \Hom_{\F_2}(\ConetauwFtwo, \F_2)$. According to Definition \ref{def:Q(tau,omega)}, the pair $(\tau,\omega)$ dictates us a weighted quiver $(Q(\tau,\omega),\dtuple(\tau,\omega))$. The 1-cocycle $\xi$ will dictate us a modulating function $g(\tau,\xi)$ for this weighted quiver (see \cite[Definition 3.2]{Geuenich-Labardini-1} for the definition of what a modulating function is). This means in particular that we will not be working with arbitrary modulating functions on $(Q(\tau,\omega),\dtuple(\tau,\omega))$, but only with those defined by 1-cocycles; the reason for this is that only the modulating functions arising from 1-cocycles have the chance of producing species admitting non-degenerate potentials, as will become clear in the proof of our main result (Theorem \ref{thm:flip<->SP-mutation}) and in Section \ref{sec:classification-of-nondeg-SPs} (see Lemma \ref{lemma:why-cocycles} and Corollary \ref{coro:only-cocycles-produce-good-species}).

\subsection{The species}
\label{subsec:species-of-a-triangulation}

Let $(\tau,\xi)$ be a colored triangulation of $\SSigmaw=(\Sigma,\marked,\orb,\omega)$. Set $d=\lcm(d(\tau,\omega)_k\suchthat k\in\tau)\in\{2,4\}$, let $F$ be a field containing a primitive $d^{\operatorname{th}}$ root of unity, and let $E/F$ be a degree-$d$ cyclic Galois field extension.
Following \cite[Equation (3.4)]{Geuenich-Labardini-1}, once and for all we fix an element $v\in E$ with the property that $\{v^\ell\suchthat \ell\in\{0,1,\ldots,v^{d-1}\}\}$ is an eigenbasis of $E/F$.

The degree $[E:F]=d$ is $2$ or $4$. In any case, $E$ always contains a unique field extension $L$ of $F$ such that $[L:F]=2$. Once and for all, we fix an element $u\in L$ with the property that $\{1,u\}$ is an eigenbasis of $L/F$. We always can, and will, assume that
$$
u=\begin{cases}
v^2 & \text{if $d=4$;}\\
v & \text{if $d=2$}.
\end{cases}
$$

We will denote by $\theta$ the unique non-identity element of $\Gal(L/F)$, so that $\Gal(L/F)=\{\myid_L,\theta\}$. If $[E:F]=4$, then $\Gal(E/F)$ is a cyclic group with four elements, and we fix a generator $\rho$ once and for all. This generator necessarily satisfies $\myid_E|_L=\myid_L=\rho^2|_L$ and $\rho|_L=\theta=\rho^3|_L$.

For each $k\in\tau$ we set $F_k/F$ to be the unique degree-$d(\tau,\omega)_k$ field subextension of $E/F$, and denote $G_k=\Gal(F_k/F)$. We also denote $G_{j,k}=\Gal(F_j\cap F_k/F)$ for $i,j\in\tau$. Thus:
$$
G_k=\begin{cases}
\{\myid_{F_k},\rho,\rho^2,\rho^3\} & \text{if $[F_k:F]=4$;}\\
\{\myid_{F_k},\theta\} & \text{if $[F_k:F]=2$;}\\
\{\myid_{F_k}\} & \text{if $[F_k:F]=1$.}
\end{cases}
$$

\begin{defi}\label{def:cocycle->modulating-function} Let $(\tau,\xi)$ be a colored triangulation of $\SSigmaw$.
We define a modulating function $g(\tau,\xi):Q(\tau,\omega)_1\rightarrow\bigcup_{j,k\in\tau}G_{j,k}$ as follows. Take $\alpha\in Q(\tau,\xi)_1$.
\begin{enumerate}
\item If $d(\tau,\omega)_{h(\alpha)}=1$ or $d(\tau,\omega)_{t(\alpha)}=1$, set
$$
g(\tau,\xi)_{\alpha}=\myid_{F_{h(\alpha),t(\alpha)}}\in G_{h(\alpha),t(\alpha)}.
$$
\item If $d(\tau,\omega)_{h(\alpha)}\neq 1\neq d(\tau,\omega)_{t(\alpha)}$, and $d(\tau,\omega)_{h(\alpha)}d(\tau,\omega)_{t(\alpha)}<16$, set
    $$
    g(\tau,\xi)_{\alpha}=\theta^{\xi(\alpha)}\in\ G_{h(\alpha),t(\alpha)}.
    $$
\item If $d(\tau,\omega)_{h(\alpha)}=4=d(\tau,\omega)_{t(\alpha)}$, then
\begin{enumerate}
\item $t(\alpha)$ and $h(\alpha)$ are pending arcs contained in a twice orbifolded triangle $\triangle$, and the quiver $Q'(\tau,\omega)$ has exactly one arrow $t(\alpha)\rightarrow h(\alpha)$, induced by $\triangle$; let $\delta_0^\triangle$ be this arrow of $Q'(\tau,\omega)$; notice that we can evaluate $\xi$ at $\delta_0^\triangle$;
\item the quiver $Q(\tau,\omega)$ has exactly two arrows going from $t(\alpha)$ to $h(\alpha)$, one of which is $\delta_0^\triangle$; let $\delta_1^\triangle$ be the other such arrow of $Q(\tau,\omega)$; of course, $\alpha\in\{\delta_0^\triangle,\delta_1^\triangle\}$;
\item $[E:F]=4$ and $F_{h(\alpha)}=E=F_{t(\alpha)}$; let $\ell$ be the unique element of $\{0,1\}$ whose residue class modulo 2 is $\xi(\delta_0^\triangle)\in\F_2=\mathbb{Z}/2\mathbb{Z}$ (equivalently, let $\ell$ be the unique element of $\{0,1\}$ such that $\rho^\ell|_{L}=\theta^{\xi(\delta_0^\triangle)}=\rho^{\ell+2}|_{L}$).
\end{enumerate}
We set
    $$
    g(\tau,\xi)_{\alpha}=\begin{cases}
    \rho^\ell & \text{if $\alpha=\delta_0^\triangle$;}\\
    \rho^{\ell+2} & \text{if $\alpha=\delta_1^\triangle$.}
    \end{cases}
    $$
\end{enumerate}
\end{defi}

\begin{defi}\label{def:species-of-colored-triangulation} The \emph{species over $E/F$ of the colored triangulation $(\tau,\xi)$} is the species of the triple $(Q(\tau,\omega),\dtuple(\tau,\omega),g(\tau,\xi))$ (cf. \cite{Geuenich-Labardini-1}). We shall denote it by $A(\tau,\xi)$.
\end{defi}

For the next two examples, let $p$ be a positive prime number congruent to 1 modulo $4$ (e.g. $p=5$), and let $F$ be the field with $p$ elements. 
Let $z\in F$ be an element which is not a square in $F$ (e.g. $z=2$ if $p=5$, or, more generally, $z=2$ if $p^2\not\equiv 1 \ \operatorname{mod}16$). Then the polynomial $X^4-z$ is irreducible over $F$ (see for example \cite[Theorem 5.4.1]{Chambert-Loir}), which means that if $E$ is the field with $p^4$ elements, then there exists an element $v\in E$ such that $v^4=z$ and $E=F(v)$. Letting $\rho:E\rightarrow E$ be the Frobenius automorphism of the extension $E/F$, and writing $p=4q+1$ (e.g. $5=4\cdot 1+1$), we have $\rho(v)=v^{p}=v^{4q+1}=z^qv$ (e.g. $\rho(v)=v^5=2v$ if $p=5$ or, more generally, if $p^2\not\equiv1 \ \operatorname{mod}16$), which means that $\{1,v,v^2,v^3\}$ is an eigenbasis of $E/F$.

\begin{ex}\label{ex:hexagon_1_orb_pt_wt4_species} Let $\surf$ be an unpunctured hexagon with one orbifold point, and let $\omega:\orb\rightarrow\{1,4\}$ be the function that takes the value 4 at the only element of $\orb$.
\begin{figure}[!ht]
                \centering
                \includegraphics[scale=.08]{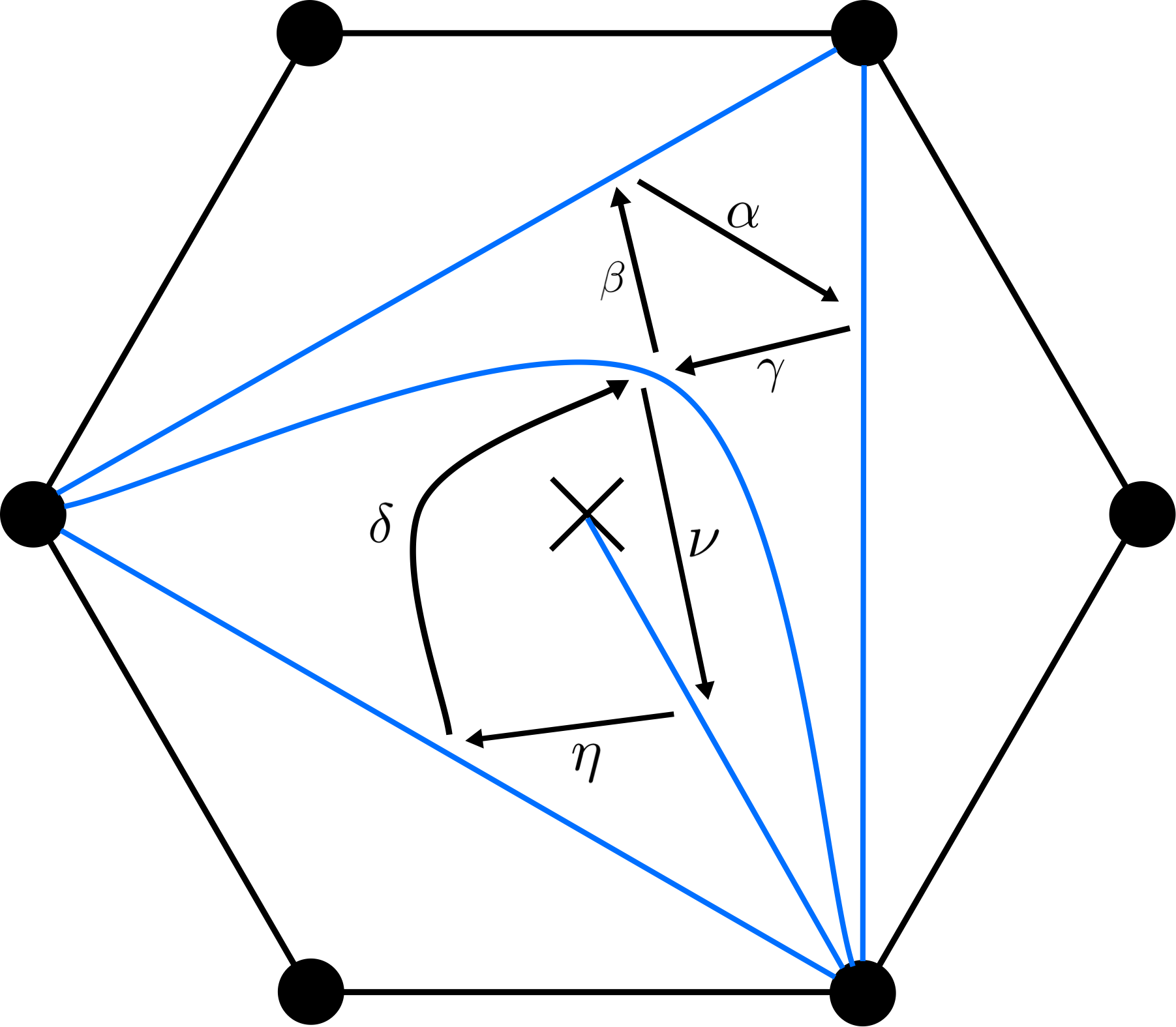}\caption{.}
                \label{Fig:hexagon_1_orb_pt_wt4}
        \end{figure}
        In Figure \ref{Fig:hexagon_1_orb_pt_wt4}, the reader can see a triangulation $\tau$ of $\surf$ and the quiver $Q'(\tau,\omega)$, which in this particular example coincides with $Q(\tau,\omega)$. For any cocycle $\xi\in Z^1(\tau,\omega)$, the following identities hold in the (complete) path algebra of the corresponding species $A(\tau,\xi)$:
\begin{center}
\begin{tabular}{cc}
$\alpha\beta\gamma v^2=v^2\alpha\beta\gamma$, &
$\nu\delta\eta v^2 =  v^2\nu\delta\eta$.
\end{tabular}
\end{center}
\end{ex}

\begin{ex}\label{ex:pentagon_2_orb_pts_wt4_species} Let $\surf$ be an unpunctured pentagon with two orbifold points, and let $\omega:\orb\rightarrow\{1,4\}$ be the function that takes the value 4 at both elements of $\orb$.
\begin{figure}[!ht]
                \centering
                \includegraphics[scale=.1]{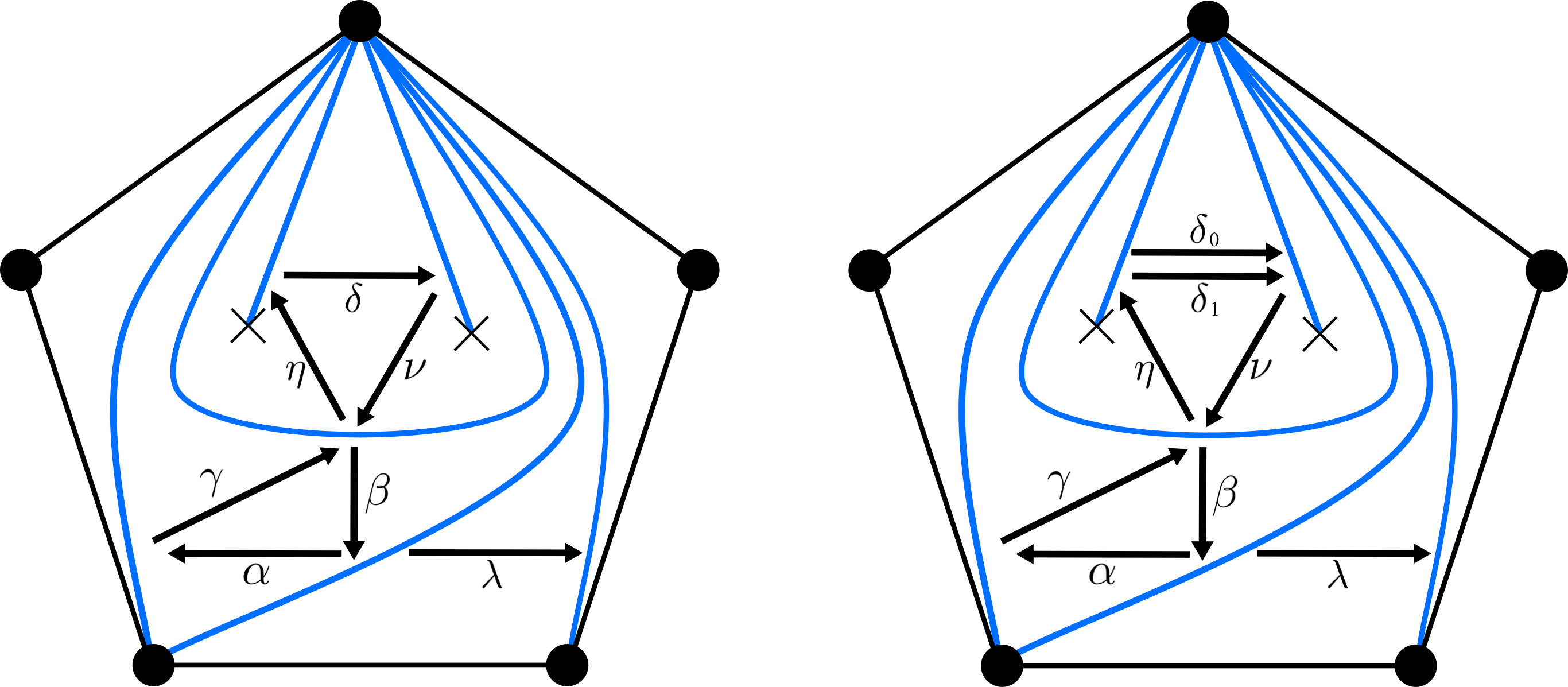}\caption{.}
                \label{Fig:pentagon_2_orb_pts_wt4}
        \end{figure}
In Figure \ref{Fig:pentagon_2_orb_pts_wt4}, the reader can see a triangulation $\tau$ of $\surf$ and the quivers $Q'(\tau,\omega)$ (left) and $Q(\tau,\omega)$ (right), which in this example do not coincide. For any cocycle $\xi\in Z^1(\tau,\omega)$, the following identities hold in the (complete) path algebra of the corresponding species $A(\tau,\xi)$:
\begin{center}
\begin{tabular}{cccc}
$\delta_0\eta\nu v^2 = v^2\delta_0\eta\nu$, & 
$\delta_1\eta\nu v^2 = v^2\delta_1\eta\nu$,&
$\delta_0v = z^{\ell q}v\delta_0$,
$\delta_1v = z^{(\ell+2) q}v\delta_1$,
\end{tabular}
\end{center}
where $\ell$ is defined as the unique element of $\{0,1\}\subseteq \mathbb{Z}$ whose class modulo $2$ is $\xi(\delta)\in\F_2=\mathbb{Z}/2\mathbb{Z}$. 
For instance, if $p=5$, $z=2$ and $\xi(\delta)=0$, then $\delta_0v = v\delta_0$, and
$\delta_1v = 4v\delta_1$, whereas if $p=5$, $z=2$ and $\xi(\delta)=1$, then $\delta_0v = 2v\delta_0$, and
$\delta_1v = 8v\delta_1$
\end{ex}

We leave the easy proof of the following two results in the hands of the reader.

\begin{prop}\label{prop:our-species-realize-FeShTu-matrices} Let $\SSigma=\surf$ be a surface as in Section \ref{sec:surfaces-and-triangs} and $(\tau,\omega)$ be any pair consisting of a triangulation $\tau$ of $\SSigma$ and a function $\omega:\orb\rightarrow\{1,4\}$. Let $B(\tau,\omega)=(b_{kj}(\tau,\omega))_{k,j}$ be the skew-symmetrizable matrix that corresponds to the weighted quiver $(Q(\tau,\omega),\dtuple(\tau,\omega))$ under \cite[Lemma 2.3]{LZ} (that is, the matrix associated to $(\tau,w)$ by Felikson-Shapiro-Tumarkin, where $w(q)=\frac{2}{\omega(q)}$ for $q\in\orb$, see Remark \ref{rem:all-possible-matrices}-(2)). For any 1-cocycle $\xi\in Z^1(\tau,\omega)$, the pair $(\mathbf{F},\mathbf{A})$ is a species realization of $B(\tau,\omega)$ (cf. \cite[Definition 2.22]{Geuenich-Labardini-1}), where $\mathbf{F}=(F_k)_{k\in \tau}$ and $\mathbf{A}=(e_kA(\tau,\xi)e_j)_{k,j\in \tau}$ ($e_j$ being the idempotent element of $R:=\bigoplus_{k\in\tau} F_k$ that has a $1$ in its $j^{\operatorname{th}}$ entry and zeros elsewhere). More precisely,  $\mathbf{F}$ is a tuple of division rings and for every pair $(k,j)\in\tau\times\tau$ such that $b_{kj}(\tau,\omega)\geq 0$ we have that:
\begin{enumerate}
\item $e_kA(\tau,\xi)e_j$ is an $F_k$-$F_j$-bimodule;
\item $\dim_{F_k}(e_kA(\tau,\xi)e_j)=b_{kj}(\tau,\omega)$ and $\dim_{F_j}(e_kA(\tau,\xi)e_j)=-b_{jk}(\tau,\omega)$;
\item $\Hom_{F_k}(e_kA(\tau,\xi)e_j,F_k)$ and $\Hom_{F_j}(e_kA(\tau,\xi)e_j,F_j)$ are isomorphic as $F_j$-$F_k$-bimodules.
\end{enumerate}
\end{prop}

\begin{prop}\label{prop:colored-flip-vs-species-mutation} Let $\SSigma=\surf$ be a surface as in Section \ref{sec:surfaces-and-triangs}, $\omega:\orb\rightarrow\{1,4\}$ a function, $(\tau,\xi)$ a colored triangulation of $\SSigmaw$, and $k\in\tau$. If $W\in\RA{A(\tau,\xi)}$ is a potential such that the SP $(A',W'):=\mu_k(A(\tau,\xi),W)$ is 2-acyclic, then $A'\cong A(\flip(\tau,\xi))$ as $R$-$R$-bimodules.
\end{prop}

%% file: 06_potentials.tex
\subsection{The potential}

\label{sec:potentials-for-colored-triangulations}

\begin{defi}{(Cycles from non-orbifolded triangles).}\label{def:cycles-from-non-orb-triangles} Let $(\tau,\xi)$ be a colored triangulation of $\SSigmaw$ and $\triangle$ be an interior triangle of $\tau$ not containing any orbifold point. Then there is a 3-cycle $\alpha^\triangle\beta^\triangle\gamma^\triangle$ on $Q(\tau,\omega)$ formed with the arrows $\alpha^\triangle,\beta^\triangle,\gamma^\triangle$ contained in $\triangle$ (see the picture on the upper left in Figure \ref{Fig:triangles_quivers}). We set $S^\triangle(\tau,\xi)=\alpha^\triangle\beta^\triangle\gamma^\triangle$.
\end{defi}

\begin{figure}[!ht]
                \centering
                \includegraphics[scale=.65]{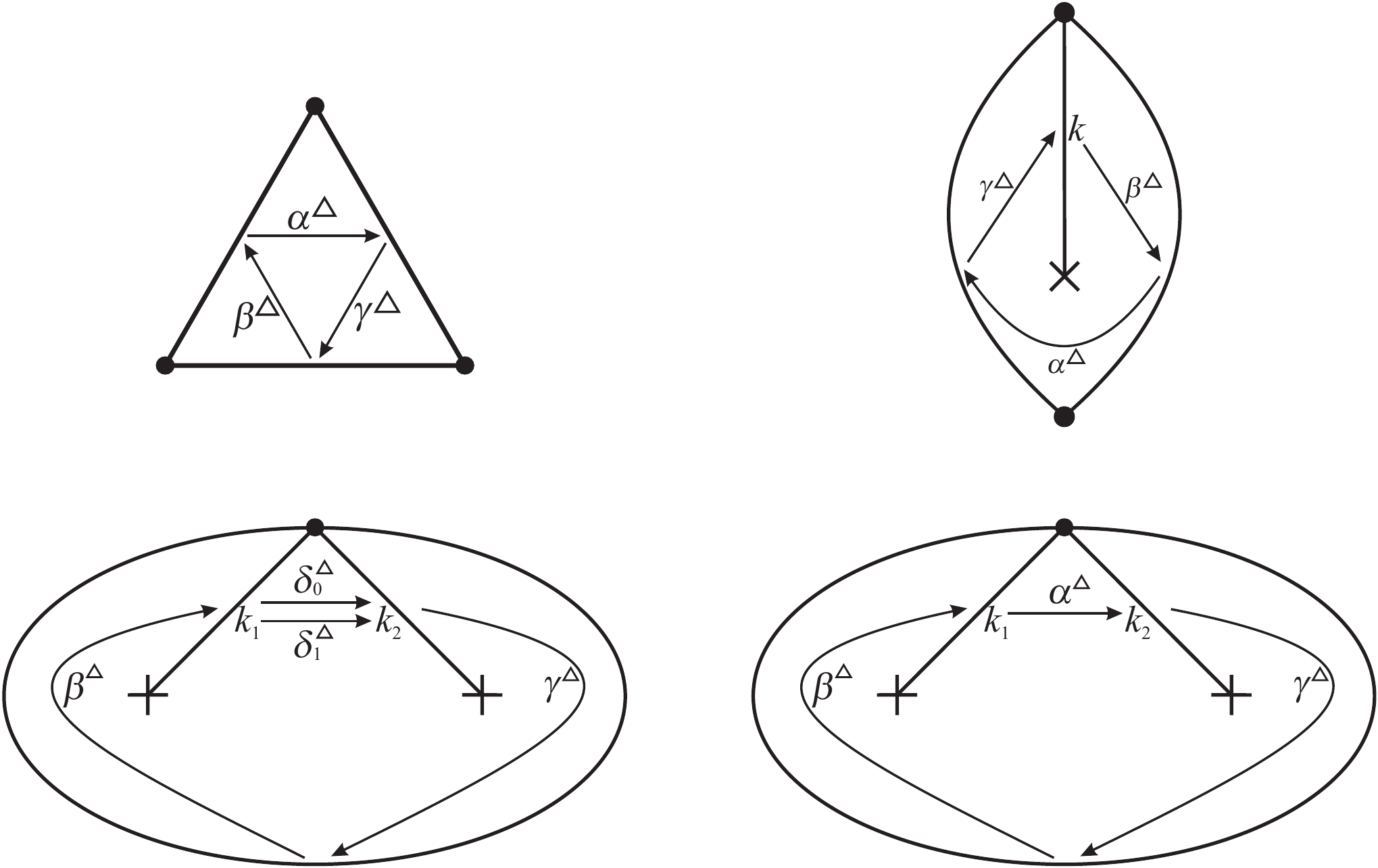}\caption{Notation for the definition of $\Stauc$.}
                \label{Fig:triangles_quivers}
        \end{figure}

\begin{defi}{(Cycles from triangles with exactly one orbifold point).}\label{def:cycles-from-1-orb-triangles} Let $(\tau,\xi)$ be a colored triangulation of $\SSigmaw$ and $\triangle$ be an interior triangle of $\tau$ containing exactly one orbifold point. Let $k$ be the unique pending arc of $\tau$ contained in $\triangle$. Using the notation from the picture on the upper right in Figure \ref{Fig:triangles_quivers}, we set $S^\triangle(\tau,\xi)=\alpha^\triangle\beta^\triangle\gamma^\triangle$, regardless of whether $d(\tau,\omega)_k=1$ or $d(\tau,\omega)_k=4$.
\end{defi}

\begin{defi}[Cycles from triangles with exactly two orbifold points]\label{def:cycles-from-triangs-with-2-orb-pts} Let $(\tau,\xi)$ be a colored triangulation of $\SSigmaw$ and $\triangle$ be an interior triangle of $\tau$ containing exactly two orbifold points. Let $k_1$ and $k_2$ be the two pending arcs of $\tau$ contained in $\triangle$, and assume that $Q(\tau,\omega)$ has (at least) one arrow going from $k_1$ to $k_2$.
\begin{itemize}
\item If $d(\tau,\omega)_{k_1}=1=d(\tau,\omega)_{k_2}$, then, with the notation of the picture on the bottom left in Figure \ref{Fig:triangles_quivers}, we set $S^\triangle(\tau,\xi)=\delta_0^\triangle\beta^\triangle\gamma^\triangle+\delta_1^\triangle\beta^\triangle u\gamma^\triangle$.
\item If $d(\tau,\omega)_{k_1}=1$ and $d(\tau,\omega)_{k_2}=4$, then, with the notation of the picture on the bottom right in Figure \ref{Fig:triangles_quivers}, we set $S^\triangle(\tau,\xi)=\alpha^\triangle\beta^\triangle\gamma^\triangle$.
\item If $d(\tau,\omega)_{k_1}=4$ and $d(\tau,\omega)_{k_2}=1$, then, with the notation of the picture on the bottom right in Figure \ref{Fig:triangles_quivers}, we set $S^\triangle(\tau,\xi)=\alpha^\triangle\beta^\triangle\gamma^\triangle$.
\item If $d(\tau,\omega)_{k_1}=4$ and $d(\tau,\omega)_{k_2}=4$, then, with the notation of the picture on the bottom left in Figure \ref{Fig:triangles_quivers}, we set $S^\triangle(\tau,\xi)=(\delta_0^\triangle+\delta_1^\triangle)\beta^\triangle\gamma^\triangle$.
\end{itemize}
\end{defi}

\begin{defi}[The potential of a colored triangulation]\label{def:S(tau,xi)} Let $\SSigma=(\Sigma,\marked,\orb)$ be an unpunctured surface with marked points and orbifold points of order 2, $\omega:\orb\rightarrow\{1,4\}$ any function, and $(\tau,\xi)$ a colored triangulation of $\SSigmaw=(\Sigma,\marked,\orb,\omega)$. The potential associated to $(\tau,\xi)$ is
$$
S(\tau,\xi)=\sum_{\triangle}S^\triangle(\tau,\xi)\in  R\langle A(\tau,\xi)\rangle\subseteq\RA{A(\tau,\xi)},
$$
where the sum runs over all interior triangles of $\tau$.
\end{defi}

\begin{ex}\label{ex:hexagon_1_orb_pt_wt4_potential} Let $\SSigmaw=(\Sigma,\marked,\orb,\omega)$ and $\tau$ be as in Example \ref{ex:hexagon_1_orb_pt_wt4_species}. For any 1-cocycle $\xi\in Z^1(\tau,\omega)$, the potential $S(\tau,\xi)$ is
$$
S(\tau,\xi)=\alpha\beta\gamma+\delta\eta\nu,
$$
whose cyclic derivatives are (see \cite[Definition 3.11]{Geuenich-Labardini-1} for the definition of cyclic derivative)
\begin{center}
\begin{tabular}{ccc}
$\partial_\alpha(S(\tau,\xi)) = \pi_{g(\tau,\xi)_\alpha^{-1}}\left(\beta\gamma\right)=\beta\gamma$, &
$\partial_\beta(S(\tau,\xi)) = \pi_{g(\tau,\xi)_\beta^{-1}}\left(\gamma\alpha \right)=\gamma\alpha$, &
$\partial_\gamma(S(\tau,\xi)) = \pi_{g(\tau,\xi)_\gamma^{-1}}\left(\alpha\beta \right)=\alpha\beta$,\\
$\partial_\delta(S(\tau,\xi)) =\pi_{g(\tau,\xi)_\delta^{-1}}\left(\eta\nu\right)=\eta\nu$, &
$\partial_\eta(S(\tau,\xi)) = \pi_{g(\tau,\xi)_\eta^{-1}}\left(\nu\delta\right)=\nu\delta$, &
$\partial_\nu(S(\tau,\xi)) = \pi_{g(\tau,\xi)_\nu^{-1}}\left(\delta\eta \right)=\delta\eta$,
\end{tabular}
\end{center}
where $\pi_{\rho}(x)=\frac{1}{2}\left(x+\rho(v^{-2})xv^2\right)$ for $\rho\in\Gal(L/F)$. 
Thus, the following identities hold in the Jacobian algebra $\Jacalg{A(\tau,\xi),S(\tau,\xi)}$ besides the ones listed in Example \ref{ex:hexagon_1_orb_pt_wt4_species} (see \cite[Definition 3.11]{Geuenich-Labardini-1} for the definition of the Jacobian algebra of an SP):
\begin{equation*}
\alpha\beta=\beta\gamma=\gamma\alpha=\delta\eta=\eta\nu=\nu\delta=0,
\end{equation*}
from which it easily follows that $\dim_{F}(\Jacalg{A(\tau,\xi),S(\tau,\xi)})<\infty$.
Note, however, that $\eta v\nu\neq 0$ in $\Jacalg{A(\tau,\xi),S(\tau,\xi)}$.
\end{ex}

\begin{ex}Let $\SSigmaw=(\Sigma,\marked,\orb,\omega)$ and $\tau$ be as in Example \ref{ex:pentagon_2_orb_pts_wt4_species}. For any 1-cocycle $\xi\in Z^1(\tau,\omega)$, the potential $S(\tau,\xi)$ is
$$
S(\tau,\xi)=\alpha\beta\gamma+(\delta_0+\delta_1)\eta\nu,
$$
whose cyclic derivatives are
\begin{center}
\begin{tabular}{ccc}
$\partial_\alpha(S(\tau,\xi)) = 
\beta\gamma$, &
$\partial_\beta(S(\tau,\xi)) = 
\gamma\alpha$, &
$\partial_\gamma(S(\tau,\xi)) = 
\alpha\beta$,
\end{tabular}
\begin{tabular}{c} 
$\partial_{\delta_0}(S(\tau,\xi)) = \pi_{g(\tau,\xi)_{\delta_0}^{-1}}\left(\eta\nu\right)=\frac{1}{2}\left(\eta\nu+g(\tau,\xi)_{\delta_0}^{-1}(v^{-1})\eta\nu v\right)$, \\
$\partial_{\delta_1}(S(\tau,\xi)) =\pi_{g(\tau,\xi)_{\delta_1}^{-1}}\left(\eta\nu\right)=\frac{1}{2}\left(\eta\nu+g(\tau,\xi)_{\delta_1}^{-1}(v^{-1})\eta\nu v\right)$,\\
$\partial_\eta(S(\tau,\xi)) = \frac{1}{2}\left(\nu\delta_0+g(\tau,\xi)_\eta^{-1}(v^{-2})\nu\delta_0 v^2\right)+\frac{1}{2}\left(\nu\delta_1+g(\tau,\xi)_\eta^{-1}(v^{-2})\nu\delta_1 v^2\right)=\nu\delta_0+\nu\delta_1$,\\
$\partial_\nu(S(\tau,\xi)) = \frac{1}{2}\left(\delta_0\eta+g(\tau,\xi)_\nu^{-1}(v^{-2})\delta_0\eta v^2\right)+\frac{1}{2}\left(\delta_1\eta+g(\tau,\xi)_\nu^{-1}(v^{-2})\delta_1\eta v^2\right)=\delta_0\eta+\delta_1\eta$,
\end{tabular}
\end{center}
where $\pi_{\rho}(x)=\frac{1}{4}\left(x+\rho(v^{-1})x v+\rho(v^{-2})x v^2+\rho(v^{-3})x v^3\right)$ for $\rho\in\Gal(E/F)$.
Thus, the following identities hold in the Jacobian algebra $\Jacalg{A(\tau,\xi),S(\tau,\xi)}$ besides the ones listed in Example \ref{ex:pentagon_2_orb_pts_wt4_species}:
\begin{equation*}
\alpha\beta=\beta\gamma=\gamma\alpha=\eta\nu=\nu\delta_0+\nu\delta_1=\delta_0\eta+\delta_1\eta=0,
\end{equation*}
from which it easily follows that $\dim_{F}(\Jacalg{A(\tau,\xi),S(\tau,\xi)})<\infty$.
\end{ex}

%% file: 07_flips_vs_mutation.tex

\section{Flip is compatible with SP-mutation}

\label{sec:flip=>SP-mutation}

Here we present the main result of this paper, which says that whenever two colored triangulations are related by the flip of an arc, then their associated SPs are related by the SP-mutation defined in \cite[Definitions 3.19 and 3.22]{Geuenich-Labardini-1}. The precise statement is:

\begin{thm}\label{thm:flip<->SP-mutation} Let $\SSigma=(\Sigma,\marked,\orb)$ be either an unpunctured surface with marked points and order-2 orbifold points, or a once-punctured closed surface with order-2 orbifold points; let $\omega:\orb\rightarrow\{1,4\}$ be any function, and let $(\tau,\xi)$ and $(\sigma,\zeta)$ be colored triangulations of $\SSigmaw=(\Sigma,\marked,\orb,\omega)$. If $(\sigma,\zeta)$ is obtained from $(\tau,\xi)$ by the colored flip of an arc $k\in\tau$, then the SPs $(A(\sigma,\zeta),S(\tau,\zeta))$ and $\mu_k(A(\tau,\xi),S(\tau,\xi))$ are right-equivalent.
\end{thm}

The proof we give of this theorem is rather long as we achieve it by verifying that its statement is true for all the possible configurations that $\tau$ can present locally around the arc $k$. As such, the proof, which the reader can find in Section \ref{sec:proof-of-main-theorem}, is done case by case. Worth mentioning is the fact that in many of these cases it will be crucial that $A(\tau,\xi)$ and $A(\sigma,\zeta)$ are given by Definitions \ref{def:cocycle->modulating-function} and \ref{def:species-of-colored-triangulation}, and that $S(\tau,\xi)$ and $S(\sigma,\zeta)$ are given by Definitions \ref{def:cycles-from-non-orb-triangles}, \ref{def:cycles-from-1-orb-triangles}, \ref{def:cycles-from-triangs-with-2-orb-pts} and \ref{def:S(tau,xi)}.

Notice that Theorem \ref{thm:flip<->SP-mutation} immediately implies:

\begin{coro}\label{coro:our-SPs-are-nondegenerate} For $\SSigmaw$ as in the statement of Theorem \ref{thm:flip<->SP-mutation}, the species with potential associated to its colored triangulations are non-degenerate in the sense of Derksen-Weyman-Zelevinsky (cf. \cite[Definition 7.2]{DWZ1}).
\end{coro}

%% file: 08_relation_to_surface_homology.tex
\section{Geometric realization of the complexes $\Ctauw$ and $\Ctauwhat$}


\label{sec:relation-to-surface-homology}

Recall that the surfaces $\surf$ we are working with in this paper have an arbitrary number of orbifold points, and are either unpunctured with non-empty boundary, or closed with exactly one puncture. Set
\begin{eqnarray}\label{eq:def-of-Sigmatilde}
\overlineSigma &=& \begin{cases}
\Sigma & \text{if $\Sigma$ has non-empty boundary;}\\
\Sigma\setminus\marked & \text{if $\Sigma$ has empty boundary.}
\end{cases}
\end{eqnarray}
We will show that the homology $\Htauw$ is isomorphic to the singular homology~$H_\bullet(\overlineSigma; \F_2)$ of~$\overlineSigma$ and $\Htauwhat$ to the singular homology~$H_\bullet(\Sigmahat; \F_2)$  of~$\Sigmahat = \overlineSigma \setminus \{ q \in \orb \suchthat \omega(q) = 1\}$. See \cite[Lemma~2.3]{Amiot-Grimeland}, \cite[Proposition~3.4]{Amiot-Labardini-Plamondon}, \cite{Broomhead-paper} or \cite{Mozgovoy-Reineke} for similar results established before. It will be important for the next section of this paper that the isomorphisms are induced by the inclusions of specific ``geometric realizations'' of $X_\bullet(\tauw)$ and $\Xhat_\bullet(\tauw)$ as topological subspaces of $\overlineSigma$ and $\Sigmahat$, respectively.

\begin{remark} If $\surf$ is an unpunctured surface, then $\Sigmahat$ can be interpreted as the \emph{associated triangulated orbifold}~\smash{$\myhat{\mathcal{O}} = \myhat{\mathcal{O}}(\tauw)$} described in \cite[Definition~5.10]{FeShTu-orbifolds}.
\end{remark}

\noindent
To argue that the homology of the chain complexes~$\Ctauw$ and $\Ctauwhat$ can be canonically identified with the singular homology of the surfaces~$\overlineSigma$ and $\Sigmahat$ with $\F_2$-coefficients, respectively, it is convenient to construct two topological subspaces $Y(\tauw) \subseteq \Yhat(\tauw)$ of~$\Sigma$.
These subspaces should be regarded as ``geometric realizations'' of $X_\bullet(\tauw)$ and $\Xhat_\bullet(\tauw)$, respectively.
To obtain the spaces $Y(\tauw) \subseteq \Yhat(\tauw)$ we proceed as follows:
\begin{enumerate}
 \itemsep 3pt \parskip 0pt \parsep 0pt

 \item
  For each arc $i \in X_0(\tauw) = \tau \setminus \tau^{\omega=1}$ choose a point $\Ypoint_i$ in the interior of $i$.

 \item
  For each ${{\alpha}} \in \widehat{X}_1(\tauw)$ choose a simple curve~$\Ycurve_{{{\alpha}}}$ going from $\Ypoint_{t(\alpha)}$ to $\Ypoint_{h(\alpha)}$ inside the triangle $\triangle$ containing~$\alpha$.
  If $\alpha = \epstau_k$ for some $k\in\tau^{\omega=1}$, let $\Ycurve_{{{\alpha}}}$ be chosen in such a way that it crosses $k$ exactly once, and that such crossing happens in a non-endpoint of~$k$.
  Otherwise, the interior of $\Ycurve_{{{\alpha}}}$ should not intersect any arc of $\tau$.
  Moreover, we assume that the interior of~$\Ycurve_{{{\alpha}}}$ intersects neither the boundary of $\triangle$ nor any~$\Ycurve_{{{\beta}}}$ with ${{\beta}} \neq {{\alpha}}$.

 \item
  Observe that for each triangle $\triangle\in X_2(\tauw)$ the set $\Ycurve_\triangle = \bigcup_{\text{$\alpha$ induced by $\triangle$}} \Ycurve_{{{\alpha}}}$ is a closed simple curve in $\Sigma$.
  Let $\Yface_\triangle$ be the closure of the connected component of $\Sigma \setminus \Ycurve_\triangle$ not intersecting any arc of~$\tau$.

 \item
  We define
  \[
   Y(\tauw)
   \:=\:
   \bigcup_{i \in X_0(\tauw)} \{ \Ypoint_i \}
   \:\:\cup\:\:
   \bigcup_{{{\alpha}} \in X_1(\tauw)} \Ycurve_{{{\alpha}}}
   \:\:\cup\:\:
   \bigcup_{\triangle \in X_2(\tauw)} \Yface_{\triangle}
  \]
  and
  \[
  \Yhat(\tauw)
  \:=\:
  \bigcup_{i \in \Xhat_0(\tauw)} \{ \Ypoint_i \}
  \:\:\cup\:\:
  \bigcup_{{{\alpha}} \in \Xhat_1(\tauw)} \Ycurve_{{{\alpha}}}
  \:\:\cup\:\:
  \bigcup_{\triangle \in \Xhat_2(\tauw)} \Yface_{\triangle}
  \,.
  \]
\end{enumerate}



\begin{ex} Consider the triangulations $\tau$ and $\sigma$ of the unpunctured pentagon with two orbifold points depicted in Figure \ref{Fig:pentagon_two_orb_points}. In Figures \ref{Fig:example_1_Q_Qp_Qpp} and \ref{Fig:example_2_Q_Qp_Qpp} we can see the topological subspace $\Yhat(\tauw)$ of $\Sigma$ depicted for every choice of function $\omega:\orb\{1,4\}$.
\end{ex}

To relate the homology of $\Ctauw$ and $\Ctauwhat$ to the singular homology of $Y(\tauw)$ and $\Yhat(\tauw)$ we fix:
\begin{itemize}
\item for each $i \in \Xhat_0(\tauw)$, the function~$\theta_i:\Delta^0\rightarrow\Sigmahat$ taking value $\Ypoint_i$, where $\Delta^0=\{\star\}$ is the 0-simplex;
\item for each ${{\alpha}} \in \Xhat_1(\tauw)$, a parametrization $\theta_{{{\alpha}}} : \Delta^1 \to \Sigmahat$ of~$\Ycurve_{{{\alpha}}}$, where $\Delta^1 \cong [0,1]$ is the 1-simplex, such that the restriction of $\theta_{{{\alpha}}}$ to the interior of $\Delta^1$ is injective, $\theta_{{{\alpha}}}(0) = y_{t(\alpha)}$, and $\theta_{{{\alpha}}}(1) = y_{h(\alpha)}$; and
\item
for each $\triangle \in \Xhat_2(\tauw)$, a parametrization $\theta_{\triangle} : \Delta^2 \to \Sigmahat$ of $\Yface_\triangle$, where $\Delta^2$ is the 2-simplex, such that the restriction of $\theta_{\triangle}$ to the interior of $\Delta^2$ is injective and any restriction of~$\theta_{\triangle}$ to a face of~$\Delta^2$ parametrizes one of the curves $c_{{{\alpha}}}$ for some $\alpha \in X_1(\tauw)$ induced by $\triangle$.
\end{itemize}
We use the notation $C_\bullet(\Yhat(\tauw); \F_2)$ for the singular complex with coefficients in $\F_2$ of the topological space $\Yhat(\tauw)$. The following is a standard result in basic algebraic topology.

\begin{prop}
 \label{prop:delta-set-homology-is-homology-of-realization}
 The map of chain complexes $\Ctauwhat \to C_\bullet(\Yhat(\tauw); \F_2)$ that maps $x \in \coprod_{n \in \N} \Xhat_n(\tauw)$ to $\theta_x$ induces isomorphisms $\Htauwhat \to H_\bullet(\Yhat(\tauw); \F_2)$ and $\Htauw \to H_\bullet(Y(\tauw); \F_2)$.
\end{prop}

\begin{remark}
 A result similar to Proposition~\ref{prop:delta-set-homology-is-homology-of-realization} also holds when taking coefficients in $\Z$ instead of $\F_2$.
 However, one has to be a little bit more careful with the choice of the $\theta_x$ in this case, since $-1 \neq 1$ in $\Z$.
\end{remark}

\noindent
The next proposition tells us together with Proposition~\ref{prop:delta-set-homology-is-homology-of-realization} that, as claimed earlier, we have quasi-isomor\-phisms $\Ctauw \to C_\bullet(\Sigma; \F_2)$ and \smash{$\Ctauwhat \to C_\bullet(\Sigmahat; \F_2)$}.

\begin{prop}
 \label{prop:embedding-is-quasi-isomorphism}
 The canonical inclusions $\iota : Y(\tauw) \to \overlineSigma$ and $\myhat{\iota} : \Yhat(\tauw) \to \Sigmahat$ induce isomorphisms in homology $\iota_* : H_\bullet(Y(\tauw); \F_2) \to H_\bullet(\overlineSigma; \F_2)$ and $\myhat{\iota}_* : H_\bullet(\Yhat(\tauw); \F_2) \to H_\bullet(\Sigmahat; \F_2)$.
\end{prop}

\begin{remark}
 Proposition~\ref{prop:embedding-is-quasi-isomorphism} follows from the observation that $Y(\tauw)$ is a strong deformation retract of $\overlineSigma$ and \smash{$\Yhat(\tauw)$} a strong deformation retract of~\smash{$\Sigmahat$}.
 We give a proof requiring less imagination.
\end{remark}

\begin{proof}[Proof of Proposition \ref{prop:embedding-is-quasi-isomorphism}]
 Clearly, $H_0(Y(\tauw); \F_2)$ has dimension $1$ and $H_n(Y(\tauw); \F_2) = 0$ for $n \geq 2$, because $Y(\tauw)$ is path-connected and homotopy-equivalent to a CW complex of dimension less than~$2$.
 Consequently, the map $\iota_* : H_n(Y(\tauw); \F_2) \to H_n(\overlineSigma; \F_2)$ is an isomorphism for $n \neq 1$.
 It remains to treat the case $n = 1$.

 Let $\Ycurve$ be a representative of an element of $H_1(\overlineSigma; \F_2)$, i.e.\ an $\F_2$-linear combination of paths~$\Delta^1 \to \overlineSigma$.
 Fix a basepoint~$\Ypoint \in Y(\tauw)$.
 Since the canonical map $\pi_1(\overlineSigma, \Ypoint) \to H_1(\overlineSigma; \F_2)$ is surjective, we can assume that $\Ycurve$ is a closed path based at $\Ypoint$.
 Observe that every closed path in $\overlineSigma$ based at $\Ypoint$ is homotopic (through a homotopy fixing the basepoint all throughout) to a path with image in $Y(\tauw)$; so, we may assume $\Ycurve = \iota \circ \Ycurve'$ for a path $\Ycurve' : \Delta^1 \to Y(\tauw)$.
 This shows that $\iota_*$ is surjective.
 Using that $H_1(\overlineSigma; \F_2)$ has dimension
 $$
 r = \begin{cases}
 2\genus + b -1 & \text{if the boundary of $\Sigma$ is not empty (and hence $\overlineSigma=\Sigma$);}\\
 2\genus & \text{if the boundary of $\Sigma$ is empty (and hence $\overlineSigma=\Sigma\setminus\marked$);}
 \end{cases}
 $$
(recall that $\genus$ and $b$ are the genus and the number of boundary components of $\Sigma$, respectively) we can therefore verify that $\iota_*$ is an isomorphism by checking that $H_1(Y(\tauw); \F_2)$ is a vector space of the same dimension~$r$. We shall show this when $\Sigma$ has non-empty boundary and leave to the reader the case when the boundary of $\Sigma$ is empty.

 Because $H_0(Y(\tauw); \F_2)$ is one-dimensional and $H_n(Y(\tauw); \F_2)$ vanishes for $n \geq 2$ we know that the dimension of $H_1(Y(\tauw); \F_2)$ is
 \[
  r' \:=\: 1 - \chi(Y(\tauw)) \:=\: 1 - |X_0(\tauw)| + |X_1(\tauw)| - |X_2(\tauw)| \,,
 \]
 where the last equality uses Proposition~\ref{prop:delta-set-homology-is-homology-of-realization}.
 To compute this dimension explicitly, let $e$ be the number of arcs and $h$ the number of triangles of $\tau$.
 Similarly to \cite[Section~2]{Fock-Goncharov} one has:
 \[
  \begin{array}{lcl}
   e &=& 6 (\genus-1) + 3 b + m + 2 o
   \\
   h &=& 4 (\genus-1) + 2 b + m + \phantom{1} o
  \end{array}
 \]
 If we denote by $h_{pq}$ the number of triangles of $\tau$ having exactly $p$ boundary sides and $q$ orbifold points of weight~$1$, we can express $h$, $m$, and $u$ as follows:
 \[
  \arraycolsep 1pt
  \begin{array}{lclrcrcrcrcrcrcrcr}
   h &\:\:=\:\:& \sum_{p,q} & h_{pq}   &\:\:=\:\:& h_{00} &+& h_{01} &+& h_{02} &+& h_{10} &+& h_{11} &+& h_{20}
   \\ [0.3em]
   m &\:\:=\:\:& \sum_{p,q} & p h_{pq} &\:\:=\:\:& &&&&&& h_{10} &+& h_{11} &+& 2 h_{20}
   \\ [0.3em]
   u &\:\:=\:\:& \sum_{p,q} & q h_{pq} &\:\:=\:\:& && h_{01} &+& 2 h_{02} &&&+& h_{11}
  \end{array}
 \]
 A straightforward calculation, using $|X_0(\tauw)| = e - u$, $|X_1(\tauw)| = 3 h_{00} + h_{01} + h_{10}$, and $|X_2(\tauw)| = h_{00}$, yields
 \[
  r' \:=\: 2\genus + b -1 \:=\: r \,.
 \]
 The proof that \smash{$\myhat{\iota}$} induces an isomorphism in homology is similar.
\end{proof}

\begin{remark}
 Since the homology groups of $Y(\tauw)$, $\Yhat(\tauw)$, $\Sigma(\tauw)$, and $\Sigmahat(\tauw)$ with coefficients in $\mathbb{Z}$ are free, one can see with a similar proof that $\iota_* : H_\bullet(Y(\tauw)) \to H_\bullet(\overlineSigma)$ and $\myhat{\iota}_* : H_\bullet(\Yhat(\tauw)) \to H_\bullet(\Sigmahat)$ are isomorphisms.
\end{remark}

\noindent
Let $\CCtauwhatFtwo = \Hom_{\F_2}(\CtauwhatFtwo, \F_2)$.
In conclusion, Propositions~\ref{prop:delta-set-homology-is-homology-of-realization} and~\ref{prop:embedding-is-quasi-isomorphism} imply the following fact.

\begin{coro}
 \label{coro:homology-is-surface-homology}
 The map of chain complexes~$\Ctauwhat \to C_\bullet(\Sigmahat; \F_2)$ given by $x \mapsto \theta_x$ for $x \in \Xhat_n(\tauw)$ induces the following commutative diagrams, where all horizontal maps are isomorphisms and the unlabeled maps are induced by the canonical inclusion $\Sigmahat \hookrightarrow \overlineSigma$:
 \[
  \xymatrix{
   \HtauwhatFtwo
   \ar[r]^/3pt/{\,\theta_*} \ar@{->>}[d]_/-2pt/{\rho_\tau}
   &
   H_\bullet(\Sigmahat; \F_2) \ar@{->>}[d]
   &
   &
   H^\bullet(\Sigmahat; \F_2)
   \ar[r]^/-4pt/{\,\theta^*}
   &
   \CHtauwhatFtwo
   \\
   \HtauwFtwo
   \ar[r]^/3pt/{\,\theta_*}
   &
   H_\bullet(\overlineSigma; \F_2)
   &
   &
   H^\bullet(\overlineSigma; \F_2)
   \ar[r]^/-4pt/{\,\theta^*} \ar@{((->}[u]
   &
   \CHtauwFtwo \ar@{((->}[u]_{\rho^\tau}
  }
 \]
Recalling that $\genus$ and $b$ are the genus and the number of boundary components of $\Sigma$, respectively, and that $u = |\{q \in \orb \suchthat \omega(q) = 1\}|$, we therefore have
 \begin{eqnarray*}
 \HonetauwFtwo \ \cong \ \CHonetauwFtwo &\cong & \begin{cases}
 \F_2^{2\genus + b -1} & \text{if the boundary of $\Sigma$ is not empty;}\\
 \F_2^{2\genus } & \text{if the boundary of $\Sigma$ is empty;}
 \end{cases}
 \\
 \text{and} \ \ \ \HonetauwhatFtwo \ \cong \  \CHonetauwhatFtwo &\cong & \begin{cases}
 \F_2^{2\genus + b + u -1} & \text{if the boundary of $\Sigma$ is not empty;} \\
  \F_2^{2\genus + u} & \text{if the boundary of $\Sigma$ is empty.}
  \end{cases}
 \end{eqnarray*}
\end{coro}

\begin{proof}
 The statement about the diagram on the left hand side follows from Propositions~\ref{prop:delta-set-homology-is-homology-of-realization} and~\ref{prop:embedding-is-quasi-isomorphism}.
 After dualizing one obtains the right-hand-side diagram.
 Finally, use the well-known fact that
\begin{eqnarray*}
\dim_{\F_2}(H_1(\overlineSigma; \F_2)) &=& \begin{cases}
2\genus+b-1 & \text{if the boundary of $\Sigma$ is not empty;}\\
2\genus & \text{if the boundary of $\Sigma$ is empty.}
\end{cases}\\
\text{and} \ \ \ 
\dim_{\F_2}(H_1(\Sigmahat; \F_2)) &=& \begin{cases}
2\genus+b+u-1 & \text{if the boundary of $\Sigma$ is not empty;}\\
2\genus+u & \text{if the boundary of $\Sigma$ is empty.}
\end{cases}
\end{eqnarray*}
\end{proof}

We will write $\theta_\tau$ and $\theta^\tau$ for the isomorphisms $\theta_*$ and $\theta^*$ appearing in Corollary~\ref{coro:homology-is-surface-homology} whenever it seems necessary to explicitly stress the dependence on the triangulation~$\tau$.

\begin{prop}
 \label{prop:theta-and-phi-compose-to-theta}
 Let $\sigma$ and $\tau$ be triangulations of $\Sigma$ that are related by a flip.
 Then the isomorphism $\varphi_{\tau,\sigma}^*: \CHonetauwhatFtwo \to \CHonetauwhatFtwo[\sigma]$ defined in Section~\ref{sec:colored-triangulations-and-flips} makes the following diagram commutative:
\[
  \xymatrix{
   &
   H^1(\Sigmahat; \F_2)
   \ar[dl]_/-2pt/{\theta^\tau}
   \ar[dr]^/-2pt/{\theta^\sigma} &
   \\
   \CHonetauwhatFtwo
   \ar[rr]^/-2pt/{\,\varphi_{\tau,\sigma}^*}
   & &
   \CHonetauwhatFtwo[\sigma]
  }
 \]
\end{prop}

\begin{proof}
 A careful (if desired case-by-case) inspection reveals that we defined $\varphi_{\tau, \sigma} : \CtauwhatFtwo[\sigma] \to \CtauwhatFtwo$ in Section~\ref{sec:colored-triangulations-and-flips} precisely in such a way to make the diagram in Proposition~\ref{prop:theta-and-phi-compose-to-theta} commutative.
\end{proof}

Given a colored triangulation $\tauc$ of $\SSigmaw$ let us write $x_{\tauc}$ for $(\theta^\tau)^{-1}\big([\widehat{\xi}\,]\big) \in H^1(\Sigmahat; \F_2)$.
The content of the following corollary is that the cohomology class~$x_{\tauc}$ associated with $\tauc$ is invariant under flips.

\begin{coro}
 \label{coro:mutation-invariant}
 If $\tauc$ and $\sigmad$ are colored triangulations of $\SSigmaw$ that happen to be related by a sequence of colored flips, then $x_{\tauc} = x_{\sigmad}$.
\end{coro}

\begin{proof}
Since any two ideal triangulations of $\SSigma=\surf$ can be obtained from each other by a finite sequence of flips (see \cite[Theorem 4.2]{FeShTu-orbifolds}), the corollary follows by repeated applications of Proposition~\ref{prop:theta-and-phi-compose-to-theta}.
\end{proof}

%% file: 09_flip_graph.tex
\section{Connectedness of flip graphs}

\label{sec:flip-graph}

\begin{defi}
 The \emph{cocycle flip graph} of $\SSigmaw$ is the unoriented simple graph~$\mathcal{G}(\SSigmaw)$ whose vertices are the colored triangulations of $\SSigmaw$.
 Two vertices~$\tauc$ and $\sigmad$ of~$\mathcal{G}(\SSigmaw)$ are joined by an edge if and only if $\tauc$ and $\sigmad$ are related by a colored flip.
\end{defi}

\begin{defi}
 The \emph{flip graph} of $\SSigmaw$ is the unoriented simple graph~$\mathcal{H}(\SSigmaw)$ obtained from $\mathcal{G}(\SSigmaw)$ as a quotient by identifying all vertices $\tauc$ and $(\tau,\xi')$ where $\xi$ and $\xi'$ are cohomologous.

 In other words, the vertices of $\mathcal{H}(\SSigmaw)$ are pairs $(\tau, x)$ where $\tau$ is a triangulation of $\SSigma$ and $x \in \CHonetauwFtwo$.
 Two vertices~$(\tau, x)$ and $(\sigma, y)$ of~$\mathcal{H}(\SSigmaw)$ are joined by an edge if and only if there are $\xi \in x$ and $\zeta \in y$ such that the colored triangulations~$\tauc$ and $\sigmad$ are related by a flip.
\end{defi}

\begin{thm}\label{thm:flip-graph-is-disconnected}
 The flip graph and the cocycle flip graph of $\SSigmaw$ are disconnected if $\Sigma$ is not a disk or a sphere. More precisely, the flip graph~$\mathcal{H}(\SSigmaw)$ has exactly $2^{2\genus+b-1}$ connected components if $\Sigma$ has non-empty boundary, and exactly $2^{2\genus}$ connected components if the boundary of $\Sigma$ is empty. 
\end{thm}

\begin{proof}
 For every triangulation $\tau$ of $\Sigma$ and cocycles~$\xi, \xi' \in \CZonetauw$ we know by Lemma~\ref{lem:hat-preserves-being-cohomologous} and Corollary~\ref{coro:homology-is-surface-homology} that $x_{\tauc} = x_{(\tau, \xi')}\in H^1(\Sigmahat; \F_2)$ if and only if $\xi$ and $\xi'$ are cohomologous.
 Hence, the rule $(\tau, [\xi]) \mapsto  x_{\tauc}$ defines a function~$\varphi$ from the vertex set of the flip graph~$\mathcal{H}(\SSigmaw)$ to the cohomology group $H^1(\Sigmahat; \F_2)$. For a fixed $\tau$, the map $[\xi]\mapsto x_{\tauc}$ is an injective (possibly non-linear) function $H^1(C^{\bullet}(\tau,\omega))\rightarrow H^1(\Sigmahat; \F_2)$ by Lemma \ref{lem:hat-preserves-being-cohomologous} and Corollary \ref{coro:homology-is-surface-homology} (see Remark \ref{lem:hat-preserves-being-cohomologous} as well).  Since $H^1(C^{\bullet}(\tau,\omega))\cong H^1(\overlineSigma; \F_2)$ by Corollary \ref{coro:homology-is-surface-homology}, this implies that the cardinality of the image of the function $\varphi$ is at least
 $$
 \begin{cases}
 2^{2\genus+b-1} & \text{if the boundary of $\Sigma$ is not empty;}\\
 2^{2\genus} & \text{if the boundary of $\Sigma$ is empty.}
 \end{cases}
 $$ 
Corollary~\ref{coro:mutation-invariant} says that $\varphi$ is constant on every connected component of~$\mathcal{H}(\SSigmaw)$.
On the other hand, for a fixed $\tau$, every vertex $(\sigma,[\zeta])$ of~$\mathcal{H}(\SSigmaw)$ lies on the connected component of $(\tau,[\xi])$ for some $\xi\in Z^1(\tau,\omega)$ by Theorem \ref{thm:ordinary-flip-graph-is-connected}. This proves that the number of connected components of~$\mathcal{H}(\SSigmaw)$ is exactly the number claimed.

Notice that $2^{2\genus+b-1} = 1 \Leftrightarrow 2\genus + b = 1 \Leftrightarrow \text{$\Sigma$ is a disk}$, and $2^{2\genus} = 1 \Leftrightarrow 2\genus = 0 \Leftrightarrow \text{$\Sigma$ is a sphere}$. 
\end{proof}

We illustrate Theorem \ref{thm:flip-graph-is-disconnected} by means of an example.
From here to the end of this section, let $\SSigma=\surf$ be the once-punctured torus with empty boundary and exactly one orbifold point $q$. For each of the two possible values of $\omega(q)$ the flip graph of $\SSigmaw$ has exactly $2^{2}=4$ connected components by Theorem \ref{thm:flip-graph-is-disconnected}.

\label{subsec:torus-1-orb-point}
Every triangulation of $\SSigma$ has the following form:
\begin{center}
  \begin{tikzpicture}[scale=0.4]
    \coordinate (0) at ( 0, 0);
    \coordinate (1) at (-3, 3);
    \coordinate (2) at (-3,-3);
    \coordinate (3) at ( 3,-3);
    \coordinate (4) at ( 3, 3);

    \node[rotate=45] at (0) {$\times$};
    \foreach \i in {1,...,4}
      \node at (\i) {$\bullet$};

    \draw (0) to (4);
    \draw (1) to node {\rotatebox{90}{$<$}} (2);
    \draw (2) to node {$\gg$} (3);
    \draw (3) to node {\rotatebox{90}{$<$}} (4);
    \draw (4) to node {$\gg$} (1);

    \draw[bend right] (2) to (4);
    \draw[bend left]  (2) to (4);
  \end{tikzpicture}
\end{center}
Hence, for any such triangulation $\tau$ and any function $\omega:\orb\rightarrow\{1,4\}$, every weighted quiver in the mutation class of $(Q(\tau,\omega),\dtuple(\tau,\omega))$ is isomorphic to the weighted quiver in Figure \ref{eq:torus_1punct_1orb_quiver}.
\begin{figure}[!ht]
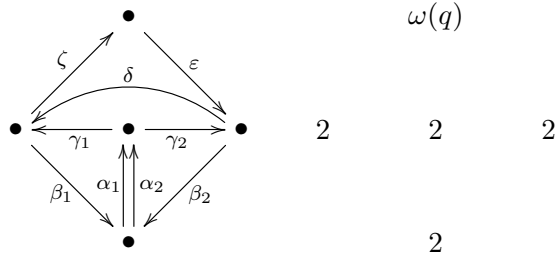

\[
 \xygraph{
   !{<0cm,0cm>;<1.5cm,0cm>:<0cm,1.5cm>::}
   !{( 0, 0)}*+{\bullet}="1"
   !{( 0,-1)}*+{\bullet}="2"
   !{(-1, 0)}*+{\bullet}="3"
   !{( 1, 0)}*+{\bullet}="4"
   !{( 0, 1)}*+{\bullet}="5"
   "1" :                 "3"  ^/-3pt/*-<2pt>{\labelstyle \gamma_1}
   "3" :                 "2"  _/0pt/*-<8pt>{\labelstyle \beta_1}
   "2" :@<-2pt>          "1"  ^/0pt/*+<2pt>{\labelstyle \alpha_1}
   "1" :                 "4"  _/-3pt/*-<2pt>{\labelstyle \gamma_2}
   "4" :                 "2"  ^/-1pt/*-<6pt>{\labelstyle \beta_2}
   "2" :@<+2pt>          "1"  _/0pt/*+<5pt>{\labelstyle \alpha_2}
   "3" :                 "5"  ^/0pt/*-<4pt>{\labelstyle \zeta}
   "5" :                 "4"  ^/0pt/*-<4pt>{\labelstyle \varepsilon}
   "4" :@<+3pt>@(lu,ru)  "3"  _/0pt/*-<3pt>{\labelstyle \delta}
 } \ \ \ \ \
 \xygraph{
   !{<0cm,0cm>;<1.5cm,0cm>:<0cm,1.5cm>::}
   !{( 0, 0)}*+{2}="1"
   !{( 0,-1)}*+{2}="2"
   !{(-1, 0)}*+{2}="3"
   !{( 1, 0)}*+{2}="4"
   !{( 0, 1)}*+{\omega(q)}="5"
 }
\]
\caption{The depicted weighted quiver is isomorphic to the weighted quiver of any triangulation of the once-punctured closed torus with one orbifold point}
\label{eq:torus_1punct_1orb_quiver}
\end{figure}

In Figure \ref{Fig:torus_1pnct_1orb_v2} we can see the (geometric realization $Y(\tau,\omega)$ of the) complex $C_\bullet(\tau,\omega)$ for the two possible values of $\omega$.
\begin{figure}[!ht]
                \centering
                \includegraphics[scale=.17]{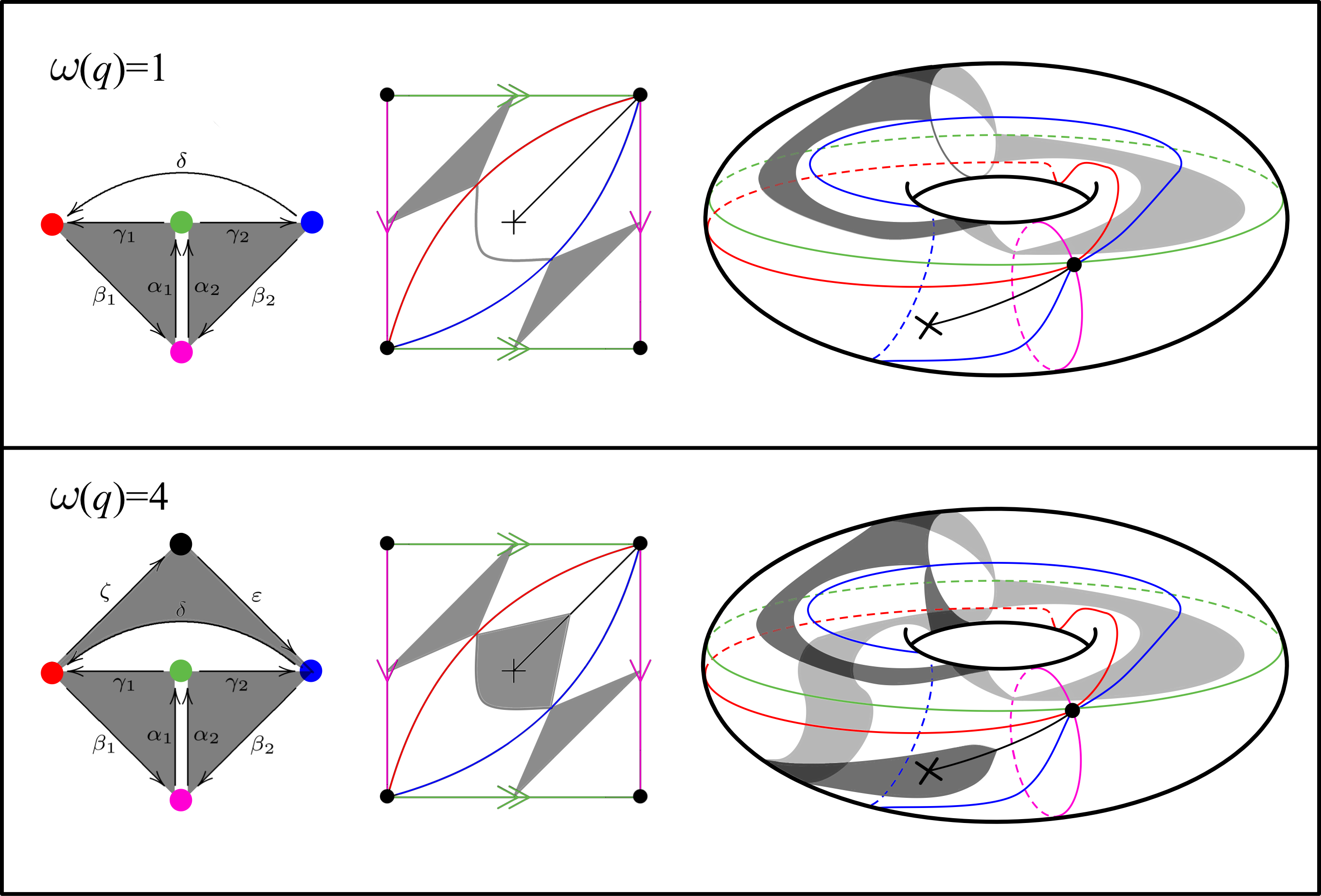}
                \caption{The geometric realization $Y(\tau,\omega)$ of the complex $C_\bullet(\tau,\omega)$ for the two possible values of $\omega$}
                \label{Fig:torus_1pnct_1orb_v2}
        \end{figure}

Let $k$ be the pending arc in $\tau$, and let $\sigma$ be the triangulation of $\SSigma$ obtained from $\tau$ by flipping~$k$. 

\begin{ex}\label{torus_1pnct_1orb_omega=1} 
%
Suppose $\omega(q)=1$. Then the bijective function $Z^1(\tau,\omega)\rightarrow Z^1(\sigma,\omega)$ that underlies the definition of colored flip is not $\F_2$-linear. However, it does have the property of being constant on any given cohomology class, hence it induces a (non-linear) bijective function $H^1(\tau,\omega)\rightarrow H^1(\sigma,\omega)$.

The set $Q'_1(\tau,\omega)=\{\alpha_1,\beta_1,\gamma_1,\alpha_2,\beta_2,\gamma_2,\delta\}$ is an $\F_2$-basis of $C_1(\tau,\omega)$, and a function $\xi\in\Hom_{\F_2}(C_1(\tau,\omega),\F_2)$ is a 1-cocycle of the cochain complex $C^\bullet(\tau,\omega)$ if and only if
$$
\xi(\alpha_1)+\xi(\beta_1)+\xi(\gamma_1)=0=\xi(\alpha_2)+\xi(\beta_2)+\xi(\gamma_2).
$$
 Furthermore, if $\{\alpha_1^\vee,\beta_1^\vee,\gamma_1^\vee,\alpha_2^\vee,\beta_2^\vee,\gamma_2^\vee,\delta^\vee\}$ is the basis of $C^1(\tau,\omega)$ which is $\F_2$-dual to $Q'_1(\tau,\omega)$, then the cohomology classes of the cocycles
$$
\alpha_1^\vee+\beta_1^\vee \ \ \ \text{and} \ \ \ \delta^\vee
$$
form an $\F_2$-basis of $H^1(C^\bullet(\tau,\omega))$. Therefore,
$$
H^1(C^\bullet(\tau,\omega))=\{[0],[\alpha_1^\vee+\beta_1^\vee],[\delta^\vee],[\alpha_1^\vee+\beta_1^\vee+\delta^\vee]\}
$$
and hence, the connected components of the flip graph~$\mathcal{H}(\SSigmaw)$ are precisely the connected components where the vertices $(\tau,[0])$, $(\tau,[\alpha_1^\vee+\beta_1^\vee])$, $(\tau,[\delta^\vee])$ and $(\tau,[\alpha_1^\vee+\beta_1^\vee+\delta^\vee])$ lie.

Let us abuse notation and use the same greek letters from Figure \ref{eq:torus_1punct_1orb_quiver} to denote the names of the arrows of $Q(\sigma,\omega)$. Then the aforementioned bijective function $Z^1(\tau,\omega)\rightarrow Z^1(\sigma,\omega)$ is given by the rule
$$
\xi\mapsto \xi+\delta^\vee,
$$
and the induced function
$H^1(\tau,\omega)\rightarrow H^1(\sigma,\omega)$ is given by
\begin{center}
\begin{tabular}{cc}
$[0]\mapsto[\delta^\vee]$, & $[\alpha_1^\vee+\beta_1^\vee]\mapsto[\alpha_1^\vee+\beta_1^\vee+\delta^\vee]$,\\
$[\delta^\vee]\mapsto[0]$, & $[\alpha_1^\vee+\beta_1^\vee+\delta^\vee]\mapsto[\alpha_1^\vee+\beta_1^\vee]$.
\end{tabular}
\end{center}
\end{ex}

\begin{ex}\label{torus_1pnct_1orb_omega=4} 
Suppose $\omega(q)=4$. Then the bijective function $Z^1(\tau,\omega)\rightarrow Z^1(\sigma,\omega)$ that underlies the definition of colored flip is $\F_2$-linear and constant on any given cohomology class, hence it induces a (linear) bijective function $H^1(\tau,\omega)\rightarrow H^1(\sigma,\omega)$.

The set $Q'_1(\tau,\omega)=\{\alpha_1,\beta_1,\gamma_1,\alpha_2,\beta_2,\gamma_2,\delta,\varepsilon,\zeta\}$ is an $\F_2$-basis of $C_1(\tau,\omega)$, and a function $\xi\in\Hom_{\F_2}(C_1(\tau,\omega),\F_2)$ is a 1-cocycle of the cochain complex $C^\bullet(\tau,\omega)$ if and only if
$$
\xi(\alpha_1)+\xi(\beta_1)+\xi(\gamma_1)=\xi(\alpha_2)+\xi(\beta_2)+\xi(\gamma_2)=\xi(\delta)+\xi(\varepsilon)+\xi(\zeta)=0.
$$
 Furthermore, if $\{\alpha_1^\vee,\beta_1^\vee,\gamma_1^\vee,\alpha_2^\vee,\beta_2^\vee,\gamma_2^\vee,\delta^\vee,\varepsilon^\vee,\zeta^\vee\}$ is the basis of $C^1(\tau,\omega)$ which is $\F_2$-dual to $Q'_1(\tau,\omega)$, then the cohomology classes of the cocycles
$$
\alpha_1^\vee+\beta_1^\vee \ \ \ \text{and} \ \ \ \delta^\vee+\varepsilon^\vee
$$
form an $\F_2$-basis of $H^1(C^\bullet(\tau,\omega))$. Therefore,
$$
H^1(C^\bullet(\tau,\omega))=\{[0],[\alpha_1^\vee+\beta_1^\vee],[\delta^\vee+\varepsilon^\vee],[\alpha_1^\vee+\beta_1^\vee+\delta^\vee+\varepsilon^\vee]\}
$$
and hence, the connected components of the flip graph~$\mathcal{H}(\SSigmaw)$ are precisely the connected components where the vertices $(\tau,[0])$, $(\tau,[\alpha_1^\vee+\beta_1^\vee])$, $(\tau,[\delta^\vee+\varepsilon^\vee])$ and $(\tau,[\alpha_1^\vee+\beta_1^\vee+\delta^\vee+\varepsilon^\vee])$ lie.

Let us abuse notation and use the same greek letters from Figure \ref{eq:torus_1punct_1orb_quiver} to denote the names of the arrows of $Q(\sigma,\omega)$. Then, using the fact that $+=-$ in every $\F_2$-vector space, the aforementioned functions $Z^1(\tau,\omega)\rightarrow Z^1(\sigma,\omega)$
and
$H^1(\tau,\omega)\rightarrow H^1(\sigma,\omega)$ are the identity.
\end{ex}

In both Examples \ref{torus_1pnct_1orb_omega=1} and \ref{torus_1pnct_1orb_omega=4}, a bit of work shows that for every cocycle $\xi\in Z^1(\tau,\omega)$ the Jacobian algebra $\mathcal{P}(A(\tau,\xi),S(\tau,\xi))$ is infinite-dimensional over $F$. Indeed, for any of the two values of $\omega$ and any cocycle $\xi\in Z^1(\tau,\omega)$ we have $S(\tau,\xi)=\alpha_1\beta_1\gamma_1+\alpha_2\beta_2\gamma_2+\delta\varepsilon\zeta\in \RA{A(\tau,\xi)}$, from which it follows that  no positive power of $\alpha_1\beta_2\varepsilon\zeta\gamma_1\alpha_2\beta_1\delta\gamma_2$ belongs to the Jacobian ideal $J(S(\tau,\xi))\subseteq \RA{A(\tau,\xi)}$.

%% file: 10_cohomology_and_jacobian_algs.tex
\section{Cohomology and Jacobian algebras}

\label{sec:cohomology-and-Jacobian-algebras}

Let $\SSigmaw=(\Sigma,\marked,\orb,\omega)$ be either unpunctured with (arbitrarily many) weighted orbifold points, or once-punctured closed with (arbitrarily many) weighted orbifold points, see Definition \ref{def:surface-with-weighted-orb-points} and the paragraphs of Section \ref{sec:surfaces-and-triangs} that precede Remark \ref{rem:why-terminology-order-2}. Let $\tau$ be a triangulation of $\SSigma=\surf$, and $E/F$ be a field extension as in Section \ref{subsec:species-of-a-triangulation}. For each 1-cocycle $\xi\in \CZonetauwFtwo\subseteq C^1(\tau,\omega) = \Hom_{\F_2}(\F_2 X_1(\tauw), \F_2)$, we have associated to the colored triangulation $\tauc$ a species $\Atauc$ over $E/F$ and a potential $\Stauc\in\RA{\Atauc}$. Recall from \cite[Definition 3.11]{Geuenich-Labardini-1} that for each arrow $a$ of $\Atauc$ the \emph{cyclic derivative} with respect to $a$ is defined to be
\begin{equation}\label{eq:cyclic-derivative}
\partial_a(\omega_0 a_1\omega_1 a_2\omega_2\ldots \omega_{\ell-1}a_\ell\omega_\ell)=\sum_{j=1}^\ell \delta_{a,a_j}\pi_{g(\tau,\xi)_a^{-1}}(\omega_j a_{j+1}\ldots a_\ell\omega_\ell\omega_0 a_1\ldots a_{j-1}\omega_{j-1})
\end{equation}
for each cyclic path
$\omega_0 a_1\omega_1 a_2\omega_2\ldots \omega_{\ell-1}a_\ell\omega_\ell\in\RA{\Atauc}$,
where $\delta_{a,a_k}$ is the Kronecker delta between $a$ and $a_k$, and
$\pi_{g(\tau,\xi)_a^{-1}}(x)=\frac{1}{d_{h(a),t(a)}}\sum_{\nu\in\B_{h(a),t(a)}}g(\tau,\xi)_a^{-1}(\nu^{-1})x\nu$.
Recall also that the \emph{Jacobian ideal} $J(\Stauc)$ is the topological closure of the two-sided ideal of
$\RA{\Atauc}$ generated by $\{\partial_a(\Stauc)\suchthat a$ is an arrow of $Q(\tau,\omega)\}$, and that the
\emph{Jacobian algebra} of $\AStauc$ is the quotient $\Jacalg{\AStauc}:=\RA{\Atauc}/J(\Stauc)$. The main result of this section is the following:

\begin{thm}\label{thm:comologous<=>isomorphic-Jacobian-algs}  Let $\xi_1,\xi_2\in Z^1(\tau,\omega)$ be 1-cocycles of the cochain complex $\CCtauwFtwo$. If $(\Sigma,\marked,\orb)$ is unpunctured, then the following two statements are equivalent:
\begin{itemize}
\item $[\xi_1]=[\xi_2]$ in the first cohomology group $H^1(C^\bullet(\tau,\omega))$;
\item the Jacobian algebras $\Jacalg{(A(\tau,\xi_1),S(\tau,\xi_1))}$ and $\Jacalg{(A(\tau,\xi_2),S(\tau,\xi_2))}$ are isomorphic through an $F$-linear ring isomorphism acting as the identity on $\{e_k\suchthat k\in Q_0\}$.
\end{itemize}
If $(\Sigma,\marked,\orb)$ is once-punctured closed, then the first statement implies the second one.
\end{thm}

\begin{proof} Suppose that $\xi_1$ and $\xi_2$ are homologous 1-cocycles of the cochain complex $\CCtauwFtwo$. This means that there exists a function $\phi:Q'_0(\tauw)\rightarrow\F_2$ such that
\begin{equation}\label{eq:cohomologous-cocycles}
\text{$\xi_2(a)-\xi_1(a)=\phi(h(a))-\phi(t(a))$ for every $a\in X_1(\tauw) = Q'_1(\tauw) $.}
\end{equation}
We shall use $\phi$ to produce a ring automorphism $\Psi^{(0)}:R\rightarrow R$ and a group isomorphism $\Psi^{(1)}:A(\tau,\xi_1)\rightarrow A(\tau,\xi_2)$ such that $\Psi^{(1)}(ra)=\Psi^{(0)}(r)\Psi^{(1)}(a)$ and $\Psi^{(1)}(ar)=\Psi^{(1)}(a)\Psi^{(0)}(r)$ for all $r\in R$ and all $a\in A(\tau,\xi_1)$.

For each $j\in\tau$ such that $\dtauomega{j}=4$ choose an element $\lambda_j\in G_k=\Gal(E/F)$ such that $\lambda_j|_L=\theta^{\phi(j)}$. For $k\in\tau$, we set
$$
\psi_k=\begin{cases}\myid_{F} & \text{if $\dtauomega{k}=1$;}\\
\theta^{\phi(k)} & \text{if $\dtauomega{k}=2$;}\\
\lambda_k & \text{if $\dtauomega{k}=4$.}
\end{cases}
$$
Then we define $\Psi^{(0)}\left(\sum_{k\in\tau}x_ke_k\right)=\sum_{k\in\tau}\psi_k(x_k)e_k$.
It is clear that $\Psi^{(0)}$ is a ring automorphism of $R$.

Now, let $a$ be an arrow of the quiver $\Qtauomega$. If $\dtauomega{h(a)}\dtauomega{t(a)}\neq 16$, then from \eqref{eq:cohomologous-cocycles} we can easily deduce that in the bimodule $F_{h(a)}^{g(\tau,\xi_2)_a}\otimes_{F_{h(a),t(a)}}F_{t(a)}$ we have
\begin{eqnarray*}
\psi_{h(a)}(x)\otimes \psi_{t(a)}(zy)&=&\psi_{h(a)}(x)\otimes \psi_{t(a)}(z)\psi_{t(a)}(y) \\
&=&
\psi_{h(a)}(x)g(\tau,\xi_2)_a(\psi_{t(a)}(z))\otimes \psi_{t(a)}(y) \\
&=&
\psi_{h(a)}(x)\psi_{h(a)}(g(\tau,\xi_1)_a(z))\otimes \psi_{t(a)}(y)
\end{eqnarray*}
for $x\in F_{h(a)}$, $y\in F_{t(a)}$ and $z\in F_{h(a),t(a)}$, and hence, that
the rule
$x\otimes y\mapsto \psi_{h(a)}(x)\otimes \psi_{t(a)}(y)$
produces a well-defined group homomorphism $F_{h(a)}^{g(\tau,\xi_1)_a}\otimes_{F_{h(a),t(a)}}F_{t(a)}\rightarrow F_{h(a)}^{g(\tau,\xi_2)_a}\otimes_{F_{h(a),t(a)}}F_{t(a)}$.

If, on the other hand, we have $\dtauomega{h(a)}\dtauomega{t(a)}=16$ instead, then $\dtauomega{h(a)}=4=\dtauomega{t(a)}$ and $\Qtauomega$ has exactly two arrows $\delta_0$ and $\delta_1$ going from $t(a)$ to $h(a)$ (of course, $a$ is one of these two arrows). From \eqref{eq:cohomologous-cocycles} and Definition \ref{def:cocycle->modulating-function} we deduce that $g(\tau,\xi_2)_{\delta_0}|_L=g(\tau,\xi_2)_{\delta_1}|_L=(\psi_{h(a)}g(\tau,\xi_1)_{\delta_0}\psi_{t(a)}^{-1})|_L=(\psi_{h(a)}g(\tau,\xi_1)_{\delta_1}\psi_{t(a)}^{-1})|_L$, and consequently, that $\{g(\tau,\xi_2)_{\delta_0}, g(\tau,\xi_2)_{\delta_1}\}=\{\psi_{h(a)}g(\tau,\xi_1)_{\delta_0}\psi_{t(a)}^{-1}, \psi_{h(a)}g(\tau,\xi_1)_{\delta_1}\psi_{t(a)}^{-1}\}$. Hence, there is a permutation $p:\{\delta_0,\delta_1\}\rightarrow\{\delta_0,\delta_1\}$ such that $g(\tau,\xi_2)_{p(\delta_i)}=\psi_{h(a)}g(\tau,\xi_1)_{\delta_{i}}\psi_{t(a)}^{-1}$ for $i\in\{0,1\}$. Therefore, there is a well-defined group homomorphism $F_{h(a)}^{g(\tau,\xi_1)_a}\otimes_{F_{h(a),t(a)}}F_{t(a)}\rightarrow F_{h(a)}^{g(\tau,\xi_2)_{p(a)}}\otimes_{F_{h(a),t(a)}}F_{t(a)}$ given by the rule $x\otimes y\mapsto \psi_{h(a)}(x)\otimes \psi_{t(a)}(y)$.

We have thus constructed a group homomorphism for each arrow of the quiver $\Qtauomega$. Assembling all the group homomorphisms constructed, and recalling that $A(\tau,\xi_j)=\bigoplus_{a\in Q_1(\tau,\omega)}F_{h(a)}^{g(\tau,\xi_j)_a}\otimes_{F_{h(a),t(a)}} F_{t(a)}$ for $j\in\{1,2\}$, we obtain a group homomorphism $\Psi^{(1)}:A(\tau,\xi_1)\rightarrow A(\tau,\xi_2)$. It is clear that $\Psi^{(1)}$ is a group isomorphism and that it satisfies $\Psi^{(1)}(ra)=\Psi^{(0)}(r)\Psi^{(1)}(a)$ and $\Psi^{(1)}(ar)=\Psi^{(1)}(a)\Psi^{(0)}(r)$ for all $r\in R$ and all $a\in A(\tau,\xi_1)$. A minor variation of \cite[Proposition 2.11]{Geuenich-Labardini-1} then implies that there exists a continuous ring isomorphism $\Psi:\RA{A(\tau,\xi_1)}\rightarrow\RA{A(\tau,\xi_2)}$ such that $\Psi|_R=\Psi^{(0)}$ and $\Psi|_{A(\tau,\xi_1)}=\Psi^{(1)}$. Since $\Psi^{(0)}$ is clearly $F$-linear, $\Psi$ is $F$-linear.

Consider the potential $\Psi(S(\tau,\xi_1))=\sum_{\triangle}\Psi(S^\triangle(\tau,\xi_1))$ (which may be not equal to $S(\tau,\xi_2)$ because of the presence of the factor $u$ in the first item of Definition \ref{def:cycles-from-triangs-with-2-orb-pts}). Using the fact that $u$ is an eigenvector of the two elements of $\Gal(L/F)$ with the corresponding eigenvalues lying in $F$ (in fact, these eigenvalues are $1$ and $-1$), it is fairly easy to check that $\Psi(\partial_{a}(S(\tau,\xi_1)))=\partial_{\Psi(a)}(\Psi(S(\tau,\xi_1)))$ for every arrow $a$ of the quiver $\Qtauomega$. It follows that $\Psi(J(S(\tau,\xi_1)))\subseteq J(\Psi(S(\tau,\xi_1)))$. Applying the same reasoning to $\Psi^{-1}$ we obtain $\Psi^{-1}(J(\Psi(S(\tau,\xi_1))))\subseteq J(\Psi^{-1}\Psi(S(\tau,\xi_1)))=J(S(\tau,\xi_1))$, and hence  $J(\Psi(S(\tau,\xi_1)))=\Psi\Psi^{-1}(J(\Psi(S(\tau,\xi_1))))\subseteq \Psi(J(S(\tau,\xi_1)))$. Therefore, $\Psi(J(S(\tau,\xi_1)))= J(\Psi(S(\tau,\xi_1)))$.

Finally, using again the fact that $u$ is an eigenvector of the two elements of $\Gal(L/F)$ with the corresponding eigenvalues lying in $F$, we see that the relation that $\Psi(S(\tau,\xi_1))$ keeps with $S(\tau,\xi_2)$ is that it can be obtained from it by multiplying some of its constituent cyclic paths by elements of $F$ (actually, these elements are $1$ and $-1$).  And noticing that $S(\tau,\xi_2)$ has an expression as an $F$-linear combination of cyclic paths that has the property that every cyclic path appearing in it involves at least one arrow that does not appear in any other cyclic path in the expression, it is easy to produce an $R$-algebra automorphism $\Phi$ of $\RA{A(\tau,\xi_2)}$ that sends $\Psi(S(\tau,\xi_1))$ to a potential cyclically equivalent to $S(\tau,\xi_2)$. By the previous paragraph and \cite[Lemma 10.3 and the paragraph that precedes it]{Geuenich-Labardini-1}, we deduce that $\Phi\Psi(J(S(\tau,\xi_1)))= \Phi(J(\Psi(S(\tau,\xi_1))))=J(\Phi\Psi(S(\tau,\xi_1)))=J(S(\tau,\xi_2))$.

We have thus proved that, regardless of whether $\SSigma$ is unpunctured or once-punctured closed, if $[\xi_1]=[\xi_2]$ in $H^1(C^{\bullet}(\tau,\omega))$, then the Jacobian algebras $\Jacalg{(A(\tau,\xi_1),S(\tau,\xi_1))}$ and $\Jacalg{(A(\tau,\xi_2),S(\tau,\xi_2))}$ are isomorphic through an $F$-linear ring isomorphism acting as the identity on $\{e_k\suchthat k\in Q_0\}$. The converse implication for unpunctured surfaces requires some preparation.

From this point to the end of this section we shall suppose that $\SSigma$ is unpunctured with non-empty boundary. We start by establishing a result of independent interest, namely:

\begin{thm}\label{thm:Jac-alg-is-fd-for-unpunctured} If $\SSigmaw=(\Sigma,\marked,\orb,\omega)$ is an unpunctured surface with weighted orbifold points, then for any colored triangulation $(\tau,\xi)$ of $\SSigmaw$ the Jacobian algebra $\mathcal{P}(A(\tau,\xi),S(\tau,\xi))$ has finite dimension over $F$; more precisely, there exists a positive integer $t$ such that $\mathfrak{m}^t\subseteq J(S(\tau,\xi))\subseteq\mathfrak{m}^2$, where $\mathfrak{m}$ is the two-sided ideal of $\RA{A(\tau,\xi)}$ generated by the arrows of the underlying quiver $Q(\tau,\omega)$.
\end{thm}

%
%

\begin{proof}
Let $\SSigmaw$ be as in the hypothesis of the theorem, and let $(\tau,\xi)$ be any colored triangulation of $\SSigmaw$. For every arrow $a$ of $Q(\tau,\omega)$, the explicit expression of $\partial_a(S(\tau,\xi))$ as an element of $\RA{A(\tau,\xi)}$ is given by the table depicted in Figure \ref{Fig:cyclic_derivative}. Hence, for any two arrows $b$ and $c$ of $Q(\tau,\omega)$ that are induced by the same triangle, if $t(b)=h(c)$ and $t(b)$ is not a pending arc, then $bzc\in J(S(\tau,\xi))$ for every element $z\in F_{t(b)}e_{t(b)}= Le_{t(b)}\subseteq R$.

\begin{figure}
\begin{center}
{
  \renewcommand{\arraystretch}{1.5}
  \newcommand{\anshift}{-3pt}
  \newcommand{\myincg}[1]{%
   \raisebox{-0.4cm}[0.7cm][0.5cm]{\includegraphics[height=1cm]{#1}}%
  }

  \tiny

  \begin{tabular}{|c|c|c|c|c|c|}
  \hline
  &
  $\partial_\alpha(S(\tau,\xi))$
  &
  $\partial_\beta(S(\tau,\xi))$
  &
  $\partial_\gamma(S(\tau,\xi))$
  &
  $\partial_{\delta_0}(S(\tau,\xi))$
  &
  $\partial_{\delta_1}(S(\tau,\xi))$
  \\
  \hline
  \myincg{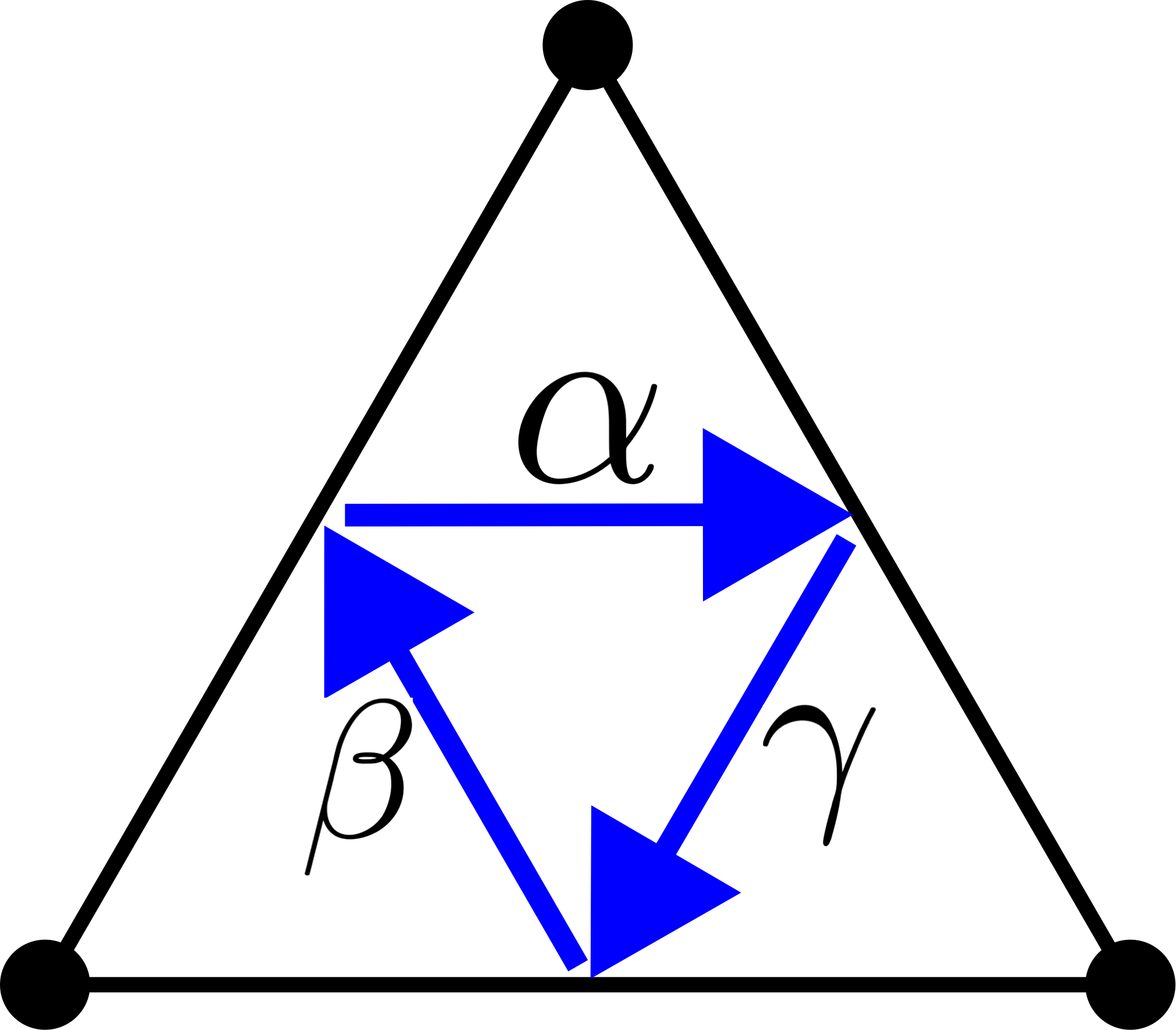}
  &
  $\beta\gamma$
  &
  $\gamma\alpha$
  &
  $\alpha\beta$
  &
  &
  \\
  \hline
  \myincg{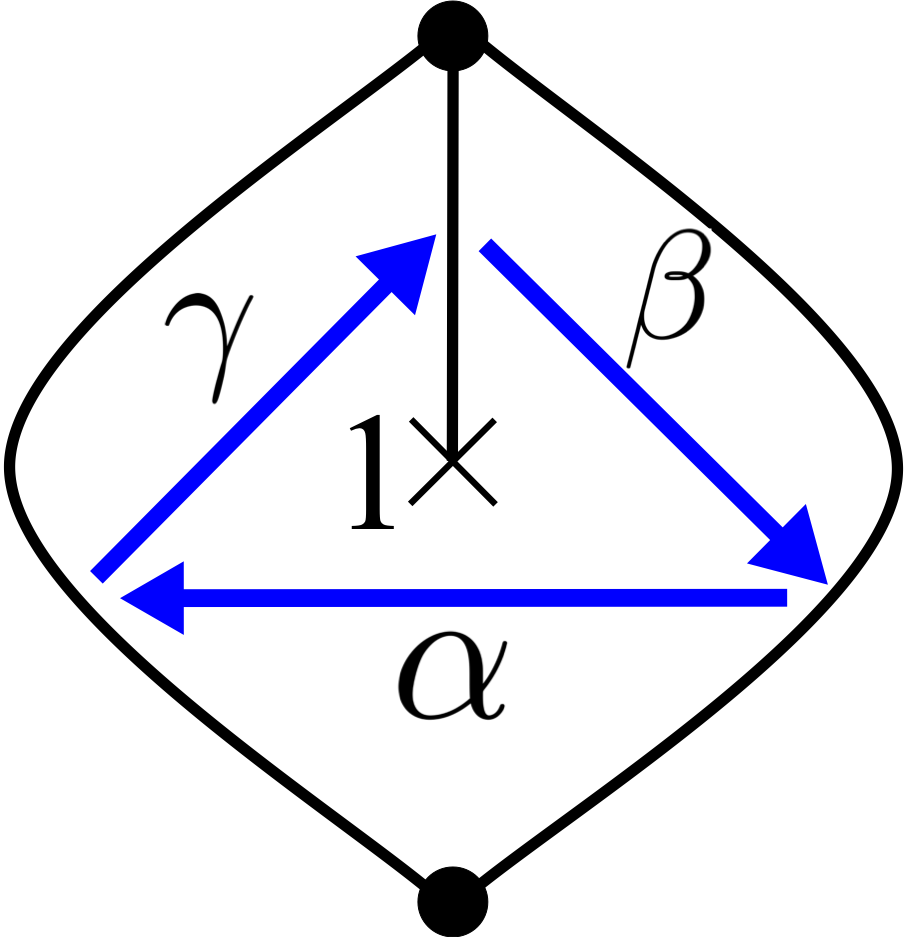}
  &
  $
   \begin{array}{l}
    \frac{1}{2}(\beta\gamma +
    \\ [\anshift]
    \theta^{-\xi(\alpha)}(u^{-1})\beta\gamma u)
   \end{array}
  $
  &
  $\gamma\alpha$
  &
  $\alpha\beta$
  &
  &
  \\
  \hline
  \myincg{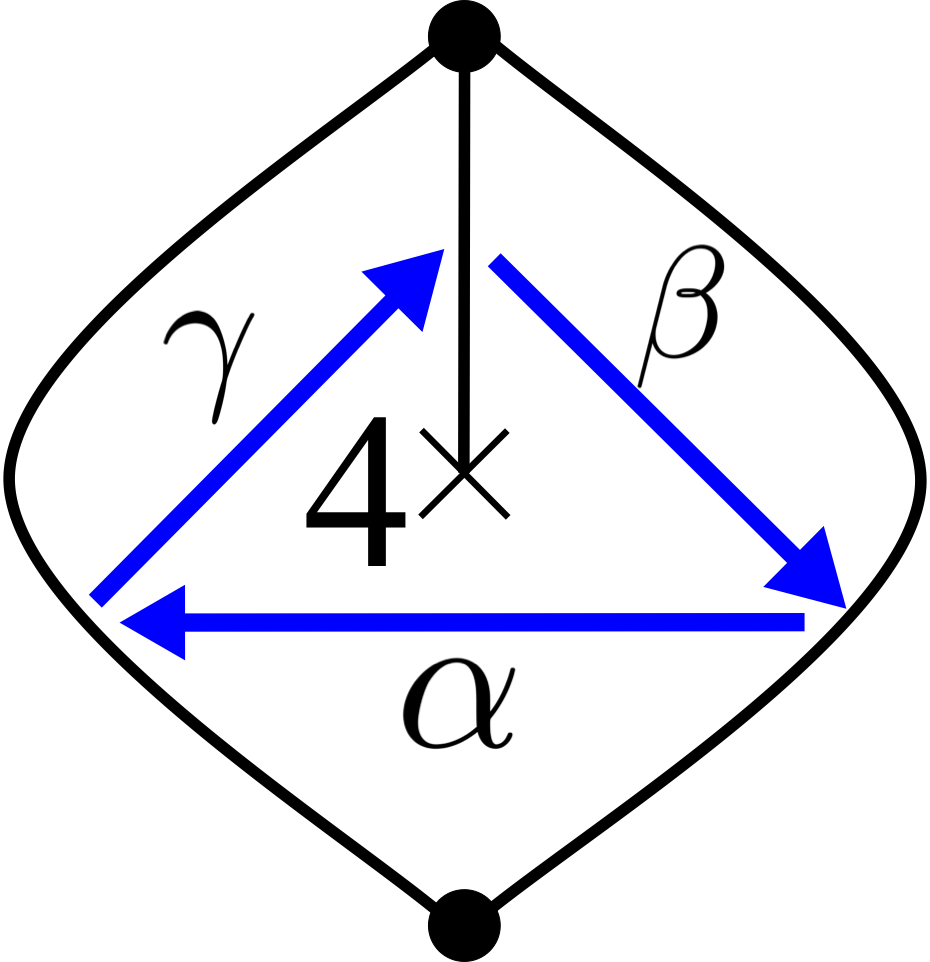}
  &
  $\beta\gamma$
  &
  $\gamma\alpha$
  &
  $\alpha\beta$
  &
  &
  \\
  \hline
  \myincg{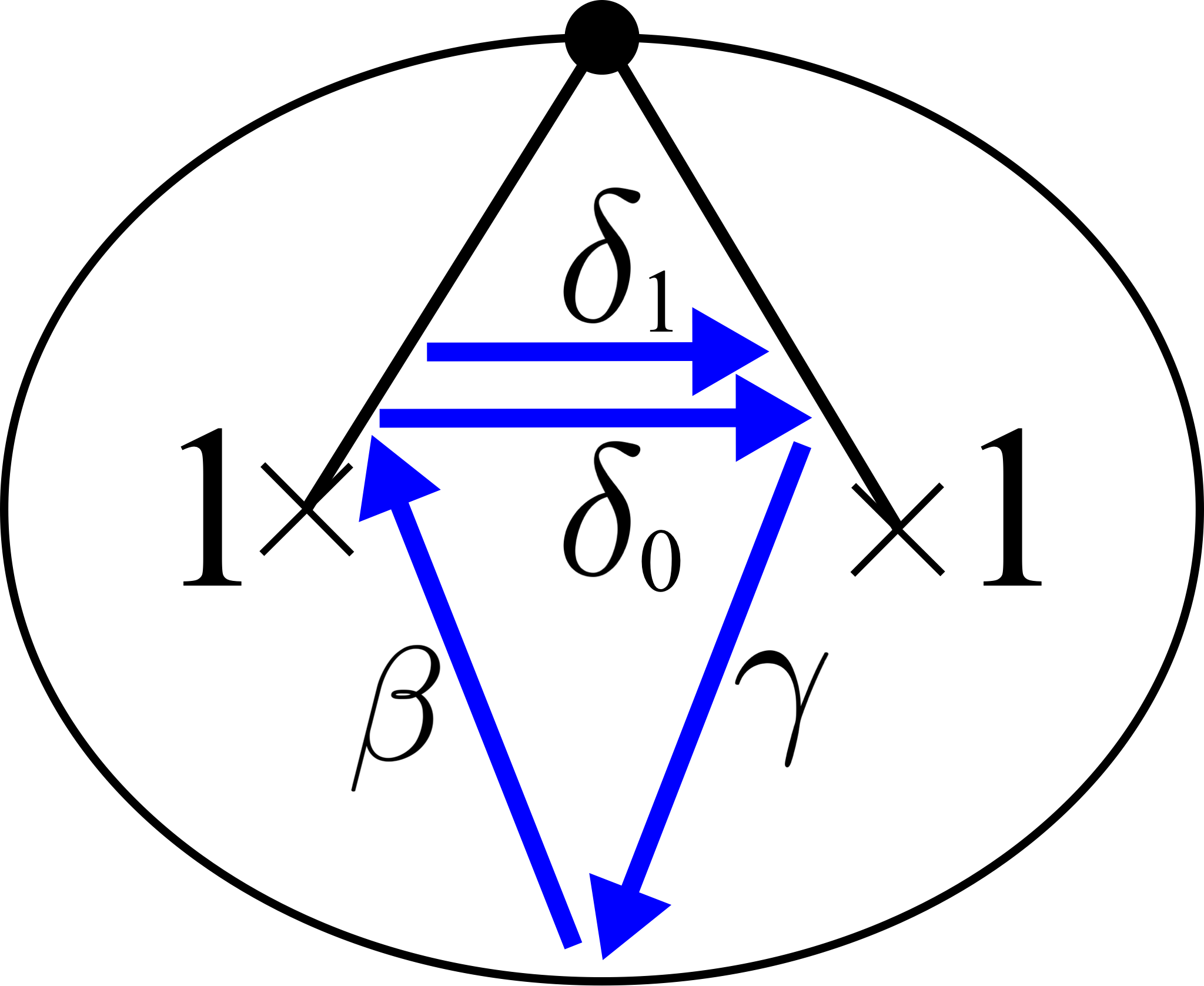}
  &
  &
  $\gamma\delta_0+u\gamma\delta_1$
  &
  $\delta_0\beta+\delta_1\beta u$
  &
  $\beta\gamma$
  &
  $\beta u\gamma$
  \\
  \hline
  \myincg{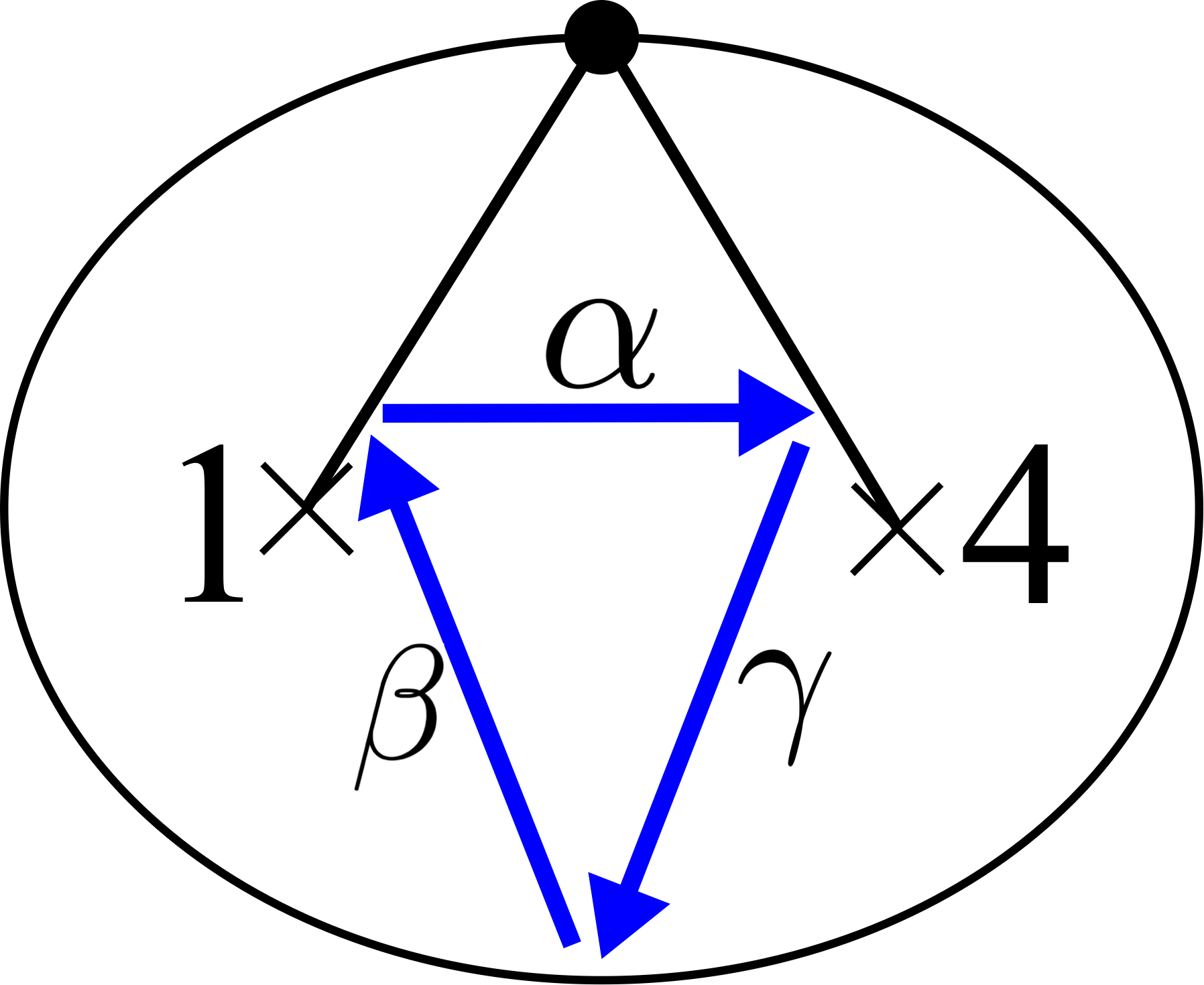}
  &
  $\beta\gamma$
  &
  $\gamma\alpha$
  &
  $
   \begin{array}{l}
    \frac{1}{2}(\alpha\beta +
    \\ [\anshift]
    \theta^{-\xi(\gamma)}(u^{-1})\alpha\beta u)
   \end{array}
  $
  &
  &
  \\
  \hline
  \myincg{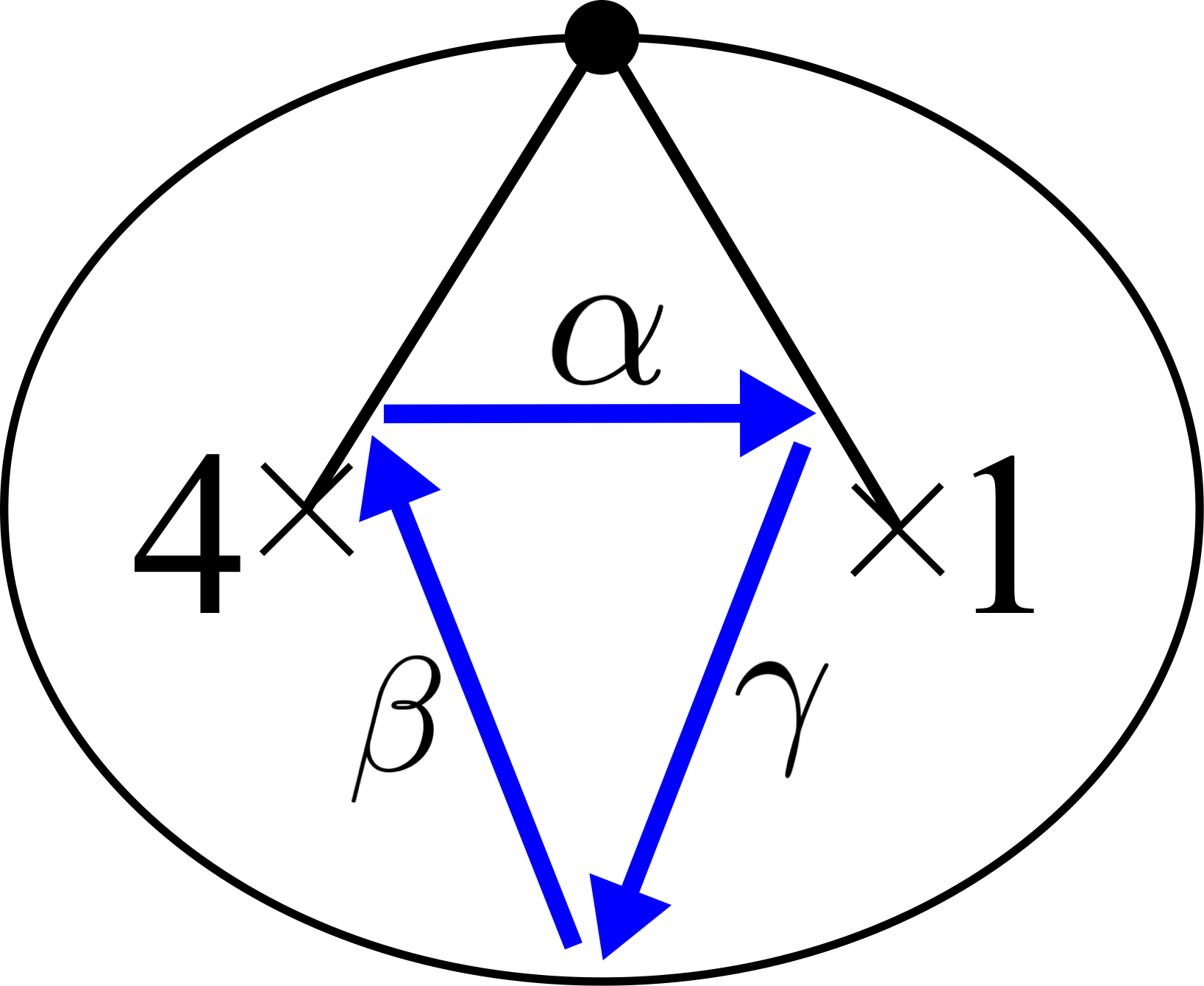}
  &
  $\beta\gamma$
  &
  $
   \begin{array}{l}
    \frac{1}{2}(\gamma\alpha +
    \\ [\anshift]
    \theta^{-\xi(\beta)}(u^{-1})\gamma\alpha u)
   \end{array}
  $
  &
  $\alpha\beta$
  &
  &
  \\
  \hline
  \myincg{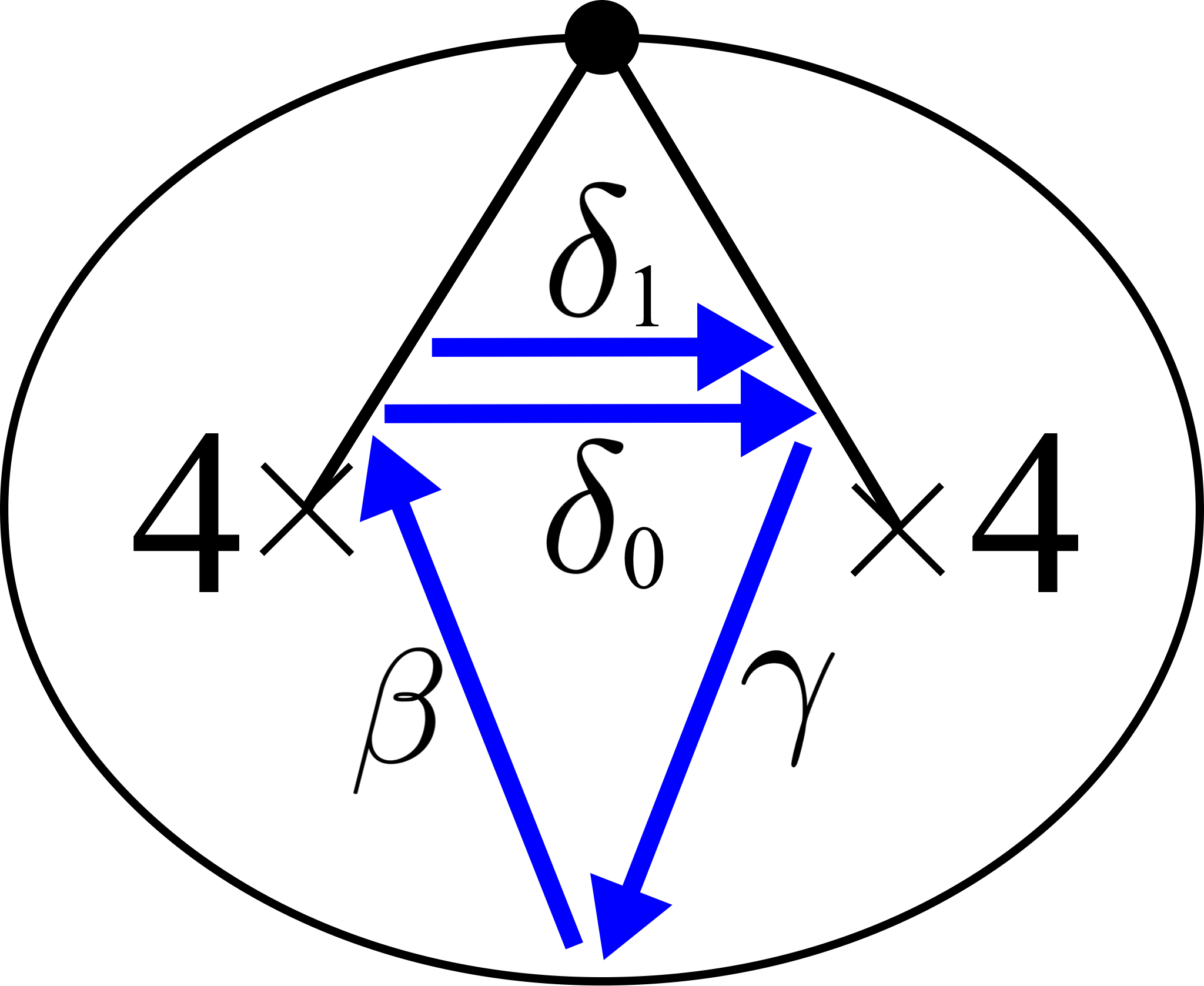}
  &
  &
  $\gamma(\delta_0+\delta_1)$
  &
  $(\delta_0+\delta_1)\beta$
  &
  $
   \begin{array}{l}
    \frac{1}{2}(\beta\gamma +
    \\ [\anshift]
    \rho^{-\ell}(v^{-1})\beta\gamma v)
   \end{array}
  $
  &
  $
   \begin{array}{l}
    \frac{1}{2}(\beta\gamma +
    \\ [\anshift]
    \rho^{-\ell-2}(v^{-1})\beta\gamma v)
   \end{array}
  $
  \\
  \hline
  \end{tabular}
}
\end{center}
\caption{}

\label{Fig:cyclic_derivative}
\end{figure}

For each marked point $m\in\marked$, let $t_m$ be the number of arcs in $\tau$ that are incident to $m$, counted with multiplicity (so that loops based at $m$ contribute twice to $t_m$), and let $t=\max\{t_m\suchthat m\in\marked\}$. Then the previous paragraph and the fact that the boundary of $\Sigma$ is not empty imply that any element of $\RA{A(\tau,\xi)}$ which happens to be a path\footnote{Our notion of path is the one given in \cite[Definition 3,6]{Geuenich-Labardini-1}} of length $t$ belongs to $J(S(\tau,\xi))$. The theorem follows.
%
\end{proof}

\begin{defi} We will say that an ideal~$I$ of the complete path algebra~$\RA{\Atauc}$ is \emph{admissible}, if there exists $t \in \N$ with $\maxid^t \subseteq I \subseteq \maxid^2$, where $\maxid = \maxid(\Atauc)$ is the (closed) ideal generated by $\Atauc$.
\end{defi}

The quotient~$\Lambda = \RA{\Atauc} / I$ by an admissible ideal~$I$ is a finite-dimensional $F$-algebra with Jacobson radical $\mathrm{rad}(\Lambda) = \maxid / I$.
In particular, we can recover the $F$-algebra~$R$ and the $R$-$R$-bimodule~$\Atauc$ from the $F$-algebra~$\Lambda$ as $R \cong \RA{\Atauc} / \maxid \cong \Lambda / \mathrm{rad}(\Lambda)$ and $\Atauc \cong \maxid / \maxid^2 \cong \mathrm{rad}(\Lambda) / \mathrm{rad}^2(\Lambda)$.

Assume $\Lambda_1 = \RA{A(\tau,\xi_1)} / I_1$ and $\Lambda_2 = \RA{A(\tau,\xi_2)} / I_2$ for admissible ideals $I_1$ and $I_2$.
For every $F$-algebra isomorphism $\Psi : \Lambda_1 \to \Lambda_2$,
we denote by~$\Psi^{(0)}$ the $F$-algebra automorphism $R \cong \Lambda_1 / \mathrm{rad}(\Lambda_1) \to \Lambda_2 / \mathrm{rad}(\Lambda_2) \cong R$.
Furthermore, for $R$-$R$-bimodules~$M$ and $F$-algebra automorphisms $\psi : R \to R$, let $\psi_* M$ be the $R$-$R$-bimodule whose underlying $F$-vector space is~$M$ with $R$-$R$-bimodule structure $rms := \psi(r) \cdot_M m \cdot_M \psi(s)$ for $r, s \in R$ and $m \in \psi_* M$, where the symbol $\cdot_M$ is used for the left and right $R$-module action of the $R$-$R$-bimodule~$M$.

\begin{coro}
 \label{coro:jacobian-algebra-yields-species-isomorphism}
 If $\Psi : \Jacalg{(A(\tau,\xi_1),S(\tau,\xi_1))} \to \Jacalg{(A(\tau,\xi_2),S(\tau,\xi_2))}$ is an $F$-algebra isomorphism,
 then the $R$-$R$-bimodules $A(\tau,\xi_1)$ and $\Psi^{(0)}_* A(\tau,\xi_2)$ are isomorphic.
\end{coro}

\begin{proof}
 Let~$\mathfrak{r}_1$ and~$\mathfrak{r}_2$ be the Jacobson radicals of~$\Jacalg{(A(\tau,\xi_1),S(\tau,\xi_1))}$ and $\Jacalg{(A(\tau,\xi_2),S(\tau,\xi_2))}$, respectively.
 Given that~$\Psi$ is an $F$-algebra isomorphism, it induces an isomorphism $\mathfrak{r}_1 / \mathfrak{r}_1^2 \to \Psi^{(0)}_* \big(\mathfrak{r}_2 / \mathfrak{r}_2^2\big)$ of $R$-$R$-bimodules.
 This proves the corollary, since $\mathfrak{r}_1 / \mathfrak{r}_1^2 \cong A(\tau,\xi_1)$ and $\mathfrak{r}_2 / \mathfrak{r}_2^2 \cong A(\tau,\xi_2)$ since $J(S(\tau,\xi_i))$ is an admissible ideal of $\RA{A(\tau,\xi_i)}$ for $i\in\{1,2\}$.
\end{proof}

\begin{lemma}
 \label{lemma:induced-bimodule-commutes-with-idempotents}
 Let~$\psi : R \to R$ be an $F$-algebra automorphism with $\psi(e_i) = e_i$ for all~$i \in \tau$ such that the $R$-$R$-bimodules $A(\tau,\xi_1)$ and $\psi_* A(\tau,\xi_2)$ are isomorphic.
 Then, for all~$i, j \in \tau$, the $F_j$-$F_i$-bimodules $e_j A(\tau,\xi_1) e_i$ and $\psi_*(e_j A(\tau,\xi_2) e_i)$ are isomorphic.
\end{lemma}

\begin{proof}
It is clear that $e_j A(\tau,\xi_1) e_i \cong e_j( \psi_*( A(\tau,\xi_2))) e_i = \psi_*(\psi(e_j) A(\tau,\xi_2) \psi(e_i)) = \psi_*(e_j A(\tau,\xi_2) e_i)$.
\end{proof}

For $F$-algebra automorphisms~$\psi : R \to R$ satisfying~$
\psi(e_i) = e_i$ for all~$i \in \tau$, denote by~$\psi_i$ the automorphism $F_i \cong e_i R e_i \to e_i R e_i \cong F_i$
in $\Gal(F_i / F)$ induced by $\psi$.

\begin{lemma}
 \label{lemma:induced-bimodule-decomposition}
 Let $\tauc$ be a colored triangulation and $\psi : R \to R$ an $F$-algebra automorphism satisfying $\psi(e_i) = e_i$ for all~$i \in \tau$.
 Then, for all~$i, j \in \tau$, the $F_j$-$F_i$-bimodule $\psi_* (e_j \Atauc e_i)$ is isomorphic to $\bigoplus_a F_j^{\rho_a}\otimes_{F_{j,i}} F_i$, where $\rho_a := \psi_j^{-1} g\tauc_a \psi_i$ and the summation variable~$a$ runs through all arrows in $Q_1(\tau,\omega)$ with $h(a) = j$ and $t(a) = i$.
\end{lemma}

\begin{proof}
 By definition $e_j \Atauc e_i = \bigoplus_{a} F_j^{g\tauc_a} \otimes_{F_{j,i}} F_i$.
 This implies the lemma, since each $\psi_* (F_j^{g\tauc_a} \otimes_{F_{j,i}} F_i)$ is a simple $F_j$-$F_i$-bimodule on which $F$ acts centrally and $\rho_a(r) x = xr$ for all $x \in \psi_* (F_j^{g\tauc_a} \otimes_{F_{j,i}} F_i)$ and $r \in F_{j,i}$.
\end{proof}

\begin{coro}
 \label{coro:permutation-of-modulating-function}
 Let~$(\tau,\xi_1)$ and $(\tau,\xi_2)$ be colored triangulations and $\psi : R \to R$ an $F$-algebra automorphism with $\psi(e_i) = e_i$ for all~$i \in \tau$ such that~$A(\tau,\xi_1)$ and $\psi_* A(\tau,\xi_2)$ are isomorphic $R$-$R$-bimodules.
 Then there is a permutation~$\pi$ of $Q_1(\tau,\omega)$ such that $h(\pi(a)) = h(a)$, $t(\pi(a)) = t(a)$, and $g(\tau,\xi_1)_a = \psi_{h(a)}^{-1} g(\tau,\xi_2)_{\pi(a)} \psi_{t(a)}$ for all~$a \in Q_1(\tau,\omega)$.
\end{coro}

\begin{proof}
 For all $i, j \in \tau$ we have $\bigoplus_{j \xleftarrow{a} i} F_j^{g(\tau,\xi_1)_a} \otimes_{F_{j,i}} F_i = e_j A(\tau,\xi_1) e_i \cong \psi_* (e_j A(\tau,\xi_2) e_i) \cong \bigoplus_{j \xleftarrow{a} i} F_j^{\rho_a} \otimes_{F_{j,i}} F_i$ with $\rho_a = \psi_j^{-1} g(\tau,\xi_2)_a \psi_i$ by Lemmas~\ref{lemma:induced-bimodule-commutes-with-idempotents} and~\ref{lemma:induced-bimodule-decomposition}.
 By the Krull-Schmidt theorem there is a permutation~$\pi$ of $Q_1(\tau,\omega)$ such that, for all $a \in Q_1(\tau,\omega)$, it is $h(\pi(a)) = h(a)$, $t(\pi(a)) = t(a)$, and $F_{h(a)}^{g(\tau,\xi_1)_a} \otimes_{F_{h(a),t(a)}} F_{t(a)} \cong F_{h(a)}^{\rho_{\pi(a)}} \otimes_{F_{h(a),t(a)}} F_{t(a)}$ as simple $F_{h(a)}$-$F_{t(a)}$-bimodules on which $F$ acts centrally.
 Hence, $g(\tau,\xi_1)_a = \rho_{\pi(a)} = \psi_{h(a)}^{-1} g(\tau,\xi_2)_{\pi(a)} \psi_{t(a)}$.
\end{proof}

\begin{prop}
 \label{prop:permutation-of-modulating-function-implies-equality-in-cohomology}
 Let~$(\tau,\xi_1)$ and $(\tau,\xi_2)$ be colored triangulations, let $\psi : R \to R$ be an $F$-algebra automorphism satisfying $\psi(e_i) = e_i$ for all~$i \in \tau$,
 and let $\pi$ be a permutation of $Q_1(\tau,\omega)$ such that $h(\pi(a)) = h(a)$, $t(\pi(a)) = t(a)$, and $g(\tau,\xi_1)_a = \psi_{h(a)}^{-1} g(\tau,\xi_2)_{\pi(a)} \psi_{t(a)}$ for all~$a \in Q_1(\tau,\omega)$.
 Then one has $[\xi_1] = [\xi_2]$
 in $\CHonetauwFtwo$.
\end{prop}

\begin{proof}
 Recall from Subsection~\ref{subsec:species-of-a-triangulation} that $\Gal(L/F) = \{ \myid_L, \theta \}$.
 Define a function $\phi : Q'_0(\tauw) \to \F_2$ by
 \[
  \phi(k)
  \::=\:
  \begin{cases}
   0
   &
   \text{if $\psi_k|_L = \myid_L$,}
   \\
   1
   &
   \text{if $\psi_k|_L = \theta$.}
  \end{cases}
 \]

 For every $a \in Q'_1(\tauw)$, we can write $g(\tau,\xi_1)_a|_L = \theta^{\xi_1(a)}$ and $g(\tau,\xi_2)_{\pi(a)}|_L = \theta^{\xi_2(\pi(a))}$ according to Definition~\ref{def:cocycle->modulating-function}.
 Therefore $\theta^{\xi_1(a)} = g(\tau,\xi_1)_a|_L = \psi_{h(a)}^{-1}|_L \cdot g(\tau,\xi_2)_{\pi(a)}|_L \cdot \psi_{t(a)}|_L = \theta^{-\phi(h(a))} \cdot \theta^{\xi_2(\pi(a))} \cdot \theta^{\phi(t(a))}$.
 This can be rewritten as
 \smash{$\theta^{\xi_2(\pi(a)) - \xi_1(a)} = \theta^{\phi(h(a))- \phi(t(a))}$} and is equivalent to
 \[
  \xi_2(\pi(a)) - \xi_1(a)
  \:=\:
  \phi(h(a)) - \phi(t(a))
  \hspace{15pt}
  \text{for every $a \in Q'_1(\tauw)$.}
 \]
 For all $i, j \in \tau$, denote by~$q_{j,i}$ the number of arrows from $i$ to $j$ in $Q'_1(\tauw)$.
 We already may conclude that
 \[
  \xi_2(a) - \xi_1(a)
  \:=\:
  \phi(h(a)) - \phi(t(a))
  \hspace{15pt}
  \text{for every $a \in Q'_1(\tauw)$ with $q_{h(a),t(a)} = 1$.}
 \]
 Observe that, for all~$a \in Q'_1(\tauw)$, it is
 $q_{h(a),t(a)} \leq 2$.
 Moreover, $q_{h(a),t(a)} = 1$, if $d(\tau,\omega)_{h(a)} = 4$ or $d(\tau,\omega)_{t(a)} = 4$.
 For every arrow~$a \in Q'_1(\tauw)$ with
 $q_{h(a),t(a)} \neq 1$, we therefore have $d(\tau,\omega)_{h(a)} = 2 = d(\tau,\omega)_{t(a)}$ and $q_{h(a),t(a)} = 2$.
 From now on, let us assume that $a$ is such an arrow.

 If $a$ is induced by an interior triangle~$\triangle$ of $\tau$, then necessarily~$\triangle \in X_2(\tauw)$.
 In this case, let $b, c \in Q'_1(\tauw)$ be the other two arrows induced by~$\triangle$ with $t(b) = h(c)$, $t(c) = h(a)$, and $t(a) = h(b)$.
 By inspecting the puzzle-piece decomposition of $\tau$, it is not hard to see that $q_{h(b),t(b)} = 1 = q_{h(c),t(c)}$.
 Consequently, $\xi_2(b) - \xi_1(b) = \phi(h(b)) - \phi(t(b))$ and $\xi_2(c) - \xi_1(c) = \phi(h(c)) - \phi(t(c))$.
 In addition, since $\xi_1$ and $\xi_2$ are $1$-cocycles, we also have~$\xi_1(a) + \xi_1(b) + \xi_1(c) = 0$ and $\xi_2(a) + \xi_2(b) + \xi_2(c) = 0$.
 Combining all this yields
 \[
  \arraycolsep 3pt
  \begin{array}{lclcl}
   \xi_2(a) - \xi_1(a)
   &=&
   -(\xi_2(b) + \xi_2(c)) + (\xi_1(b) + \xi_1(c))
   \\
   &=&
   -(\xi_2(b) - \xi_1(b)) - (\xi_2(c) - \xi_1(c))
   \\
   &=&
   -(\phi(h(b)) - \phi(t(b)))
   -(\phi(h(c)) - \phi(t(c)))
   \\
   &=&
   -(\phi(t(a)) - \phi(h(c)))
   -(\phi(h(c)) - \phi(h(a)))
   &=&
   \phi(h(a)) - \phi(t(a))
   \,.
  \end{array}
 \]

 If $a$ is not induced by an interior triangle of~$\tau$, but the parallel arrow~$\pi(a)$ is induced by an interior triangle of~$\tau$, the argument just given (with $a$ replaced by $\pi(a)$) shows that
 $\xi_2(\pi(a)) - \xi_1(\pi(a)) = \phi(h(a)) - \phi(t(a))$.
 On the other hand,
  we already know that $\xi_2(a) - \xi_1(\pi(a)) = \phi(h(a)) - \phi(t(a))$, since $\pi(\pi(a)) = a$.
 Hence, $\xi_2(\pi(a)) = \xi_2(a)$.

 Let us finally consider the case where neither $a$ nor $\pi(a)$ is induced by an interior triangle of~$\tau$.
 Then $a$ and $\pi(a)$ are each induced by a triangle with a boundary segment as one of its sides. Moreover, the two triangles inducing $a$ and $\pi(a)$ share the two sides $h(a)$ and $t(a)$. It is easy to see (from the puzzle-piece decomposition of $\tau$) that $\SSigma$ is a cylinder with two marked points and without orbifold points.
 In particular, the quiver $Q'(\tauw)$ consists just of the two parallel arrows $a$ and $\pi(a)$.
 Define $\phi' : Q'_0(\tauw) \to \F_2$ by $\phi'(h(a)) := \xi_2(a)$ and $\phi'(t(a)) := \xi_1(a)$.
 Then, obviously, $\xi_2(a) - \xi_1(a) = \phi'(h(a)) - \phi'(t(a))$.
 Furthermore, using $\xi_2(\pi(a)) - \xi_1(a) = \phi(h(a)) - \phi(t(a)) = \xi_2(a) - \xi_1(\pi(a))$,
 \[
  \arraycolsep 3pt
  \begin{array}{lclclcl}
   \xi_2(\pi(a)) - \xi_1(\pi(a))
   &=&
   (\xi_2(\pi(a)) - \xi_2(a)) + (\xi_2(a) - \xi_1(\pi(a))
   \\
   &=&
   (\xi_2(\pi(a)) - \xi_2(a)) + (\xi_2(\pi(a)) - \xi_1(a))
   &=&
   \xi_2(a) - \xi_1(a)
   &=&
   \phi'(h(a)) - \phi'(t(a))
   \,.
  \end{array}
 \]
 So, if~$\SSigma$ is a cylinder with two marked points and without orbifold points, we get $[\xi_1] = [\xi_2]$ in cohomology, since the $\F_2$-linear extension~$\widehat{\phi}'$ of $\phi'$ defines a coboundary $\xi_2 - \xi_1 = \widehat{\phi}' \circ \partial_1$.

 For the general situation, but excluding the case in which $\SSigma$ is a cylinder with two marked points and without orbifold points, we have seen before that $\xi_2(a) - \xi_1(a) = \phi(h(a)) - \phi(t(a))$ for all $a \in Q'_1(\tauw)$.
 Again, this readily implies the identity $[\xi_1] = [\xi_2]$ in cohomology.
\end{proof}

To finish the proof of Theorem \ref{thm:comologous<=>isomorphic-Jacobian-algs},
suppose that $\Psi : \Jacalg{(A(\tau,\xi_1),S(\tau,\xi_1))} \to \Jacalg{(A(\tau,\xi_2),S(\tau,\xi_2))}$ is an $F$-linear ring isomorphism satisfying $\Psi(e_k) = e_k$ for all~$k \in \tau$.
 Let us abbreviate~$\Psi^{(0)}$ as $\psi$.
 Corollary~\ref{coro:jacobian-algebra-yields-species-isomorphism} shows that $A(\tau,\xi_1)$ and $\psi_* A(\tau,\xi_2)$ are isomorphic $R$-$R$-bimodules.
 According to Corollary~\ref{coro:permutation-of-modulating-function} there exists a permutation~$\pi$ of $Q_1(\tau,\omega)$ such that $h(\pi(a)) = h(a)$, $t(\pi(a)) = t(a)$, and $g(\tau,\xi_1)_a = \psi_{h(a)}^{-1} g(\tau,\xi_2)_{\pi(a)} \psi_{t(a)}$ for all~$a \in Q_1(\tau,\omega)$.
 Now Proposition~\ref{prop:permutation-of-modulating-function-implies-equality-in-cohomology} yields $[\xi_1] = [\xi_2]$.

Theorem \ref{thm:comologous<=>isomorphic-Jacobian-algs} is proved.
\end{proof}

%% file: 11_uniqueness_of_sps.tex
\section{Classification of non-degenerate SPs}
\label{sec:classification-of-nondeg-SPs}

The goal of this section is to show that for unpunctured surfaces the SPs $(A(\tau,\xi),S(\tau,\xi))$ defined in Section~\ref{sec:sp-of-a-colored-triangulation} are unique in the sense that if the pair $(\tau,\omega)$ and the ground degree-$d$ cyclic Galois extension $E/F$ are fixed, and if we are given a non-degenerate SP $(A,W)$ such that $A$ is a species realization over $E/F$ of the skew-symmetrizable matrix $B(\tau,\omega)$, then  $(A,W)$ is right-equivalent to $(A(\tau,\xi),S(\tau,\xi))$ for some 1-cocycle $\xi\in Z^1(\tau,\omega)\subseteq C^1(\tau,\omega)$ (see Theorem \ref{thm:uniqueness-of-nondegSP-for-unpunctured} below, compare to \cite[Theorems 8.4 and 8.21]{Geiss-Labardini-Schroer}).

The following lemma is the very reason why in this paper we have restricted our attention to the species that arise from 1-cocycles of the cochain complexes $C^\bullet(\tau,\omega)$. Roughly speaking, it says that certain ``cocycle condition'' is necessarily satisfied by some 3-vertex species if they are to admit non-degenerate potentials.

\begin{lemma}\label{lemma:why-cocycles} Let $(Q,\dtuple)$ be any of the three weighted quivers depicted in Figure \ref{Fig:3_critical_quivers}, $d$ the least common multiple of the integers conforming the tuple $\dtuple$, $E/F$ a degree-$d$ cyclic Galois field extension such that $F$ contains a primitive $d^{\operatorname{th}}$ root of unity, $g:Q_1\rightarrow\bigcup_{i,j\in Q_0}\Gal(F_{i,j}/F)$ a modulating function for $(Q,\dtuple)$ over $E/F$, and $A$ the species of the triple $(Q,\dtuple,g)$.
\begin{figure}[!ht]
                \centering
                \includegraphics[scale=.065]{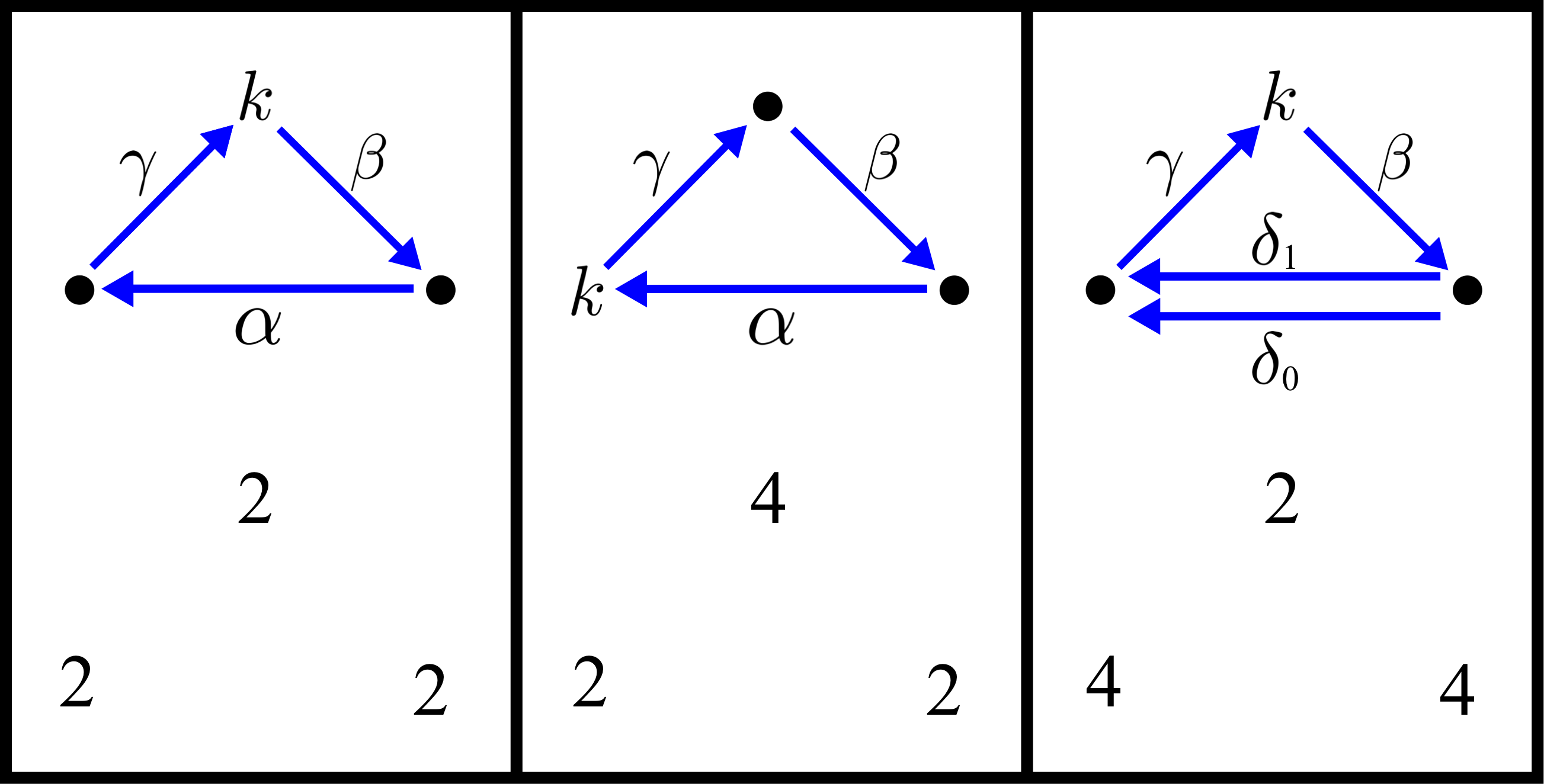}\caption{}
                \label{Fig:3_critical_quivers}
        \end{figure}
If $A$ admits a non-degenerate potential, then the following identities hold in the Galois group $\Gal(L/F)\cong\Z/2\Z$:
\begin{equation}\label{eq:nondegeneracy=>cocycle-condition}
g_\beta g_\gamma = \begin{cases}
g_\alpha^{-1} & \text{if $(Q,\dtuple)$ is the weighted quiver}\\
& \text{on the left or in the middle of Figure \ref{Fig:3_critical_quivers};}\\
g_{\delta_0}^{-1}|_L=g_{\delta_1}^{-1}|_L & \text{if $(Q,\dtuple)$ is the weighted quiver}\\
& \text{on the farthest right of Figure \ref{Fig:3_critical_quivers}}.
\end{cases}
\end{equation}
If $(Q,\dtuple)$ is the weighted quiver on the farthest right of Figure \ref{Fig:3_critical_quivers}, then in the Galois group $\Gal(E/F)\cong\Z/4\Z$ we have $g_{\delta_0}\neq g_{\delta_1}$.
\end{lemma}

\begin{proof} In virtue of \cite[Definitions 3.19 and 3.22]{Geuenich-Labardini-1}, Lemma \ref{lemma:why-cocycles} is a direct consequence of \cite[Example 3.12]{Geuenich-Labardini-1}. We shall elaborate only for the sake of clarity. Assume that $W\in\RA{A}$ is a potential such that $(A,W)$ is a non-degenerate SP.

\setcounter{case}{0}

\begin{case}
Suppose that $(Q,\dtuple)$ is the weighted quiver on the farthest left of Figure \ref{Fig:3_critical_quivers}. Up to cyclical equivalence, we can assume that $W=\sum_{n\geq 1}x_n(\alpha\beta\gamma)^n$ for some elements $x_n\in E$. By \cite[Definition 3.19-(3)]{Geuenich-Labardini-1}, in the complete path algebra of the species $\widetilde{\mu}_k(A)$ we have
$$
\widetilde{\mu}_k(W)=\left(\sum_{n\geq 1}x_n(\alpha[\beta\gamma]_{g_\beta g_\gamma})^n\right)+\gamma^*\beta^*[\beta\gamma]_{g_\beta g_\gamma},
$$
which implies $x_1\alpha[\beta\gamma]_{g_\beta g_\gamma}\not\sim_{\operatorname{cyc}}0$ since $(A,W)$ is non-degenerate. Hence, by \cite[Definition 3.19-(2) and Example 3.12]{Geuenich-Labardini-1}, we have
$$
g_\alpha^{-1}=\widetilde{\mu}_k(g)_\alpha^{-1}=\widetilde{\mu}_k(g)_{[\beta\gamma]_{g_\beta g_\gamma}}=g_\beta g_\gamma.
$$
\end{case}

\begin{case}
Next, suppose that $(Q,\dtuple)$ is the weighted quiver in the middle of Figure \ref{Fig:3_critical_quivers}. Up to cyclical equivalence, we can assume that $W=x\gamma\alpha\beta+W^{(\geq 6)}$ for some elements $x\in E$ and $W^{(\geq 6)}\in \maxid^6$. By \cite[Definition 3.19-(3)]{Geuenich-Labardini-1}, in the complete path algebra of the species $\widetilde{\mu}_k(A)$ we have
$$
\widetilde{\mu}_k(W)=x[\gamma\alpha]_{g_\gamma g_\alpha}\beta+[W^{(\geq 6)}]+[\gamma\alpha]_{g_\gamma g_\alpha}\alpha^*\gamma^*,
$$
which implies $x[\gamma\alpha]_{g_\gamma g_\alpha}\beta\not\sim_{\operatorname{cyc}}0$ since $(A,W)$ is non-degenerate. Hence, by \cite[Definition 3.19-(2) and Example 3.12]{Geuenich-Labardini-1}, we have
$$
g_{\beta}^{-1}=\widetilde{\mu}_k(g)^{-1}_{\beta}=\widetilde{\mu}_k(g)_{[\gamma\alpha]_{g_\gamma g_\alpha}}=g_\gamma g_\alpha.
$$
\end{case}

\begin{case}
Finally, suppose that $(Q,\dtuple)$ is the weighted quiver on the farthest right of Figure \ref{Fig:3_critical_quivers}. Up to cyclical equivalence, we can assume that $W=x\delta_0\beta\gamma+y\delta_1\beta\gamma+W^{(\geq 6)}$ for some elements $x,y\in E$ and $W^{(\geq 6)}\in\maxid^6$. By \cite[Definition 3.19 and Example 3.12]{Geuenich-Labardini-1}, in the complete path algebra of the species $\widetilde{\mu}_k(A)$ we have
$$
\widetilde{\mu}_k(W)\sim_{\operatorname{cyc}}x\delta_0\pi_{g_{\delta_0}^{-1}}\left([\beta\gamma]_{\nu_0}+[\beta\gamma]_{\nu_1}\right)+y\delta_1\pi_{g_{\delta_1}^{-1}}\left([\beta\gamma]_{\nu_0}+[\beta\gamma]_{\nu_1}\right)+[W^{(\geq 6)}],
$$
where $\nu_0$ and $\nu_1$ are the two different field automorphisms of $E$ whose restrictions to $L$ equal $g_\beta g_\gamma\in\Gal(L/F)$, and 
$$
\pi_{g_{\delta_\ell}^{-1}}(x) = \frac{1}{4}\sum_{t=0}^3g_{\delta_\ell}^{-1}(v^{-t})xv^t \ \ \ \ \  \text{for $\ell\in\{0,1\}$}.
$$
Since $(A,W)$ is non-degenerate, this implies the equality of sets $\{g_{\delta_0}^{-1},g_{\delta_1}^{-1}\}=\{\nu_0,\nu_1\}$, which is equivalent to saying that $g_{\delta_0}^{-1}|_L=g_\beta g_\gamma=g_{\delta_1}^{-1}|_L$ and $g_{\delta_0}^{-1}\neq g_{\delta_1}$.\end{case}\end{proof}

Let $\genus(\Sigma)$ be the genus of $\Sigma$. Notice that for $\SSigma=\surf$ unpunctured, the inequality $\genus(\Sigma)+|\marked|+|\orb|\leq 2$ holds if and only if $\SSigma$ is one of the following:
\begin{itemize}
\item An unpunctured monogon without orbifold points;
\item an unpunctured monogon with exactly one orbifold point;
\item an unpunctured digon without orbifold points;
\item an unpunctured annulus with exactly one marked point on each boundary component, and  without orbifold points;
\item an unpunctured torus with exactly one boundary component, exactly one marked point on such component, and without orbifold points.
\end{itemize}
The first three surfaces have been explicitly excluded from the considerations of this paper.

We leave the easy proof of the following lemma in the hands of the reader.

\begin{lemma}\label{lemma:triangulations-without-double-arrows} Suppose that $\SSigma=\surf$ is an unpunctured surface with order-2 orbifold points that satisfies $\genus(\Sigma)+|\marked|+|\orb|>2$. Then there  exists a triangulation $\sigma$ of $\SSigma$ with the property that the quiver $\overline{Q}(\sigma)$ does not have double arrows.
\end{lemma}


\begin{lemma} Let $\SSigma=\surf$ and $\sigma$ be as in Lemma \ref{lemma:triangulations-without-double-arrows}, $\omega:\orb\rightarrow\{1,4\}$ be any function, and $E/F$ any degree-$d$ cyclic Galois field extension such that $F$ contains a primitive $d^{\operatorname{th}}$ root of unity, where $d=\operatorname{lcm}(\dtuple(\sigma,\omega))$. If $g:Q_1(\sigma,\omega)\rightarrow\bigcup_{i,j\in\sigma}\Gal(F_{i,j}/F)$ is a modulating function for $(Q(\sigma,\omega),\dtuple(\sigma,\omega))$ over $E/F$ whose associated species $A$ admits a non-degenerate potential, then $g=g(\sigma,\zeta)$ and $A=A(\sigma,\zeta)$ for some 1-cocycle $\zeta\in Z^1(\sigma,\omega)\subseteq C^1(\sigma,\omega)$.
\end{lemma}

\begin{proof} Let $\triangle$ be a triangle of $\sigma$ which is interior and does not contain any $q\in\orb$ such that $\omega(q)=1$. Let $i$, $j$ and $k$ be the three arcs of $\sigma$ that are contained in $\triangle$, $Q(\sigma,\omega)|_{\triangle}$ denote the full subquiver of $Q(\sigma,\omega)$ determined by $\{i,j,k\}$, and $\dtuple(\sigma,\omega)|_{\triangle}$ denote the triple $(d(\sigma,\omega)_i,d(\sigma,\omega)_j),d(\sigma,\omega)_k)$. The fact that $\sigma$ satisfies the conclusion of Lemma \ref{lemma:triangulations-without-double-arrows} implies that the weighted quiver $(Q(\sigma,\omega)|_{\triangle},\dtuple(\sigma,\omega)|_{\triangle})$ is one one of the three weighted quivers of Figure \ref{Fig:3_critical_quivers}. With the notation of such figure, if $(Q(\sigma,\omega)|_{\triangle},\dtuple(\sigma,\omega)|_{\triangle})$ happens to be the weighted quiver on the farthest right of the figure, write $\alpha:=\delta_0$ and assume that $\alpha$ belongs to the arrow set of $\overline{Q}(\sigma)$.

With the notation just established, and regardless of which of the three weighted quivers in Figure \ref{Fig:3_critical_quivers} $(Q(\sigma,\omega)|_{\triangle},\dtuple(\sigma,\omega)|_{\triangle})$ is, define $\zeta(\alpha)$, $\zeta(\beta)$ and $\zeta(\gamma)$ to be the elements of $\Z/2\Z$ that respectively correspond to $g_\alpha|_L$, $g_\beta$ and $g_\gamma$ under the isomorphism $\Z/2\Z\cong\Gal(L/F)$.

Now, suppose that $a$ is an arrow of $Q'(\sigma,\omega)$ contained in a triangle of $\sigma$ which either is non-interior or contains some $q\in\orb$ with $\omega(q)=1$. Set $\zeta(a)$ to be the element of $\Z/2\Z$ corresponding to $g_a|_L$ under the isomorphism $\Z/2\Z\cong\Gal(L/F)$.

Notice that at this point we have already defined a function $\zeta:Q'_1(\sigma,\omega)\rightarrow\Z/2\Z=\F_2$; abusing notation, we write $\zeta$ as well for its unique $\F_2$-linear extension $C_1(\sigma,\omega)\rightarrow\F_2$. Then $\zeta$ is a 1-cocycle of $C^\bullet(\sigma,\omega)$ by Lemma \ref{lemma:why-cocycles}, and we clearly have $g=g(\sigma,\zeta)$.
\end{proof}

\begin{coro}\label{coro:only-cocycles-produce-good-species} Let $\SSigma=\surf$ be as in Lemma \ref{lemma:triangulations-without-double-arrows}, $\omega:\orb\rightarrow\{1,4\}$ any function, $\tau$ a triangulation of $\SSigma$,  and $E/F$ any degree-$d$ cyclic Galois field extension such that $F$ contains a primitive $d^{\operatorname{th}}$ root of unity, where $d=\operatorname{lcm}(\dtuple(\tau,\omega))$. If $g:Q_1(\tau,\omega)\rightarrow\bigcup_{i,j\in\sigma}\Gal(F_{i,j}/F)$ is a modulating function for $(Q(\tau,\omega),\dtuple(\tau,\omega))$ over $E/F$ whose associated species $A$ admits a non-degenerate potential, then there is an $R$-$R$-bimodule isomorphism $A\cong A(\tau,\xi)$ for some 1-cocycle $\xi\in Z^1(\tau,\omega)\subseteq C^1(\tau,\omega)$.
\end{coro}

\begin{proof} Let $W\in\RA{A}$ be a potential such that $(A,W)$ is a non-degenerate SP, and let $\sigma$ be as in Lemma \ref{lemma:triangulations-without-double-arrows}. By~\cite[Theorem~4.2]{FeShTu-orbifolds}, it is possible to obtain $\sigma$ from $\tau$ by applying a sequence of flips to the latter. Since $(A,W)$ is non-degenerate, we can apply the corresponding sequence of SP-mutations to $(A,W)$, the result being an SP $(A',W')$ with the property that $A'$ is the species associated to the triple $(Q(\sigma,\omega),\dtuple(\sigma,\omega),g')$ for some modulating function $g':Q_1(\sigma)\rightarrow\bigcup_{i,j\in\sigma}\Gal(F_{i,j}/F)$ over $E/F$. By Lemma \ref{lemma:triangulations-without-double-arrows}, there exists a 1-cocycle $\zeta\in Z^1(\sigma,\omega)\subseteq C^1(\sigma,\omega)$ such that $A'=A(\sigma,\zeta)$. If we apply to $(\sigma,\zeta)$ the sequence of colored flips that goes in the opposite direction of the one we took above, we obtain a colored triangulation $(\tau,\xi)$ such that, by Proposition \ref{prop:colored-flip-vs-species-mutation}, $A\cong A(\tau,\xi)$ as $R$-$R$-bimodules.
\end{proof}

We have thus established the fact that, amongst the species realizations of the skew-symmetrizable matrices $B(\tau,\omega)$ over degree-$d$ cyclic Galois extensions, only those arising from 1-cocycles of the complexes $C^\bullet(\tau,\omega)$ admit non-degenerate potentials. We now move on to analyzing the non-degenerate potentials.

\begin{lemma}\label{lemma:degree-3-part-forced} Let $(Q,\dtuple)$ be any of the seven weighted quivers depicted in Figure \ref{Fig:3_critical_quivers_2}, $d$ the least common multiple of the integers conforming the tuple $\dtuple$, $E/F$ a degree-$d$ cyclic Galois field extension such that $F$ contains a primitive $d^{\operatorname{th}}$ root of unity, $g:Q_1\rightarrow\bigcup_{i,j\in Q_0}\Gal(F_{i,j}/F)$ a modulating function for $(Q,\dtuple)$ over $E/F$, $A$ the species of the triple $(Q,\dtuple,g)$, and $W\in\RA{A}$ a potential on $A$.
\begin{figure}[!ht]
                \centering
                \includegraphics[scale=.15]{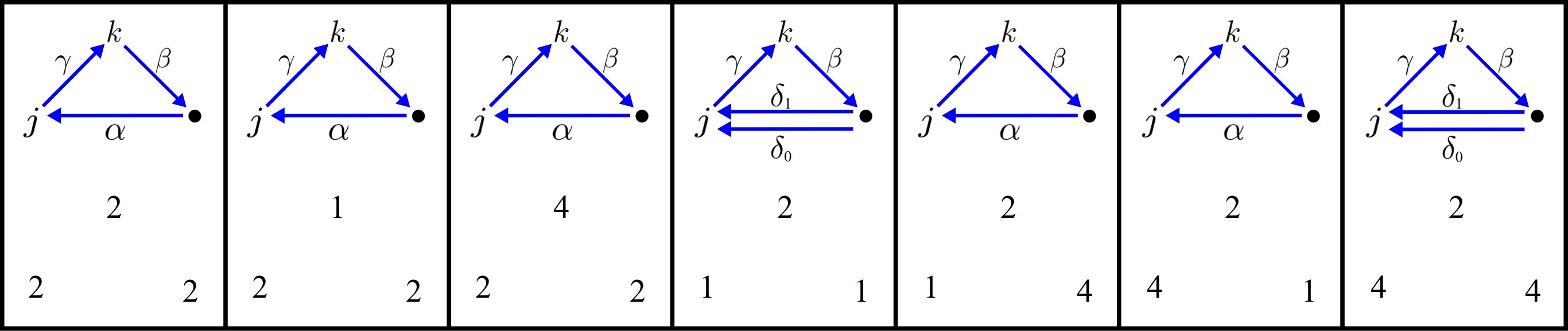}\caption{}
                \label{Fig:3_critical_quivers_2}
        \end{figure}
If $(A,W)$ is a non-degenerate SP, then there exists an $R$-algebra automorphism $\varphi$ of $\RA{A}$ such that:
\begin{equation}\label{eq:nondegeneracy=>degree-3-part-forced}
\varphi(W)^{(3)} \sim_{\operatorname{cyc}} \begin{cases}
 \alpha\beta\gamma & \text{if $(Q,\dtuple)$ is the weighted quiver appearing}\\
 & \text{in the $1^{\operatorname{st}}$, $2^{\operatorname{nd}}$, $3^{\operatorname{rd}}$, $5^{\operatorname{th}}$ or $6^{\operatorname{th}}$ column of Figure \ref{Fig:3_critical_quivers_2};}\\
 \delta_0\beta\gamma + \delta_1\beta u\gamma & \text{if $(Q,\dtuple)$ is the weighted quiver appearing}\\
  & \text{in the $4^{\operatorname{th}}$ column of Figure \ref{Fig:3_critical_quivers_2};}\\
 (\delta_0+\delta_1)\beta\gamma & \text{if $(Q,\dtuple)$ is the weighted quiver appearing}\\
 & \text{in the $7^{\operatorname{th}}$ column of Figure \ref{Fig:3_critical_quivers_2}.}
\end{cases}
\end{equation}
\end{lemma}

\begin{proof} The lemma follows by verifying that its statement is true for each of the seven weighted quivers in Figure \ref{Fig:3_critical_quivers_2}. The seven verifications are all very similar, so we include only three of them and leave the remaining four in the hands of the reader.

\setcounter{case}{0}\begin{case}\label{case:uniqueness-1} Suppose that $(Q,\dtuple)$ weighted quiver in the $1^{\operatorname{st}}$ column of Figure \ref{Fig:3_critical_quivers_2}. We claim that $W^{(3)}\not\sim_{\operatorname{cyc}} 0$. Indeed, if $W^{(3)}\sim_{\operatorname{cyc}}0$, then the SP $\widetilde{\mu}_k(A,W)$ is reduced and hence $\mu_k(A,W)=\widetilde{\mu}_k(A,W)$; since the underlying species of $\widetilde{\mu}_k(A,W)$ is not 2-acyclic, this contradicts the non-degeneracy of $(A,W)$. Therefore, $W^{(3)}\not\sim_{\operatorname{cyc}} 0$.

Up to cyclical equivalence, we can write $W^{(3)}=x\alpha\beta\gamma$ for some $x\in L$. The element $x$ is not zero since $W^{(3)}\not\sim_{\operatorname{cyc}} 0$. The $R$-algebra automorphism $\varphi$ of $\RA{A}$ that sends $\alpha$ to $x^{-1}\alpha$ clearly satisfies $\varphi(W)^{(3)}\sim_{\operatorname{cyc}} \alpha\beta\gamma$.
\end{case}




\begin{case}
Suppose that $(Q,\dtuple)$ weighted quiver in the $4^{\operatorname{th}}$ column of Figure \ref{Fig:3_critical_quivers_2}. The weighted quiver $\mu_k(Q,\dtuple)$  is obviously acyclic. Hence, by \cite[Theorem 3.24]{Geuenich-Labardini-1}, $(A,W)$ is right-equivalent to $(A,W')$, where
$$
W'=\frac{1}{2}\left(\delta_1\beta\gamma+\delta_0\beta u^{-1}\gamma\right).
$$
The SP $(A,W')$ is right-equivalent to $(A,\delta_1\beta u\gamma+\delta_0\beta \gamma)$ by means of the $R$-algebra automorphism of $\RA{A}$ that sends $\beta$ to $2\beta u$.
\end{case}




\begin{case}
Finally, suppose that $(Q,\dtuple)$ weighted quiver in the $7^{\operatorname{th}}$ column of Figure \ref{Fig:3_critical_quivers_2}. The weighted quiver $\mu_k(Q,\dtuple)$  is obviously acyclic. Hence, by \cite[Theorem 3.24]{Geuenich-Labardini-1}, $(A,W)$ is right-equivalent to $(A,W')$, where
$$
W'=\delta_0\beta\gamma+\delta_1\beta\gamma.
$$
\end{case}

Lemma \ref{lemma:degree-3-part-forced} is proved.
\end{proof}

\begin{lemma}\label{lemma:W-rightequiv-to-S(sigma,zeta)} Let $\SSigma=\surf$ and $\sigma$ be as in Lemma \ref{lemma:triangulations-without-double-arrows}, $\omega:\orb\rightarrow\{1,4\}$ be any function, and $\zeta\in Z^1(\sigma,\omega)$ be any 1-cocycle of the cochain complex $C^\bullet(\sigma,\omega)$. If $W\in\RA{A(\sigma,\zeta)}$ is any non-degenerate potential for $A(\sigma,\zeta)$, then $(A(\sigma,\zeta),W)$ is right-equivalent to $(A(\sigma,\zeta),S(\sigma,\zeta))$.
\end{lemma}

\begin{proof} Suppose that $W\in\RA{A(\sigma,\zeta)}$ is a non-degenerate potential for $A(\sigma,\zeta)$. Let $\triangle$ be an interior triangle of $\sigma$, $i$, $j$ and $k$ be the three arcs of $\sigma$ that are contained in $\triangle$, and $(A(\sigma,\zeta)|_{\triangle},W|_{\triangle})$ be the restriction of $(A(\sigma,\zeta),W)$ to $\{i,j,k\}$ (see \cite[Definition 8.1]{Geuenich-Labardini-1}). The fact that $\sigma$ satisfies the conclusion of Lemma \ref{lemma:triangulations-without-double-arrows} implies that the weighted quiver which underlies the species $A(\sigma,\zeta)|_{\triangle}$ is one one of the seven weighted quivers of Figure \ref{Fig:3_critical_quivers_2}. Since $(A(\sigma,\zeta),W)$ is non-degenerate, $(A(\sigma,\zeta)|_\triangle,W|_{\triangle})$ is non-degenerate (this is an easy consequence of \cite[Lemma 8.2]{Geuenich-Labardini-1}). Hence, by Lemma \ref{lemma:degree-3-part-forced}, there exists an $R$-algebra automorphism $\varphi_\triangle$ of $\RA{A(\sigma,\zeta)}$ such that $\varphi_{\triangle}(W|_{\triangle})^{(3)}$ is cyclically equivalent to the right hand side of \eqref{eq:nondegeneracy=>degree-3-part-forced}.

Assembling all the automorphisms $\varphi_\triangle$, with $\triangle$ running in the set of interior triangles of $\sigma$, we obtain an $R$-algebra automorphism $\varphi$ of $\RA{A(\sigma,\zeta)}$ such that
$\varphi(W)^{(3)}\sim_{\operatorname{cyc}}S(\sigma,\zeta)+S'$ for some $S'\in\maxid^{4}$, where $\maxid$ is the two-sided ideal of $\RA{A(\sigma,\zeta)}$ generated by $A(\sigma,\zeta)$.

Since $\SSigma$ is unpunctured, every cyclic path\footnote{Our notion of path is the one given in \cite[Definition 3,6]{Geuenich-Labardini-1}} on $A(\sigma,\zeta)$ is cyclically equivalent to a cyclic path that has a factor of the form $bzc$ for some arrows $b$ and $c$ of $Q(\sigma,\omega)$ that are induced by the same triangle and satisfy $t(b)=h(c)$ and $d(\sigma,\omega)_{t(b)}=2$, and some  element $z\in \{1,u\}e_{t(b)}\subseteq F_{t(b)}e_b$. Using the table depicted in Figure \ref{Fig:cyclic_derivative}, it is not hard to see that $bzc=\sum_{a\in I}\partial_a(S(\sigma,\zeta))$ for some non-empty set $I$ consisting of at most two arrows of $Q(\sigma,\zeta)$. Consequently, every cyclic path on $A(\sigma,\zeta)$ is cyclically equivalent to a cyclic path of the form $\sum_{a\in I}\partial_a(S(\sigma,\zeta))x_a$ for some $x_a\in\maxid$. Therefore,
$$
\varphi(W)^{(3)}\sim_{\operatorname{cyc}}S(\sigma,\zeta)+\sum_{a\in Q_1(\sigma,\zeta)}\partial_a(S(\sigma,\zeta))x_a.
$$

A straightforward adaptation of \cite[Lemma 8.19 and its proof]{Geiss-Labardini-Schroer} shows that there exists an $R$-algebra automorphism $\psi$ of $\RA{A(\sigma,\zeta)}$ such that $\psi(S(\sigma,\zeta)+\sum_{a\in Q_1(\sigma,\zeta)}\partial_a(S(\sigma,\zeta))x_a)\sim_{\operatorname{cyc}}S(\sigma,\zeta)$. Lemma \ref{lemma:W-rightequiv-to-S(sigma,zeta)} is proved. 
\end{proof}

\begin{coro}\label{coro:W-rightequiv-to-S(tau,xi)-for-tau-arbitrary} Let $\SSigma=\surf$ be as in Lemma \ref{lemma:triangulations-without-double-arrows}, $\omega:\orb\rightarrow\{1,4\}$ be any function, and $(\tau,\xi)$ be any colored triangulation of $\SSigmaw$. If $W\in\RA{A(\tau,\xi)}$ is any non-degenerate potential for $A(\tau,\xi)$, then $(A(\tau,\xi),W)$ is right-equivalent to $(A(\tau,\xi),S(\tau,\xi))$.
\end{coro}

\begin{proof} This follows from a straightforward combination of Proposition \ref{prop:colored-flip-vs-species-mutation}, Lemmas \ref{lemma:triangulations-without-double-arrows} and \ref{lemma:W-rightequiv-to-S(sigma,zeta)},~\cite[Theorem~4.2]{FeShTu-orbifolds}, and \cite[Theorems 3.21 and 3.24]{Geuenich-Labardini-1}.  For the sake of clarity, we give a detailed argument.

Let $\sigma$ be a triangulation of $\SSigma$ as in Lemma \ref{lemma:triangulations-without-double-arrows}. By~\cite[Theorem~4.2]{FeShTu-orbifolds}, $\sigma$ is related to $\tau$ by a finite sequence of flips. Consequently, there exists a 1-cocycle $\zeta\in Z^1(\sigma,\omega)\subseteq C^1(\sigma,\omega)$ such that the colored triangulations $(\sigma,\zeta)$ and $(\tau,\xi)$ are related by a finite sequence of colored flips, say $(\sigma,\zeta)=\flip_{k_n}\flip_{k_{n-1}}\ldots\flip_{k_1}(\tau,\xi)$.

The underlying species of the SP $\mu_{k_n}\ldots\mu_{k_1}(A(\tau,\xi),W)$ is $A(\sigma,\zeta)$ by Proposition \ref{prop:colored-flip-vs-species-mutation}, so we can write $(A(\sigma,\zeta),W')=\mu_{k_n}\ldots\mu_{k_1}(A(\tau,\xi),W)$. Since $(A(\tau,\xi),W)$ is non-degenerate, $(A(\sigma,\zeta),W')$ is non-degenerate too, and therefore, $(A(\sigma,\zeta),W')$ is right-equivalent to $(A(\sigma,\zeta),S(\sigma,\zeta))$ by Lemma \ref{lemma:W-rightequiv-to-S(sigma,zeta)}.

On the other hand, the SPs $\mu_{k_n}\ldots\mu_{k_1}(A(\tau,\xi),S(\tau,\xi))$ and $(A(\sigma,\zeta),S(\sigma,\zeta))$ are right-equivalent by Theorem \ref{thm:flip<->SP-mutation}. Thus, $\mu_{k_n}\ldots\mu_{k_1}(A(\tau,\xi),S(\tau,\xi))$ is right-equivalent to $(A(\sigma,\zeta),W')$. The involutivity of SP-mutations up to right-equivalence, cf. \cite[Theorem 3.24]{Geuenich-Labardini-1}, and the fact that SP-mutations send right-equivalent SPs to right-equivalent SPs, cf. \cite[Theorem 3.21]{Geuenich-Labardini-1}, imply that $(A(\tau,\xi),S(\tau,\xi))$ is right-equivalent to $\mu_{k_1}\ldots\mu_{k_n}(A(\sigma,\zeta),W')$, SP which is in turn right-equivalent to $(A(\tau,\xi),W)$ by the involutivity of SP-mutations up to right-equivalence. 
\end{proof}

Putting Corollaries \ref{coro:only-cocycles-produce-good-species} and \ref{coro:W-rightequiv-to-S(tau,xi)-for-tau-arbitrary} together, we obtain the main result of this section, namely:

\begin{thm}\label{thm:uniqueness-of-nondegSP-for-unpunctured} Let $\SSigma=\surf$ be a surface with marked points and order-2 orbifold points which is unpunctured and different from a torus with exactly one marked point and without orbifold points, $\omega:\orb\rightarrow\{1,4\}$ any function, $\tau$ any triangulation of $\SSigma$, $F$ any field containing a primitive $d^{\operatorname{th}}$ root of unity, where $d=\operatorname{lcm}(\dtuple(\tau,\omega))$, and $E/F$ any degree-$d$ cyclic Galois extension. Any realization of the skew-symmetrizable matrix $B(\tau,\omega)$ via a non-degenerate SP over $E/F$ is right-equivalent to $(A(\tau,\xi),S(\tau,\xi))$ for some 1-cocycle $\xi\in Z^1(\tau,\omega)\subseteq C^1(\tau,\omega)$.
\end{thm}

%% file: 14_problems.tex
\section{Some problems}
\label{sec:problems}

\begin{conjecture}\label{conj:cluster-tilted-algebras} If $\SSigma$ admits a triangulation $\sigma$ such that the quiver $\overline{Q}(\sigma)$ (equivalently, any of the quivers $Q(\sigma,\omega)$ for $\omega:\orb\rightarrow\{1,4\}$) is acyclic\footnote{This happens, for instance, when $\SSigma$ is an unpunctured polygon with at most 2 orbifold points, or an unpunctured annulus without orbifold points.}, then for any function $\omega:\orb\rightarrow\{1,4\}$, and any colored triangulation $(\rho,\xi)$ of $\SSigmaw$, the Jacobian algebra $\Jacalg{(A(\rho,\xi),S(\rho,\xi))}$ is cluster-tilted in the sense \cite{BMRRT,BMR} of Buan-Marsh-Reineke-Reiten-Todorov and Buan-Marsh-Reiten. More precisely, we conjecture that for $(\rho,\xi)$ there is a 1-cocycle $\zeta\in Z(\sigma,\omega)$ such that $\Jacalg{(A(\rho,\xi),S(\rho,\xi))}$ is isomorphic as an $F$-algebra to the endomorphism algebra of a cluster-tilting object of the orbit category $\mathcal{D}/(\tau^{-1}\circ[1])$, where $\mathcal{D}$ is the bounded derived category of the module category of the hereditary algebra $\Jacalg{(A(\sigma,\zeta),S(\sigma,\zeta))}$ (note that $S(\sigma,\zeta)=0$ by the acyclicity of $\overline{Q}(\sigma)$), and $\tau$ and $[1]$ are the Auslander-Reiten translation and the shift functor of $\mathcal{D}$, respectively.
\end{conjecture}

\begin{question} Can the results of Amiot \cite{Amiot-gldim2,Amiot-survey}, Keller-Yang \cite{Keller-Yang} and Plamondon \cite{Plamondon-characters,Plamondon} be applied or adapted in order to be able to associate 2- and 3-Calabi-Yau triangulated categories to the surfaces $\SSigmaw=(\Sigma,\marked,\orb,\omega)$ treated in \cite{Geuenich-Labardini-1} and in this paper?
\end{question}

\begin{problem} Given an arbitrarily punctured surface $\SSigma=(\Sigma,\marked,\orb)$ with order-2 orbifold points, and an arbitrary function $\omega:\orb\rightarrow\{1,4\}$, determine whether it is possible to associate species with potential to the (colored) triangulations of $\SSigmaw$, in such a way that the SPs associated to (colored) triangulations related by a (colored) flip are correspondingly related by the SP-mutation defined in \cite[Definitions 3.19 and 3.22]{Geuenich-Labardini-1}.
\end{problem}

\begin{question} Let $\SSigma=\surf$ be either unpunctured or once-punctured closed, and let $\omega:\orb\rightarrow\{1,4\}$ be any function. Can the cluster algebra associated to $\SSigmaw$ by Felikson-Shapiro-Tumarkin be realized as a suitable Caldero-Chapoton algebra of the Jacobian algebras of the colored triangulations of $\SSigmaw$ (see \cite{CLFS})? More precisely, is it possible to define a Caldero-Chapoton algebra for the Jacobian algebra of each colored triangulation of $\SSigmaw$, in such a way that it turns out to be independent of the colored triangulation chosen and sits between the upper cluster algebra and the cluster algebra associated to $\SSigmaw$ with respect to some choice of coefficient system?
\end{question}

In \cite{Chekhov-Shapiro}, Chekhov-Shapiro have associated a \emph{generalized cluster algebra} to each surface with marked points and orbifold points of arbitrary order. The \emph{cluster mutation rule} inside a generalized cluster algebra differs from that of a cluster algebra in that it allows its \emph{exchange polynomials} to not be binomials, while in a cluster algebra all exchange polynomials are binomials. The (algebraic) combinatorics of generalized cluster algebras is very similar to the combinatorics of cluster algebras, though. For instance, they possess the remarkable \emph{Laurent phenomenon}. 

\begin{question} Can the generalized cluster algebras of Chekhov-Shapiro be realized as Caldero-Chapoton algebras of (suitably constructed) associative algebras?
\end{question}


%% file: 12_examples.tex
\section{Examples}
\label{sec:examples}

\subsection{Types $\AAA_n$, $\CC_n$, $\tildeC_n$, $\BB_n$, $\tildeB_n$ and $\tildeBC_n$ via polygons with at most two orbifold points}

Let us say that the matrices
{\tiny\begin{center}\begin{tabular}{ll}
$
B=\left[\begin{array}{rrrrrr}
0 & -2 & 0 &        & 0 & 0\\
1 & 0 & -1 & \ldots & 0 & 0\\
0 & 1 & 0 &         &0 & 0\\
 & \vdots &  & \ddots & & \\
0 & 0 & 0 &  & 0 & -1\\
0  & 0 & 0 &  & 1 & 0
\end{array}\right]\in\mathbb{Z}^{(n-1)\times (n-1)}$, &
$\widetilde{B}=\left[\begin{array}{rrrrrrr}
0 & -2 & 0 &        & 0 & 0 & 0\\
1 & 0 & -1 & \ldots & 0 & 0 &0\\
0 & 1 & 0 &         &0 & 0&0\\
 & \vdots &  & \ddots & & &\\
0 & 0 & 0 &  & 0 & -1 & 0 \\
0  & 0 & 0 &  & 1 & 0 & -1\\
0  & 0 & 0 &  & 0 & 2 & 0
\end{array}\right]\in\mathbb{Z}^{(n+1)\times (n+1)}$,\\
$
C=\left[\begin{array}{rrrrrr}
0 & -1 & 0 &        & 0 & 0\\
2 & 0 & -1 & \ldots & 0 & 0\\
0 & 1 & 0 &         &0 & 0\\
 & \vdots &  & \ddots & & \\
0 & 0 & 0 &  & 0 & -1\\
0  & 0 & 0 &  & 1 & 0
\end{array}\right]\in\mathbb{Z}^{(n-1)\times (n-1)}$,&
$\widetilde{C}=\left[\begin{array}{rrrrrrr}
0 & -1 & 0 &        & 0 & 0 & 0\\
2 & 0 & -1 & \ldots & 0 & 0 &0\\
0 & 1 & 0 &         &0 & 0&0\\
 & \vdots &  & \ddots & & &\\
0 & 0 & 0 &  & 0 & -1 & 0 \\
0  & 0 & 0 &  & 1 & 0 & -2\\
0  & 0 & 0 &  & 0 & 1 & 0
\end{array}\right]\in\mathbb{Z}^{(n+1)\times (n+1)}$,\\
$
X=\left[\begin{array}{rrrrrrr}
0 & -1 & 0 &        & 0 & 0\\
1 & 0 & -1 & \ldots & 0 & 0\\
0 & 1 & 0 &         &0 & 0\\
 & \vdots &  & \ddots & & \\
0 & 0 & 0 &  & 0 & -1\\
0  & 0 & 0 &  & 1 & 0
\end{array}\right]\in\mathbb{Z}^{(n-3)\times (n-3)} 
$
 \ \ \ and &
$
\widetilde{Y}=\left[\begin{array}{rrrrrrr}
0 & -2 & 0 &        & 0 & 0 & 0\\
1 & 0 & -1 & \ldots & 0 & 0 &0\\
0 & 1 & 0 &         &0 & 0&0\\
 & \vdots &  & \ddots & & &\\
0 & 0 & 0 &  & 0 & -1 & 0 \\
0  & 0 & 0 &  & 1 & 0 & -2\\
0  & 0 & 0 &  & 0 & 1 & 0
\end{array}\right]\in\mathbb{Z}^{(n+1)\times (n+1)}
$
\end{tabular}
\end{center}}
\noindent are of type $\BB_{n-1}$ $(n\geq 3)$, $\tildeB_n$ ($n\geq 2$), $\CC_{n-1}$ ($n\geq 3$), $\tildeC_n$ ($n\geq 2$), $\AAA_{n-3}$ ($n\geq 4$) and $\tildeBC_n$ ($n\geq 2$), respectively\footnote{This nomenclature is in sync with \cite[Section 6]{DR} and \cite[Section 3]{Musiker-classical} (see also \cite[Remark 6]{Musiker-classical})}. These matrices can respectively be skew-symmetrized by
\begin{center}
\begin{tabular}{ll}$D_{\BB}=\diag(1,2,2\ldots,2,2)$, & $D_{\tildeB}=\diag(1,2,2\ldots,2,2,1)$,\\
$D_{\CC}=\diag(4,2,2,\ldots,2,2)$, & $D_{\tildeC}=\diag(4,2,2,\ldots,2,2,4)$,\\
$D_{\AAA}=\diag(2,2,2,\ldots,2,2)$ \ \ \ \text{and} & $D_{\tildeBC}=\diag(1,2,2,\ldots,2,2,4)$,
\end{tabular}
\end{center}
\noindent
The weighted quivers corresponding under \cite[Lemma 2.3]{LZ} to the pairs $(B,D_{\BB})$, $(\widetilde{B},D_{\tildeB})$, $(C,D_{\CC})$, $(\widetilde{C},D_{\tildeC})$, $(X,D_{\AAA})$ and $(Y,D_{\tildeBC})$ are
\begin{center}
\begin{tabular}{rcll}
$(Q_B,\dtuple_{\BB})$ & $=$ & $(1 \to 2 \to 3 \ldots \to n-2 \to n-1,$ & $1 \ \ 2 \ \ 2 \ \ \ldots  \ \ 2 \ \ 2)$,\\ 
$(Q_{\widetilde{B}},\dtuple_{\tildeB})$ & $=$ & $(1 \to 2 \to 3 \ldots \to n-1 \to n \to n+1,$ & $1 \ \ 2 \ \ 2 \ \ \ldots  \ \ 2 \ \ 2 \ \ 1)$,\\
$(Q_C,\dtuple_{\CC})$ & $=$ &
$(1 \to 2 \to 3 \ldots \to n-2 \to n-1,$ & $4 \ \ 2 \ \ 2 \ \ \ldots  \ \ 2 \ \ 2)$,\\
$(Q_{\widetilde{C}},\dtuple_{\tildeC})$ & $=$ & $(1 \to 2 \to 3 \ldots \to n-1 \to n \to n+1,$ &$4 \ \ 2 \ \ 2 \ \ \ldots  \ \ 2 \ \ 2 \ \ 4)$,\\
$(Q_X,\dtuple_{\AAA})$ & $=$ &
$(1 \to 2 \to 3 \ldots \to n-4 \to n-3,$ & $2 \ \ 2 \ \ 2 \ \ \ldots  \ \ 2 \ \ 2)$,\\
and \ \ \ 
$(Q_{\widetilde{Y}},\dtuple_{\tildeBC})$ & $=$ & $(1 \to 2 \to 3 \ldots \to n-1 \to n \to n+1,$ &$1 \ \ 2 \ \ 2 \ \ \ldots  \ \ 2 \ \ 2 \ \ 4)$.
\end{tabular}
\end{center}

\begin{ex}\label{ex:finite-types-B-C} Let $\SSigma=\surf$ be an $n$-gon with exactly one orbifold point $q$ (that is, $\orb=\{q\}$), and let $\omega_{\BB},\omega_{\CC}:\orb\rightarrow\{1,4\}$ be the functions given by $\omega_{\BB}(q)=1$ and $\omega_{\CC}(q)=4$. Then the weighted quivers $(Q_B,\dtuple_{\BB})$ and $(Q_C,\dtuple_{\CC})$ respectively coincide with $(Q(\tau,\omega_{\BB}),\dtuple(\tau,\omega_{\BB}))$ and $(Q(\tau,\omega_{\CC}),\dtuple(\tau,\omega_{\CC}))$, where $\tau$ is the following triangulation of $\SSigma$:
\begin{center}
  \begin{tikzpicture}[scale=0.4]
    \def \n {11}
    \def \m {9} 
    \def \k {10} 
    \def \s {360/\n}
    \def \r {5cm}

    \foreach \i in {2,...,\n} {
      \coordinate (\i) at (\i*360/\n+90:\r);
    }
    \coordinate (O) at ($(2)!0.3!(\n)$);
    \node[rotate=45] at (O) {$\times$};
    \coordinate (Oq) at ($(2)!0.2!(\n)$);
    \node at (Oq) {$q$};
    \coordinate (middle) at (270:\r/2);
    \node at (middle) {$\ldots$};

    \foreach \i in {3,...,\m}
      \node at (\i) {$\bullet$};
    \node at (\n) {$\bullet$};

    \draw[dotted] (0,0) circle (\r);

    \draw (O) to (\n);
    
    \foreach \i [evaluate=\i as \is using int(\i-2)] in {5,...,\k} {
      \draw (\is) to (\n);
    }
  \end{tikzpicture}
  \end{center}
\end{ex}

\begin{ex}\label{ex:affine-types-B-C-BC} Let $\SSigma=\surf$ be an $n$-gon with exactly two orbifold points $q_1$ and $q_2$ (that is, $\orb=\{q_1,q_2\}$), and let $\omega_{\tildeB},\omega_{\tildeC},\omega_{\tildeBC}:\orb\rightarrow\{1,4\}$ be the functions given by $\omega_{\tildeB}(q_1)=1=\omega_{\tildeB}(q_2)$, $\omega_{\tildeC}(q_1)=4=\omega_{\tildeC}(q_2)$, $\omega_{\tildeBC}(q_1)=1$ and $\omega_{\tildeBC}(q_2)=4$
Then the weighted quivers $(Q_{\widetilde{B}},\dtuple_{\tildeB})$, $(Q_{\widetilde{C}},\dtuple_{\tildeC})$ and $(Q_{\widetilde{Y}},\dtuple_{\tildeBC})$ respectively coincide with $(Q(\tau,\omega_{\tildeB}),\dtuple(\tau,\omega_{\tildeB}))$, $(Q(\tau,\omega_{\tildeC}),\dtuple(\tau,\omega_{\tildeC}))$ and $(Q(\tau,\omega_{\tildeBC}),\dtuple(\tau,\omega_{\tildeBC}))$, where $\tau$ is the following triangulation of $\SSigma$:
\begin{center}
  \begin{tikzpicture}[scale=0.4]
    \def \n {11}
    \def \m {8} 
    \def \k {10} 
    \def \s {360/\n}
    \def \r {5cm}

    \foreach \i in {2,...,\n} {
      \coordinate (\i) at (\i*360/\n+90:\r);
    }
    \coordinate (O) at ($(2)!0.3!(\n)$);
    \node[rotate=45] at (O) {$\times$};
    \coordinate (Oq) at ($(2)!0.2!(\n)$);
    \node at (Oq) {$q_1$};
    \coordinate (x) at ($(9)!0.3!(\n)$);
    \node[rotate=45] at (x) {$\times$};
    \coordinate (xq) at ($(9)!0.2!(\n)$);
    \node at (xq) {$q_2$};
    \coordinate (middle) at (270:\r/2);
    \node at (middle) {$\ldots$};

    \foreach \i in {3,...,\m}
      \node at (\i) {$\bullet$};
    \node at (\n) {$\bullet$};

    \draw[dotted] (0,0) circle (\r);

    \draw (O) to (\n);
     \draw (x) to (\n);
    \foreach \i [evaluate=\i as \is using int(\i-2)] in {5,...,\k} {
      \draw (\is) to (\n);
    }
  \end{tikzpicture}
  \end{center}
\end{ex}

\begin{ex}\label{ex:finite-type-A} Let $\SSigma=\surf$ be an $n$-gon without orbifold points (that is, $\orb=\varnothing$), and let $\omega_{\AAA}:\orb\rightarrow\{1,4\}$ be the empty function. Then the weighted quiver $(Q_X,\dtuple_{\AAA})$ coincides with $(Q(\tau,\omega_{\AAA}),\dtuple(\tau,\omega_{\AAA}))$, where $\tau$ is the following triangulation of $\SSigma$:
\begin{center}
  \begin{tikzpicture}[scale=0.4]
    \def \n {11}
    \def \m {9} 
    \def \k {10} 
    \def \s {360/\n}
    \def \r {5cm}

    \foreach \i in {2,...,\n} {
      \coordinate (\i) at (\i*360/\n+90:\r);
    }
    \coordinate (O) at ($(2)$);
    \coordinate (middle) at (270:\r/2);
    \node at (middle) {$\ldots$};

    \foreach \i in {2,...,\m}
      \node at (\i) {$\bullet$};
    \node at (\n) {$\bullet$};

    \draw[dotted] (0,0) circle (\r);

    \foreach \i [evaluate=\i as \is using int(\i-2)] in {5,...,\k} {
      \draw (\is) to (\n);
    }
  \end{tikzpicture}
  \end{center}
\end{ex}

Using Examples \ref{ex:finite-types-B-C}, \ref{ex:affine-types-B-C-BC} and \ref{ex:finite-type-A}, we see that any matrix $B'$ mutation-equivalent to a matrix of type $\BB_{n-1}$, $\tildeB_n$, $\CC_{n-1}$, $\tildeC_n$, $\AAA_{n-3}$ or $\tildeBC_n$, is the matrix associated by Felikson-Shapiro-Tumarkin to a triangulation $\sigma$ of a polygon with at most two orbifold points with respect to some function $\omega$. Whenever we choose an arbitrary 1-cocycle $\zeta\in Z^1(\sigma,\omega)$, we obtain an SP-realization of $B'$ via the non-degenerate SP $(A(\sigma,\zeta),S(\sigma,\zeta))$. Note that, by Theorem \ref{thm:comologous<=>isomorphic-Jacobian-algs}, the isomorphism class of $\mathcal{P}(A(\sigma,\zeta),S(\sigma,\zeta))$ as an $F$-algebra is independent of the chosen cocycle $\zeta$ since $H^1(C^\bullet(\sigma,\omega))$ is isomorphic to the first singular cohomology group of a disc (with coefficients in $\F_2$).

As stated in Conjecture \ref{conj:cluster-tilted-algebras}, we believe that the Jacobian algebra $\mathcal{P}(A(\sigma,\zeta),S(\sigma,\zeta))$ is cluster-tilted in the sense \cite{BMRRT} of Buan-Marsh-Reineke-Reiten-Todorov.  Thus, for instance, we believe the algebras from Examples \ref{ex:hexagon_1_orb_pt_wt4_species} and \ref{ex:pentagon_2_orb_pts_wt4_species} to be cluster-tilted of types $\CC_5$ and $\tildeC_5$, respectively.

\begin{ex} Let $\SSigma=\surf$ be the unpunctured hexagon with exactly one orbifold point $q$, $\omega:\orb\rightarrow\{1,4\}$ be the function that takes the value $1$ at $q$, and $\tau$ be the triangulation of $\SSigma$ depicted in Figure \ref{Fig:hexagon_1_orb_pt_wt4}. Take $\C/\R$ as the ground field extension $E/F$. If $\xi\in Z^1(\tau,\omega)$ is the zero cocycle, then the semisimple ring $R=\bigoplus_{k\in\tau}F_k$ and the species $A(\tau,\xi)$ can be (informally) visualized as follows:
$$
\xymatrix{
\C \ar[rr]^{\C\otimes_{\C}\C} & & \C \ar[dl]^{\C\otimes_{\C}\C} \\
 & \C \ar[ul]^{\C\otimes_{\C}\C} \ar[dr]^{\R\otimes_{\R}\C} & \\
\C \ar[ur]^{\C\otimes_{\C}\C} & & \R \ar[ll]^{\C\otimes_{\R}\R},
}
$$
Furthermore, $\R$ sits as a central subring of the Jacobian algebra $\mathcal{P}(A(\tau,\xi),S(\tau,\xi))$, and the following identities hold in $\mathcal{P}(A(\tau,\xi),S(\tau,\xi))$:
$$
\alpha z = z\alpha, \ \beta z = z\beta, \ \gamma z = z\gamma, \  \delta z = z\delta, \ \text{for all $z\in\C$},
$$
$$
\beta\gamma=\gamma\alpha=\alpha\beta=\eta\nu=\nu\delta=\delta\eta=0.
$$
Since $H^1(C^\bullet(\tau,\omega))\cong H^1(\Sigma;\F_2)$ is zero, Theorem \ref{thm:comologous<=>isomorphic-Jacobian-algs} implies that for any 1-cocycle $\xi'\in Z^1(\tau,\omega)$ the Jacobian algebra $\mathcal{P}(A(\tau,\xi'),S(\tau,\xi'))$ will be isomorphic to $\mathcal{P}(A(\tau,\xi),S(\tau,\xi))$ as an $\R$-algebra (of course, provided we do not decide to change our ground field extension $E/F$ all of the sudden). As we have said in Conjecture \ref{conj:cluster-tilted-algebras}, we believe $\mathcal{P}(A(\tau,\xi),S(\tau,\xi))$ to be a cluster-tilted algebra of type $\BB_5$.
\end{ex}

\subsection{Two non-isomorphic algebras for the Kronecker quiver}
Let $\SSigma=\surf$ be an unpunctured annulus with exactly two marked points (one on each boundary component) and without orbifold points, let $\omega:\orb\rightarrow\{1,4\}$ be the empty function, and let $\tau=\{j_1,j_2\}$ be any triangulation of $\SSigma$. Then $\overline{Q}(\tau)=Q(\tau,\omega)=Q'(\tau,\omega)=Q''(\tau,\omega)$ has exactly two arrows, say $a:j_1\rightarrow j_2$ and $b:j_1\rightarrow j_2$, and we clearly have $X_2(\tau,\omega)=\varnothing$. Hence $\dim_{\F_2}(C^1(\tau,\omega))=2$ and every element $\xi\in C^1(\tau,\omega)$ is a 1-cocycle. Let $\{\xi_a,\xi_b\}$ be the $\F_2$-vector space basis of $C^1(\tau,\omega)$ which is $\F_2$-dual to the $\F_2$-basis $\{a,b\}$ of $C_1(\tau,\omega)$. Then $C^1(\tau,\omega)=\{0,\xi_a,\xi_b,\xi_a+\xi_b\}$.

Let us take $\C/\R$ as our ground field extension. It is fairly easy to check that $\mathcal{P}(A(\tau,0),S(\tau,0))\cong \mathcal{P}(A(\tau,\xi_a+\xi_b),S(\tau,\xi_a+\xi_b))\not\cong\mathcal{P}(A(\tau,\xi_a),S(\tau,\xi_a))\cong\mathcal{P}(A(\tau,\xi_b),S(\tau,\xi_b))$ as rings and as $\R$-algebras. It is also easy to directly verify that $[0]=[\xi_a+\xi_b]$ and $[\xi_a]=[\xi_b]$ as elements of $H^1(C^\bullet(\tau,\omega))$. This is in sync with Theorem \ref{thm:comologous<=>isomorphic-Jacobian-algs}.

As the reader can readily check, the field $\mathbb{C}$ acts centrally on the ring $\mathcal{P}(A(\tau,0),S(\tau,0))$, and $\mathcal{P}(A(\tau,0),S(\tau,0))$ is isomorphic, as a $\C$-algebra, to the usual path algebra over $\C$ of the well-known Kronecker quiver. On the other hand, the right and left actions of $\C$ on $\mathcal{P}(A(\tau,\xi_a),S(\tau,\xi_a))$ are related by the identities
$$
az=\theta(z)a \ \ \ \text{and} \ \ \ bz =zb,
$$
which hold for all $z\in\C$, where $\theta:\C\rightarrow\C$ is the usual conjugation of complex numbers. This means that if $M$ is a left $\mathcal{P}(A(\tau,\xi_a),S(\tau,\xi_a))$-module, and if we set $M_{j_1}=e_{j_1}M$, $M_{j_2}=e_{j_2}$, $M_a,M_b:M_{j_1}\rightarrow M_{j_2}$, $M_a(m) = am$, $M_b(m)=b_m$, then $M_{j_1}$ and $M_{j_2}$ are $\C$-vector spaces, $M_b$ is a $\C$-linear transformation and $M_a$ is an $\R$-linear transformation which is not $\C$-linear (unless $M_a=0$) but rather satisfies $M_a(zm)=\theta(z)M_a(m)$ for $z\in \C$ and $m\in M_{j_1}$. In particular, the category of left $\mathcal{P}(A(\tau,\xi_a),S(\tau,\xi_a))$-modules is not equal to the well-known category of left modules over the usual path algebra over $\C$ of the Kronecker quiver. A little effort shows that, actually, these two categories are not equivalent as additive categories. Indeed, in $\mathcal{P}(A(\tau,0),S(\tau,0))$-$\operatorname{mod}$ every non-zero endomorphism ring has $\mathbb{C}$ as a subring, but at least one object of $\mathcal{P}(A(\tau,\xi_a),S(\tau,\xi_a))$-$\operatorname{mod}$ has $\mathbb{R}$ as its endomorphism ring, and $\mathbb{R}$ has no subring isomorphic to $\mathbb{C}$, so an additive equivalence between the two categories cannot exist.

The classification of the indecomposable $\mathcal{P}(A(\tau,\xi_a),S(\tau,\xi_a))$-modules has been carried out in \cite{Dj}.


\subsection{Relation to previous constructions}

Suppose that $\SSigma=\surf$ is unpunctured and that $\omega:\orb\rightarrow\{1,4\}$ is the constant function that takes the value $4$ at every element of $\orb$. Then for any two 
triangulations $\tau$ and $\sigma$ of $\SSigma$ that are related by the flip of an arc $k\in\tau$, the 
function $Z^1(\tau,\omega)\rightarrow Z^1(\sigma,\omega)$ given by the colored flip is $\F_2$-linear. 
Therefore, if $(\tau,\xi)$ is a colored triangulation such that $\xi\in Z^1(\tau,\omega)$ is the zero 
cocycle, then applying any finite sequence of colored flips to $(\tau,\xi)$ will always produce a colored 
triangulation $(\rho,\zeta)$ whose underlying cocycle $\zeta$ is the zero cocycle (in the corresponding 
cochain complex $C^\bullet(\rho,\omega)$). This means that if $\omega$ never takes the value $1$, then neither colored flips nor SP-mutations obstruct us if we decide to always choose the zero cocycle.

\begin{ex} Suppose that $\orb=\varnothing$, so that $\omega$ is actually the empty function. If $(\tau,\xi)$ is a colored triangulation of $\SSigmaw$ such that $\xi=0$, and $\C/\R$ is our ground field extension (note that $d(\tau,\omega)_k=2$ for every $k\in\tau$ since $\orb=\varnothing$), then $\C$ acts centrally on the path algebra
$R\langle A(\tau,\xi)\rangle$, on the complete path algebra $\RA{A(\tau,\xi)}$ and on the Jacobian 
algebra $\mathcal{P}(A(\tau,\xi),S(\tau,\xi))$ (which means that these rings are $\C$-algebras and not 
only $\R$-algebras), and the SP $(A(\tau,\xi),S(\tau,\xi))$ coincides with the quiver with potential defined 
in \cite{Labardini1,Labardini4}. The fact that $\SSigma$ is unpunctured implies then that
$\mathcal{P}(A(\tau,\xi),S(\tau,\xi))$ is the gentle algebra studied in \cite{ABCP} by
Assem-Br\"ustle-Charbonneau-Plamondon.
\end{ex}

\begin{ex} Suppose that $\orb\neq \varnothing$. If $\tau$ is any triangulation of $\SSigma$, then $d(\tau,\omega)_k\in\{2,4\}$ for every $k\in\tau$, and so, Section \ref{sec:sp-of-a-colored-triangulation} tells us that our ground field extension should have degree $4$. However, if $\xi=0\in Z^1(\tau,\omega)$, then we can actually work with a degree-2 extension $E/L$ of fields of characteristic different from $2$ (for example, $\C/\R$). Indeed, if $\xi=0$, then we can set $g(\tau,\xi)_{\alpha}=\myid_{L}$ for every arrow $\alpha$ incident to at least one non-pending arc, and
$$
g(\tau,\xi)_{\alpha}=\begin{cases}
\myid_{E} & \text{if $\alpha=\delta_0^\triangle$;}\\
\vartheta & \text{if $\alpha=\delta_1^\triangle$;}
\end{cases}
$$
for every twice-orbifolded triangle $\triangle$ of $\tau$, where $\vartheta$ is the non-identity element of $\Gal(E/L)$, and the arrows $\delta_0^\triangle$ and $\delta_1^\triangle$ are defined as in item (3) of Definition \ref{def:cocycle->modulating-function} for such a triangle $\triangle$. The SP obtained in this way coincides with the SP defined in \cite{Geuenich-Labardini-1}. This means that for unpunctured surfaces, the SPs constructed in this paper generalize those constructed in \cite{Geuenich-Labardini-1}.
\end{ex}

%% file: 13_technical.tex
\section{Proof of Theorem \ref{thm:flip<->SP-mutation}}

\label{sec:proof-of-main-theorem}

{\small

\begin{proof}[Proof of Theorem \ref{thm:flip<->SP-mutation}] Consider the three puzzle pieces shown in Figure \ref{Fig:puzzle_pieces_unpunctured}.
        \begin{figure}[!ht]
                \centering
                \includegraphics[scale=.3]{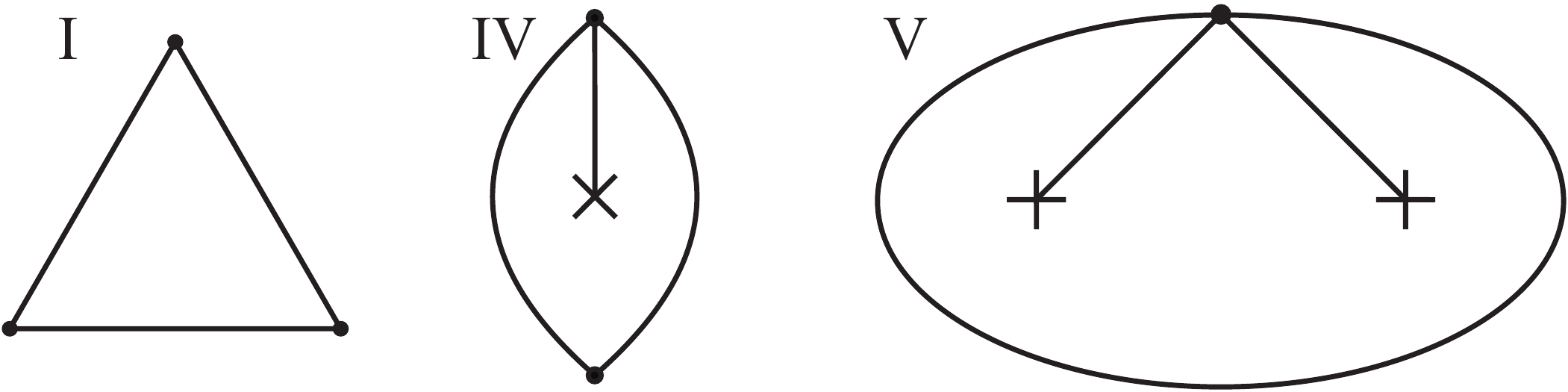}
                \caption{{\footnotesize Unpunctured puzzle pieces}}
                \label{Fig:puzzle_pieces_unpunctured}
        \end{figure}
Since $\SSigma$ is either unpunctured or once-punctured closed, a consequence of \cite[Theorem 2.7]{Geuenich-Labardini-1} is that both $\tau$ and $\sigma$ have puzzle-piece decompositions that involve only puzzle pieces of types I, IV or V. Fix such a puzzle-piece decomposition of $\tau$. The arc $k\in\tau$ is then either a pending arc inside a puzzle piece or an arc shared by two puzzle pieces. Figure \ref{Fig:all_possibilities_for_k_pending} lists the three possibilities for $k$ if $k$ happens to be a pending arc, while Figure \ref{Fig:all_possibilities_for_k_shared} lists all the possibilities for $k$ if $k$ happens to be an arc shared by two puzzle pieces.
        \begin{figure}[!ht]
                \centering
                \includegraphics[scale=.3]{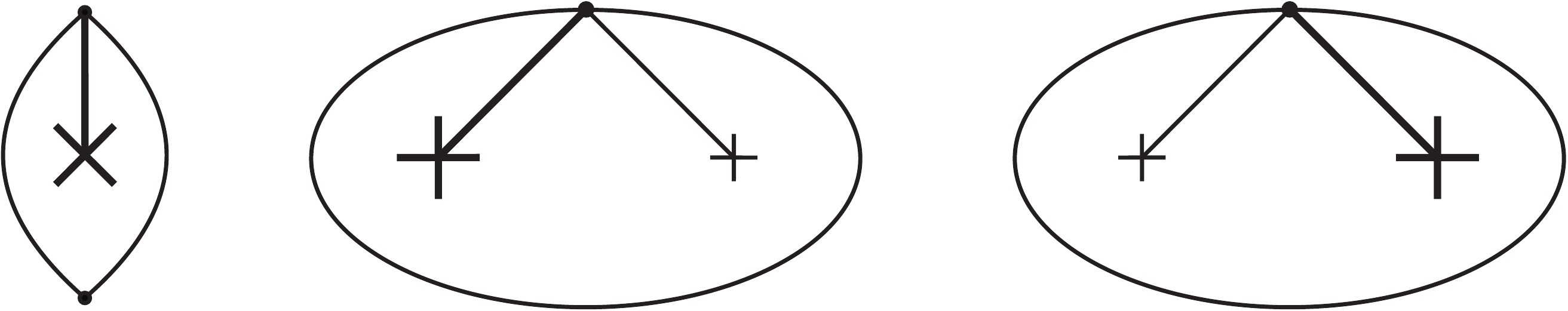}
                \caption{{\footnotesize}}
                \label{Fig:all_possibilities_for_k_pending}
        \end{figure}
        \begin{figure}[!ht]
               \begin{displaymath}
\hspace{-20em}
\vcenter{
\begin{tikzpicture}[scale=0.65]
 \coordinate (2) at (-1.2,-1.2);
 \coordinate (1) at (-1.2, 1.2);
 \coordinate (4) at ( 1.2, 1.2);
 \coordinate (3) at ( 1.2,-1.2);

 \foreach \i in {1,...,4}
  \node at (\i) {$\bullet$};

 \draw (1) to node[left] {$i$} (2);
 \draw (2) to node[below] {$j$} (3);
 \draw (3) to node[right] {$l$} (4);
 \draw (4) to node[above] {$m$} (1);

 \draw[very thick] (1) to (3);

 \node at (0.2,0.2) {$k$};
\end{tikzpicture}
}
\hspace{-20em}
\begin{array}{cccc}
\begin{tikzpicture}[scale=0.65]
 \coordinate (2) at (-30:1.4cm);
 \coordinate (3) at (210:1.4cm);
 \coordinate (1) at ( 90:1.4cm);
 \coordinate (4) at (0,0.5);
 \coordinate (H) at (-0.5,-0.3);

 \foreach \i in {1,...,3}
  \node at (\i) {$\bullet$};
 \node at (4) {$\times$};

 \draw (0,0) circle (1.4cm);
 \draw[very thick,bend angle=40,bend right,rounded corners] (1) to (H) to (2);
 \draw (1) to (4);

 \node at ($(-0.3,-0.1)+(150:1.4cm)$) {$i$};
 \node at ($(0,-0.3)+(270:1.4cm)$) {$j$};
 \node at ($(0.3,-0.1)+( 30:1.4cm)$) {$l$};

 \node at (-0.2,-0.2) {$k$};
\end{tikzpicture}
&
\begin{tikzpicture}[scale=0.65]
 \coordinate (2) at (-90:1.4cm);
 \coordinate (1) at ( 90:1.4cm);
 \coordinate (3) at (-0.5,0.5);
 \coordinate (4) at ( 0.5,0.5);
 \coordinate (H) at (0.5,-0.3);

 \foreach \i in {1,...,2}
  \node at (\i) {$\bullet$};
 \node[rotate=45] at (3) {$\times$};
 \node[rotate=45] at (4) {$\times$};

 \draw (0,0) circle (1.4cm);
 \draw[very thick] (0,0.5) circle (0.9cm);
 \draw (1) to (3);
 \draw (1) to (4);

 \node at ($(-0.2,0)+(180:1.4cm)$) {$i$};
 \node at ($( 0.2,0)+(  0:1.4cm)$) {$j$};
 \node at (0,-0.1) {$k$};

 \node at ($(0,-0.3)+(270:1.4cm)$) {\phantom{$j$}};
\end{tikzpicture}
&
\begin{tikzpicture}[scale=0.65]
 \coordinate (1) at ( 90:1.4cm);
 \coordinate (2) at (   0,-0.7);
 \coordinate (3) at (-0.5, 0.5);
 \coordinate (4) at ( 0.5, 0.5);
 \coordinate (H) at (165:1.2cm);

 \node at (1) {$\bullet$};
 \node at (2) {$\times$};
 \node[rotate=45] at (3) {$\times$};
 \node[rotate=45] at (4) {$\times$};

 \draw (0,0) circle (1.4cm);
 \draw[very thick] (0,0.5) circle (0.9cm);
 \draw[bend angle=40,bend right,rounded corners] (1) to (H) to (2);
 \draw (1) to (3);
 \draw (1) to (4);

 \node at ($(0,-0.3)+(270:1.4cm)$) {$i$};
 \node at (0,-0.1) {$k$};

 \node at ($(0,-0.3)+(270:1.4cm)$) {\phantom{$j$}};
\end{tikzpicture}
&
\begin{tikzpicture}[scale=0.65]
 \coordinate (2) at (-90:1.4cm);
 \coordinate (1) at ( 90:1.4cm);
 \coordinate (3) at (0, 0.7);
 \coordinate (4) at (0,-0.7);
 \coordinate (H) at (0,0);

 \foreach \i in {1,...,2}
  \node at (\i) {$\bullet$};
 \node at (3) {$\times$};
 \node at (4) {$\times$};

 \draw (0,0) circle (1.4cm);
 \draw[very thick,bend angle=40,bend left] (1) to (H);
 \draw[very thick,bend angle=40,bend right] (H) to (2);
 \draw (1) to (3);
 \draw (2) to (4);

 \node at ($(-0.2,0)+(180:1.4cm)$) {$i$};
 \node at ($( 0.2,0)+(  0:1.4cm)$) {$j$};
 \node at (-0.35,0) {$k$};

 \node at ($(0,-0.3)+(270:1.4cm)$) {\phantom{$j$}};
\end{tikzpicture}
\\
\begin{tikzpicture}[scale=0.65]
 \coordinate (2) at (-30:1.4cm);
 \coordinate (3) at (210:1.4cm);
 \coordinate (1) at ( 90:1.4cm);
 \coordinate (4) at (0,0.5);
 \coordinate (H) at (0.5,-0.3);

 \foreach \i in {1,...,3}
  \node at (\i) {$\bullet$};
 \node at (4) {$\times$};

 \draw (0,0) circle (1.4cm);
 \draw[very thick,bend angle=40,bend left,rounded corners] (1) to (H) to (3);
 \draw (1) to (4);

 \node at ($(-0.3,-0.1)+(150:1.4cm)$) {$i$};
 \node at ($(0,-0.3)+(270:1.4cm)$) {$j$};
 \node at ($(0.3,-0.1)+( 30:1.4cm)$) {$l$};

 \node at (0.2,-0.2) {$k$};
\end{tikzpicture}
&
\begin{tikzpicture}[scale=0.65]
 \coordinate (2) at (-90:1.4cm);
 \coordinate (1) at ( 90:1.4cm);
 \coordinate (3) at (-0.5,0.5);
 \coordinate (4) at ( 0.5,0.5);

 \foreach \i in {1,...,2}
  \node at (\i) {$\bullet$};
 \node[rotate=45] at (3) {$\times$};
 \node[rotate=45] at (4) {$\times$};

 \draw (0,0) circle (1.4cm);
 \draw[very thick] (1) to (2);
 \draw (1) to (3);
 \draw (1) to (4);

 \node at ($(-0.2,0)+(180:1.4cm)$) {$i$};
 \node at ($( 0.2,0)+(  0:1.4cm)$) {$j$};
 \node at (-0.25,0) {$k$};

 \node at ($(0,-0.3)+(270:1.4cm)$) {\phantom{$j$}};
\end{tikzpicture}
&
\begin{tikzpicture}[scale=0.65]
 \coordinate (1) at ( 90:1.4cm);
 \coordinate (2) at (   0,-0.7);
 \coordinate (3) at (-0.5, 0.5);
 \coordinate (4) at ( 0.5, 0.5);
 \coordinate (H) at (15:1.2cm);

 \node at (1) {$\bullet$};
 \node at (2) {$\times$};
 \node[rotate=45] at (3) {$\times$};
 \node[rotate=45] at (4) {$\times$};

 \draw (0,0) circle (1.4cm);
 \draw[very thick] (0,0.5) circle (0.9cm);
 \draw[bend angle=40,bend left,rounded corners] (1) to (H) to (2);
 \draw (1) to (3);
 \draw (1) to (4);

 \node at ($(0,-0.3)+(270:1.4cm)$) {$i$};
 \node at (0,-0.1) {$k$};

 \node at ($(0,-0.3)+(270:1.4cm)$) {\phantom{$j$}};
\end{tikzpicture}
&
\begin{tikzpicture}[scale=0.65]
 \coordinate (2) at (-90:1.4cm);
 \coordinate (1) at ( 90:1.4cm);
 \coordinate (3) at (0, 0.7);
 \coordinate (4) at (0,-0.7);
 \coordinate (H) at (0,0);

 \foreach \i in {1,...,2}
  \node at (\i) {$\bullet$};
 \node at (3) {$\times$};
 \node at (4) {$\times$};

 \draw (0,0) circle (1.4cm);
 \draw[very thick,bend angle=40,bend right] (1) to (H);
 \draw[very thick,bend angle=40,bend left] (H) to (2);
 \draw (1) to (3);
 \draw (2) to (4);

 \node at ($(-0.2,0)+(180:1.4cm)$) {$i$};
 \node at ($( 0.2,0)+(  0:1.4cm)$) {$j$};
 \node at (0.35,0) {$k$};

 \node at ($(0,-0.3)+(270:1.4cm)$) {\phantom{$j$}};
\end{tikzpicture}
\end{array}
\end{displaymath}
                \caption{{\footnotesize}}
                \label{Fig:all_possibilities_for_k_shared}
        \end{figure}
Note that each of the sides $i, j, l, m$ of the configurations depicted in Figure \ref{Fig:all_possibilities_for_k_shared} may be either an arc of $\tau$ or a boundary segment.

We will prove Theorem \ref{thm:flip<->SP-mutation} by showing that its assertion is true for each of the three configurations depicted in Figure \ref{Fig:all_possibilities_for_k_pending} and each of the nine configurations depicted in Figure \ref{Fig:all_possibilities_for_k_shared}. Now, all but one of these twelve configurations involve at least one (and at most three) orbifold point. Since $\omega$ is an arbitrary function from $\orb$ to $\{1,4\}$, this means that for each configuration, if it involves $\nu$ orbifold points ($0\leq\nu\leq 3$), we have to take into account the $2^\nu$ possible combinations of values that $\omega$ can take on the $\nu$ orbifold points involved. This gives us a total of $47=(2+4+4)+(1+2+2+4+4+8+8+4+4)$ possible cases to take into account; these cases have been depicted in Figure \ref{Fig:all_possibilities_for_k_and_weights}.
        \begin{figure}[!ht]
                \centering
                \includegraphics[scale=.5]{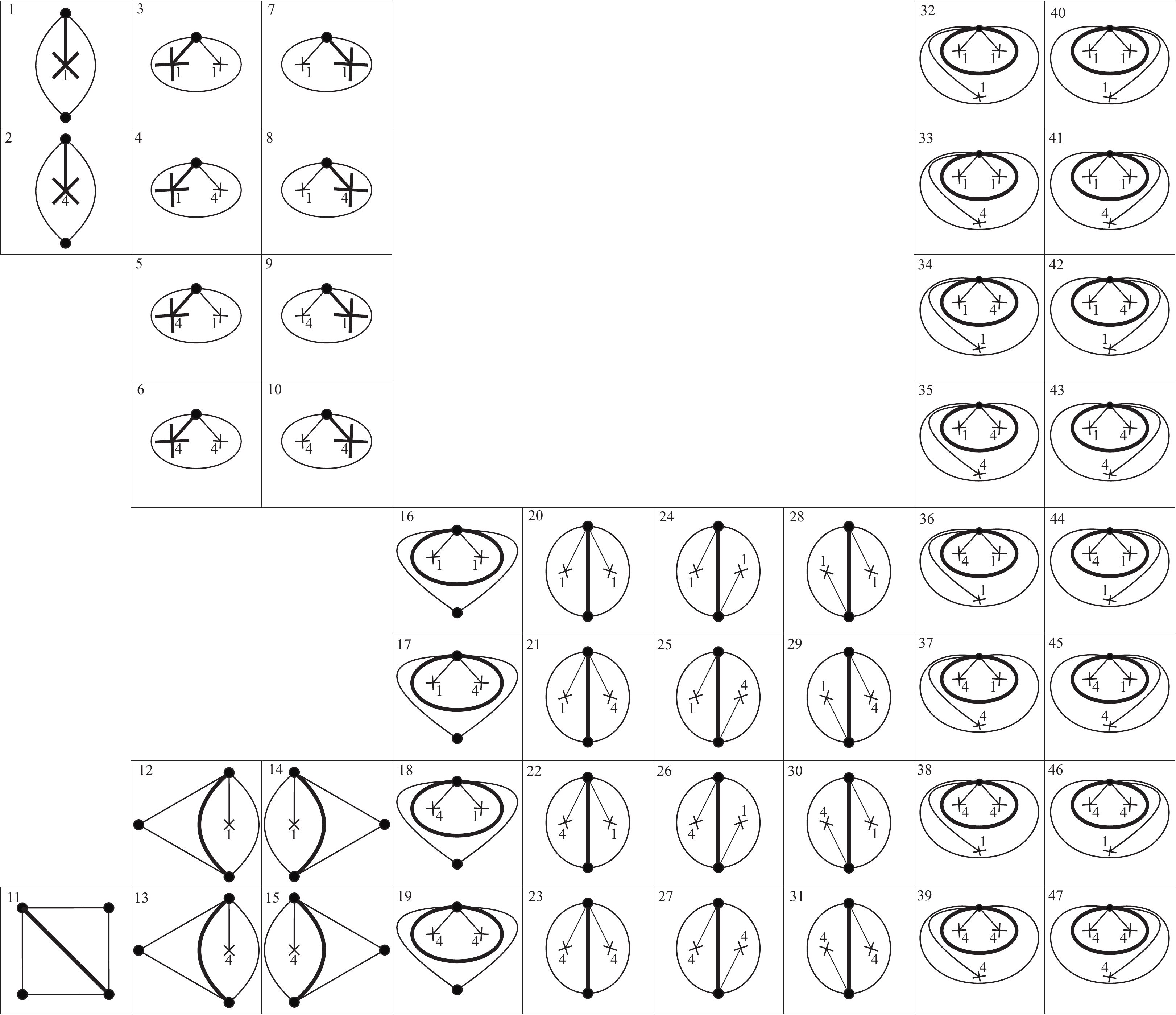}
                \caption{{\footnotesize}}
                \label{Fig:all_possibilities_for_k_and_weights}
        \end{figure}
However, note that in Figure \ref{Fig:all_possibilities_for_k_and_weights} the following pairs of configurations correspond to each other under the flip of $k$:
$$
\begin{array}{cccc}
3 \leftrightarrow 7 &
4 \leftrightarrow 9 &
5 \leftrightarrow 8 &
6 \leftrightarrow 10 \\
12 \leftrightarrow 14 &
13 \leftrightarrow 15 &
16 \leftrightarrow 20 &
17 \leftrightarrow 21 \\
18 \leftrightarrow 22 &
19 \leftrightarrow 23 &
24 \leftrightarrow 28 &
25 \leftrightarrow 30 \\
26 \leftrightarrow 29 &
27 \leftrightarrow 31 &
32 \leftrightarrow 40 &
33 \leftrightarrow 44 \\
34 \leftrightarrow 41 &
35 \leftrightarrow 45 &
36 \leftrightarrow 42 &
37 \leftrightarrow 46 \\
&
38 \leftrightarrow 43 &
39 \leftrightarrow 47. &
\end{array}
$$
Note also that configuration 25 coincides with configuration 26.
Since flips and SP-mutations are involutive (the latter up to right-equivalence, cf. \cite[Theorem 3.24]{Geuenich-Labardini-1}), this means that, in order to know that the statement of Theorem \ref{thm:flip<->SP-mutation} is true for the 47 configurations of Figure \ref{Fig:all_possibilities_for_k_and_weights}, it suffices to show it is true for the following 24 configurations:
$$
\begin{array}{ccccc}
1 &
2 &
3 &
4 &
5  \\
6 &
11 &
12 &
13 &
20 \\
21 &
22 &
23 &
24 &
25  \\
\phantom{26}  &
27  &
32  &
33 &
34  \\
35  &
36  &
37  &
38  &
39.
\end{array}
$$
This is what we will do. In all the 24 cases, the strategy we shall adopt is the following:
\begin{enumerate}
\item Draw $\tau$ and $\sigma$, and draw the quivers $Q(\tau,\xi)$ and $\widetilde{\mu}_k(Q(\tau,\xi))$ on top of $\tau$ and $\sigma$, respectively.
 The elements of certain set of arrows whose removal from $\widetilde{\mu}_k(Q(\tau,\xi))$ gives (a quiver isomorphic to) $\mu_k(Q(\tau,\xi))$ will be drawn as dotted arrows. We will identify $Q(\sigma,\zeta)$ with the subquiver of $\widetilde{\mu}_k(Q(\tau,\xi))$ obtained by deleting the dotted arrows. This identification will have the property that for every arrow $a$ of $Q(\sigma,\zeta)$, the value of the modulating function $g(\sigma,\zeta)$ at $a$ coincides with the value of $\widetilde{\mu}_k(g(\tau,\xi))$ at $a$. [This property, whose verification we will leave in the hands of the reader in all cases, will be an easy consequence of the definition of $\widetilde{\mu}_k(g(\tau,\xi))$, the definition of flips of colored triangulations (recall that we have $(\sigma,\zeta)=f_k(\tau,\xi)$ by hypothesis), and the definition of $g(\sigma,\zeta)$.] The 1-cocycles $\xi$ and $\zeta$ will not be indicated in the figures.
\item Define certain potential $\Ssigmad^\sharp\in\RA{\widetilde{\mu}_k(A(\tau,\xi))}$ as the sum of $\Ssigmad$ and a degree-2 potential on $\widetilde{\mu}_k(A(\tau,\xi))$. The SP $(\widetilde{\mu}_k(A(\tau,\xi)),\Ssigmad^\sharp)$ will have the property of having $\ASsigmad$ as a reduced part. [This property will be a consequence of the fact that every 2-cycle $ab$ appearing in (the degree-2 component of) $\Ssigmad^\sharp$ will satisfy $\widetilde{\mu}_k(g(\tau,\xi))(a)=\left(\widetilde{\mu}_k(g(\tau,\xi))(b)\right)^{-1}$. This fact, whose verification we will leave in the hands of the reader in all cases, is an easy consequence of the definition of $\widetilde{\mu}_k(g(\tau,\xi))$ and of the fact that $\xi$ is a $1$-cocycle in the cochain complex $C^\bullet(\tau,\omega)$.]
\item Write down $\Stauc$ and a potential cyclically equivalent to $\widetilde{\mu}_k(\Stauc)$ explicitly. The potential cyclically equivalent to $\widetilde{\mu}_k(\Stauc)$ will always be the result of applying the following fact to some\footnote{Possibly none, depending on the case.} of the summands of $\widetilde{\mu}_k(\Stauc)$: For any arrow $a:j\rightarrow i$ and any element $x$ of the complete path algebra, if $xa$ is a potential, then $xa\sim_{\operatorname{cyc}}\pi_{g_a^{-1}}(x)a$, where
$\pi_{g_{a}^{-1}}(x)=\frac{1}{d_{i,j}}\sum_{\omega\in\B_{i,j}}g_a^{-1}(\omega^{-1})x\omega$ (this is proved in \cite[Example 3.12]{Geuenich-Labardini-1}).
\item Exhibit a finite sequence of explicitly-defined $R$-algebra automorphisms of $\RA{\widetilde{\mu}_k(A(\tau,\xi))}$ with the property that its composition sends $\widetilde{\mu}_k(\Stauc)$ to a potential which is cyclically equivalent to $\Ssigmad^\sharp$. The fact that the automorphisms in the sequence are indeed well defined will be a consequence of the definition of $\widetilde{\mu}_k(g(\tau,\xi))$ and of the fact that $\xi$ is a $1$-cocycle in the cochain complex $C^\bullet(\tau,\omega)$. The property that the composition of these automorphisms sends $\widetilde{\mu}_k(\Stauc)$ to a potential cyclically equivalent to $\Ssigmad^\sharp$ will follow from direct computation; we will omit this computation.
\end{enumerate}
Note that once we have this last step, that is, once we know that $\widetilde{\mu}_k(A(\tau,\xi),\Stauc)$ is right-equivalent to $(\widetilde{\mu}_k(A(\tau,\xi),\Ssigmad^\sharp)$, it will follow from \cite[Theorem 3.16]{Geuenich-Labardini-1} that $\mu_k(\AStauc)$ is right-equivalent to $\ASsigmad$.

\begin{remark}\label{rem:notation_g=g(tau,xi)}\begin{enumerate}\item In order to keep the notation as light as possible throughout the proof, we will write $g$ for the modulating function $g(\tau,\xi)$ and $\widetilde{\mu}_k(g)$ for the modulating function $\widetilde{\mu}_k(g(\tau,\xi))$.
\item It is easy to deduce from \cite[Example 3.12]{Geuenich-Labardini-1} that whenever we have an arrow $a:j\rightarrow i$ of $\widetilde{\mu}_k(A(\tau,\xi))$ and an element $x\in e_j\RA{\widetilde{\mu}_k(A(\tau,\xi))} e_i$, we necessarily have $xa\sim_{\operatorname{cyc}}\pi_{(\widetilde{\mu}_k(g)_a)^{-1}}(x)a$, where $\widetilde{\mu}_k(g)_a$ is the value of the modulating function $\widetilde{\mu}_k(g)$ at $a$ (see \cite[Definition 3.19]{Geuenich-Labardini-1}) and $\pi_{(\widetilde{\mu}_k(g)_a)^{-1}}(x):=\frac{1}{d_{i,j}}\sum_{\omega\in\B_{i,j}}(\widetilde{\mu}_k(g)_a)^{-1}(\omega^{-1})x\omega$. We will use this fact repeatedly throughout the proof without any further apology ($a$ will usually be a composite arrow of $\widetilde{\mu}_k(A(\tau,\xi))$).
\end{enumerate}
\end{remark}

\setcounter{case}{0}


\begin{case}\label{case:1}\emph{Configuration 1}.
\begin{figure}[!ht]
                \centering
                \includegraphics[scale=.65]{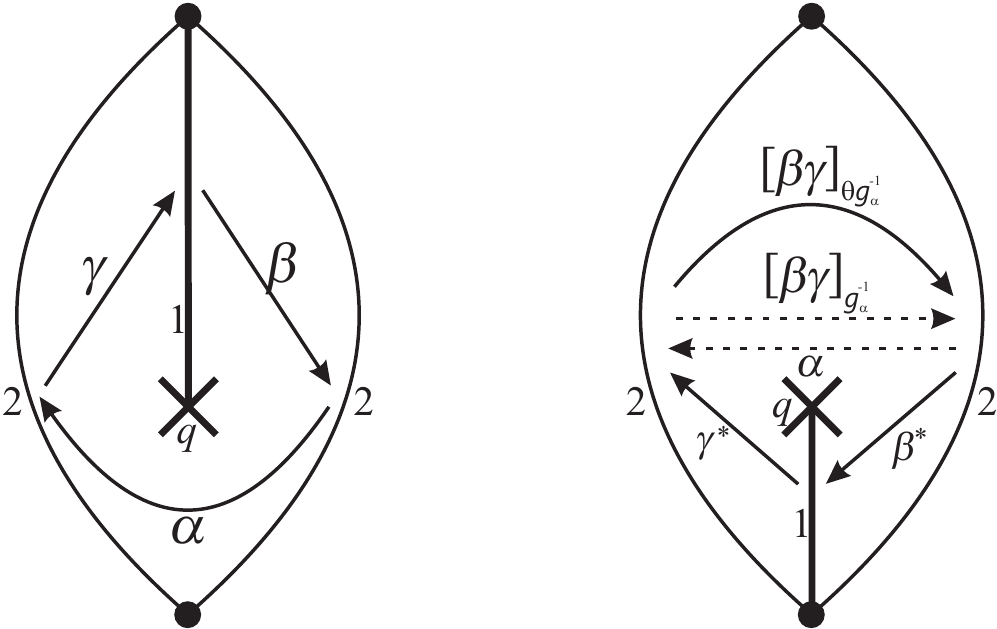}\caption{{\footnotesize First configuration of Figure \ref{Fig:all_possibilities_for_k_and_weights}. Left: $\tau$ and $Q(\tau,\omega)$. Right: $\sigma$ and $\widetilde{\mu}_k(Q(\tau,\omega))$.  The numbers next to the arcs are the corresponding values of the tuple $\dtuple(\tau,\omega)$.}}
                \label{Fig:flip_mut_1}
        \end{figure}
\begin{enumerate} \item We use the notation in Figure \ref{Fig:flip_mut_1} and identify the quiver $Q(\sigma,\omega)$ with the subquiver of $\widetilde{\mu}_k(Q(\tau,\omega))$ obtained by deleting the dotted arrows in the figure.

\item The SP $(A(\sigma,\zeta),S(\sigma,\zeta))$ is the reduced part of $(\widetilde{\mu}_k(A(\tau,\xi)),S(\sigma,\zeta)^\sharp)$, where
\begin{eqnarray}\nonumber
\Ssigmad^\sharp  & = &
X\left(\alpha[\beta\gamma]_{g_\alpha^{-1}}
+\gamma^*\beta^*[\beta\gamma]_{\theta g_\alpha^{-1}}\right)
+S(\tau,\sigma)
\in\RA{\widetilde{\mu}_k(A(\tau,\xi))},
\end{eqnarray}
with $S(\tau,\sigma)\in\RA{A(\tau,\xi)}\cap\RA{A(\sigma,\zeta)}$.
Here, $X$ is an element of $\{0,1\}\subseteq F$ with the property of being equal to $1\in F$ (resp. $0\in F$) if and only if the two sides of the digon in Figure \ref{Fig:flip_mut_1} are indeed arcs in $\tau$ (resp. at least one of the two sides of the digon is a boundary segment).

\item The potentials $\Stauc$ and $\widetilde{\mu}_k(\Stauc)$ are
\begin{eqnarray*}
\Stauc & = &
X\alpha\beta\gamma
+S(\tau,\sigma) \ \ \ \ \ \text{and}\\
\widetilde{\mu}_k(\Stauc)
&\sim_{\operatorname{cyc}}&
X\left(\alpha[\beta\gamma]_{g_\alpha^{-1}}
+\pi_{g_\alpha}(\gamma^*\beta^*)[\beta\gamma]_{g_\alpha^{-1}}
+\gamma^*\beta^*[\beta\gamma]_{\theta g_\alpha^{-1}}\right)
+S(\tau,\sigma).
\end{eqnarray*}

\item
Define an $R$-algebra automorphism $\varphi:\RA{\widetilde{\mu}_k(A(\tau,\xi))}\rightarrow\RA{\widetilde{\mu}_k(A(\tau,\xi))}$ according to the rule
\begin{eqnarray*}
\varphi &:& X\alpha\mapsto X\left(\alpha-\pi_{g_\alpha}(\gamma^*\beta^*)\right)
\end{eqnarray*}
(that this rule indeed yields a well-defined $R$-algebra automorphism is a consequence of \cite[Propositions 2.15-(6) and 3.7]{Geuenich-Labardini-1}).
It is obvious that $\varphi$ is a right-equivalence $(\widetilde{\mu}_k(A(\tau,\xi)),\widetilde{\mu}_k(S(\tau,\xi)))\rightarrow(\widetilde{\mu}_k(A(\tau,\xi)),S(\sigma,\zeta)^\sharp)$.
\end{enumerate}
\end{case}


\begin{case}\label{case:2}\emph{Configuration 2}. \begin{enumerate}\item We use the notation in Figure \ref{Fig:flip_mut_2} and identify the quiver $Q(\sigma,\omega)$ with the subquiver of $\widetilde{\mu}_k(Q(\tau,\omega))$ obtained by deleting the dotted arrows in the figure.

\begin{figure}[!ht]
                \centering
                \includegraphics[scale=.65]{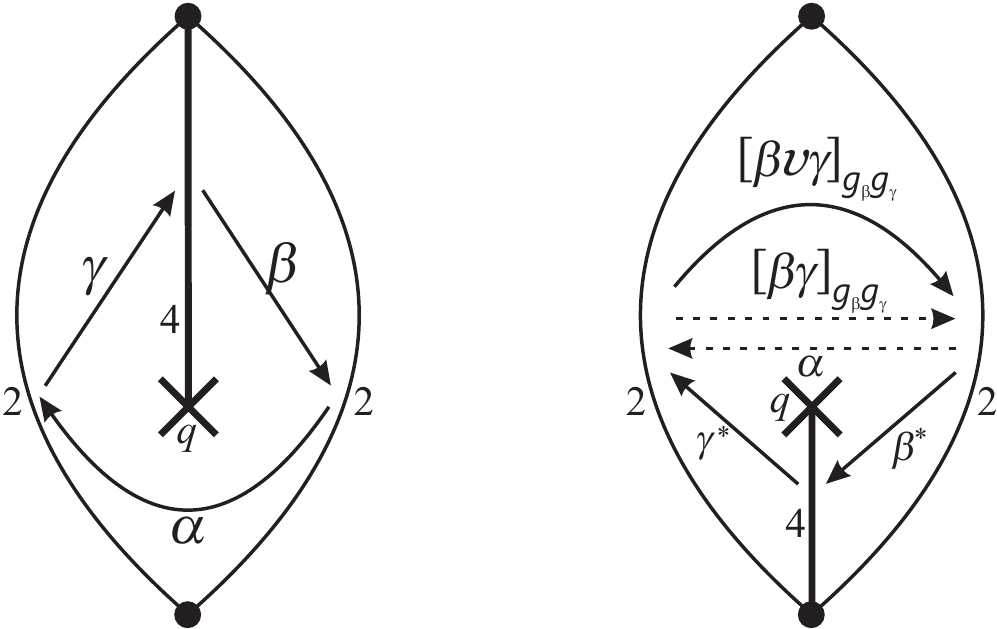}\caption{{\footnotesize Second configuration of Figure \ref{Fig:all_possibilities_for_k_and_weights}. Left:  $\tau$ and $Q(\tau,\omega)$. Right: $\sigma$ and $\widetilde{\mu}_k(Q(\tau,\omega))$.  The numbers next to the arcs are the corresponding values of the tuple $\dtuple(\tau,\omega)$.}}
                \label{Fig:flip_mut_2}
        \end{figure}

\item
The SP $(A(\sigma,\zeta),S(\sigma,\zeta))$ is the reduced part of $(\widetilde{\mu}_k(A(\tau,\xi)),S(\sigma,\zeta)^\sharp)$, where
\begin{eqnarray}\nonumber
\Ssigmad^\sharp  & = &
X\left(\alpha[\beta \gamma]_{g_\beta g_\gamma}
+\gamma^*\beta^*[\beta v\gamma]_{g_\beta g_\gamma}\right)+S(\tau,\sigma)
\in\RA{\widetilde{\mu}_k(A(\tau,\xi))},
\end{eqnarray}
with $S(\tau,\sigma)\in\RA{A(\tau,\xi)}\cap\RA{A(\sigma,\zeta)}$.
Here, $X$ is an element of $\{0,1\}\subseteq F$ with the property of being equal to $1\in F$ (resp. $0\in F$) if and only if the two sides of the digon in Figure \ref{Fig:flip_mut_2} are indeed arcs in $\tau$ (resp. at least one of the two sides of the digon is a boundary segment).

\item The potentials $\Stauc$ and $\widetilde{\mu}_k(\Stauc)$ are
\begin{eqnarray*}
\Stauc & = &
X\alpha\beta\gamma
+S(\tau,\sigma) \ \ \ \ \ \text{and}\\
\widetilde{\mu}_k(\Stauc)
&\sim_{\operatorname{cyc}}&
X\left(\alpha[\beta\gamma]_{g_\beta g_\gamma}
+\frac{1}{2}\left(\pi_{(g_\beta g_\gamma)^{-1}}(\gamma^*\beta^*)[\beta\gamma]_{g_\beta g_\gamma}
+\gamma^*v^{-1}\beta^*[\beta v\gamma]_{g_\beta g_\gamma} \right)\right)
+S(\tau,\sigma).
\end{eqnarray*}

\item
Since $\xi$ is a $1$-cocycle, from the definition of the modulating function $g$ it follows that $(g_\beta g_\gamma)^{-1}=g_\alpha$, and hence the rule
\begin{eqnarray*}
\varphi_1 &:& X\alpha\mapsto X\left(\alpha-\frac{1}{2}\pi_{(g_\beta g_\gamma)^{-1}}(\gamma^*\beta^*)\right)
\end{eqnarray*}
yields a well-defined $R$-algebra automorphism $\varphi:\RA{\widetilde{\mu}_k(A(\tau,\xi))}\rightarrow\RA{\widetilde{\mu}_k(A(\tau,\xi))}$ (see \cite[Propositions 2.15-(6) and 3.7]{Geuenich-Labardini-1}). Letting $\varphi_2:\RA{\widetilde{\mu}_k(A(\tau,\xi))}\rightarrow\RA{\widetilde{\mu}_k(A(\tau,\xi))}$ defined by the rule
\begin{eqnarray*}
\varphi_2 &:& X\beta^*\mapsto X2v\beta^*,
\end{eqnarray*}
it is obvious that $\varphi_2\varphi_1$ is a right-equivalence $(\widetilde{\mu}_k(A(\tau,\xi)),\widetilde{\mu}_k(S(\tau,\xi)))\rightarrow(\widetilde{\mu}_k(A(\tau,\xi)),S(\sigma,\zeta)^\sharp)$.
\end{enumerate}
\end{case}


\begin{case}\label{case:3}\emph{Configurations 3 and 7}. \begin{figure}[!ht]
                \centering
                \includegraphics[scale=.65]{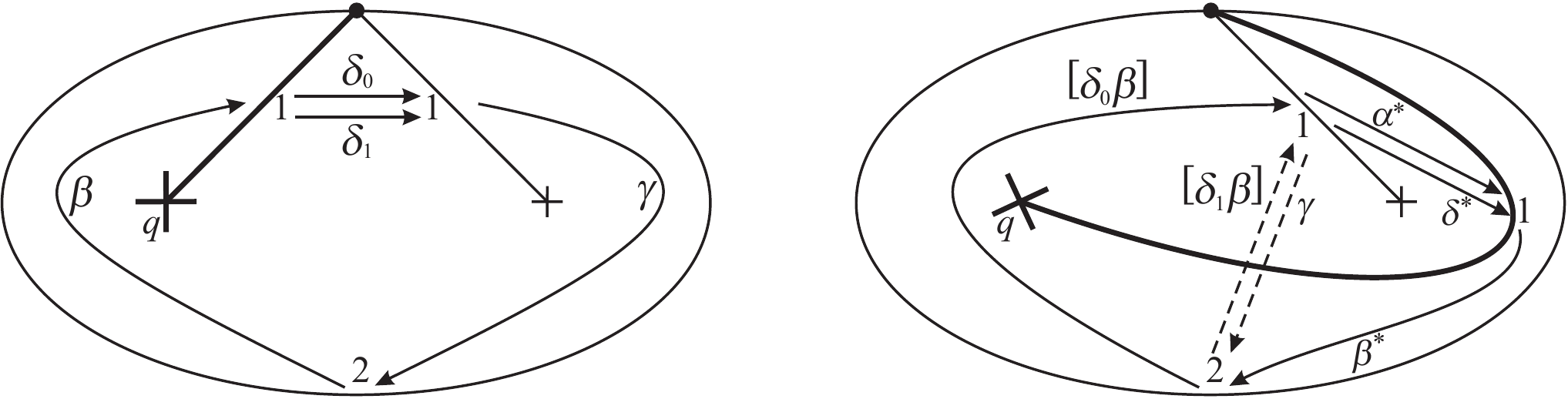}
                \caption{{\footnotesize Configurations 3 and 7 of Figure \ref{Fig:all_possibilities_for_k_and_weights}. Left: $\tau$ and $Q(\tau,\omega)$. Right: $\sigma$ and $\widetilde{\mu}_k(Q(\tau,\omega))$.  The numbers next to the arcs are the corresponding values of the tuple $\dtuple(\tau,\omega)$.}}
                \label{Fig:flip_mut_3}
        \end{figure}
\begin{enumerate}\item We use the notation in Figure \ref{Fig:flip_mut_3} and identify the quiver $Q(\sigma,\omega)$ with the subquiver of $\widetilde{\mu}_k(Q(\tau,\omega))$ obtained by deleting the dotted arrows in the quiver.

\item The SP $(A(\sigma,\zeta),S(\sigma,\zeta))$ is the reduced part of $(\widetilde{\mu}_k(A(\tau,\xi)),S(\sigma,\zeta)^\sharp)$, where
\begin{eqnarray}\nonumber
\Ssigmad^\sharp  & = &
\gamma[\delta_1\beta]+u\beta^*\delta_0^*[\delta_0\beta]+\beta^*\delta_1^*[\delta_0\beta]
+S(\tau,\sigma)
\in\RA{\widetilde{\mu}_k(A(\tau,\xi))},
\end{eqnarray}
with $S(\tau,\sigma)\in\RA{A(\tau,\xi)}\cap\RA{A(\sigma,\zeta)}$.

\item The potentials $\Stauc$ and $\widetilde{\mu}_k(\Stauc)$ are
\begin{eqnarray*}
\Stauc & = &
\gamma\delta_0\beta+u\gamma\delta_1\beta
+S(\tau,\sigma) \ \ \ \ \ \text{and}\\
\widetilde{\mu}_k(\Stauc) & = &
\gamma[\delta_0\beta]+u\gamma[\delta_1\beta]+
\beta^*\delta_0^*[\delta_0\beta]+\beta^*\delta_1^*[\delta_1\beta]
+S(\tau,\sigma).
\end{eqnarray*}

\item Define $R$-algebra automorphisms $\psi,\varphi,\Phi:\RA{\widetilde{\mu}_k(A(\tau,\xi))}\rightarrow\RA{\widetilde{\mu}_k(A(\tau,\xi))}$ according to the rules
\begin{eqnarray*}
\psi &:& [\delta_1\beta]\mapsto[\delta_1\beta]u^{-1}-[\delta_0\beta]u^{-1}\\
\varphi &:& \gamma\mapsto\gamma-u^{-1}\beta^*\delta_1^*\\
\Phi &:& \delta_1^*\mapsto -\delta_1^*, \ \ \ \ \ [\delta_0\beta]\mapsto[\delta_0\beta]u
\end{eqnarray*}
Direct computation shows that $\Phi\varphi\psi$ is a right-equivalence $(\widetilde{\mu}_k(A(\tau,\xi)),\widetilde{\mu}_k(S(\tau,\xi)))\rightarrow(\widetilde{\mu}_k(A(\tau,\xi)),S(\sigma,\zeta)^\sharp)$.
\end{enumerate}
\end{case}



\begin{case}\label{case:4}\emph{Configuration 4 and 9}.
\begin{figure}[!ht]
                \centering
                \includegraphics[scale=.65]{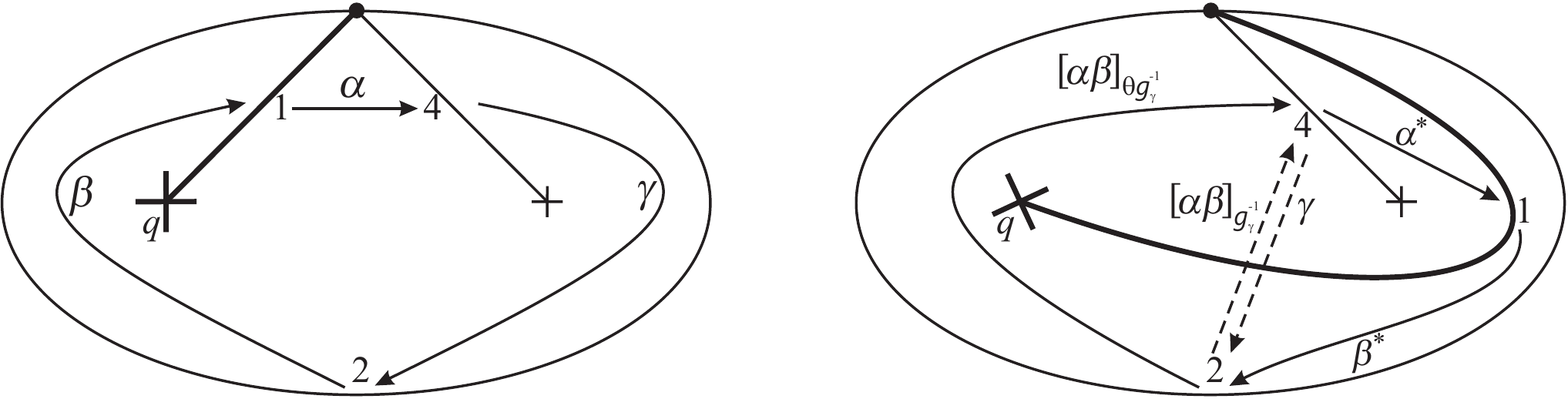}
                \caption{{\footnotesize Configurations 4 and 9 of Figure \ref{Fig:all_possibilities_for_k_and_weights}. Left: $\tau$ and $Q(\tau,\omega)$. Right: $\sigma$ and $\widetilde{\mu}_k(Q(\tau,\omega))$.  The numbers next to the arcs are the corresponding values of the tuple $\dtuple(\tau,\omega)$.}}
                \label{Fig:flip_mut_4}
        \end{figure}
\begin{enumerate}\item We use the notation in Figure \ref{Fig:flip_mut_4} and identify the quiver $Q(\sigma,\omega)$ with the subquiver of $\widetilde{\mu}_k(Q(\tau,\omega))$ obtained by deleting the dotted arrows in the figure.

\item The SP $(A(\sigma,\zeta),S(\sigma,\zeta))$ is the reduced part of $(\widetilde{\mu}_k(A(\tau,\xi)),S(\sigma,\zeta)^\sharp)$, where
\begin{eqnarray}\nonumber
\Ssigmad^\sharp  & = &
\gamma[\alpha\beta]_{g_\gamma^{-1}}
+\beta^*\alpha^*[\alpha\beta]_{\theta g_\gamma^{-1}}+S(\tau,\sigma)
\in\RA{\widetilde{\mu}_k(A(\tau,\xi))},
\end{eqnarray}
with $S(\tau,\sigma)\in\RA{A(\tau,\xi)}\cap\RA{A(\sigma,\zeta)}$.

\item The potentials $\Stauc$ and $\widetilde{\mu}_k(\Stauc)$ are
\begin{eqnarray*}
\Stauc & = &
\gamma\alpha\beta+S(\tau,\sigma) \ \ \ \ \ \text{and}\\
\widetilde{\mu}_k(\Stauc)
&\sim_{\operatorname{cyc}}&
\gamma[\alpha\beta]_{g_\gamma^{-1}}
+\pi_{g_\gamma}(\beta^*\alpha^*)[\alpha\beta]_{g_\gamma^{-1}}+\beta^*\alpha^*[\alpha\beta]_{\theta g_\gamma^{-1}}
+S(\tau,\sigma).
\end{eqnarray*}

\item
Define an $R$-algebra automorphism $\varphi:\RA{\widetilde{\mu}_k(A(\tau,\xi))}\rightarrow\RA{\widetilde{\mu}_k(A(\tau,\xi))}$ according to the rule
\begin{eqnarray*}
\varphi &:& \gamma\mapsto \gamma-\pi_{g_\gamma}(\beta^*\alpha^*)
\end{eqnarray*}
(that this rule indeed yields a well-defined $R$-algebra automorphism is a consequence of \cite[Propositions 2.15-(6) and 3.7]{Geuenich-Labardini-1}).
It is obvious that $\varphi$ is a right-equivalence $(\widetilde{\mu}_k(A(\tau,\xi)),\widetilde{\mu}_k(S(\tau,\xi)))\rightarrow(\widetilde{\mu}_k(A(\tau,\xi)),S(\sigma,\zeta)^\sharp)$.
\end{enumerate}
\end{case}



\begin{case}\label{case:5}\emph{Configurations 5 and 8}.
\begin{figure}[!ht]
                \centering
                \includegraphics[scale=.65]{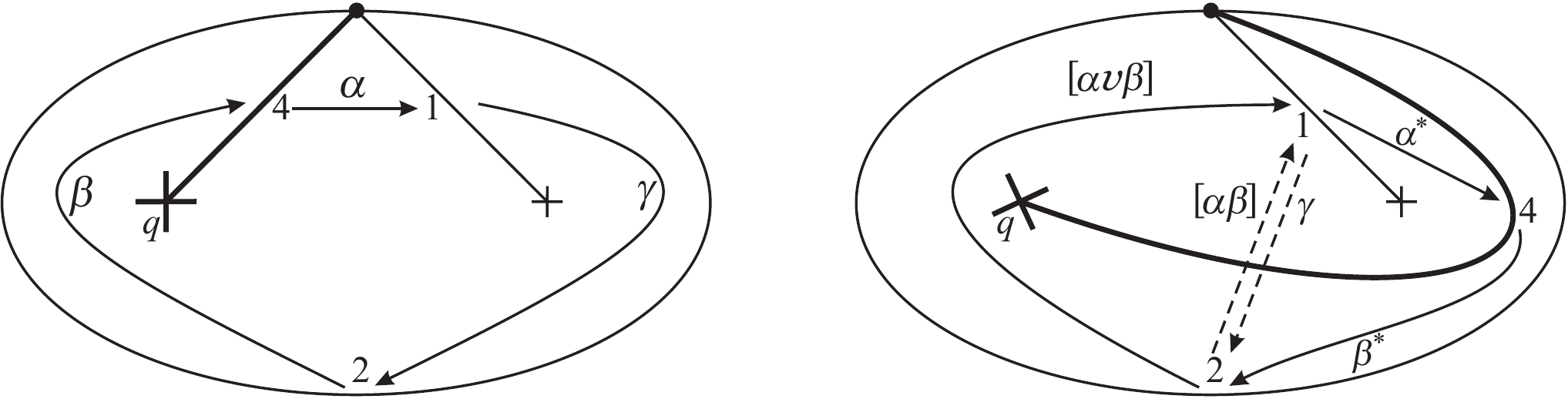}
                \caption{{\footnotesize Configurations 5 and 8 of Figure \ref{Fig:all_possibilities_for_k_and_weights}. Left: $\tau$ and $Q(\tau,\omega)$. Right: $\sigma$ and $\widetilde{\mu}_k(Q(\tau,\omega))$.  The numbers next to the arcs are the corresponding values of the tuple $\dtuple(\tau,\omega)$.}}
                \label{Fig:flip_mut_5}
        \end{figure}
\begin{enumerate}\item We use the notation in Figure \ref{Fig:flip_mut_5} and identify the quiver $Q(\sigma,\omega)$ with the subquiver of $\widetilde{\mu}_k(Q(\tau,\omega))$ obtained by deleting the dotted arrows in the figure.

\item
The SP $(A(\sigma,\zeta),S(\sigma,\zeta))$ is the reduced part of $(\widetilde{\mu}_k(A(\tau,\xi)),S(\sigma,\zeta)^\sharp)$, where
\begin{eqnarray}\nonumber
\Ssigmad^\sharp  & = &
\gamma[\alpha\beta]
+\beta^*\alpha^*[\alpha v\beta]
+S(\tau,\sigma)
\in\RA{\widetilde{\mu}_k(A(\tau,\xi))},
\end{eqnarray}
with $S(\tau,\sigma)\in\RA{A(\tau,\xi)}\cap\RA{A(\sigma,\zeta)}$.

\item The potentials $\Stauc$ and $\widetilde{\mu}_k(\Stauc)$ are
\begin{eqnarray*}
\Stauc & = &
\gamma\alpha\beta
+S(\tau,\sigma) \ \ \ \ \ \text{and}\\
\widetilde{\mu}_k(\Stauc) & = &
\gamma[\alpha\beta]+\frac{1}{2}\left(\beta^*\alpha^*[\alpha\beta]+\beta^* v^{-1}\alpha^*[\alpha v\beta]\right)
+S(\tau,\sigma).
\end{eqnarray*}

\item
Define $R$-algebra automorphisms $\varphi,\Phi:\RA{\widetilde{\mu}_k(A(\tau,\xi))}\rightarrow\RA{\widetilde{\mu}_k(A(\tau,\xi))}$ according to the rules
\begin{center}
\begin{tabular}{cclcccl}
$\varphi$ &:& $\gamma\mapsto\gamma-\frac{1}{2}\beta^*\alpha^*$, & &
$\Phi$ &:& $\alpha^*\mapsto 2v\alpha^*$.
\end{tabular}
\end{center}
An easy computation shows that $\Phi\varphi$ is a right-equivalence $(\widetilde{\mu}_k(A(\tau,\xi)),\widetilde{\mu}_k(S(\tau,\xi)))\rightarrow(\widetilde{\mu}_k(A(\tau,\xi)),S(\sigma,\zeta)^\sharp)$.
\end{enumerate}
\end{case}



\begin{case}\label{case:6}\emph{Configurations 6 and 10}.
\begin{figure}[!ht]
                \centering
                \includegraphics[scale=.65]{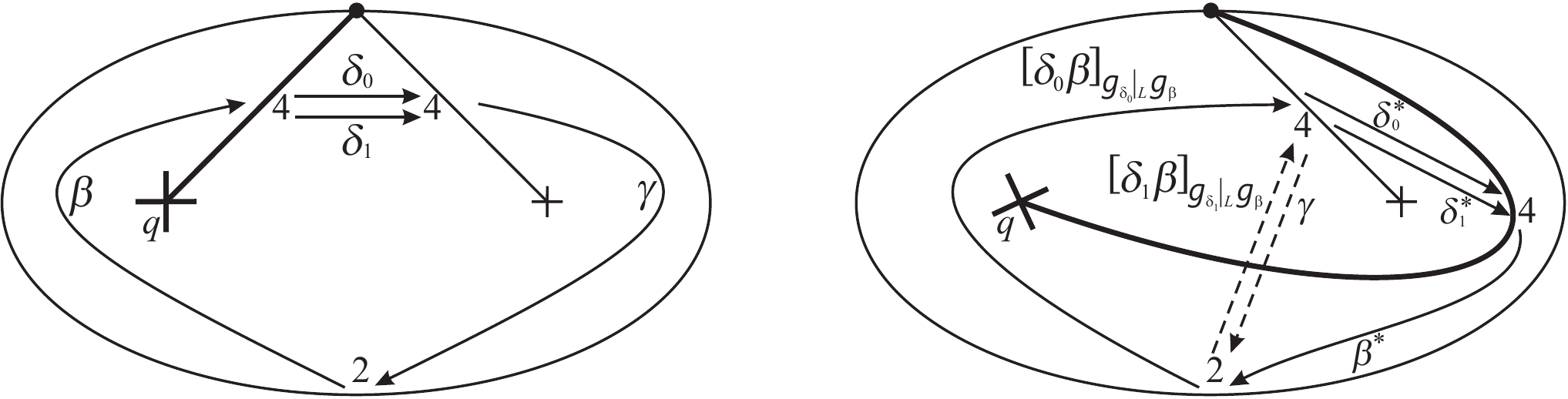}
                \caption{{\footnotesize Configurations 6 and 10 of Figure \ref{Fig:all_possibilities_for_k_and_weights}. Left: $\tau$ and $Q(\tau,\omega)$. Right: $\sigma$ and $\widetilde{\mu}_k(Q(\tau,\omega))$.  The numbers next to the arcs are the corresponding values of the tuple $\dtuple(\tau,\omega)$.}}
                \label{Fig:flip_mut_6}
        \end{figure}
\begin{enumerate}\item We use the notation in Figure \ref{Fig:flip_mut_6} and identify the quiver $Q(\sigma,\omega)$ with the subquiver of $\widetilde{\mu}_k(Q(\tau,\omega))$ obtained by deleting the dotted arrows in the figure.

\item
The SP $(A(\sigma,\zeta),S(\sigma,\zeta))$ is the reduced part of $(\widetilde{\mu}_k(A(\tau,\xi)),S(\sigma,\zeta)^\sharp)$, where
\begin{eqnarray}\nonumber
\Ssigmad^\sharp  & = &
\gamma[\delta_1\beta]_{g_{\delta_1}|_{L}g_\beta}
+\beta^*(\delta_0^*+\delta_1^*)[\delta_0\beta]_{g_{\delta_0}|_{L}g_\beta}
+S(\tau,\sigma)
\in\RA{\widetilde{\mu}_k(A(\tau,\xi))},
\end{eqnarray}
with $S(\tau,\sigma)\in\RA{A(\tau,\xi)}\cap\RA{A(\sigma,\zeta)}$.

\item The potentials $\Stauc$ and $\widetilde{\mu}_k(\Stauc)$ are
\begin{eqnarray*}
\Stauc & = &
(\delta_0+\delta_1)\beta\gamma
+S(\tau,\sigma) \ \ \ \ \ \text{and}\\
\widetilde{\mu}_k(\Stauc)
&\sim_{\operatorname{cyc}}&
\gamma[\delta_0\beta]_{g_{\delta_0}|_{L}g_\beta}
+\gamma[\delta_1\beta]_{g_{\delta_1}|_{L}g_\beta}\\
&&
+\beta^*\delta_0^*[\delta_0\beta]_{g_{\delta_0}|_{L}g_\beta}
+\pi_{(g_{\delta_1}|_{L}g_\beta)^{-1}}(\beta^*\delta_1^*)[\delta_1\beta]_{g_{\delta_1}|_{L}g_\beta}
+S(\tau,\sigma)
.
\end{eqnarray*}

\item
Since $\xi$ is a $1$-cocycle, from the definition of the modulating function $g$ it follows that $g_{\delta_0}|_{L}g_\beta=g_{\delta_1}|_{L}g_\beta=g_\gamma^{-1}$, from which we deduce that the rules
\begin{eqnarray*}
\psi &:&
[\delta_1\beta]_{g_{\delta_1}|_{L}g_\beta}
\mapsto
[\delta_1\beta]_{g_{\delta_1}|_{L}g_\beta}-[\delta_0\beta]_{g_{\delta_0}|_{L}g_\beta|_{L}}\\
\varphi &:&
\gamma
\mapsto
\gamma-\pi_{(g_{\delta_1}|_{L}g_\beta)^{-1}}(\beta^*\delta_1^*)
\end{eqnarray*}
produce well-defined $R$-algebra automorphisms $\psi,\varphi:\RA{\widetilde{\mu}_k(A(\tau,\xi))}\rightarrow\RA{\widetilde{\mu}_k(A(\tau,\xi))}$ (see \cite[Propositions 2.15-(6) and 3.7]{Geuenich-Labardini-1}). Letting $\Phi:\RA{\widetilde{\mu}_k(A(\tau,\xi))}\rightarrow\RA{\widetilde{\mu}_k(A(\tau,\xi))}$ be the $R$-algebra automorphism defined by the rule
\begin{eqnarray*}
\Phi &:& \delta_1^*\mapsto -\delta_1^*,
\end{eqnarray*}
direct computation shows that $\Phi\varphi\psi$ is a right-equivalence $(\widetilde{\mu}_k(A(\tau,\xi)),\widetilde{\mu}_k(S(\tau,\xi)))\rightarrow(\widetilde{\mu}_k(A(\tau,\xi)),S(\sigma,\zeta)^\sharp)$.
\end{enumerate}
\end{case}



\begin{case}\label{case:11}\emph{Configuration 11}.
\begin{figure}[!ht]
                \centering
                \includegraphics[scale=.75]{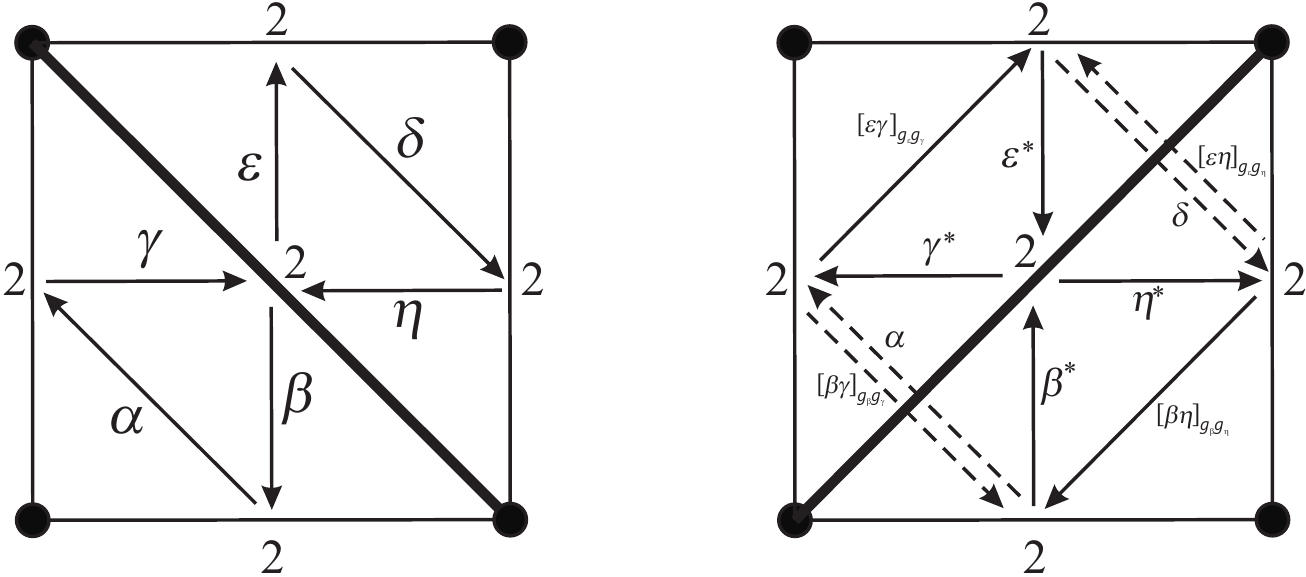}
                \caption{{\footnotesize Configuration 11 of Figure \ref{Fig:all_possibilities_for_k_and_weights}. Left: $\tau$ and $Q(\tau,\omega)$. Right: $\sigma$ and $\widetilde{\mu}_k(Q(\tau,\omega))$.  The numbers next to the arcs are the corresponding values of the tuple $\dtuple(\tau,\omega)$.}}
                \label{Fig:flip_mut_11}
        \end{figure}
\begin{enumerate}\item We use the notation in Figure \ref{Fig:flip_mut_11} and identify the quiver $Q(\sigma,\omega)$ with the subquiver of $\widetilde{\mu}_k(Q(\tau,\omega))$ obtained by deleting the dotted arrows in the figure.

\item
The SP $(A(\sigma,\zeta),S(\sigma,\zeta))$ is the reduced part of $(\widetilde{\mu}_k(A(\tau,\xi)),S(\sigma,\zeta)^\sharp)$, where
\begin{eqnarray}\nonumber
\Ssigmad^\sharp  & = &
\alpha[\beta\gamma]_{g_{\beta}g_\gamma}
+ \delta[\varepsilon\eta]_{g_{\varepsilon}g_\eta}
+\gamma^*\varepsilon^*[\varepsilon\gamma]_{g_{\varepsilon}g_{\gamma}}
+\eta^*\beta^*[\beta\eta]_{g_{\beta}g_{\eta}}
+S(\tau,\sigma)
\in\RA{\widetilde{\mu}_k(A(\tau,\xi))},
\end{eqnarray}
with $S(\tau,\sigma)\in\RA{A(\tau,\xi)}\cap\RA{A(\sigma,\zeta)}$.

\item The potentials $\Stauc$ and $\widetilde{\mu}_k(\Stauc)$ are
\begin{eqnarray*}
\Stauc & = &
\alpha\beta\gamma+\delta\varepsilon\eta
+S(\tau,\sigma) \ \ \ \ \ \text{and}\\
\widetilde{\mu}_k(\Stauc)
&\sim_{\operatorname{cyc}}&
\alpha[\beta\gamma]_{g_{\beta}g_\gamma}
+\delta[\varepsilon\eta]_{g_{\varepsilon}g_\eta}\\
&&
+\pi_{(g_{\beta}g_\gamma)^{-1}}(\gamma^*\beta^*)[\beta\gamma]_{_{\beta}g_\gamma}
+\pi_{(g_{\varepsilon}g_\eta)^{-1}}(\eta^*\varepsilon^*)[\varepsilon\eta]_{g_{\varepsilon}g_\eta}\\
&&
+\gamma^*\varepsilon^*[\varepsilon\gamma]_{g_{\varepsilon}g_{\gamma}}
+\eta^*\beta^*[\beta\eta]_{g_{\beta}g_{\eta}}+S(\tau,\sigma)
.
\end{eqnarray*}

\item
Since $\xi$ is a $1$-cocycle, from the definition of the modulating function $g$ it follows that $g_{\beta}g_\gamma=g_\alpha^{-1}$ and $g_{\varepsilon}g_\eta=g_\delta^{-1}$, from which we deduce that the rule
\begin{center}
\begin{tabular}{ccll}
$\varphi$ &:&
$\alpha\mapsto\alpha-\pi_{(g_{\beta}g_\gamma)^{-1}}(\gamma^*\beta^*)$,
&
$\delta\mapsto\delta-\pi_{(g_{\varepsilon}g_\eta)^{-1}}(\eta^*\varepsilon^*)$,
\end{tabular}
\end{center}
produces a well-defined $R$-algebra automorphisms $\varphi:\RA{\widetilde{\mu}_k(A(\tau,\xi))}\rightarrow\RA{\widetilde{\mu}_k(A(\tau,\xi))}$ (see \cite[Propositions 2.15-(6) and 3.7]{Geuenich-Labardini-1}). It is obvious that $\varphi$ is a right-equivalence $(\widetilde{\mu}_k(A(\tau,\xi)),\widetilde{\mu}_k(S(\tau,\xi)))\rightarrow(\widetilde{\mu}_k(A(\tau,\xi)),S(\sigma,\zeta)^\sharp)$.
\end{enumerate}
\end{case}



\begin{case}\label{case:12}\emph{Configurations 12 and 14}.
\begin{figure}[!ht]
                \centering
                \includegraphics[scale=.5]{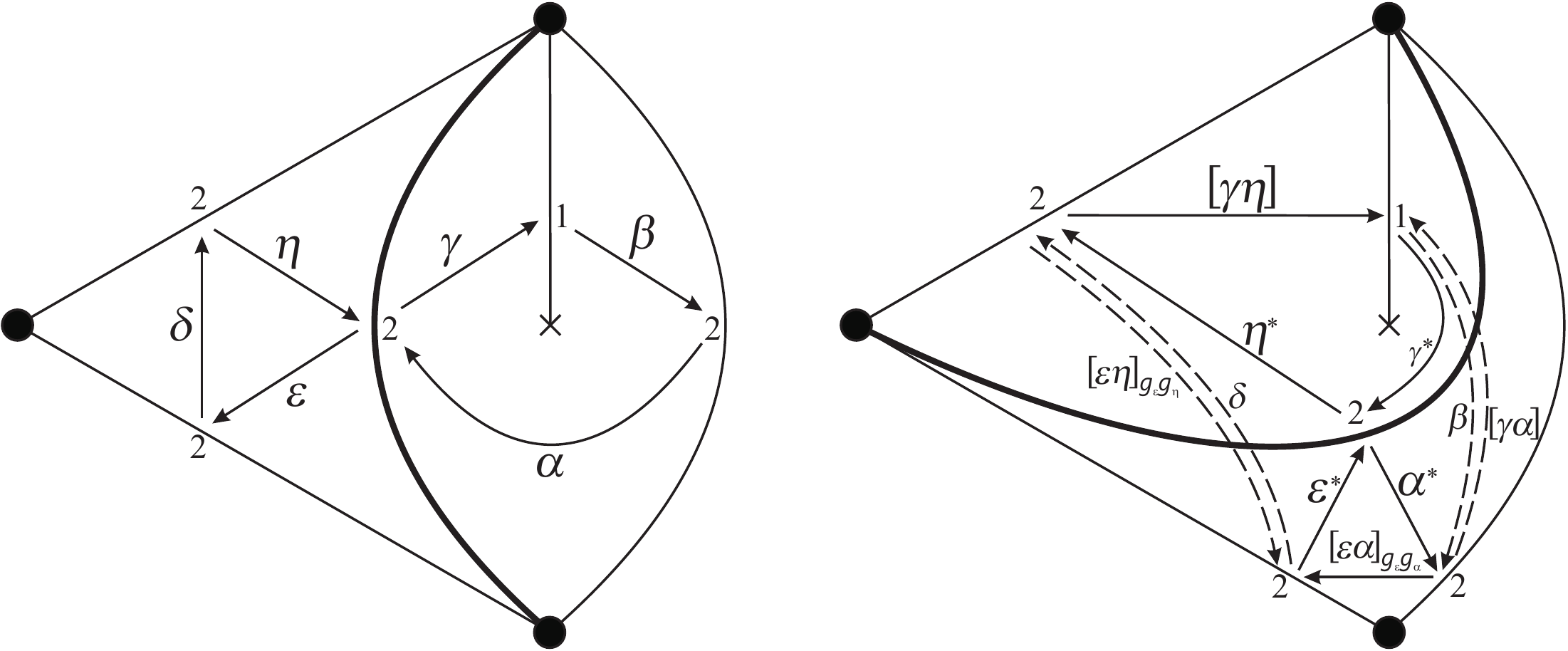}
                \caption{{\footnotesize Configurations 12 and 14 of Figure \ref{Fig:all_possibilities_for_k_and_weights}. Left: $\tau$ and $Q(\tau,\omega)$. Right: $\sigma$ and $\widetilde{\mu}_k(Q(\tau,\omega))$.  The numbers next to the arcs are the corresponding values of the tuple $\dtuple(\tau,\omega)$.}}
                \label{Fig:flip_mut_12}
        \end{figure}
\begin{enumerate}\item
We use the notation in Figure \ref{Fig:flip_mut_12} and identify the quiver $Q(\sigma,\omega)$ with the subquiver of $\widetilde{\mu}_k(Q(\tau,\omega))$ obtained by deleting the dotted arrows in the figure.

\item
The SP $(A(\sigma,\zeta),S(\sigma,\zeta))$ is the reduced part of $(\widetilde{\mu}_k(A(\tau,\xi)),S(\sigma,\zeta)^\sharp)$, where
\begin{eqnarray}\nonumber
\Ssigmad^\sharp  & = &
\beta[\gamma\alpha]
+\delta[\varepsilon\eta]_{g_{\varepsilon}g_\eta}
+\alpha^*\varepsilon^*[\varepsilon\alpha]_{g_{\varepsilon}g_{\alpha}}
+\eta^*\gamma^*[\gamma\eta]
+S(\tau,\sigma)
\in\RA{\widetilde{\mu}_k(A(\tau,\xi))},
\end{eqnarray}
with $S(\tau,\sigma)\in\RA{A(\tau,\xi)}\cap\RA{A(\sigma,\zeta)}$.

\item
The potentials $\Stauc$ and $\widetilde{\mu}_k(\Stauc)$ are
\begin{eqnarray*}
\Stauc & = &
\alpha\beta\gamma+\varepsilon\eta\delta
+S(\tau,\sigma) \ \ \ \ \ \text{and}\\
\widetilde{\mu}_k(\Stauc)
&\sim_{\operatorname{cyc}}&
\beta[\gamma\alpha]
+\delta[\varepsilon\eta]_{g_{\varepsilon}g_\eta}\\
&&
+\alpha^*\gamma^*[\gamma\alpha]
+\pi_{(g_{\varepsilon}g_\eta)^{-1}}(\eta^*\varepsilon^*)[\varepsilon\eta]_{g_{\varepsilon}g_\eta}
+\alpha^*\varepsilon^*[\varepsilon\alpha]_{g_{\varepsilon}g_{\alpha}}
+\eta^*\gamma^*[\gamma\eta]
+S(\tau,\sigma)
.
\end{eqnarray*}

\item
Since $\xi$ is a $1$-cocycle, from the definition of the modulating function $g$ it follows that $g_{\varepsilon}g_\eta=g_\delta^{-1}$, from which we deduce that the rule
\begin{center}
\begin{tabular}{ccll}
$\varphi$ &:&
$\beta\mapsto\beta-\alpha^*\gamma^*$,
&
$\delta\mapsto\delta-\pi_{(g_{\varepsilon}g_\eta)^{-1}}(\eta^*\varepsilon^*)$,
\end{tabular}
\end{center}
produces a well-defined $R$-algebra automorphism $\varphi:\RA{\widetilde{\mu}_k(A(\tau,\xi))}\rightarrow\RA{\widetilde{\mu}_k(A(\tau,\xi))}$ (see \cite[Propositions 2.15-(6) and 3.7]{Geuenich-Labardini-1}). It is obvious that $\varphi$ is a right-equivalence $(\widetilde{\mu}_k(A(\tau,\xi)),\widetilde{\mu}_k(S(\tau,\xi)))\rightarrow(\widetilde{\mu}_k(A(\tau,\xi)),S(\sigma,\zeta)^\sharp)$.
\end{enumerate}
\end{case}



\begin{case}\label{case:13}\emph{Configurations 13 and 15}.
\begin{figure}[!ht]
                \centering
                \includegraphics[scale=.5]{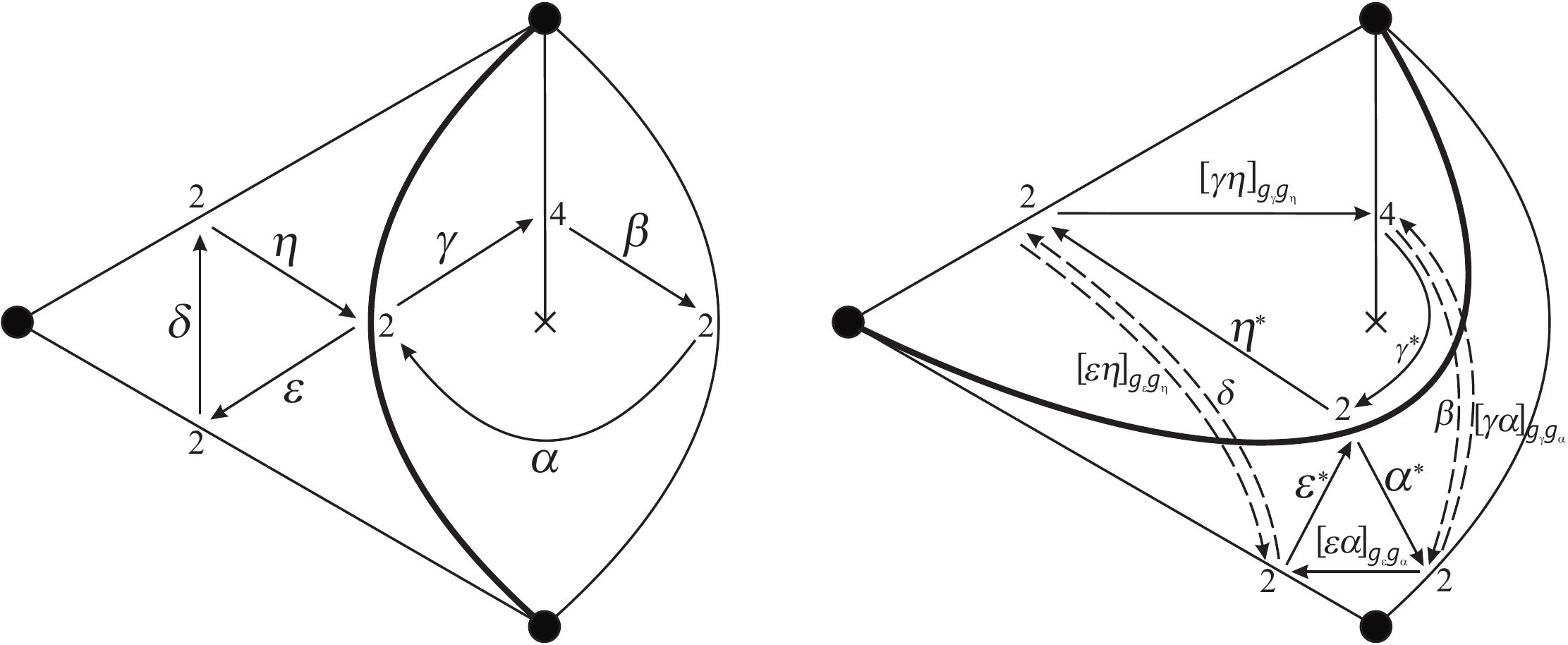}
                \caption{{\footnotesize Configurations 13 and 15 of Figure \ref{Fig:all_possibilities_for_k_and_weights}. Left: $\tau$ and $Q(\tau,\omega)$. Right: $\sigma$ and $\widetilde{\mu}_k(Q(\tau,\omega))$.  The numbers next to the arcs are the corresponding values of the tuple $\dtuple(\tau,\omega)$.}}
                \label{Fig:flip_mut_13}
        \end{figure}
\begin{enumerate}
\item We use the notation in Figure \ref{Fig:flip_mut_13} and identify the quiver $Q(\sigma,\omega)$ with the subquiver of $\widetilde{\mu}_k(Q(\tau,\omega))$ obtained by deleting the dotted arrows in the figure.

\item
The SP $(A(\sigma,\zeta),S(\sigma,\zeta))$ is the reduced part of $(\widetilde{\mu}_k(A(\tau,\xi)),S(\sigma,\zeta)^\sharp)$, where
\begin{eqnarray}\nonumber
\Ssigmad^\sharp  & = &
\delta[\varepsilon\eta]_{g_\varepsilon g_\eta}
+\beta[\gamma\alpha]_{g_\gamma g_\alpha}
+\eta^*\gamma^*[\gamma\eta]_{g_\gamma g_\eta}
+\alpha^*\varepsilon^*[\varepsilon\alpha]_{g_\varepsilon g_\alpha}
+S(\tau,\sigma)
\in\RA{\widetilde{\mu}_k(A(\tau,\xi))},
\end{eqnarray}
with $S(\tau,\sigma)\in\RA{A(\tau,\xi)}\cap\RA{A(\sigma,\zeta)}$.

\item
The potentials $\Stauc$ and $\widetilde{\mu}_k(\Stauc)$ are
\begin{eqnarray*}
\Stauc & = &
\delta\varepsilon\eta
+\beta\gamma\alpha
+S(\tau,\sigma) \ \ \ \ \ \text{and}\\
\widetilde{\mu}_k(\Stauc)
&\sim_{\operatorname{cyc}}&
\delta[\varepsilon\eta]_{g_\varepsilon g_\eta}
+\beta[\gamma\alpha]_{g_\gamma g_\alpha}\\
&&
+\pi_{(g_\varepsilon g_\eta)^{-1}}(\eta^*\varepsilon^*)[\varepsilon\eta]_{g_\varepsilon g_\eta}
+\pi_{(g_\gamma g_\alpha)^{-1}}(\alpha^*\gamma^*)[\gamma\alpha]_{g_\gamma g_\alpha}\\
&&
+\eta^*\gamma^*[\gamma\eta]_{g_\gamma g_\eta}
+\alpha^*\varepsilon^*[\varepsilon\alpha]_{g_\varepsilon g_\alpha}
+S(\tau,\sigma)
.
\end{eqnarray*}

\item
Since $\xi$ is a $1$-cocycle, from the definition of the modulating function $g$ it follows that $g_{\varepsilon}g_\eta=g_\delta^{-1}$ and $g_\gamma g_\alpha=g_\beta^{-1}$, from which we deduce that the rule
\begin{center}
\begin{tabular}{ccll}
$\varphi$ &:&
$\delta\mapsto\delta-\pi_{(g_\varepsilon g_\eta)^{-1}}(\eta^*\varepsilon^*)$,
&
$\beta\mapsto\beta-\pi_{(g_\gamma g_\alpha)^{-1}}(\alpha^*\gamma^*)$,
\end{tabular}
\end{center}
produces a well-defined $R$-algebra automorphism $\varphi:\RA{\widetilde{\mu}_k(A(\tau,\xi))}\rightarrow\RA{\widetilde{\mu}_k(A(\tau,\xi))}$ (see \cite[Propositions 2.15-(6) and 3.7]{Geuenich-Labardini-1}). It is obvious that $\varphi$ is a right-equivalence $(\widetilde{\mu}_k(A(\tau,\xi)),\widetilde{\mu}_k(S(\tau,\xi)))\rightarrow(\widetilde{\mu}_k(A(\tau,\xi)),S(\sigma,\zeta)^\sharp)$.
\end{enumerate}
\end{case}



\begin{case}\label{case:20}\emph{Configurations 20 and 16}.
\begin{figure}[!ht]
                \centering
                \includegraphics[scale=.75]{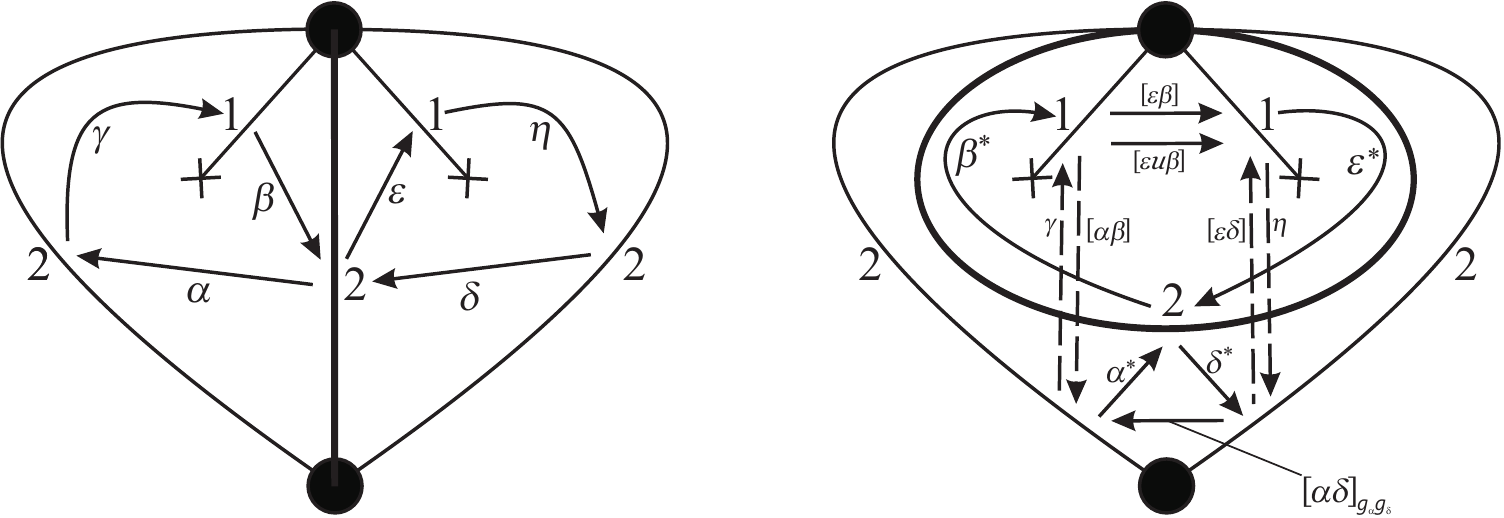}
                \caption{{\footnotesize Configurations 20 and 16 of Figure \ref{Fig:all_possibilities_for_k_and_weights}. Left: $\tau$ and $Q(\tau,\omega)$. Right: $\sigma$ and $\widetilde{\mu}_k(Q(\tau,\omega))$.  The numbers next to the arcs are the corresponding values of the tuple $\dtuple(\tau,\omega)$.}}
                \label{Fig:flip_mut_20}
        \end{figure}
\begin{enumerate}\item We use the notation in Figure \ref{Fig:flip_mut_20} and identify the quiver $Q(\sigma,\omega)$ with the subquiver of $\widetilde{\mu}_k(Q(\tau,\omega))$ obtained by deleting the dotted arrows in the figure.

\item
The SP $(A(\sigma,\zeta),S(\sigma,\zeta))$ is the reduced part of $(\widetilde{\mu}_k(A(\tau,\xi)),S(\sigma,\zeta)^\sharp)$, where
\begin{eqnarray}\nonumber
\Ssigmad^\sharp  & = &
\gamma[\alpha\beta]+\eta[\varepsilon\delta]
+\delta^*\alpha^*[\alpha\delta]_{g_\alpha g_\delta}
+[\varepsilon\beta]\beta^*u\varepsilon^*
+[\varepsilon u\beta]\beta^*\varepsilon^*
+S(\tau,\sigma)
\in\RA{\widetilde{\mu}_k(A(\tau,\xi))},
\end{eqnarray}
with $S(\tau,\sigma)\in\RA{A(\tau,\xi)}\cap\RA{A(\sigma,\zeta)}$.

\item The potentials $\Stauc$ and $\widetilde{\mu}_k(\Stauc)$ are
\begin{eqnarray*}
\Stauc & = &
\alpha\beta\gamma
+\delta\eta\varepsilon
+S(\tau,\sigma), \ \ \ \ \ \text{and}\\
\widetilde{\mu}_k(\Stauc) & = &
\gamma[\alpha\beta]+\eta[\varepsilon\delta]
+\beta^*\alpha^*[\alpha\beta]+\delta^*\varepsilon^*[\varepsilon\delta]
+\beta^*\varepsilon^*[\varepsilon\beta]+\beta^*u^{-1}\varepsilon^*[\varepsilon u\beta]
+\delta^*\alpha^*[\alpha\delta]_{g_{\alpha}g_{\delta}}
+S(\tau,\sigma).
\end{eqnarray*}

\item
From the definition of the modulating function $g$ it obviously follows that $g_{\alpha}g_\beta=g_\gamma^{-1}$, $g_\varepsilon g_\delta=g_\eta^{-1}$, $\pi_{(g_{\alpha}g_\beta)^{-1}}(\beta^*\alpha^*)=\beta^*\alpha^*$ and $\pi_{(g_\varepsilon g_\delta)^{-1}(\delta^*\varepsilon^*)}=\delta^*\varepsilon^*$, from which we deduce that the rule
\begin{center}
\begin{tabular}{ccll}
$\varphi$ &:&
$\gamma\mapsto\gamma-\beta^*\alpha^*$,&
$\eta\mapsto\eta-\delta^*\varepsilon^*$
\end{tabular}
\end{center}
produces a well-defined $R$-algebra automorphism $\varphi:\RA{\widetilde{\mu}_k(A(\tau,\xi))}\rightarrow\RA{\widetilde{\mu}_k(A(\tau,\xi))}$ (see \cite[Propositions 2.15-(6) and 3.7]{Geuenich-Labardini-1}). Letting $\Phi:\RA{\widetilde{\mu}_k(A(\tau,\xi))}\rightarrow\RA{\widetilde{\mu}_k(A(\tau,\xi))}$ be the $R$-algebra automorphism defined by the rule
\begin{eqnarray*}
\Phi &:& \varepsilon^*\mapsto u\varepsilon^*,
\end{eqnarray*}
direct computation shows that $\Phi\varphi$ is a right-equivalence $(\widetilde{\mu}_k(A(\tau,\xi)),\widetilde{\mu}_k(S(\tau,\xi)))\rightarrow(\widetilde{\mu}_k(A(\tau,\xi)),S(\sigma,\zeta)^\sharp)$.
\end{enumerate}
\end{case}



\begin{case}\label{case:21}\emph{Configurations 21 and 17}.
\begin{figure}[!ht]
                \centering
                \includegraphics[scale=.75]{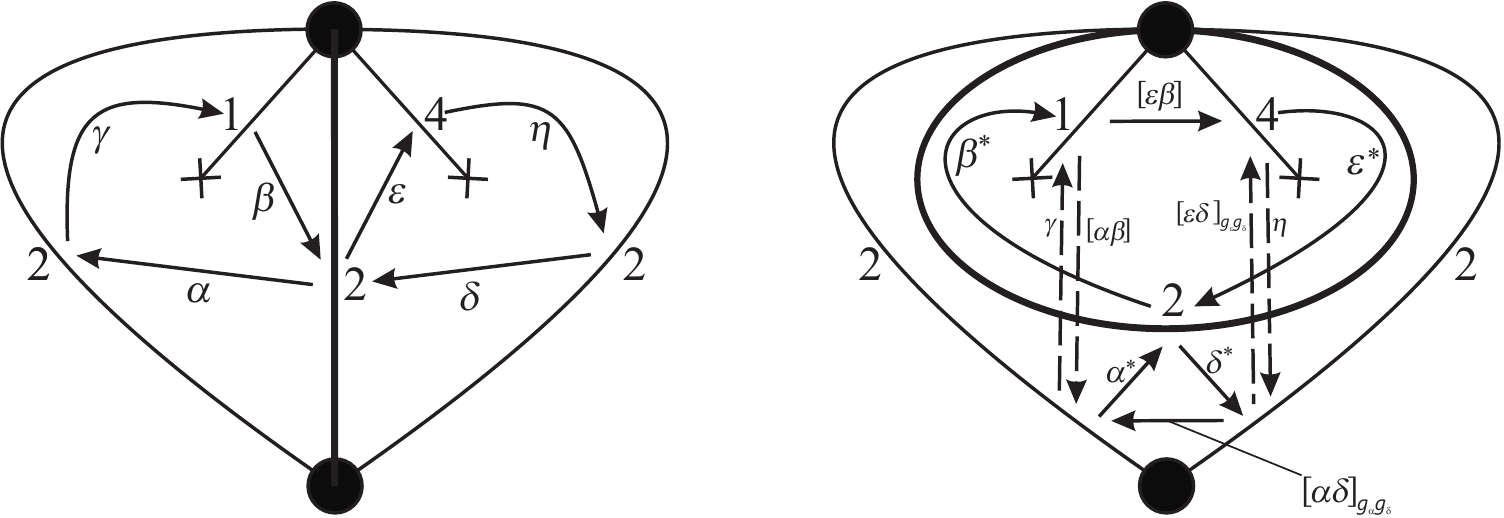}
                \caption{{\footnotesize Configurations 21 and 17 of Figure \ref{Fig:all_possibilities_for_k_and_weights}. Left: $\tau$ and $Q(\tau,\omega)$. Right: $\sigma$ and $\widetilde{\mu}_k(Q(\tau,\omega))$.  The numbers next to the arcs are the corresponding values of the tuple $\dtuple(\tau,\omega)$.}}
                \label{Fig:flip_mut_21}
        \end{figure}
\begin{enumerate}\item
 We use the notation in Figure \ref{Fig:flip_mut_21} and identify the quiver $Q(\sigma,\omega)$ with the subquiver of $\widetilde{\mu}_k(Q(\tau,\omega))$ obtained by deleting the dotted arrows in the figure.

\item
The SP $(A(\sigma,\zeta),S(\sigma,\zeta))$ is the reduced part of $(\widetilde{\mu}_k(A(\tau,\xi)),S(\sigma,\zeta)^\sharp)$, where
\begin{eqnarray}\nonumber
\Ssigmad^\sharp  & = &
\gamma[\alpha\beta]+\eta[\varepsilon\delta]_{g_\varepsilon g_\delta}
+\delta^*\alpha^*[\alpha\delta]_{g_\alpha g_\delta}
+[\varepsilon\beta]\beta^*\varepsilon^*
+S(\tau,\sigma)
\in\RA{\widetilde{\mu}_k(A(\tau,\xi))},
\end{eqnarray}
with $S(\tau,\sigma)\in\RA{A(\tau,\xi)}\cap\RA{A(\sigma,\zeta)}$.
Furthermore,
\begin{eqnarray*}
\Stauc & = &
\alpha\beta\gamma
+\delta\eta\varepsilon
+S(\tau,\sigma) \ \ \ \ \ \text{and}\\
\widetilde{\mu}_k(\Stauc)
&\sim_{\operatorname{cyc}}&
\gamma[\alpha\beta]
+\eta[\varepsilon\delta]_{g_\varepsilon g_{\delta}}
+\beta^*\alpha^*[\alpha\beta]
+\pi_{(g_{\varepsilon}g_{\delta})^{-1}}(\delta^*\varepsilon^*)[\varepsilon\delta]_{g_{\varepsilon}g_{\delta}}\\
&&
+\beta^*\varepsilon^*[\varepsilon\beta]
+\delta^*\alpha^*[\alpha\delta]_{g_\alpha g_\delta}
+S(\tau,\sigma).
\end{eqnarray*}

\item
From the equalities $g_\varepsilon g_\delta = g_{\eta}^{-1}$, $g_{\alpha}g_\beta=g_\gamma^{-1}$ and $\pi_{(g_{\alpha}g_\beta)^{-1}}(\beta^*\alpha^*)=\beta^*\alpha^*$, we deduce that the rule
\begin{center}
\begin{tabular}{ccll}
$\varphi$ &:&
$\gamma\mapsto\gamma-\beta^*\alpha^*$,
&
$\eta\mapsto\eta-\pi_{(g_{\varepsilon}g_{\delta})^{-1}}(\delta^*\varepsilon^*)$,
\end{tabular}
\end{center}
produces a well-defined $R$-algebra automorphism $\varphi:\RA{\widetilde{\mu}_k(A(\tau,\xi))}\rightarrow\RA{\widetilde{\mu}_k(A(\tau,\xi))}$ (see \cite[Propositions 2.15-(6) and 3.7]{Geuenich-Labardini-1}). It turns out to be obvious that $\varphi$ is a right-equivalence $(\widetilde{\mu}_k(A(\tau,\xi)),\widetilde{\mu}_k(S(\tau,\xi)))\rightarrow(\widetilde{\mu}_k(A(\tau,\xi)),S(\sigma,\zeta)^\sharp)$.
\end{enumerate}
\end{case}



\begin{case}\label{case:22}\emph{Configurations 22 and 18}.
\begin{figure}[!ht]
                \centering
                \includegraphics[scale=.75]{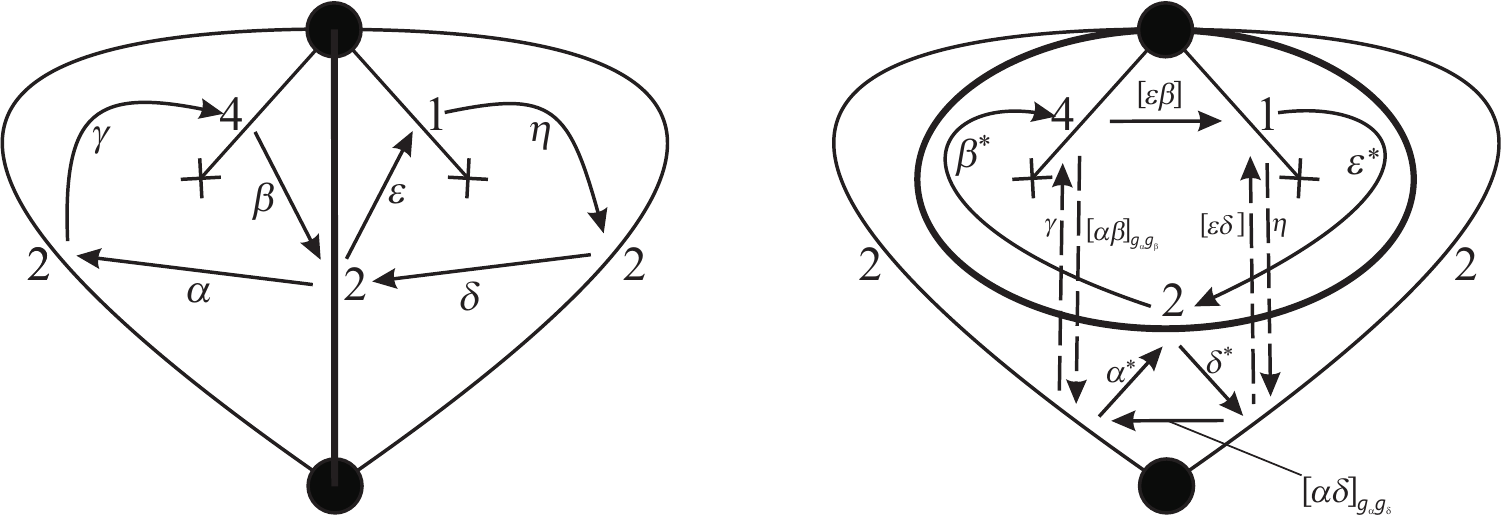}
                \caption{{\footnotesize Configurations 22 and 18 of Figure \ref{Fig:all_possibilities_for_k_and_weights}. Left: $\tau$ and $Q(\tau,\omega)$. Right: $\sigma$ and $\widetilde{\mu}_k(Q(\tau,\omega))$ appear on the right.  The numbers next to the arcs are the corresponding values of the tuple $\dtuple(\tau,\omega)$.}}
                \label{Fig:flip_mut_22}
        \end{figure}
\begin{enumerate}\item We use the notation in Figure \ref{Fig:flip_mut_22} and identify the quiver $Q(\sigma,\omega)$ with the subquiver of $\widetilde{\mu}_k(Q(\tau,\omega))$ obtained by deleting the dotted arrows in the figure.

\item
The SP $(A(\sigma,\zeta),S(\sigma,\zeta))$ is the reduced part of $(\widetilde{\mu}_k(A(\tau,\xi)),S(\sigma,\zeta)^\sharp)$, where
\begin{eqnarray}\nonumber
\Ssigmad^\sharp  & = &
\gamma[\alpha\beta]_{g_\alpha g_\beta}+\eta[\varepsilon\delta]
+\delta^*\alpha^*[\alpha\delta]_{g_\alpha g_\delta}
+[\varepsilon\beta]\beta^*\varepsilon^*
+S(\tau,\sigma)
\in\RA{\widetilde{\mu}_k(A(\tau,\xi))},
\end{eqnarray}
with $S(\tau,\sigma)\in\RA{A(\tau,\xi)}\cap\RA{A(\sigma,\zeta)}$.

\item
The potentials $\Stauc$ and $\widetilde{\mu}_k(\Stauc)$ are
\begin{eqnarray*}
\Stauc & = &
\alpha\beta\gamma
+\delta\eta\varepsilon
+S(\tau,\sigma) \ \ \ \ \ \text{and}\\
\widetilde{\mu}_k(\Stauc)
&\sim_{\operatorname{cyc}}&
\gamma[\alpha\beta]_{g_\alpha g_\beta}
+\eta[\varepsilon\delta]
+\pi_{(g_{\alpha}g_{\beta})^{-1}}(\beta^*\alpha^*)[\alpha\beta]_{g_\alpha g_\beta}
+\delta^*\varepsilon^*[\varepsilon\delta]\\
&&
+\beta^*\varepsilon^*[\varepsilon\beta]
+\delta^*\alpha^*[\alpha\delta]_{g_\alpha g_\delta}
+S(\tau,\sigma).
\end{eqnarray*}

\item
From the equalities $g_\alpha g_\beta = g_{\gamma}^{-1}$, $g_{\varepsilon}g_\delta=g_\eta^{-1}$ and $\pi_{(g_{\varepsilon}g_\delta)^{-1}}(\delta^*\varepsilon^*)=\delta^*\varepsilon^*$, we deduce that the rule
\begin{center}
\begin{tabular}{ccll}
$\varphi$ &:&
$\gamma\mapsto\gamma-\pi_{(g_{\alpha}g_{\beta})^{-1}}(\beta^*\alpha^*)$,
&
$\eta\mapsto\eta-\delta^*\varepsilon^*$,
\end{tabular}
\end{center}
produces a well-defined $R$-algebra automorphism $\varphi:\RA{\widetilde{\mu}_k(A(\tau,\xi))}\rightarrow\RA{\widetilde{\mu}_k(A(\tau,\xi))}$ (see \cite[Propositions 2.15-(6) and 3.7]{Geuenich-Labardini-1}). It turns out to be obvious that $\varphi$ is a right-equivalence $(\widetilde{\mu}_k(A(\tau,\xi)),\widetilde{\mu}_k(S(\tau,\xi)))\rightarrow(\widetilde{\mu}_k(A(\tau,\xi)),S(\sigma,\zeta)^\sharp)$.
\end{enumerate}
\end{case}



\begin{case}\label{case:23}\emph{Configurations 23 and 19}.
\begin{figure}[!ht]
                \centering
                \includegraphics[scale=.75]{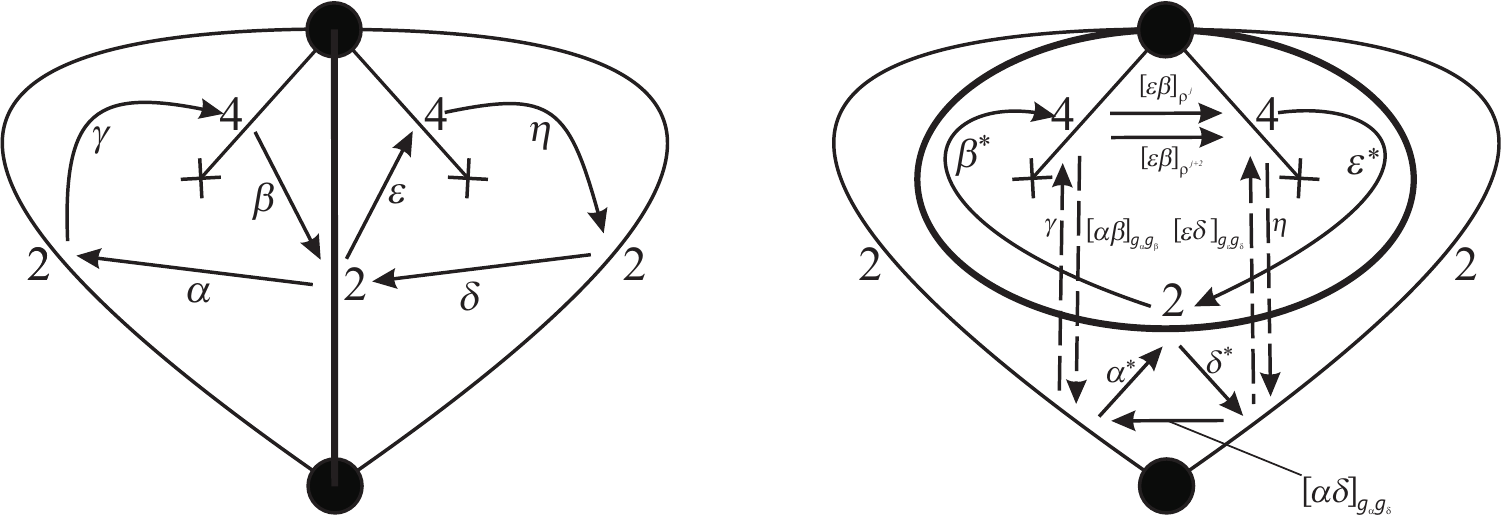}
                \caption{{\footnotesize Configurations 23 and 19 of Figure \ref{Fig:all_possibilities_for_k_and_weights}. Left: $\tau$ and $Q(\tau,\omega)$. Right: $\sigma$ and $\widetilde{\mu}_k(Q(\tau,\omega))$.  The numbers next to the arcs are the corresponding values of the tuple $\dtuple(\tau,\omega)$.}}
                \label{Fig:flip_mut_23}
        \end{figure}
\begin{enumerate}\item We use the notation in Figure \ref{Fig:flip_mut_23},
where the element $j\in\{0,1,2,3\}$ satisfies $\rho^j|_{L}=g_\varepsilon g_{\beta}=\rho^{j+2}|_{L}$ and identify the quiver $Q(\sigma,\omega)$ with the subquiver of $\widetilde{\mu}_k(Q(\tau,\omega))$ obtained by deleting the dotted arrows in the figure.
By the definition of the modulating function $\widetilde{\mu}_k(g):\widetilde{\mu}_k(Q(\tau,\xi))_1\rightarrow\bigcup_{i,j\in\tau} G_{i,j}$ (cf. \cite[Definition 3.19]{Geuenich-Labardini-1}), we have
\begin{center}
\begin{tabular}{cccc}
$\widetilde{\mu}_k(g)([\alpha\delta]_{g_\alpha g_\delta})  =  g_\alpha g_\delta$,&
$\widetilde{\mu}_k(g)([\varepsilon\beta]_{\rho^j})  =  \rho^j$& and&
$\widetilde{\mu}_k(g)([\varepsilon\beta]_{\rho^{j+2}})  =  \rho^{j+2}$,
\end{tabular}
\end{center}
which respectively coincide with $\theta^{\zeta([\alpha\delta]_{g_\alpha g_\delta})}$, $g(\sigma,\zeta)_{[\varepsilon\beta]_{\rho^j}}$ and $g(\sigma,\zeta)_{[\varepsilon\beta]_{\rho^{j+2}}}$ since $(\sigma,\zeta)=\flip_k(\tau,\xi)$.

\item
The SP $(A(\sigma,\zeta),S(\sigma,\zeta))$ is the reduced part of $(\widetilde{\mu}_k(A(\tau,\xi)),S(\sigma,\zeta)^\sharp)$, where
\begin{eqnarray}\nonumber
\Ssigmad^\sharp  & = &
\gamma[\alpha\beta]_{g_\alpha g_\beta}+\eta[\varepsilon\delta]_{g_\varepsilon g_\delta}
+\delta^*\alpha^*[\alpha\delta]_{g_\alpha g_\delta}
+([\varepsilon\beta]_{\rho^j}+[\varepsilon\beta]_{\rho^{j+2}})\beta^*\varepsilon^*
+S(\tau,\sigma)
\in\RA{\widetilde{\mu}_k(A(\tau,\xi))},
\end{eqnarray}
with $S(\tau,\sigma)\in\RA{A(\tau,\xi)}\cap\RA{A(\sigma,\zeta)}$.
Furthermore,
\begin{eqnarray*}
\Stauc & = &
\alpha\beta\gamma
+\delta\eta\varepsilon
+S(\tau,\sigma) \ \ \ \ \ \text{and}\\
\widetilde{\mu}_k(\Stauc)
&\sim_{\operatorname{cyc}}&
\gamma[\alpha\beta]_{g_\alpha g_\beta}
+\eta[\varepsilon\delta]_{g_\varepsilon g_\delta}\\
&&
+\pi_{(g_{\alpha}g_{\beta})^{-1}}(\beta^*\alpha^*)[\alpha\beta]_{g_\alpha g_\beta}
+\pi_{(g_{\varepsilon}g_{\delta})^{-1}}(\delta^*\varepsilon^*)[\varepsilon\delta]_{g_\varepsilon g_\delta}\\
&&
+\beta^*\varepsilon^*([\varepsilon\beta]_{\rho^j}+[\varepsilon\beta]_{\rho^{j+2}})
+\delta^*\alpha^*[\alpha\delta]_{g_\alpha g_\delta}
+S(\tau,\sigma).
\end{eqnarray*}

\item
From the equalities $g_\alpha g_\beta = g_{\gamma}^{-1}$ and $g_{\varepsilon}g_\delta=g_\eta^{-1}$, we deduce that the rule
\begin{center}
\begin{tabular}{ccll}
$\varphi$ &:&
$\gamma\mapsto\gamma-\pi_{(g_{\alpha}g_{\beta})^{-1}}(\beta^*\alpha^*)$,
&
$\eta\mapsto\eta-\pi_{(g_{\varepsilon}g_{\delta})^{-1}}(\delta^*\varepsilon^*)$,
\end{tabular}
\end{center}
produces a well-defined $R$-algebra automorphism $\varphi:\RA{\widetilde{\mu}_k(A(\tau,\xi))}\rightarrow\RA{\widetilde{\mu}_k(A(\tau,\xi))}$ (see \cite[Propositions 2.15-(6) and 3.7]{Geuenich-Labardini-1}). It turns out to be obvious that $\varphi$ is a right-equivalence $(\widetilde{\mu}_k(A(\tau,\xi)),\widetilde{\mu}_k(S(\tau,\xi)))\rightarrow(\widetilde{\mu}_k(A(\tau,\xi)),S(\sigma,\zeta)^\sharp)$.
\end{enumerate}
\end{case}



\begin{case}\label{case:24}\emph{Configurations 24 and 28}.
\begin{figure}[!ht]
                \centering
                \includegraphics[scale=.75]{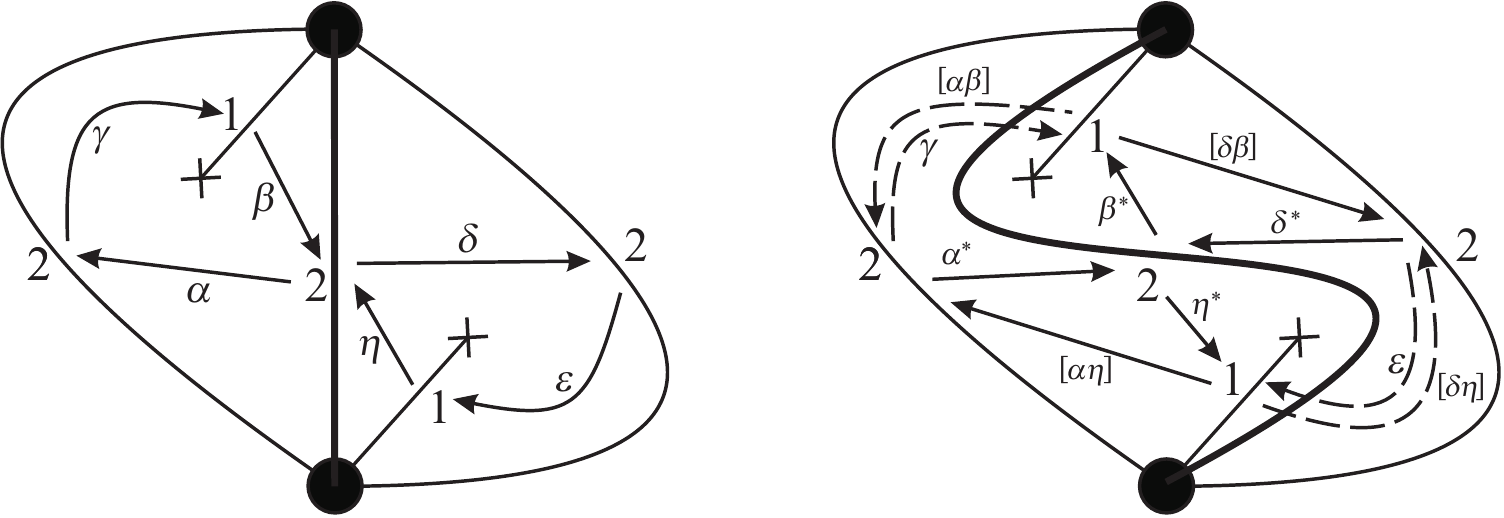}
                \caption{{\footnotesize Configurations 24 and 28 of Figure \ref{Fig:all_possibilities_for_k_and_weights}. Left: $\tau$ and $Q(\tau,\omega)$. Right: $\sigma$ and $\widetilde{\mu}_k(Q(\tau,\omega))$.  The numbers next to the arcs are the corresponding values of the tuple $\dtuple(\tau,\omega)$.}}
                \label{Fig:flip_mut_24}
        \end{figure}
\begin{enumerate}\item We use the notation in Figure \ref{Fig:flip_mut_24} and identify the quiver $Q(\sigma,\omega)$ with the subquiver of $\widetilde{\mu}_k(Q(\tau,\omega))$ obtained by deleting the dotted arrows in the figure.

\item
It is straightforward to see that $(A(\sigma,\zeta),S(\sigma,\zeta))$ is the reduced part of $(\widetilde{\mu}_k(A(\tau,\xi)),S(\sigma,\zeta)^\sharp)$, where
\begin{eqnarray}\nonumber
\Ssigmad^\sharp  & = &
\gamma[\alpha\beta]+\varepsilon[\delta\eta]+
\eta^*\alpha^*[\alpha\eta]+\beta^*\delta^*[\delta\beta]
+S(\tau,\sigma)
\in\RA{\widetilde{\mu}_k(A(\tau,\xi))},
\end{eqnarray}
with $S(\tau,\sigma)\in\RA{A(\tau,\xi)}\cap\RA{A(\sigma,\zeta)}$.

\item
The potentials $\Stauc$ and $\widetilde{\mu}_k(\Stauc)$ are
\begin{eqnarray*}
\Stauc & = &
\alpha\beta\gamma
+\delta\eta\varepsilon
+S(\tau,\sigma) \ \ \ \ \ \text{and} \\
\widetilde{\mu}_k(\Stauc) & = &
\gamma[\alpha\beta]+ \varepsilon[\delta\eta]
+\beta^*\alpha^*[\alpha\beta]+\eta^*\delta^*[\delta\eta]+\eta^*\alpha^*[\alpha\eta]+\beta^*\delta^*[\delta\beta]
+S(\tau,\sigma).
\end{eqnarray*}

\item
From the equalities $g_\alpha g_\beta = g_{\gamma}^{-1}$,
$g_{\delta}g_\eta=g_\varepsilon^{-1}$,
$\pi_{(g_{\alpha}g_{\beta})^{-1}}(\beta^*\alpha^*)=\beta^*\alpha^*$ and
$\pi_{(g_{\delta}g_{\eta})^{-1}}(\eta^*\delta^*)=\eta^*\delta^*$, we deduce that the rule
\begin{center}
\begin{tabular}{ccll}
$\varphi$ &:&
$\gamma\mapsto\gamma-\beta^*\alpha^*$,
&
$\varepsilon\mapsto\varepsilon-\eta^*\delta^*$,
\end{tabular}
\end{center}
produces a well-defined $R$-algebra automorphism $\varphi:\RA{\widetilde{\mu}_k(A(\tau,\xi))}\rightarrow\RA{\widetilde{\mu}_k(A(\tau,\xi))}$ (see \cite[Propositions 2.15-(6) and 3.7]{Geuenich-Labardini-1}). It turns out to be obvious that $\varphi$ is a right-equivalence $(\widetilde{\mu}_k(A(\tau,\xi)),\widetilde{\mu}_k(S(\tau,\xi)))\rightarrow(\widetilde{\mu}_k(A(\tau,\xi)),S(\sigma,\zeta)^\sharp)$.
\end{enumerate}
\end{case}



\begin{case}\label{case:25}\emph{Configurations 25 and 30}.
\begin{figure}[!ht]
                \centering
                \includegraphics[scale=.75]{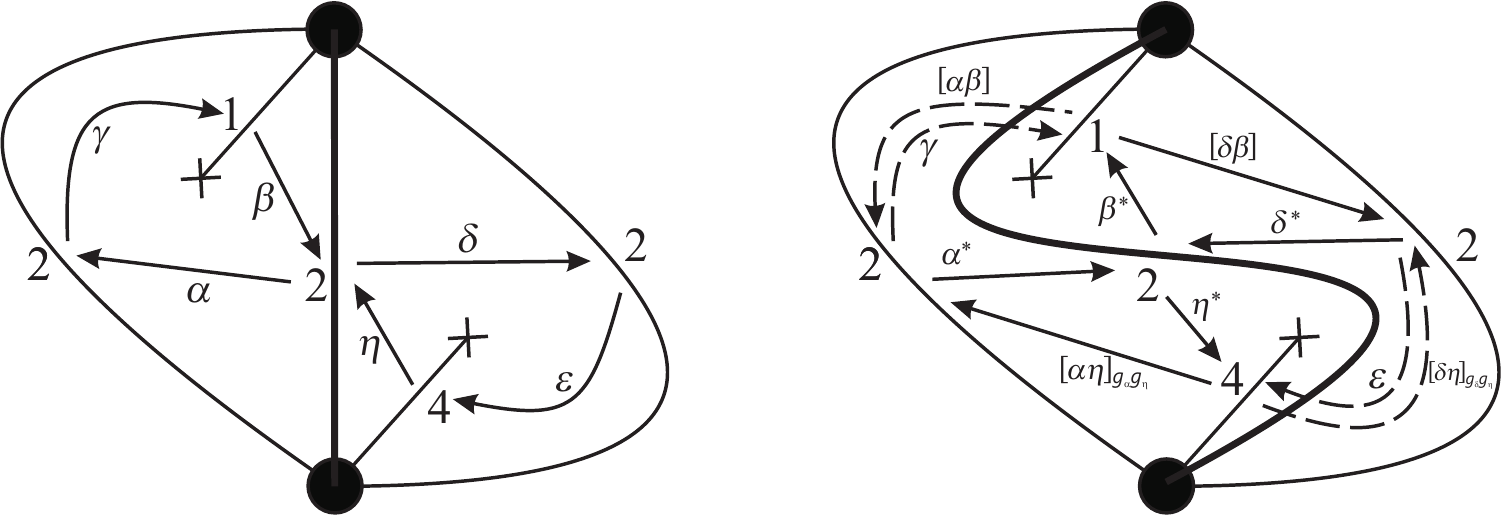}
                \caption{{\footnotesize Configurations 25 and 30 of Figure \ref{Fig:all_possibilities_for_k_and_weights}. Left: $\tau$ and $Q(\tau,\omega)$. Right: $\sigma$ and $\widetilde{\mu}_k(Q(\tau,\omega))$.  The numbers next to the arcs are the corresponding values of the tuple $\dtuple(\tau,\omega)$.}}
                \label{Fig:flip_mut_25}
        \end{figure}
\begin{enumerate}\item We use the notation in Figure \ref{Fig:flip_mut_25} and identify the quiver $Q(\sigma,\omega)$ with the subquiver of $\widetilde{\mu}_k(Q(\tau,\omega))$ obtained by deleting the dotted arrows in the figure.

\item
The SP $(A(\sigma,\zeta),S(\sigma,\zeta))$ is the reduced part of $(\widetilde{\mu}_k(A(\tau,\xi)),S(\sigma,\zeta)^\sharp)$, where
\begin{eqnarray}\nonumber
\Ssigmad^\sharp  & = &
\gamma[\alpha\beta]+\varepsilon[\delta\eta]_{g_\delta g_\eta}
+\eta^*\alpha^*[\alpha\eta]_{g_\alpha g_\eta}
+\beta^*\delta^*[\delta\beta]
+S(\tau,\sigma)
\in\RA{\widetilde{\mu}_k(A(\tau,\xi))},
\end{eqnarray}
with $S(\tau,\sigma)\in\RA{A(\tau,\xi)}\cap\RA{A(\sigma,\zeta)}$.

\item The potentials $\Stauc$ and $\widetilde{\mu}_k(\Stauc)$ are
\begin{eqnarray*}
\Stauc & = &
\alpha\beta\gamma
+\delta\eta\varepsilon
+S(\tau,\sigma) \ \ \ \ \ \text{and}\\
\widetilde{\mu}_k(\Stauc)
&\sim_{\operatorname{cyc}}&
\gamma[\alpha\beta]
+\varepsilon[\delta\eta]_{g_\delta g_\eta}
+\beta^*\alpha^*[\alpha\beta]
+\pi_{(g_{\delta}g_{\eta})^{-1}}(\eta^*\delta^*)[\delta\eta]_{g_\delta g_\eta}\\
&&
+\eta^*\alpha^*[\alpha\eta]_{g_\alpha g_\eta}
+\beta^*\delta^*[\delta\beta]
+S(\tau,\sigma)
,
\end{eqnarray*}
where $\pi_{(g_{\delta}g_{\eta})^{-1}}(x)
=\frac{1}{2}\left(x+(g_{\delta}g_{\eta})^{-1}(u^{-1})xu\right)$.

\item
From the equalities $g_\alpha g_\beta = g_{\gamma}^{-1}$, $g_{\delta}g_\eta=g_\varepsilon^{-1}$ and $\pi_{(g_{\alpha}g_{\beta})^{-1}}(\beta^*\alpha^*)=\beta^*\alpha^*$, we deduce that the rule
\begin{center}
\begin{tabular}{ccll}
$\varphi$ &:&
$\gamma\mapsto\gamma-\beta^*\alpha^*$,
&
$\varepsilon\mapsto\varepsilon-\pi_{(g_{\delta}g_{\eta})^{-1}}(\eta^*\delta^*)$,
\end{tabular}
\end{center}
produces a well-defined $R$-algebra automorphism $\varphi:\RA{\widetilde{\mu}_k(A(\tau,\xi))}\rightarrow\RA{\widetilde{\mu}_k(A(\tau,\xi))}$ (see \cite[Propositions 2.15-(6) and 3.7]{Geuenich-Labardini-1}). It turns out to be obvious that $\varphi$ is a right-equivalence $(\widetilde{\mu}_k(A(\tau,\xi)),\widetilde{\mu}_k(S(\tau,\xi)))\rightarrow(\widetilde{\mu}_k(A(\tau,\xi)),S(\sigma,\zeta)^\sharp)$.
\end{enumerate}
\end{case}



\begin{case}\label{case:27}\emph{Configurations 27 and 31}.
\begin{figure}[!ht]
                \centering
                \includegraphics[scale=.75]{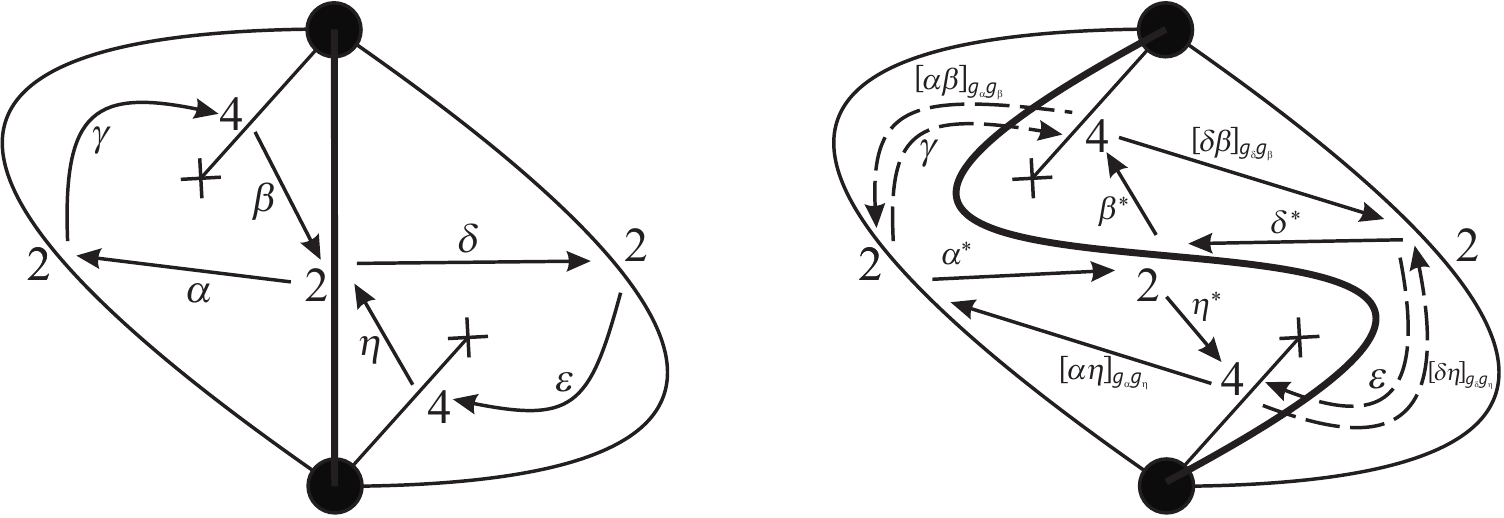}
                \caption{{\footnotesize Configurations 27 and 31 of Figure \ref{Fig:all_possibilities_for_k_and_weights}. Left: $\tau$ and $Q(\tau,\omega)$. Right: $\sigma$ and $\widetilde{\mu}_k(Q(\tau,\omega))$.  The numbers next to the arcs are the corresponding values of the tuple $\dtuple(\tau,\omega)$.}}
                \label{Fig:flip_mut_27}
        \end{figure}
\begin{enumerate}\item We use the notation in Figure \ref{Fig:flip_mut_27} and identify the quiver $Q(\sigma,\omega)$ with the subquiver of $\widetilde{\mu}_k(Q(\tau,\omega))$ obtained by deleting the dotted arrows in the figure.

\item
The SP $(A(\sigma,\zeta),S(\sigma,\zeta))$ is the reduced part of $(\widetilde{\mu}_k(A(\tau,\xi)),S(\sigma,\zeta)^\sharp)$, where
\begin{eqnarray*}
\Ssigmad^\sharp  & = &
\gamma[\alpha\beta]_{g_\alpha g_\beta}+\varepsilon[\delta\eta]_{g_\delta g_\eta}
+\eta^*\alpha^*[\alpha\eta]_{g_\alpha g_\eta}
+\beta^*\delta^*[\delta\beta]_{g_\beta g_\delta}
+S(\tau,\sigma)
\in\RA{\widetilde{\mu}_k(A(\tau,\xi))},
\end{eqnarray*}
with $S(\tau,\sigma)\in\RA{A(\tau,\xi)}\cap\RA{A(\sigma,\zeta)}$.

\item
The potentials $\Stauc$ and $\widetilde{\mu}_k(\Stauc)$ are
\begin{eqnarray*}
\Stauc & = &
\alpha\beta\gamma
+\delta\eta\varepsilon
+S(\tau,\sigma) \ \ \ \ \ \text{and}\\
\widetilde{\mu}_k(\Stauc)
&\sim_{\operatorname{cyc}}&
\gamma[\alpha\beta]_{g_\alpha g_\beta}
+\varepsilon[\delta\eta]_{g_\delta g_\eta}\\
&&
+\pi_{(g_{\alpha}g_{\beta})^{-1}}(\beta^*\alpha^*)[\alpha\beta]_{g_\alpha g_\beta}
+\pi_{(g_{\delta}g_{\eta})^{-1}}(\eta^*\delta^*)[\delta\eta]_{g_\delta g_\eta}\\
&&
+\eta^*\alpha^*[\alpha\eta]_{g_\alpha g_\eta}
+\beta^*\delta^*[\delta\eta]_{g_\delta g_\eta}
+S(\tau,\sigma)
.
\end{eqnarray*}

\item
From the equalities $g_\alpha g_\beta = g_{\gamma}^{-1}$ and $g_{\delta}g_\eta=g_\varepsilon^{-1}$, we deduce that the rule
\begin{center}
\begin{tabular}{ccll}
$\varphi$ &:&
$\gamma\mapsto\gamma-\pi_{(g_{\alpha}g_{\beta})^{-1}}(\beta^*\alpha^*)$,
&
$\varepsilon\mapsto\varepsilon-\pi_{(g_{\delta}g_{\eta})^{-1}}(\eta^*\delta^*)$,
\end{tabular}
\end{center}
produces a well-defined $R$-algebra automorphism $\varphi:\RA{\widetilde{\mu}_k(A(\tau,\xi))}\rightarrow\RA{\widetilde{\mu}_k(A(\tau,\xi))}$ (see \cite[Propositions 2.15-(6) and 3.7]{Geuenich-Labardini-1}). It turns out to be obvious that $\varphi$ is a right-equivalence $(\widetilde{\mu}_k(A(\tau,\xi)),\widetilde{\mu}_k(S(\tau,\xi)))\rightarrow(\widetilde{\mu}_k(A(\tau,\xi)),S(\sigma,\zeta)^\sharp)$.
\end{enumerate}
\end{case}



\begin{case}\label{case:32}\emph{Configurations 32 and 40}.
\begin{figure}[!ht]
                \centering
                \includegraphics[scale=.75]{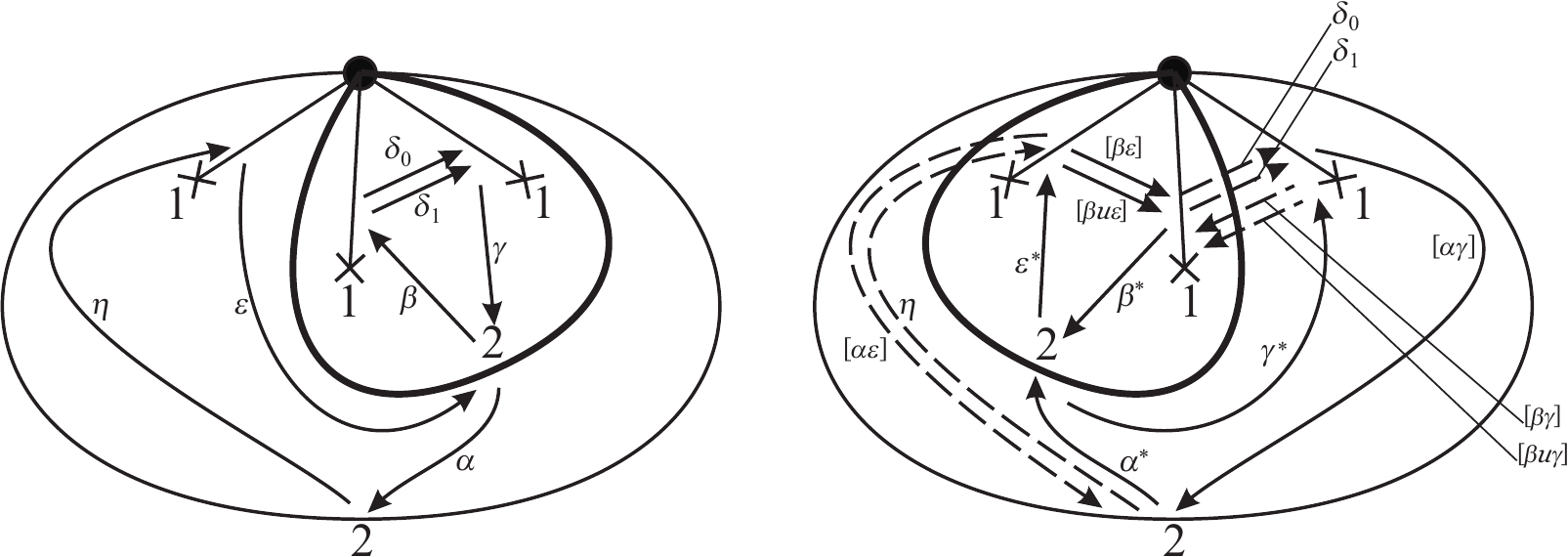}
                \caption{{\footnotesize Configurations 32 and 40 of Figure \ref{Fig:all_possibilities_for_k_and_weights}. Left: $\tau$ and $Q(\tau,\omega)$. Right: $\sigma$ and $\widetilde{\mu}_k(Q(\tau,\omega))$.  The numbers next to the arcs are the corresponding values of the tuple $\dtuple(\tau,\omega)$.}}
                \label{Fig:flip_mut_32}
        \end{figure}
\begin{enumerate}\item We use the notation in Figure \ref{Fig:flip_mut_32} and identify the quiver $Q(\sigma,\omega)$ with the subquiver of $\widetilde{\mu}_k(Q(\tau,\omega))$ obtained by deleting the dotted arrows in the figure.

\item
The SP $(A(\sigma,\zeta),S(\sigma,\zeta))$ is the reduced part of $(\widetilde{\mu}_k(A(\tau,\xi)),S(\sigma,\zeta)^\sharp)$, where
\begin{eqnarray*}
\Ssigmad^\sharp  & = &
\delta_0[\beta\gamma]+\delta_1[\beta u\gamma]+\eta[\alpha\varepsilon]
+\alpha^*[\alpha\gamma]\gamma^*+[\beta\varepsilon]\varepsilon^*u\beta^*+[\beta u\varepsilon]\varepsilon^*\beta^*+S(\tau,\sigma)
\in\RA{\widetilde{\mu}_k(A(\tau,\xi))},
\end{eqnarray*}
with $S(\tau,\sigma)\in\RA{A(\tau,\xi)}\cap\RA{A(\sigma,\zeta)}$.

\item The potentials $\Stauc$ and $\widetilde{\mu}_k(\Stauc)$ are
\begin{eqnarray*}
\Stauc & = &
\delta_0\beta\gamma+\delta_1\beta u\gamma
+\eta\alpha\varepsilon+S(\tau,\sigma) \ \ \ \ \ \text{and}\\
\widetilde{\mu}_k(\Stauc) & = &
\delta_0[\beta\gamma]+\delta_1[\beta u\gamma]
+\eta[\alpha\varepsilon]
\\
&&
+\gamma^*\beta^*[\beta\gamma]+\gamma^*u^{-1}\beta^*[\beta u\gamma]
+\varepsilon^*\alpha^*[\alpha\varepsilon]
+\alpha^*[\alpha\gamma]\gamma^*
+[\beta\varepsilon]\varepsilon^*\beta^*+[\beta u\varepsilon]\varepsilon^*u^{-1}\beta^*
+S(\tau,\sigma).
\end{eqnarray*}

\item
The rules
\begin{center}
\begin{tabular}{cclll}
$\varphi$ &:&
$\delta_0\mapsto\delta_0-\gamma^*\beta^*$,
&
$\delta_1\mapsto\delta_1-\gamma^*u^{-1}\beta^*$,
&
$\eta\mapsto\eta-\varepsilon^*\alpha^*$,\\
$\Phi$ &:& $\varepsilon^*\mapsto\varepsilon^*u$ & &
\end{tabular}
\end{center}
produce a well-defined $R$-algebra automorphisms $\varphi,\Phi:\RA{\widetilde{\mu}_k(A(\tau,\xi))}\rightarrow\RA{\widetilde{\mu}_k(A(\tau,\xi))}$ (see \cite[Propositions 2.15-(6) and 3.7]{Geuenich-Labardini-1}). It turns out to be obvious that the composition $\Phi\varphi$ is a right-equivalence $(\widetilde{\mu}_k(A(\tau,\xi)),\widetilde{\mu}_k(S(\tau,\xi)))\rightarrow(\widetilde{\mu}_k(A(\tau,\xi)),S(\sigma,\zeta)^\sharp)$.
\end{enumerate}
\end{case}



\begin{case}\label{case:33}\emph{Configurations 33 and 44}.
\begin{figure}[!ht]
                \centering
                \includegraphics[scale=.75]{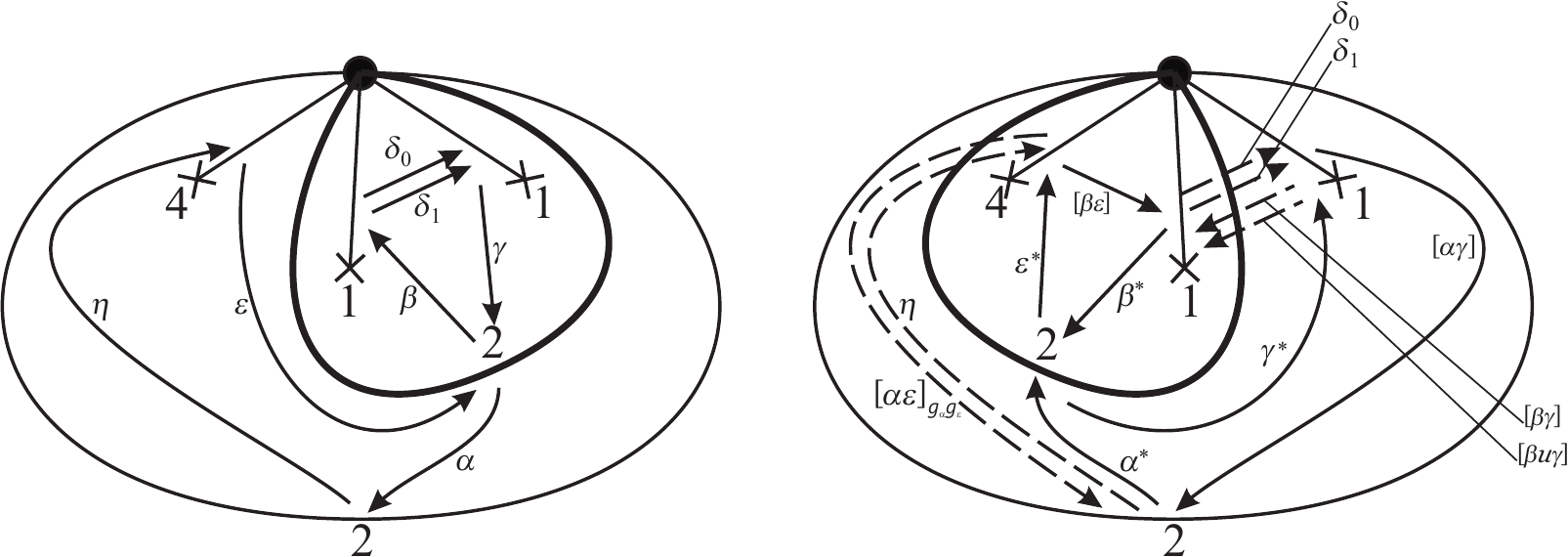}
                \caption{{\footnotesize Configurations 33 and 44 of Figure \ref{Fig:all_possibilities_for_k_and_weights}. Left: $\tau$ and $Q(\tau,\omega)$. Right: $\sigma$ and $\widetilde{\mu}_k(Q(\tau,\omega))$.  The numbers next to the arcs are the corresponding values of the tuple $\dtuple(\tau,\omega)$.}}
                \label{Fig:flip_mut_33}
        \end{figure}
\begin{enumerate}\item We use the notation in Figure \ref{Fig:flip_mut_33} and identify the quiver $Q(\sigma,\omega)$ with the subquiver of $\widetilde{\mu}_k(Q(\tau,\omega))$ obtained by deleting the dotted arrows in the figure.

\item
The SP $(A(\sigma,\zeta),S(\sigma,\zeta))$ is the reduced part of $(\widetilde{\mu}_k(A(\tau,\xi)),S(\sigma,\zeta)^\sharp)$, where
\begin{eqnarray}\nonumber
\Ssigmad^\sharp  & = &
\delta_0[\beta\gamma]+\delta_1[\beta u\gamma]+\eta[\alpha\varepsilon]_{g_\alpha g_\varepsilon}
+\beta^*[\beta\varepsilon]\varepsilon^*+\alpha^*[\alpha\gamma]\gamma^*+S(\tau,\sigma)
\in\RA{\widetilde{\mu}_k(A(\tau,\xi))},
\end{eqnarray}
with $S(\tau,\sigma)\in\RA{A(\tau,\xi)}\cap\RA{A(\sigma,\zeta)}$.

\item The potentials $\Stauc$ and $\widetilde{\mu}_k(\Stauc)$ are
\begin{eqnarray*}
\Stauc & = &
\delta_0\beta\gamma + \delta_1\beta u\gamma+\eta\alpha\varepsilon
+S(\tau,\sigma) \ \ \ \ \ \text{and}\\
\widetilde{\mu}_k(\Stauc)
&\sim_{\operatorname{cyc}}&
\delta_0[\beta\gamma] + \delta_1[\beta u\gamma]+\eta[\alpha\varepsilon]_{g_\alpha g_\varepsilon}
\\
&&
+\gamma^*\beta^*[\beta\gamma]+\gamma^* u^{-1}\beta^*[\beta u\gamma]+\pi_{(g_\alpha g_\varepsilon)^{-1}}(\varepsilon^*\alpha^*)[\alpha\varepsilon]_{g_\alpha g_\varepsilon}
\\
&&
+\alpha^*[\alpha\gamma]\gamma^*+\beta^*[\beta\varepsilon]\varepsilon^*
+S(\tau,\sigma),
\end{eqnarray*}
where $\pi_{(g_\alpha g_\varepsilon)^{-1}}(x)
=\frac{1}{2}\left(x+(g_\alpha g_\varepsilon)^{-1}(u^{-1})xu\right)$.

\item
From the equality $(g_\alpha g_\varepsilon)^{-1} = g_\eta$, we deduce that the rule
\begin{center}
\begin{tabular}{cclll}
$\varphi$ &:&
$\delta_0\mapsto\delta_0-\gamma^*\beta^*$,
&
$\delta_1\mapsto\delta_1-\gamma^*u^{-1}\beta^*$,
&
$\eta\mapsto\eta-\pi_{(g_\alpha g_\varepsilon)^{-1}}(\varepsilon^*\alpha^*)$
\end{tabular}
\end{center}
produces a well-defined $R$-algebra automorphism $\varphi:\RA{\widetilde{\mu}_k(A(\tau,\xi))}\rightarrow\RA{\widetilde{\mu}_k(A(\tau,\xi))}$ (see \cite[Propositions 2.15-(6) and 3.7]{Geuenich-Labardini-1}). It turns out to be obvious that $\varphi$ is a right-equivalence $(\widetilde{\mu}_k(A(\tau,\xi)),\widetilde{\mu}_k(S(\tau,\xi)))\rightarrow(\widetilde{\mu}_k(A(\tau,\xi)),S(\sigma,\zeta)^\sharp)$.
\end{enumerate}
\end{case}



\begin{case}\label{case:34}\emph{Configurations 34 and 41}.
\begin{figure}[!ht]
                \centering
                \includegraphics[scale=.75]{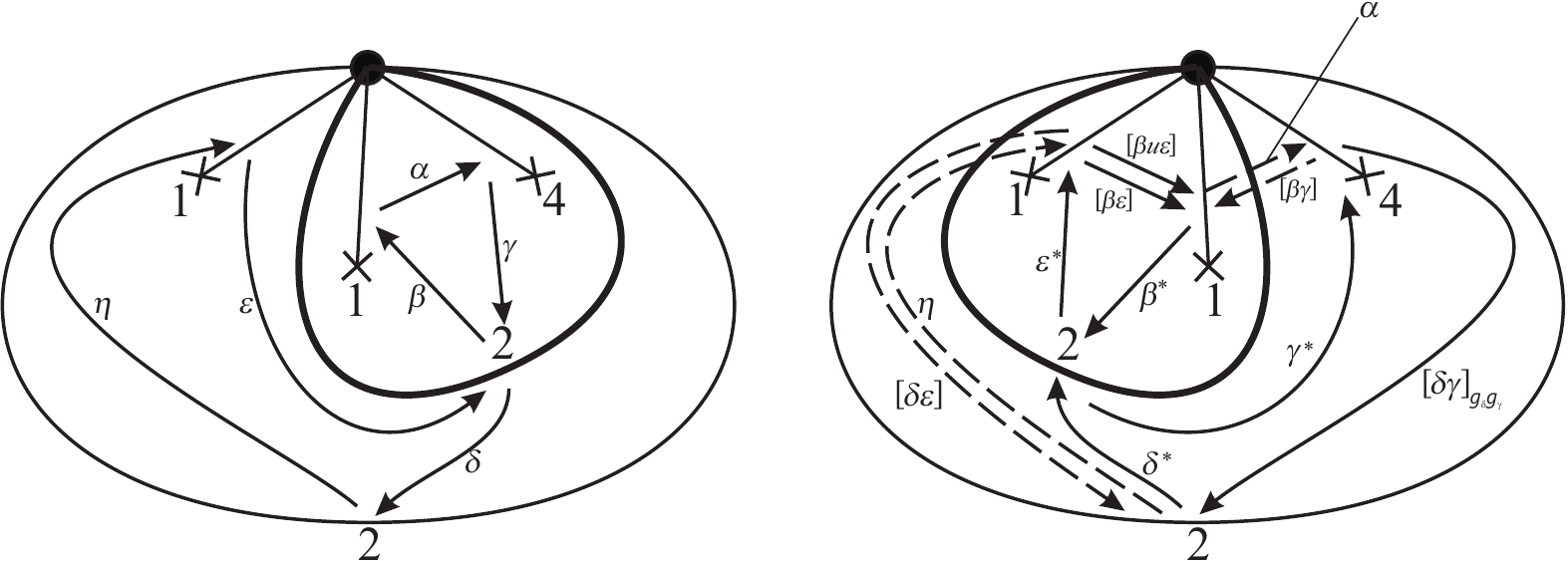}
                \caption{{\footnotesize Configurations 34 and 41 of Figure \ref{Fig:all_possibilities_for_k_and_weights}. Left: $\tau$ and $Q(\tau,\omega)$. Right: $\sigma$ and $\widetilde{\mu}_k(Q(\tau,\omega))$.  The numbers next to the arcs are the corresponding values of the tuple $\dtuple(\tau,\omega)$.}}
                \label{Fig:flip_mut_34}
        \end{figure}
\begin{enumerate}\item We use the notation in Figure \ref{Fig:flip_mut_34} and identify the quiver $Q(\sigma,\omega)$ with the subquiver of $\widetilde{\mu}_k(Q(\tau,\omega))$ obtained by deleting the dotted arrows in the figure.

\item
The SP $(A(\sigma,\zeta),S(\sigma,\zeta))$ is the reduced part of $(\widetilde{\mu}_k(A(\tau,\xi)),S(\sigma,\zeta)^\sharp)$, where
\begin{eqnarray}\nonumber
\Ssigmad^\sharp  & = &
\alpha[\beta\gamma]+\eta[\delta\varepsilon]
+[\beta\varepsilon]\varepsilon^*u\beta^*+[\beta u\varepsilon]\varepsilon^*\beta^*+
\delta^*[\delta\gamma]_{g_\delta g_\gamma}\gamma^*+S(\tau,\sigma)
\in\RA{\widetilde{\mu}_k(A(\tau,\xi))},
\end{eqnarray}
with $S(\tau,\sigma)\in\RA{A(\tau,\xi)}\cap\RA{A(\sigma,\zeta)}$.

\item The potentials $\Stauc$ and $\widetilde{\mu}_k(\Stauc)$ are
\begin{eqnarray*}
\Stauc & = &
\alpha\beta\gamma+\eta\delta\varepsilon
+S(\tau,\sigma) \ \ \ \ \ \text{and}\\
\widetilde{\mu}_k(\Stauc) & = &
\alpha[\beta\gamma]+\eta[\delta\varepsilon]
\\
&&
+\gamma^*\beta^*[\beta\gamma]+\varepsilon^*\delta^*[\delta\varepsilon]
+[\beta\varepsilon]\varepsilon^*\beta^*+[\beta u\varepsilon]\varepsilon^*u^{-1}\beta^*
+\delta^*[\delta\gamma]_{g_\delta g_\gamma}\gamma^*
+S(\tau,\sigma).
\end{eqnarray*}

\item
The rules
\begin{center}
\begin{tabular}{ccllcccl}
$\varphi$ &:&
$\alpha\mapsto\alpha-\gamma^*\beta^*$,
&
$\eta\mapsto\eta-\varepsilon^*\delta^*$,
&and&
$\Phi$ &:&
$\varepsilon^*\mapsto\varepsilon^*u$
\end{tabular}
\end{center}
produce well-defined $R$-algebra automorphisms $\varphi,\Phi:\RA{\widetilde{\mu}_k(A(\tau,\xi))}\rightarrow\RA{\widetilde{\mu}_k(A(\tau,\xi))}$ (see \cite[Propositions 2.15-(6) and 3.7]{Geuenich-Labardini-1}). It turns out to be obvious that the composition $\Phi\varphi$ is a right-equivalence $(\widetilde{\mu}_k(A(\tau,\xi)),\widetilde{\mu}_k(S(\tau,\xi)))\rightarrow(\widetilde{\mu}_k(A(\tau,\xi)),S(\sigma,\zeta)^\sharp)$.
\end{enumerate}
\end{case}



\begin{case}\label{case:35}\emph{Configurations 35 and 45}.
\begin{figure}[!ht]
                \centering
                \includegraphics[scale=.75]{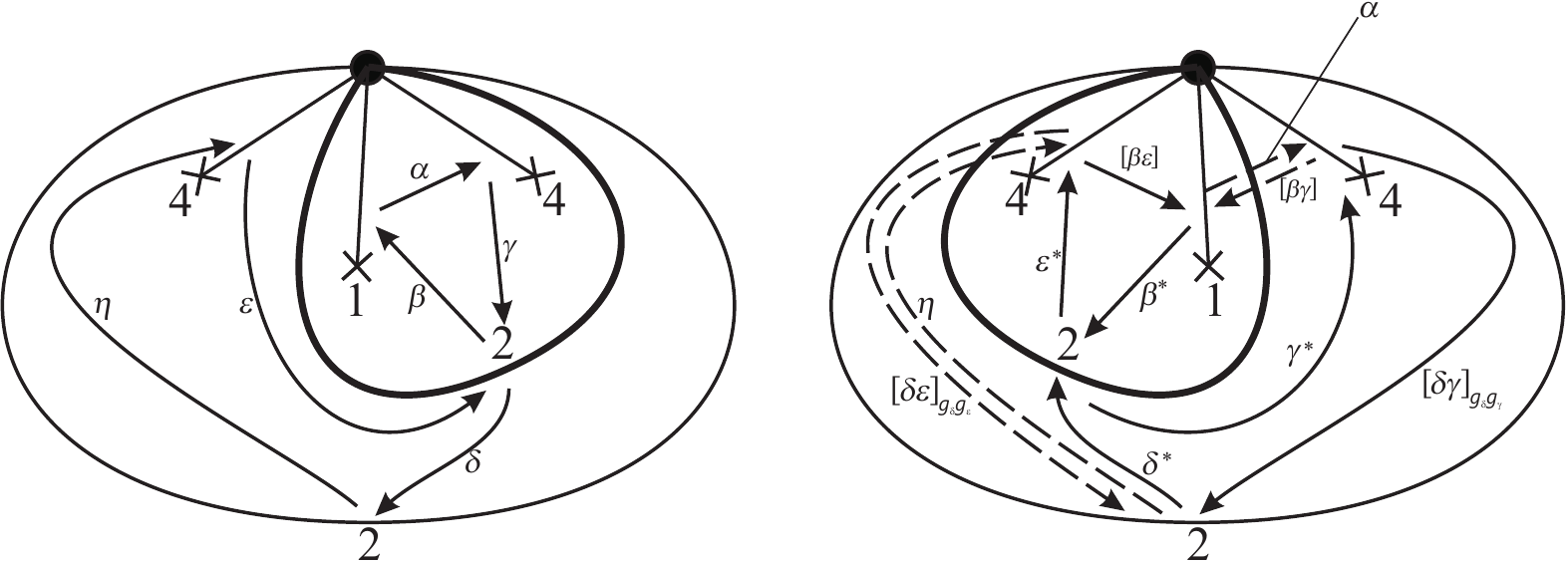}
                \caption{{\footnotesize Configurations 35 and 45 of Figure \ref{Fig:all_possibilities_for_k_and_weights}. Left: $\tau$ and $Q(\tau,\omega)$. Right: $\sigma$ and $\widetilde{\mu}_k(Q(\tau,\omega))$.  The numbers next to the arcs are the corresponding values of the tuple $\dtuple(\tau,\omega)$.}}
                \label{Fig:flip_mut_35}
        \end{figure}
\begin{enumerate}\item We use the notation in Figure \ref{Fig:flip_mut_35} and identify the quiver $Q(\sigma,\omega)$ with the subquiver of $\widetilde{\mu}_k(Q(\tau,\omega))$ obtained by deleting the dotted arrows in the figure.

\item
The SP $(A(\sigma,\zeta),S(\sigma,\zeta))$ is the reduced part of $(\widetilde{\mu}_k(A(\tau,\xi)),S(\sigma,\zeta)^\sharp)$, where
\begin{eqnarray}\nonumber
\Ssigmad^\sharp  & = &
\alpha[\beta\gamma]+\eta[\delta\varepsilon]_{g_{\delta}g_{\varepsilon}}
+\beta^*[\beta\varepsilon]\varepsilon^*+\delta^*[\delta\gamma]_{g_{\delta}g_{\gamma}}\gamma^*+S(\tau,\sigma)
\in\RA{\widetilde{\mu}_k(A(\tau,\xi))},
\end{eqnarray}
with $S(\tau,\sigma)\in\RA{A(\tau,\xi)}\cap\RA{A(\sigma,\zeta)}$.

\item The potentials $\Stauc$ and $\widetilde{\mu}_k(\Stauc)$ are
\begin{eqnarray*}
\Stauc & = &
\alpha\beta\gamma+\eta\delta\varepsilon
+S(\tau,\sigma) \ \ \ \ \ \text{and}\\
\widetilde{\mu}_k(\Stauc)
&\sim_{\operatorname{cyc}}&
\alpha[\beta\gamma]+\eta[\delta\varepsilon]_{g_{\delta}g_{\varepsilon}}
\\
&&
+\gamma^*\beta^*[\beta\gamma]
+\pi_{(g_{\delta}g_{\varepsilon})^{-1}}(\varepsilon^*\delta^*)[\delta\varepsilon]_{g_{\delta}g_{\varepsilon}}
+\beta^*[\beta\varepsilon]\varepsilon^*+\delta^*[\delta\gamma]_{g_{\delta}g_{\gamma}}\gamma^*
+S(\tau,\sigma).
\end{eqnarray*}

\item
From the equality $(g_{\delta}g_{\varepsilon})^{-1} = g_{\eta}$, we deduce that the rule
\begin{center}
\begin{tabular}{ccll}
$\varphi$ &:&
$\alpha\mapsto\alpha-\gamma^*\beta^*$,
&
$\eta\mapsto\eta-\pi_{(g_{\delta}g_{\varepsilon})^{-1}}(\varepsilon^*\delta^*)$,
\end{tabular}
\end{center}
produces a well-defined $R$-algebra automorphism $\varphi:\RA{\widetilde{\mu}_k(A(\tau,\xi))}\rightarrow\RA{\widetilde{\mu}_k(A(\tau,\xi))}$ (see \cite[Propositions 2.15-(6) and 3.7]{Geuenich-Labardini-1}). It turns out to be obvious that $\varphi$ is a right-equivalence $(\widetilde{\mu}_k(A(\tau,\xi)),\widetilde{\mu}_k(S(\tau,\xi)))\rightarrow(\widetilde{\mu}_k(A(\tau,\xi)),S(\sigma,\zeta)^\sharp)$.
\end{enumerate}
\end{case}



\begin{case}\label{case:36}\emph{Configurations 36 and 42}.
\begin{figure}[!ht]
                \centering
                \includegraphics[scale=.75]{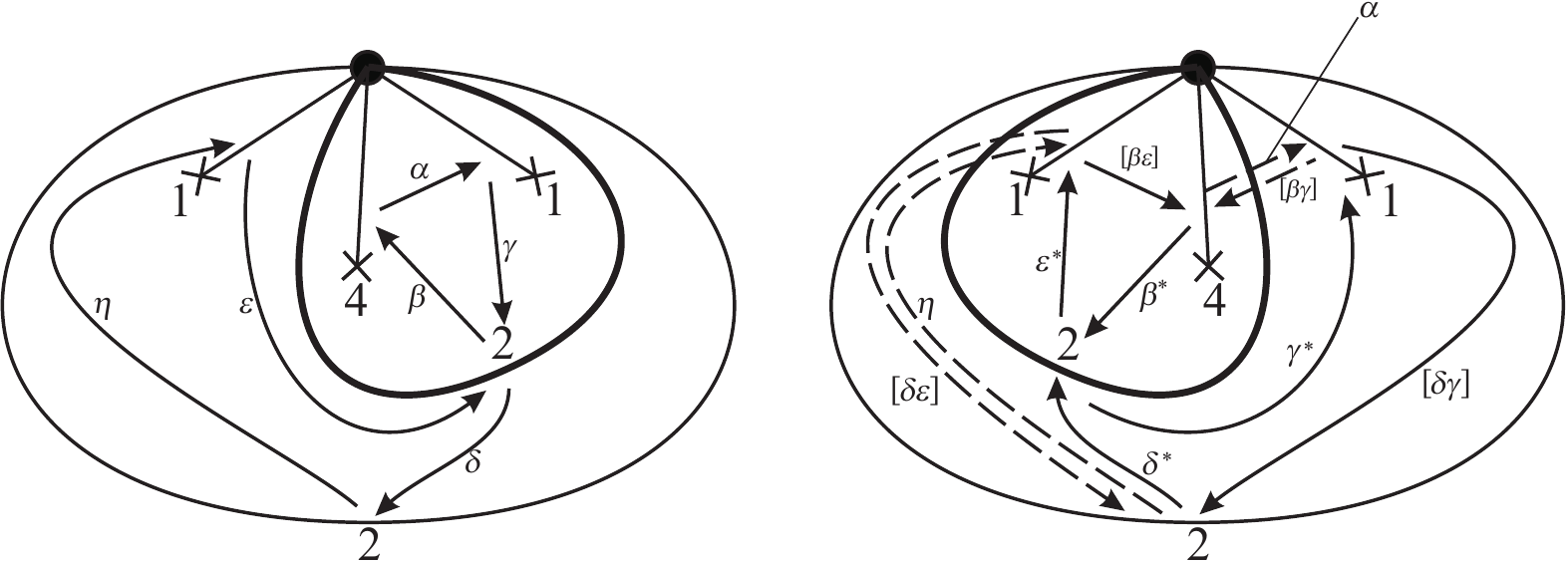}
                \caption{{\footnotesize Configurations 36 and 42 of Figure \ref{Fig:all_possibilities_for_k_and_weights}. Left: $\tau$ and $Q(\tau,\omega)$. Right: $\sigma$ and $\widetilde{\mu}_k(Q(\tau,\omega))$.  The numbers next to the arcs are the corresponding values of the tuple $\dtuple(\tau,\omega)$.}}
                \label{Fig:flip_mut_36}
        \end{figure}
\begin{enumerate}\item We use the notation in Figure \ref{Fig:flip_mut_36} and identify the quiver $Q(\sigma,\omega)$ with the subquiver of $\widetilde{\mu}_k(Q(\tau,\omega))$ obtained by deleting the dotted arrows in the figure.

\item
The SP $(A(\sigma,\zeta),S(\sigma,\zeta))$ is the reduced part of $(\widetilde{\mu}_k(A(\tau,\xi)),S(\sigma,\zeta)^\sharp)$, where
\begin{eqnarray}\nonumber
\Ssigmad^\sharp  & = &
\alpha[\beta\gamma]+\eta[\delta\varepsilon]
+\beta^*[\beta\varepsilon]\varepsilon^*+\delta^*[\delta\gamma]\gamma^*
+S(\tau,\sigma)
\in\RA{\widetilde{\mu}_k(A(\tau,\xi))},
\end{eqnarray}
with $S(\tau,\sigma)\in\RA{A(\tau,\xi)}\cap\RA{A(\sigma,\zeta)}$.

\item The potentials $\Stauc$ and $\widetilde{\mu}_k(\Stauc)$ are
\begin{eqnarray*}
\Stauc & = &
\alpha\beta\gamma+\eta\delta\varepsilon
+S(\tau,\sigma) \ \ \ \ \ \text{and}\\
\widetilde{\mu}_k(\Stauc) & = &
\alpha[\beta\gamma]+\eta[\delta\varepsilon]
\\
&&
+\gamma^*\beta^*[\beta\gamma]+\varepsilon^*\delta^*[\delta\varepsilon]
+\beta^*[\beta\varepsilon]\varepsilon^*+\delta^*[\delta\gamma]\gamma^*
+S(\tau,\sigma).
\end{eqnarray*}

\item
The rule
\begin{center}
\begin{tabular}{ccll}
$\varphi$ &:&
$\alpha\mapsto\alpha-\gamma^*\beta^*$,
&
$\eta\mapsto\eta-\varepsilon^*\delta^*$,
\end{tabular}
\end{center}
produces a well-defined $R$-algebra automorphism $\varphi:\RA{\widetilde{\mu}_k(A(\tau,\xi))}\rightarrow\RA{\widetilde{\mu}_k(A(\tau,\xi))}$ (see \cite[Propositions 2.15-(6) and 3.7]{Geuenich-Labardini-1}). It turns out to be obvious that $\varphi$ is a right-equivalence $(\widetilde{\mu}_k(A(\tau,\xi)),\widetilde{\mu}_k(S(\tau,\xi)))\rightarrow(\widetilde{\mu}_k(A(\tau,\xi)),S(\sigma,\zeta)^\sharp)$.
\end{enumerate}
\end{case}



\begin{case}\label{case:37}\emph{Configurations 37 and 46}.
\begin{figure}[!ht]
                \centering
                \includegraphics[scale=.75]{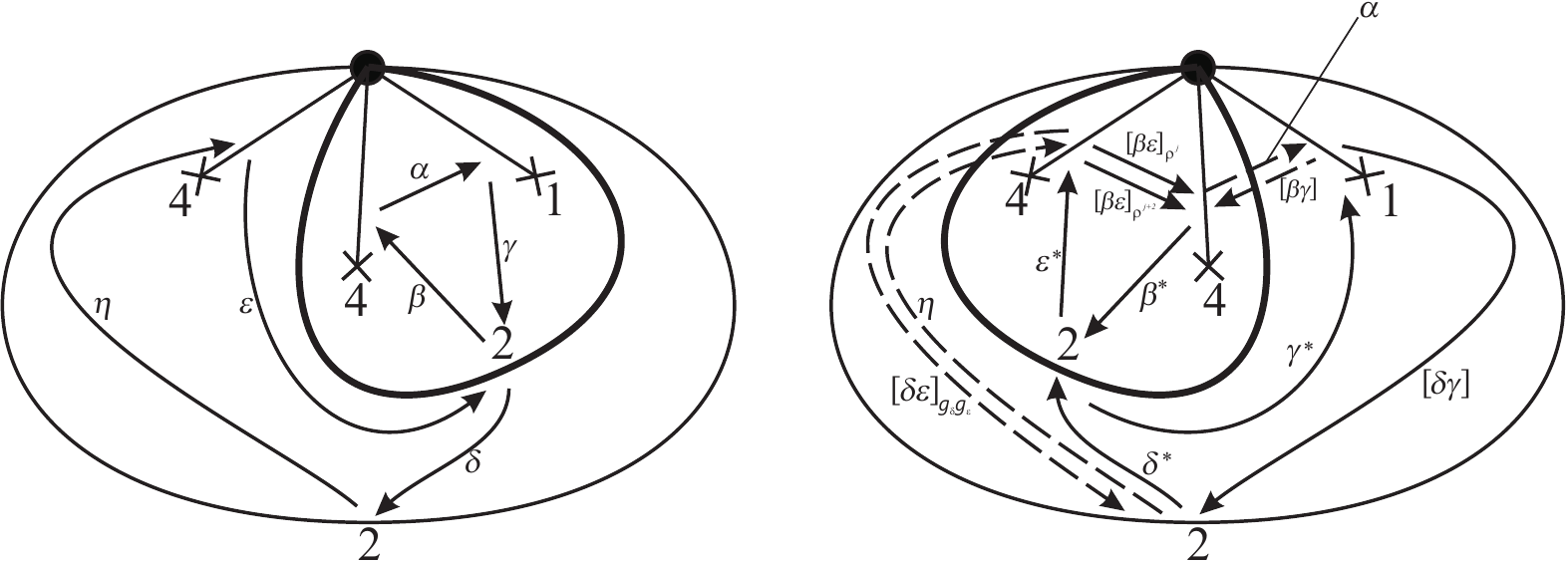}
                \caption{{\footnotesize Configurations 37 and 46 of Figure \ref{Fig:all_possibilities_for_k_and_weights}. Left: $\tau$ and $Q(\tau,\omega)$. Right: $\sigma$ and $\widetilde{\mu}_k(Q(\tau,\omega))$.  The numbers next to the arcs are the corresponding values of the tuple $\dtuple(\tau,\omega)$.}}
                \label{Fig:flip_mut_37}
        \end{figure}
\begin{enumerate}\item We use the notation in Figure \ref{Fig:flip_mut_37},
where the element $j\in\{0,1,2,3\}$ satisfies $\rho^j|_{L}=g_\beta g_{\varepsilon}=\rho^{j+2}|_{L}$.
Furthermore, we identify the quiver $Q(\sigma,\omega)$ with the subquiver of $\widetilde{\mu}_k(Q(\tau,\omega))$ obtained by deleting the dotted arrows in the figure.
By the definition of the modulating function $\widetilde{\mu}_k(g):\widetilde{\mu}_k(Q(\tau,\xi))_1\rightarrow\bigcup_{i,j\in\tau} G_{i,j}$ (cf. \cite[Definition 3.19]{Geuenich-Labardini-1}), we have
\begin{center}
\begin{tabular}{cccc}
$\widetilde{\mu}_k(g)([\delta\varepsilon]_{g_\delta g_\varepsilon})  =  g_\delta g_\varepsilon$,&
$\widetilde{\mu}_k(g)([\beta\varepsilon]_{\rho^j})  =  \rho^j$ & and &
$\widetilde{\mu}_k(g)([\beta\varepsilon]_{\rho^{j+2}})  =  \rho^{j+2}$,
\end{tabular}
\end{center}
which respectively coincide with $\theta^{\zeta([\delta\varepsilon]_{g_\delta g_\varepsilon})}$, $g(\sigma,\zeta)_{[\beta\varepsilon]_{\rho^j}}$ and $g(\sigma,\zeta)_{[\beta\varepsilon]_{\rho^{j+2}}}$ since $(\sigma,\zeta)=\flip_k(\tau,\xi)$.

\item
The SP $(A(\sigma,\zeta),S(\sigma,\zeta))$ is the reduced part of $(\widetilde{\mu}_k(A(\tau,\xi)),S(\sigma,\zeta)^\sharp)$, where
\begin{eqnarray}\nonumber
\Ssigmad^\sharp  & = &
\alpha[\beta\gamma]+\eta[\delta\varepsilon]_{g_{\delta}g_{\varepsilon}}
+([\beta\varepsilon]_{\rho^j}+[\beta\varepsilon]_{\rho^{j+2}})\varepsilon^*\beta^*+\delta^*[\delta\gamma]\gamma^*+S(\tau,\sigma)
\in\RA{\widetilde{\mu}_k(A(\tau,\xi))},
\end{eqnarray}
with $S(\tau,\sigma)\in\RA{A(\tau,\xi)}\cap\RA{A(\sigma,\zeta)}$.

\item The potentials $\Stauc$ and $\widetilde{\mu}_k(\Stauc)$ are
Furthermore,
\begin{eqnarray*}
\Stauc & = &
\alpha\beta\gamma+\eta\delta\varepsilon
+S(\tau,\sigma) \ \ \ \ \ \text{and}\\
\widetilde{\mu}_k(\Stauc)
&\sim_{\operatorname{cyc}}&
\alpha[\beta\gamma]+\eta[\delta\varepsilon]_{g_{\delta}g_{\varepsilon}}
\\
&&
\gamma^*\beta^*[\beta\gamma]+\pi_{(g_{\delta}g_{\varepsilon})^{-1}}(\varepsilon^*\delta^*)[\delta\varepsilon]_{g_{\delta}g_{\varepsilon}}
+([\beta\varepsilon]_{\rho^j}+[\beta\varepsilon]_{\rho^{j+2}})\varepsilon^*\beta^*+\delta^*[\delta\gamma]\gamma^*
+S(\tau,\sigma).
\end{eqnarray*}

\item
From the equality $(g_{\delta}g_{\varepsilon})^{-1} = g_{\eta}$, we deduce that the rule
\begin{center}
\begin{tabular}{ccll}
$\varphi$ &:&
$\alpha\mapsto\alpha-\gamma^*\beta^*$,
&
$\eta\mapsto\eta-\pi_{(g_{\delta}g_{\varepsilon})^{-1}}(\varepsilon^*\delta^*)$,
\end{tabular}
\end{center}
produces a well-defined $R$-algebra automorphism $\varphi:\RA{\widetilde{\mu}_k(A(\tau,\xi))}\rightarrow\RA{\widetilde{\mu}_k(A(\tau,\xi))}$ (see \cite[Propositions 2.15-(6) and 3.7]{Geuenich-Labardini-1}). It turns out to be obvious that $\varphi$ is a right-equivalence $(\widetilde{\mu}_k(A(\tau,\xi)),\widetilde{\mu}_k(S(\tau,\xi)))\rightarrow(\widetilde{\mu}_k(A(\tau,\xi)),S(\sigma,\zeta)^\sharp)$.
\end{enumerate}
\end{case}



\begin{case}\label{case:38}\emph{Configurations 38 and 43}.
\begin{figure}[!ht]
                \centering
                \includegraphics[scale=.75]{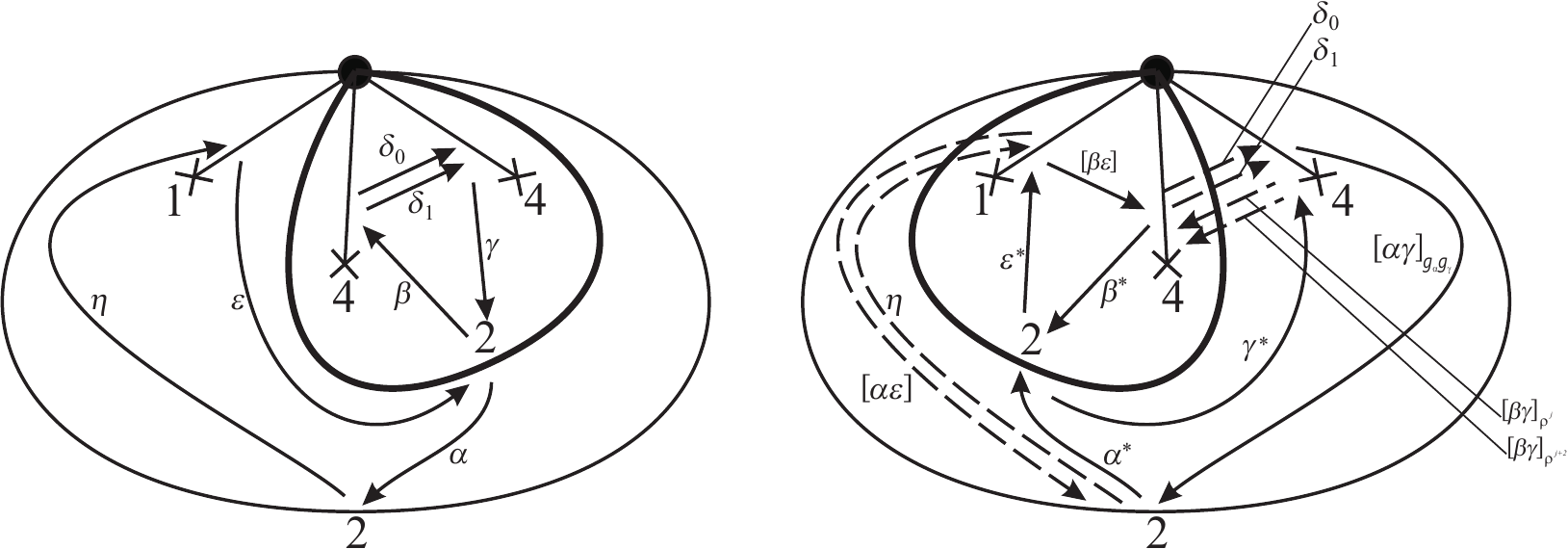}
                \caption{{\footnotesize Configurations 38 and 43 of Figure \ref{Fig:all_possibilities_for_k_and_weights}. Left: $\tau$ and $Q(\tau,\omega)$. Right: $\sigma$ and $\widetilde{\mu}_k(Q(\tau,\omega))$.  The numbers next to the arcs are the corresponding values of the tuple $\dtuple(\tau,\omega)$.}}
                \label{Fig:flip_mut_38}
        \end{figure}
\begin{enumerate}\item We use the notation in Figure \ref{Fig:flip_mut_38},
where $j$ is the unique element of $\{0,1\}$ that satisfies $\rho^j|_{L}=g_\beta g_{\gamma}=\rho^{j+2}|_{L}$. Furthermore, we identify the quiver $Q(\sigma,\omega)$ with the subquiver of $\widetilde{\mu}_k(Q(\tau,\omega))$ obtained by deleting the dotted arrows in the figure.

By the definition of the modulating function $\widetilde{\mu}_k(g):\widetilde{\mu}_k(Q(\tau,\xi))_1\rightarrow\bigcup_{i,j\in\tau} G_{i,j}$ (cf. \cite[Definition 3.19]{Geuenich-Labardini-1}), we have
\begin{center}
\begin{tabular}{cccc}
$\widetilde{\mu}_k(g)([\alpha\gamma]_{g_\alpha g_\gamma})  =  g_\alpha g_\gamma$,&
$\widetilde{\mu}_k(g)([\beta\gamma]_{\rho^j})  =  \rho^j$ & and &
$\widetilde{\mu}_k(g)([\beta\gamma]_{\rho^{j+2}})  =  \rho^{j+2}$,
\end{tabular}
\end{center}
which respectively coincide with $\theta^{\zeta([\alpha\gamma]_{g_\alpha g_\gamma})}$, $g(\sigma,\zeta)_{[\beta\gamma]_{\rho^j}}$ and $g(\sigma,\zeta)_{[\beta\gamma]_{\rho^{j+2}}}$ since $(\sigma,\zeta)=\flip_k(\tau,\xi)$.

Let $\ell$ be the unique element of $\{0,1\}$ whose congruence class modulo $2$ is $\xi(\delta_0)$. By the definition of the modulating function $g$ (Definition \ref{def:cocycle->modulating-function}), we have
\begin{center}
\begin{tabular}{cc}
    $g_{\delta_0}= \rho^\ell$,&
    $g_{\delta_1}=\rho^{\ell+2}$.
    \end{tabular}
    \end{center}
On the other hand, since $\xi$ is a 1-cocycle we have $\xi([\delta_0])=-\xi([\beta])-\xi([\gamma])$, hence $\rho^\ell|_{L}=\rho^{\ell+2}|_{L}=\theta^{\xi([\delta_0])}=\theta^{-\xi([\beta])-\xi([\gamma])}=(g_\beta g_\gamma)^{-1}=(\rho^j|_{L})^{-1}=(\rho^{j+2}|_{L})^{-1}$ (we have used the definition of $g_\beta$ and $g_\gamma$ in terms of $\xi$), and therefore, $\{\rho^\ell, \rho^{\ell+2}\}=\{(\rho^j)^{-1},(\rho^{j+2})^{-1}\}$. In particular, $\ell=j$ and
\begin{center}
\begin{tabular}{cc}
$g_{\delta_j}=\rho^{-j}$,&
$g_{\delta_{|j-1|}}= \rho^{-j-2}$.
\end{tabular}
\end{center}

\item
The SP $(A(\sigma,\zeta),S(\sigma,\zeta))$ is the reduced part of $(\widetilde{\mu}_k(A(\tau,\xi)),S(\sigma,\zeta)^\sharp)$, where
\begin{eqnarray}\nonumber
\Ssigmad^\sharp  & = &
\delta_j[\beta\gamma]_{\rho^j}+\delta_{|j-1|}[\beta\gamma]_{\rho^{j+2}}+\eta[\alpha\varepsilon]
+\beta^*[\beta\varepsilon]\varepsilon^*+\alpha^*[\alpha\gamma]_{g_{\alpha}g_{\gamma}}\gamma^*
+S(\tau,\sigma)
\in\RA{\widetilde{\mu}_k(A(\tau,\xi))},
\end{eqnarray}
with $S(\tau,\sigma)\in\RA{A(\tau,\xi)}\cap\RA{A(\sigma,\zeta)}$.

\item The potentials $\Stauc$ and $\widetilde{\mu}_k(\Stauc)$ are
\begin{eqnarray*}
\Stauc & = &
(\delta_0+\delta_1)\beta\gamma+\eta\alpha\varepsilon
+S(\tau,\sigma) \ \ \ \ \ \text{and}\\
\widetilde{\mu}_k(\Stauc)
&\sim_{\operatorname{cyc}}&
\delta_j[\beta\gamma]_{\rho^j}+\delta_{|j-1|}[\beta\gamma]_{\rho^{j+2}}+\eta[\alpha\varepsilon]
\\
&&
+\pi_{\rho^{-j}}(\gamma^*\beta^*)[\beta\gamma]_{\rho^j}+\pi_{\rho^{-j-2}}(\gamma^*\beta^*)[\beta\gamma]_{\rho^{j+2}}
+\varepsilon^*\alpha^*[\alpha\varepsilon]\\
&&
+\beta^*[\beta\varepsilon]\varepsilon^*+\alpha^*[\alpha\gamma]_{g_{\alpha}g_{\gamma}}\gamma^*
+S(\tau,\sigma).
\end{eqnarray*}

\item
From the equalities $g_{\delta_j}=\rho^{-j}$ and
$g_{\delta_{|j-1|}}= \rho^{-j-2}$, we deduce that the rule
\begin{center}
\begin{tabular}{cclll}
$\varphi$ &:&
$\delta_j\mapsto\delta_j-\pi_{\rho^{-j}}(\gamma^*\beta^*)$,
&
$\delta_{|j-1|}\mapsto\delta_{|j-1|}-\pi_{\rho^{-j-2}}(\gamma^*\beta^*)$,
&
$\eta\mapsto\eta-\varepsilon^*\alpha^*$,
\end{tabular}
\end{center}
produces a well-defined $R$-algebra automorphism $\varphi:\RA{\widetilde{\mu}_k(A(\tau,\xi))}\rightarrow\RA{\widetilde{\mu}_k(A(\tau,\xi))}$ (see \cite[Propositions 2.15-(6) and 3.7]{Geuenich-Labardini-1}). It turns out to be obvious that $\varphi$ is a right-equivalence $(\widetilde{\mu}_k(A(\tau,\xi)),\widetilde{\mu}_k(S(\tau,\xi)))\rightarrow(\widetilde{\mu}_k(A(\tau,\xi)),S(\sigma,\zeta)^\sharp)$.
\end{enumerate}
\end{case}



\begin{case}\label{case:39}\emph{Configurations 39 and 47}.
\begin{figure}[!ht]
                \centering
                \includegraphics[scale=.75]{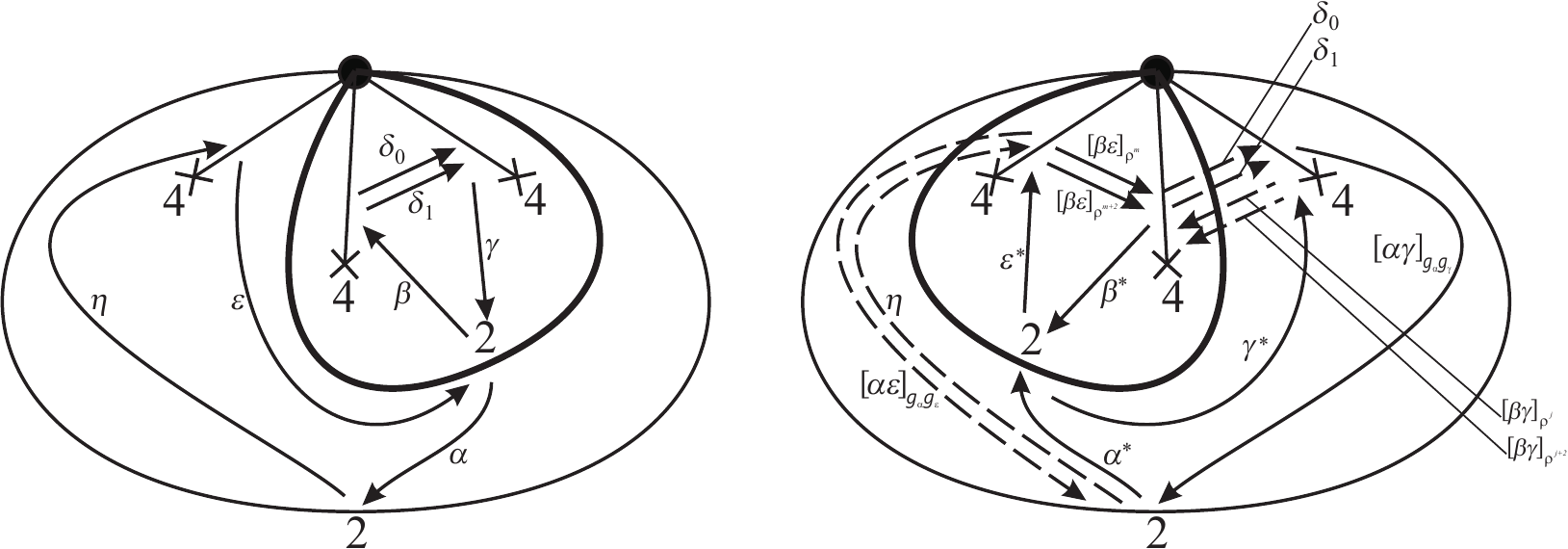}
                \caption{{\footnotesize Configurations 39 and 47 of Figure \ref{Fig:all_possibilities_for_k_and_weights}. Left: $\tau$ and $Q(\tau,\omega)$. Right: $\sigma$ and $\widetilde{\mu}_k(Q(\tau,\omega))$.  The numbers next to the arcs are the corresponding values of the tuple $\dtuple(\tau,\omega)$.}}
                \label{Fig:flip_mut_39}
        \end{figure}
\begin{enumerate}\item We use the notation in Figure \ref{Fig:flip_mut_39},
where $j$ (resp. $m$) is the unique element of $\{0,1\}$ that satisfies $\rho^j|_{L}=g_\beta g_{\gamma}=\rho^{j+2}|_{L}$ (resp. $\rho^m|_{L}=g_\beta g_{\varepsilon}=\rho^{m+2}|_{L}$). Furthermore, we identify the quiver $Q(\sigma,\omega)$ with the subquiver of $\widetilde{\mu}_k(Q(\tau,\omega))$ obtained by deleting the dotted arrows in the figure.

By the definition of the modulating function $\widetilde{\mu}_k(g):\widetilde{\mu}_k(Q(\tau,\xi))_1\rightarrow\bigcup_{i,j\in\tau} G_{i,j}$ (cf. \cite[Definition 3.19]{Geuenich-Labardini-1}), we have
\begin{center}
\begin{tabular}{ccc}
$\widetilde{\mu}_k(g)([\alpha\gamma]_{g_\alpha g_\gamma})  =  g_\alpha g_\gamma$,&
$\widetilde{\mu}_k(g)([\beta\gamma]_{\rho^j})  =  \rho^j$,&
$\widetilde{\mu}_k(g)([\beta\gamma]_{\rho^{j+2}})  =  \rho^{j+2}$,\\
$\widetilde{\mu}_k(g)([\beta\varepsilon]_{\rho^m})  =  \rho^m$ & and &
$\widetilde{\mu}_k(g)([\beta\varepsilon]_{\rho^{m+2}})  =  \rho^{m+2}$,
\end{tabular}
\end{center}
which respectively coincide with $\theta^{\zeta([\alpha\gamma]_{g_\alpha g_\gamma})}$, $g(\sigma,\zeta)_{[\beta\gamma]_{\rho^j}}$ and $g(\sigma,\zeta)_{[\beta\gamma]_{\rho^{j+2}}}$ since $(\sigma,\zeta)=\flip_k(\tau,\xi)$.

Let $\ell$ be the unique element of $\{0,1\}$ whose congruence class modulo $2$ is $\xi(\delta_0)$. By the definition of the modulating function $g$ (Definition \ref{def:cocycle->modulating-function}), we have
\begin{center}
\begin{tabular}{cc}
    $g_{\delta_0}= \rho^\ell$,&
    $g_{\delta_1}=\rho^{\ell+2}$.
    \end{tabular}
    \end{center}
On the other hand, since $\xi$ is a 1-cocycle we have $\xi([\delta_0])=-\xi([\beta])-\xi([\gamma])$, hence $\rho^\ell|_{L}=\rho^{\ell+2}|_{L}=\theta^{\xi([\delta_0])}=\theta^{-\xi([\beta])-\xi([\gamma])}=(g_\beta g_\gamma)^{-1}=(\rho^j|_{L})^{-1}=(\rho^{j+2}|_{L})^{-1}$ (we have used the definition of $g_\beta$ and $g_\gamma$ in terms of $\xi$), and therefore, $\{\rho^\ell, \rho^{\ell+2}\}=\{(\rho^j)^{-1},(\rho^{j+2})^{-1}\}$. In particular, $\ell=j$ and
\begin{center}
\begin{tabular}{cc}
$g_{\delta_j}=\rho^{-j}$, &
$g_{\delta_{|j-1|}}= \rho^{-j-2}$.
\end{tabular}
\end{center}

\item
The SP $(A(\sigma,\zeta),S(\sigma,\zeta))$ is the reduced part of $(\widetilde{\mu}_k(A(\tau,\xi)),S(\sigma,\zeta)^\sharp)$, where
\begin{eqnarray}\nonumber
\Ssigmad^\sharp  & = &
\delta_j[\beta\gamma]_{\rho^j}+\delta_{|j-1|}[\beta\gamma]_{\rho^{j+2}}+\eta[\alpha\varepsilon]_{g_\alpha g_{\varepsilon}}
+\beta^*[\beta\varepsilon]_{\rho^m}\varepsilon^*+\beta^*[\beta\varepsilon]_{\rho^{m+2}}\varepsilon^*
+\alpha^*[\alpha\gamma]_{g_{\alpha}g_{\gamma}}\gamma^*
+S(\tau,\sigma)
\in\RA{\widetilde{\mu}_k(A(\tau,\xi))},
\end{eqnarray}
with $S(\tau,\sigma)\in\RA{A(\tau,\xi)}\cap\RA{A(\sigma,\zeta)}$.

\item The potentials $\Stauc$ and $\widetilde{\mu}_k(\Stauc)$ are
\begin{eqnarray*}
\Stauc & = &
(\delta_0+\delta_1)\beta\gamma+\eta\alpha\varepsilon
+S(\tau,\sigma) \ \ \ \ \ \text{and}\\
\widetilde{\mu}_k(\Stauc)
&\sim_{\operatorname{cyc}}&
(\delta_0+\delta_1)([\beta\gamma]_{\rho^j}+[\beta\gamma]_{\rho^{j+2}})+\eta[\alpha\varepsilon]_{g_\alpha g_\varepsilon}
\\
&&
+\gamma^*\beta^*[\beta\gamma]_{\rho^j}+\gamma^*\beta^*[\beta\gamma]_{\rho^{j+2}}+\varepsilon^*\alpha^*[\alpha\varepsilon]_{g_\alpha g_\varepsilon}
+\beta^*[\beta\varepsilon]_{\rho^m}\varepsilon^*+\beta^*[\beta\varepsilon]_{\rho^{m+2}}\varepsilon^*
+\alpha^*[\alpha\gamma]_{g_{\alpha}g_{\gamma}}\gamma^*
+S(\tau,\sigma)\\
&\sim_{\operatorname{cyc}}&
\delta_j[\beta\gamma]_{\rho^j}+\delta_{|j-1|}[\beta\gamma]_{\rho^{j+2}}+\eta[\alpha\varepsilon]_{g_\alpha g_\varepsilon}
\\
&&
+\pi_{\rho^{-j}}(\gamma^*\beta^*)[\beta\gamma]_{\rho^j}+\pi_{\rho^{-j-2}}(\gamma^*\beta^*)[\beta\gamma]_{\rho^{j+2}}
+\pi_{(g_\alpha g_\varepsilon)^{-1}}(\varepsilon^*\alpha^*)[\alpha\varepsilon]_{g_\alpha g_\varepsilon}\\
&&
+\beta^*[\beta\varepsilon]_{\rho^m}\varepsilon^*+\beta^*[\beta\varepsilon]_{\rho^{m+2}}\varepsilon^*
+\alpha^*[\alpha\gamma]_{g_{\alpha}g_{\gamma}}\gamma^*
+S(\tau,\sigma).
\end{eqnarray*}

\item
From the equalities $g_{\delta_j}=\rho^{-j}$,
$g_{\delta_{|j-1|}}= \rho^{-j-2}$ and $(g_\alpha g_\varepsilon)^{-1}=g_\eta$, we deduce that the rule
\begin{center}
\begin{tabular}{cclll}
$\varphi$ &:&
$\delta_j\mapsto\delta_j-\pi_{\rho^{-j}}(\gamma^*\beta^*)$,
&
$\delta_{|j-1|}\mapsto\delta_{|j-1|}-\pi_{\rho^{-j-2}}(\gamma^*\beta^*)$,
&
$\eta\mapsto\eta-\pi_{(g_\alpha g_\varepsilon)^{-1}}(\varepsilon^*\alpha^*)$,
\end{tabular}
\end{center}
produces a well-defined $R$-algebra automorphism $\varphi:\RA{\widetilde{\mu}_k(A(\tau,\xi))}\rightarrow\RA{\widetilde{\mu}_k(A(\tau,\xi))}$ (see \cite[Propositions 2.15-(6) and 3.7]{Geuenich-Labardini-1}). It turns out to be obvious that $\varphi$ is a right-equivalence $(\widetilde{\mu}_k(A(\tau,\xi)),\widetilde{\mu}_k(S(\tau,\xi)))\rightarrow(\widetilde{\mu}_k(A(\tau,\xi)),S(\sigma,\zeta)^\sharp)$.
\end{enumerate}
\end{case}

We have thus shown, in all cases, that the SPs $(\widetilde{\mu}_k(A(\tau,\xi)),\widetilde{\mu}_k(S(\tau,\xi)))$ and $(\widetilde{\mu}_k(A(\tau,\xi)),S(\sigma,\zeta)^\sharp)$ are right-equivalent. Whence their reduced parts are right-equivalent as well by \cite[Theorem 3.16]{Geuenich-Labardini-1}. But these reduced parts are precisely $\mu_k(A(\tau,\xi),S(\tau,\xi))$ and $(A(\sigma,\zeta),S(\sigma,\zeta))$, respectively. Theorem \ref{thm:flip<->SP-mutation} is proved.
\end{proof}
}